\documentclass{article} 
\usepackage{mystyle,times}
\usepackage{hyperref}
\usepackage{url}
\usepackage{algpseudocode}
\title{Block Coordinate Descent Methods for Structured Nonconvex Optimization with Nonseparable Constraints: Optimality Conditions and Global Convergence}

\author{Zhijie Yuan\\
	Sun Yat-sen University, China\\
	\texttt{yuanzhj5@mail2.sysu.edu.cn}\\
	Ganzhao Yuan\\
	Peng Cheng Laboratory, China\\
	\texttt{yuangzh@pal.ac.cn}\\
	Lei Sun\thanks{Corresponding author.}\\
	Sun Yat-sen University, China\\
	\texttt{sunlei8@mail.sysu.edu.cn}
}

\iclrfinalcopy 
\usepackage{graphicx}
\usepackage{epstopdf}
\usepackage{enumitem,amsfonts,amssymb,mathtools,amsthm,bm,pifont,bbm,listings,xcolor,threeparttable,xcolor}
\usepackage{ulem,graphicx,subcaption,hyperref,comment,amsmath}
\usepackage[ruled,lined]{algorithm2e}\usepackage{algpseudocode}
\usepackage{multirow}
\usepackage{natbib}

\newtheorem{theorem}{Theorem}[section]

\newtheorem{lemma}[theorem]{Lemma}

\newtheorem{definition}[theorem]{Definition}

\def\d{\mathbf{d}}\def\p{\mathbf{p}}\def\r{\mathbf{r}}\def\w{\mathbf{w}}\def\x{\mathbf{x}}\def\y{\mathbf{y}}\def\z{\mathbf{z}}\def\A{\mathbf{A}}\def\B{\mathbf{B}}\def\E{\mathbf{E}}\def\H{\mathbf{H}}\def\I{\mathbf{I}}\def\N{\mathbf{N}}\def\P{\mathbf{P}}\def\Q{\mathbf{Q}}\def\R{\mathbf{R}}\def\S{\mathbf{S}}\def\U{\mathbf{U}}\def\Z{\mathbf{Z}}\def\p{\mathbf{p}}

\def\trans{^\mathsf{T}}    
 \def\trans {^\mathsf{T}}

\DeclareMathOperator{\argmin}{arg\,min}

\newcommand{\beq}{\begin{eqnarray}}
	\newcommand{\eeq}{\end{eqnarray}}
\newcommand{\beqq}{\begin{equation}}
	\newcommand{\eeqq}{\end{equation}}
\newcommand{\bel}{\begin{align}}
	\newcommand{\eel}{\end{align}}
\newcommand{\la}{\langle}
\newcommand{\ra}{\rangle}
\newcommand{\noi}{\noindent}
\newcommand{\nn}{\nonumber}
\def\noi{\noindent}
\def\nn{\nonumber}
\def\la{\langle}
\def\ra{\rangle}

\newcommand{\step}[1]{\text{\ding{\numexpr#1+171\relax}}}

\def\trans{^\mathsf{T}}

\renewcommand{\frac}[2]{\tfrac{#1}{#2}}

\usepackage{dsfont}

\usepackage{listings}\definecolor{backcolour}{rgb}{0.95,0.95,0.92}\definecolor{codegreen}{rgb}{0,0.6,0}\lstdefinestyle{myStyle}{backgroundcolor=\color{backcolour},commentstyle=\color{codegreen},basicstyle=\ttfamily\footnotesize,breakatwhitespace=false,breaklines=true,keepspaces=true,numbers=left,numbersep=5pt,showspaces=false,showstringspaces=false,showtabs=false,tabsize=2,frame = single,numbers=right,}\usepackage{caption} 

\renewcommand{\frac}[2]{\tfrac{#1}{#2}}

\usepackage{multirow}

\def\p{{\mathbf{p}}}

\def\E{{\mathbb{E}}}
\def\A{{\mathbf{A}}}
\def\x{{\mathbf{x}}}
\def\y{{\mathbf{y}}}
\def\z{{\mathbf{z}}}
\def\d{{\mathbf{d}}}

\def\r{{\mathbf{r}}}

\def\I{{\mathbf{I}}}

\def\H{{\mathbf{Q}}}
\def\w{{\mathbbm{w}}}

\def\argmin{\operatorname*{arg\,min}}
\def\argmax{\operatorname*{arg\,max}}

\begin{document}

	\maketitle

	\begin{abstract}
		Coordinate descent algorithms are widely used in machine learning and large-scale data analysis due to their strong optimality guarantees and impressive empirical performance in solving nonconvex problems. In this work, we introduce Block Coordinate Descent (BCD) method for structured nonconvex optimization with nonseparable constraints. Unlike traditional large-scale Coordinate Descent (CD) approaches, we do not assume the constraints are separable. Instead, we account for the possibility of nonlinear coupling among them. By leveraging the inherent problem structure, we propose new CD methods to tackle this specific challenge. Under the relatively mild condition of locally bounded non-convexity, we demonstrate that achieving coordinate-wise stationary points offer a stronger optimality criterion compared to standard critical points. Furthermore, under the Luo-Tseng error bound conditions, our BCD methods exhibit Q-linear convergence to coordinate-wise stationary points or critical points. To demonstrate the practical utility of our methods, we apply them to various machine learning and signal processing models. We also provide the geometry analysis for the models. Experiments on real-world data consistently demonstrate the superior objective values of our approaches compared to existing methods.


	\end{abstract}

	\section{Introduction}
	
	This paper mainly focuses on a class of structured nonconvex minimization problems and establishes following nonconvex optimization framework ('$\triangleq$' means define):
	\beq
	\min_{\x \in \mathbb{R}^n} F(\x) \triangleq f(\x) + h(\x) + g(\x),\ s.t.\ u(\x) = 0. \label{problem original}
	\eeq
	\noi Throughout this paper, we make the following assumptions on Problem (\ref{problem original}). \textbf{(i)} $F(\x)$ only takes finite values for any feasible solution. \textbf{(ii)} For any $\x$ and $\x^+$ in feasible set $\Omega \triangleq \{\x | u(\x) =  0\}$, we assume $f:\R^n \rightarrow \R$ is convex and continuously differentiable with a Lipschitz constant matrix $\Q \in \R^{n \times n}$ that \citep{nesterov2012efficiency, beck2013convergence,patrascu2015efficient}: $f(\x^+) \leq f(\x) + \la \x^+ - \x, \nabla f(\x) \ra  + \tfrac{1}{2} \|\x^+ - \x \|_{\Q}^2$, where its gradient is coordinate-wise Lipschitz continuous with $\forall i,j,\Q_{ij} \triangleq L_{ij}$ and $\|\x\|_{\Q}^2 \triangleq \x^{\top}\Q\x$. $L_{ij}$ represents the Lipschitz constant corresponding to any indices $i$ and $j$. \textbf{(iii)} $h(\cdot)$ is convex and coordinate-wise separable with $h(\x) = \sum_{i=1}^nh_i(\x_i)$. \textbf{(iv)} $g(\x)$ is a concave function. \textbf{(v)} $u(\x) = 0$ is a nonseparable constraint.

	\noi $\bullet$ \textbf{BCD Method Literature.} Block coordinate descent is a powerful optimization technique in machine learning that has gained attention for its ability to efficiently solve large-scale problems. Due to its simplicity and efficiency, BCD has regained popularity and has been widely used for many years in structured high-dimensional machine learning and data mining applications, including Support Vector Machines (SVM) \citep{hsieh2008dual}, Principal Component Analysis (PCA)\citep{jenatton2010structured}, and sparse multiclass classification \citep{blondel2013block}. Furthermore, traditional methods for LASSO (Least Absolute Shrinkage and Selection Operator) often struggle with high-dimensional data, whereas BCD offers superior performance over these approaches \citep{tseng2009coordinate, escribe2021block, blondel2013block}.

	\noi $\bullet$ \textbf{Nonconvex BCD Method.} Nonconvex optimization problems play a crucial role in the development of machine learning when training the learning models. The nonconvexity enables the capture of complex learning and prediction problems more accurately than convex optimization problems \citep{jain2017non}. However, these problems are notoriously difficult to solve due to their NP-hard nature. Recently, the popularity of BCD has grown due to its strong optimality guarantees and superior empirical performance in addressing nonconvex problems \citep{pmlr-v235-yuan24a, yuan2019decomposition, nie2021coordinate, beck20142, liu2014asynchronous, hong2017iteration}. BCD improves the efficiency and accuracy of optimization algorithms by effectively leveraging the problem structure to solve optimization tasks. Sparse optimization is a crucial problem in both machine learning and computer vision. \citep{yuan2020block, patrascu2015random, liu2014blockwise} analyze the random block coordinate descent and cyclic order coordinate decent method for minimizing the sparse optimization problem, providing their convergence results respectively. \citep{massart2022coordinate, asadi2021block} propose novel block coordinate descent methods for solving the nonsmooth problem involving orthogonality constraints with improved accuracy and reducing computational complexity. 
	
	
	\noi $\bullet$ \textbf{Existing Challenges.} We focus on the challenges posed by the non-convexity and non-smoothness of $g(\x)$, as well as the nonseparable constraint $u(\x) = 0$ in Problem (\ref{problem original}), which is recognized as NP-hard. Existing approaches are limited to ensuring convergence to a critical point, without providing guarantees of stronger optimality. Additionally, the current literature lacks a comprehensive convergence analysis on the BCD method and offers no geometric insights into Problem (\ref{problem original}).
	
	\noi $\bullet$ \textbf{Our Contributions.} This paper makes the following contributions: \textbf{(i)} We propose the BCD algorithms for nonconvex composite optimization with nonseparable constraints in Problem (\ref{problem original}). Our goal is to provide a superior solution compared to existing methods, addressing the challenges of nonconvexity and nonseparable constraints. \textbf{(ii)} Theoretically, we establish the optimality of our approach, demonstrating that a coordinate-wise stationary point is also a critical point. Additionally, this is the first work to study the convergence rate of Problem (\ref{problem original}). We offer the convergence analysis of our methods, revealing their Q-linear convergence rate while requiring only the continuity of $g(\x)$ rather than smoothness. \textbf{(iii)} We introduce breakpoint searching strategies to solve four nonconvex optimization problems, resulting in more precise solutions. \textbf{(iv)} We propose two novel semi-greedy index strategies to accelerate the BCD methods and enhance their computational efficiency. \textbf{(v)} We provide the geometry analysis for three applications in Section \ref{sec:app}. \textbf{(vi)} Empirically, we conduct extensive experiments to demonstrate that our method outperforms existing full-gradient algorithms.

	\noi $\bullet$ \textbf{Organization.} The remainder of this article is organized as follows: Section \ref{sec:app} introduces the applications relevant for our framework, as outlined in Problem \ref{problem original}. Section \ref{sec:related} reviews an overview of existing methods designed to solve the framework \ref{problem original}. Section \ref{sec:propose} proposes a novel block coordinate descent method. Section \ref{sec:opt} and Section \ref{sec:conv} provide a detailed analysis of the optimality and convergence properties of the BCD algorithm. Section \ref{sec:block2} leverages the structure of the optimizaiton problems and propose four solid algorithms when $k=2$. Section \ref{sec:greedy} introduces two innovative semi-greedy index selection strategies to accelerate the BCD methods. Section \ref{sec:geo} provide the geometry analysis for three applications. Section \ref{sec:exp} presents experimental results, comparing the performance of our BCD methods with other methods. Finally, Section \ref{sec:conclusion} offers a conclusion.

	\noi $\bullet$ \textbf{Notations.} Vectors are denoted by boldface lowercase letters, and matrices by boldface uppercase letters. The Euclidean inner product between $\x$ and $\y$ is denoted by $\la \x,\y\ra$ or $\x^{\top}\y$. $\x_i$ is the $i$-th element of the vector $\x$. We denote $\|\x\|_{\Q}^2 \triangleq \sum_i \sum_j \x_i\Q_{ij}\x_j$. $\E[\cdot]$ represents the expectation of a random variable. $\odot$ denote the element-wise multiplication between two vectors. The expressions $\mathbf{0}_n$ and $\mathbf{1}_n$ represent the zero vector and the all-ones vector with $n$ elements. $e_i$ refers to the column vector with 1 at the $i$-th position and 0 elsewhere. $\I_k \in \R^{k \times k}$ is the $k \times k$ identity matrix.

	\section{Motivating Applications}\label{sec:app}
	Problem (\ref{problem original}) serves as a fundamental optimization framework across various fields of  machine learning, signal processing, and computational biology.  We now present four specific instances of this framework.

	\subsection{Sparse Index Tracking (SIT)}
	Sparse index tracking task aims to address the challenges of asset selection and capital allocation simultaneously \citep{benidis2017sparse}. Given a return matrix $\A \in \mathbb{R}^{m \times n}$ for $n$ stocks and the index return $\y \in \mathbb{R}^m$ during $m$ days, $\x \in \mathbb{R}^n$ is a decision vector, each component indicating the proportion of the investment budget. The task tries to find the optimal investment proportions $\x$ that minimize the difference between the portfolio return $\A\x$ and the index return $\y$. An important contribution to this field is the introduction of $\ell_0$ norm constraint-based sparse index optimization, which enables direct control over sparsity in portfolio construction: $\min_{\x}\,\tfrac{1}{2}\|\A\x-\y\|_2^2,\ s.t.\ \|\x\|_0 \leq s, \x\geq 0, \|\x\|_1 = 1$. The $\ell_0$ norm constraint $\|\x\|_0$ can be represented as a Difference of Convex (DC) function: $ \|\x\|_1 - \|\x\|_{[s]}$. While both constraints describe the same feasible set, it is important to note that $\|\x\|_1$ and $\|\x\|_{[s]}$ are continuous functions with efficiently computable subgradients \citep{gotoh2018dc}. The sparse index tracking problem above can be reformulated as follows: $\min_{\x}\,\tfrac{1}{2}\|\A\x-\y\|_2^2 + \lambda ( \|\x\|_1 - \|\x\|_{[s]}),\ s.t.\ \x\geq 0,\,\|\x\|_1 = 1$, where $\|\x\|_{[s]}$ returns the sum of top $s$ large elements in $\x$ and $\lambda > 0$ is a constant. Since $\|\x\|_1 = 1$ and $\x \geq 0$, we seek to address the subsequent problem:
	\beq
	\min_{\x}\,\tfrac{1}{2}\|\A\x-\y\|_2^2 + \mathcal{I}_{\geq 0}(\x) - \lambda\|\x\|_{[s]},\ s.t.\ \|\x\|_1=1, \label{plb:app:topk}
	\eeq
	where $\mathcal{I}_{\geq 0}(\x)$ is the indicator function on the non-negative constraint $\x \geq 0$. Under this setting, Problem (\ref{plb:app:topk}) coincides with framework (\ref{problem original}) with $f(\x) = \tfrac{1}{2}\|\A\x-\y\|_2^2$, $h(\x) = \mathcal{I}_{\geq 0}(\x)$, $g(\x) = - \lambda \|\x\|_{[s]}$ and $u(\x) = \|\x\|_1 - 1$.
	
	\subsection{Non-negative Sparse PCA (NNSPCA)}
	Non-negative sparse PCA (NNSPCA) extends the conventional PCA by imposing constraints of non-negativity and sparsity on the principal components. NNSPCA ensures that the principal components are sparse and do not include many non-zero loadings. This makes the principal components easier to interpret and understand \citep{asteris2014nonnegative, drikvandi2023sparse}. Given the matrix $\A \in \mathbb{R}^{m \times n}$, the naive NNSPCA model can be presented as: $\min_{\x \in \mathbb{R}^n}\,-\tfrac{1}{2}\|\A\x\|_2^2,\ s.t.\ \|\x\|_0 \leq s,\x \geq 0,\|\x\|^2_2=1$. Using the equivalent variational reformulation of $\|\x\|_0 \leq s \Leftrightarrow \|\x\|_1 = \|\x\|_{[s]}$, the original NNSPCA problem can be represented as: $\min_{\x \in \mathbb{R}^n}\,-\tfrac{1}{2}\|\A\x\|_2^2 + \lambda ( \|\x\|_1 - \|\x\|_{[s]}),\ s.t.\ \x \geq 0,\|\x\|^2_2=1$, where $\lambda > 0$ is a constant. 
	Since $\x \geq 0$ and $\|\x\|_2^2 = 1$, we focus on following NNSPCA problem:
	\beq
	\min_{\x \in \mathbb{R}^n}\,\tfrac{1}{2}\x^{\top}\hat{\Q}\x + \mathcal{I}_{\geq 0}(\x) + \lambda( \x^{\top}\mathbf{1}_n  - \|\x\|_{[s]}),\ s.t.\ \|\x\|^2_2=1, \label{plb:app:nnspca2}
	\eeq
	where $\hat{\Q} = -\A^{\top}\A + \gamma\I_n \succeq 0$ with $\gamma$ being sufficiently large. Therefore, Problem (\ref{plb:app:nnspca2}) fits to framework (\ref{problem original}) with $f(\x) = \tfrac{1}{2}\x^{\top}\hat{\Q}\x$, $h(\x) = \mathcal{I}_{\geq 0}(\x) + \lambda \x^{\top}\mathbf{1}_n$, $g(\x) = - \lambda \|\x\|_{[s]}$ and $u(\x) = \|\x\|_2^2 - 1$.

	\subsection{DC Penalized Binary (DCPB) Optimization }
	Given a matrix $\A \in \mathbb{R}^{m \times n}$ and a vector $\y \in \mathbb{R}^{m}$, the binary optimization problem can be formulated as follows when $\x$ has a binary structure: $\min_{\x \in \mathbb{R}^{n}} \tfrac{1}{2} \|\A\x - \y\|_2^2,\ s.t.\ \x \in \{-1,+1\}^n,\ \x^{\top}\textbf{1} = c$. Using the equivalent variational reformulation of the binary constraint, $\x \in \{-1,+1\}^n$ can be relaxed to $\{\x|-1\leq \x \leq 1,\|\x\|_2^2 = n\}$. This relaxation simplifies the problem by allowing it to be solved over a continuous domain rather than a discrete one. Binary optimization problems appear in various fields, including K-means, spectral clustering, and fuzzy clustering \citep{nie2021coordinate, yan2024binary, ruspini2019fuzzy}. The following approximate binary optimization problem: 
	\beq
	&&\min_{\x \in \mathbb{R}^{n}} \tfrac{1}{2} \|\A\x - \y\|_2^2 + \lambda(n - \|\x\|^2_2),\ s.t.\ -1\leq \x \leq 1,\ \x^{\top}\textbf{1} = c\nn\\
	\Rightarrow && \min_{\x \in \mathbb{R}^{n}}  \p^{\top}\x + \mathcal{I}_{[-1,1]}(\x) - \tfrac{1}{2}\x^{\top}\hat{\Q}\x,\ s.t.\ \x^{\top}\textbf{1} = c,\label{plb:app:bin2}
	\eeq
	where $\p = -\A^{\top}\y$, $\lambda > 0$ is a given constant, and $\hat{\Q} = 2\lambda \I_n - \A^{\top}\A \succeq 0$ with $\lambda$ being sufficiently large. With $f(\x) = \p^{\top}\x$, $h(\x) = \mathcal{I}_{[-1,1]}(\x)$, $g(\x) = - \tfrac{1}{2}\x^{\top}\hat{\Q}\x$ and $u(\x) = \x^{\top}\textbf{1} - c$, Problem (\ref{plb:app:bin2}) is also falls into framework (\ref{problem original}). 
	
	Using an other variational reformulation of the binary constraint $\{-1, +1\}^n \Leftrightarrow \{\x \mid -1 \leq \x \leq 1, |\x|_2 = \sqrt{n}\}$, the naive binary optimization problem can be also transformed into the following approximate binary optimization problem:
	\beq
	&& \min_{\x \in \mathbb{R}^{n}} \tfrac{1}{2} \|\A\x - \y\|_2^2 + \lambda(\sqrt{n} - \|\x\|_2),\,s.t \,-1\leq \x \leq 1,\ \x^{\top}\textbf{1} = c\nn\\
	\Rightarrow && \min_{\x \in \mathbb{R}^{n}}  \tfrac{1}{2} \|\A\x - \y\|_2^2 +  \mathcal{I}_{[-1,1]}(\x) -  \lambda \|\x\|_2,\ s.t.\ \x^{\top}\textbf{1} = c,\label{plb:app:bin3}
	\eeq
	where $\lambda > 0$. Problem (\ref{plb:app:bin3}) coincides with framework (\ref{problem original}) with  $f(\x) = \tfrac{1}{2} \|\A\x - \y\|_2^2$, $h(\x) = \mathcal{I}_{[-1,1]}(\x)$, $g(\x) =  - \lambda \|\x\|_2$ and $u(\x) = \x^{\top}\textbf{1} - c$.

	\section{Related Work}\label{sec:related}
	Nonconvex optimization has emerged as a critical tool of study in various fields. It proposes the challenge due to its NP-hard nature. One common strategy is convex relaxation, which approximates the original problem and provides a bound on the optimal value to assess the suboptimality of candidate solutions \citep{eltved2021convex}. Direct approaches is another strategy to solve the nonconvex optimization, often yielding better performance than relaxation-based techniques. These direct methods are particularly important in machine learning and communication systems, including methods such as projected gradient descent, proximal gradient method, and coordinate descent \citep{hauswirth2016projected, jain2014iterative, yuan2020block}.  In this work, we introduce several related algorithms to solve Problem (\ref{problem original}).

	\subsection{Projected Subgradient Method (PSG)}
	PSG is one of the extension on subgradient method. It first lets $\x^{t}$ take a step opposite the subgradient direction then provides a solution  $\x^{t+1} $ in feasible set $\Omega$ with a projection function \citep{alber1998projected, mainge2008strong, beck2003mirror}. It generates a sequence $\{\x^t\}$ as: 
	$
	\x^{t+1} = P_{\Omega}(\x^{t} - \alpha_{t}\textbf{g}^{t}),
	$
	while $\textbf{g}^{t}$ is the subgradient of $F(\cdot)$ at $\x^t$ and $\alpha_{t}$ is the step size at $t$ iteration.
	
	\subsection{ Multi-Stage Convex Relaxation (MSCR)} 
	MSCR derives a general multi-stage convex relaxation for solving learning formulations with nonconvex regularization \citep{zhang2010analysis, zhang2013multi}. With $K$ representing the number of outer iterations, MSCR minimizes an upper bound the original function $K$ times. As a result, the computational cost of MSCR could be expensive for large-scale problems. $\x^{t+1}$ is updated by $K$ times as follows: $\x^{t+1} =\text{arg}\min_{\x} f(\x) + h(\x) + \la\x-\x^{t},\partial g(\x^t)\ra$.

	\subsection{Proximal DC Algorithm (PDCA)}  
	PDCA is a powerful tool for addressing sparse optimization problems \citep{gotoh2018dc,gong2013general}. To avoid the computational issue of MSCR, PDCA exploits the structure of $f(\cdot)$ and the sequence of $\{\x^t\}$ is generated as follows: $\x^{t+1} =\text{arg}\min_{\x} \mathcal{S}(\x,\x^{t}) + h(\x) + \la\x-\x^{t},\partial g(\x^t)\ra$, where $\mathcal{S}(\x,\x^{t}) \triangleq f(\x^{t})+\la\nabla f(\x),\x-\x^{t}\ra+\tfrac{L}{2}\|\x-\x^{t}\|^2_2$ is a Lipschitz gradient surrogate and $L$ is the Lipschitz constant of $\nabla f(\cdot)$.

	\section{The Proposed Block Coordinate Descent Methods}\label{sec:propose}
	In this section, we introduce two novel BCD methods specifically designed to address structured nonconvex optimization problems under nonseparable constraints $u(\x) = 0$. We first propose the BCD-g method for solving the subproblem (\ref{plb:subproblem}) globally. The second BCD method offers a local solution for the subproblem (\ref{plb:subproblem}) denoted as BCD-l-$k$, where $k$ represents the number of the block. 
	
	The BCD methods concentrate on only updating a single block exclusively in each iteration with the remaining blocks being fixed. Therefore, the entire problem is decomposed into numerous subproblems which is converted into a lower-dimensional minimization problem. To ensure feasibility with the constraint $u(\x) = 0$, it is necessary to update at least two variables in each iteration. In each iteration, the variables in $\x$ are divided into two sets: $\B \in \mathbb{N}^k$ and $\B^c \in \mathbb{N}^{n-k}$, with $\B$ designated as the working set. The objective function $F(\x)$ is minimized with respect to the working set $\B$, while the variables in $\B^c$ remain unchanged. We first define $\U_\B \in \mathbb{N}^{n \times k}$ and $\U_{\B^c} \in \mathbb{N}^{n \times (n-k)}$ as:
	$
	[\U_\B]_{ji} = \left\{
	\begin{aligned}
		&1, && \B_i = j; \\
		&0, && \text{else}.
	\end{aligned}
	\right.
	,\ 
	[\U_{\B^c}]_{ji} = \left\{
	\begin{aligned}
		&1, && \B^c_i = j; \\
		&0, && \text{else}.
	\end{aligned}
	\right.
	$. 
	Hence, for any $\x \in \mathbb{R}^n$ and $\Q \in \mathbb{R}^{n\times n}$, we have: $\x_\B = \U_\B^{\top}\x$, $\x = (\U_{\B^c}\U_{\B^c}^{\top} + \U_{\B}\U_{\B}^{\top})\x = \U_\B\x_\B + \U_{\B^c}\x_{\B^c}$ and $\Q_{\B\B} = \U_\B^{\top}\Q\U_\B$. Then we define $\r^t \triangleq \U_{\B^t}\d_{\B^t}$. Given the working set $\B^t$, Problem (\ref{problem original}) is equivalent to perform the following step size searching problem: $\bar{\d}_{\B^t} \in \arg\min_{\d_{\B^t}} \, f(\x^t + \r^t) + h(\x^t + \r^t) + g(\x^t + \r^t),\ s.t.\ \x^t + \r^t \in \Omega$. This problem aims at searching the optimal step size $\bar{\d}_{\B^t}$ in $t$ iteration along the coordinates in working set $\B^t$. Finally $\x^{t+1}$ is updated by $\x^{t} + \U_{\B^t}\bar{\d}_{\B^t}$.

	We now discuss how to determine the optimal step size $\bar{\d}_{\B^t}$. With the benefits of Lipschitz continuity of its gradient $\nabla f(\x)$: $\|\nabla f(\x) - \nabla f(\y)\| \leq L \|\x-\y\|$, we have: $f(\x + \U_\B\d_\B) \leq \mathcal{S}_{\B}(\x,\d_\B)$, where $\mathcal{S}_{\B}(\x,\d_\B) \triangleq f(\x) + \la\U_\B\d_\B, \nabla f(\x) \ra  + \tfrac{1}{2} \|\d_\B\|_{\Q_{\B\B}}^2$. With the proximal parameter $\theta \geq 0$, Problem (\ref{problem original}) can be addressed using the BCD-g method and reformulated as follows:
	\beq
	\min_{\d_{\B^t}} \,  \mathcal{M}_{\B^t}(\x,\d_{\B^t}) \triangleq \mathcal{S}_{\B^t}(\x^t,\d_{\B^t}) + h(\x^t + \r^t) + g(\x^t + \r^t) + \tfrac{\theta}{2} \|\d_{\B^t}\|_2^2,\ s.t.\ \x^t + \r^t \in \Omega.\label{plb:subproblem}
	\eeq
	
	\noi Now we define  $g(\x + \U_\B\d_\B) \leq \mathcal{G}_{\B}(\x,\d_{\B}) \triangleq g(\x) + \la\U_\B\d_\B, \partial g(\x) \ra$. The function $\mathcal{M}_{\B}(\x,\d_{\B})$ can be relaxed as a convex function $\mathcal{N}_{\B}(\x,\d_{\B})$. Then we can solve the subproblem (\ref{plb:subproblem}) locally using the BCD-l-$k$ method by solving following problem:
	
	\beq
	\min_{\d_{\B^t}} \,  \mathcal{N}_{\B^t}(\x,\d_{\B^t}) \triangleq \mathcal{S}_{\B^t}(\x^t,\d_{\B^t}) + h(\x^t + \r^t) + \mathcal{G}_{\B^t}(\x^t,\d_{\B^t}) + \tfrac{\theta}{2} \|\d_{\B^t}\|_2^2,\ s.t.\ \x^t + \r^t \in \Omega.\label{plb:subproblemlocal}
	\eeq
	
	\noi$\bullet$ \textbf{Breakpoint searching method.} The breakpoint searching method systematically explores the solution space to identify all breakpoints by exhaustively searching over a predefined range of candidate points. The subproblems (\ref{plb:subproblem}) and (\ref{plb:subproblemlocal}) involve $k$ unknown decision variables ($\d_{\B_1},\d_{\B_2},...,\d_{\B_k}$) that need to be solved. Generally, we divide each variable into two possible states in (\ref{plb:subproblem}) and (\ref{plb:subproblemlocal}). In other words, for a given state of $\d_{\B}$, the subproblem reduces to solving $2^k$ simple system. We then systematically traverse the entire binary tree, enumerating all potential candidate solutions, and select the one that minimizes the objective function the most as optimal solution. By thoroughly examining a wide range of candidate points, the algorithm can effectively capture the complex and nonconvex nature of the optimization problem. Although the exhaustive breakpoint searching method can be computationally demanding, it ensures a thorough exploration of the solution space and leads to robust and reliable results. Therefore, combining the breakpoint searching strategy, BCD-g and BCD-l-$k$ are enable to provide better solutions in subproblems (\ref{plb:subproblem}) and (\ref{plb:subproblemlocal}) even when $u(\x)$ is nonseparable.

	%


	\begin{algorithm} [!h]
		\DontPrintSemicolon
		\caption{Block Coordinate Descent Methods for Structured Nonconvex Optimization with Nonseparable Constraints}\label{alg1}
		Give an initial feasible solution $\x^0$, $\theta \geq 0$. Set $t = 0$. \\
		\While {not converge}{
			\textbf{(S1)} Use some strategies to find a working set $\B^t$ at $t$-iteration.\\
			\textbf{(S2)} Solve the nonconvex subproblem (\ref{plb:subproblem}) with the breakpoint searching strategy by BCD-g or BCD-l-$k$ as following:
					\beq
					\text{BCD-g}: & \bar{\d}_{\B^t} \in \bar{\mathcal{M}}_{\B^t}(\x^t) \triangleq \arg\min_{\d_{\B^t}} \mathcal{M}_{\B^t}(\x^t,\d_{\B^t}),\ s.t.\ \x^t + \r^t \in \Omega. \\ \label{plb:subproblemalg} 		
					\text{BCD-l-$k$}: & \bar{\d}_{\B^t} \in \bar{\mathcal{N}}_{\B^t}(\x^t) \triangleq \arg\min_{\d_{\B^t}} \mathcal{N}_{\B^t}(\x^t,\d_{\B^t}),\ s.t.\ \x^t + \r^t \in \Omega .\label{plb:subproblemalglocal} 	
					\eeq\\
			\textbf{(S3)} $\x^{t+1} = \x^t + \U_{\B^t}\bar{\d}_{\B^t} $  ($ \Leftrightarrow \x^{t+1}_{\B^t} = \x^t_{\B^t} + \bar{\d}_{\B^t}$). \\
			\textbf{(S4)} Increment $t$ by 1.
		}
	\end{algorithm}
	
	\noi$\bullet$ \textbf{Working set selection strategy.} There are mainly three strategies to determine which coordinate to update during coordinate descent in literature. \textbf{(i) Cyclic order strategy.}  It runs all coordinates in cyclic order $\{1,2,...,k\} \rightarrow \{n-k+1,...,n\} \rightarrow \{n-k+2,...,1\}$. This approach ensures that each coordinate is updated during the proposed method. \textbf{(ii) Random sampling strategy.} In this strategy, a set of $k$ indices are randomly and uniformly selected to form a working set $\B$ for updating. This method introduces an element of randomness into the selection process, which can help in avoiding local minima. \textbf{(iii) Greedy strategy.}  This approach involves selecting the coordinates that lead to the least objective value in the current iteration. The greedy strategy aims to make the most significant progress towards the optimal solution in each iteration. Two semi-greedy strategies for working set selection are discussed in Section \ref{sec:greedy}.


	\def\B{\mathrm{B}}
	\def\N{\mathrm{B^c}}
	\def\P{\mathbb{P}}
	
	\def\B{\texttt{\textup{B}}}
	\def\Bc{\B^c}
	\def\UB{\mathrm{U}_{\B}}
	
	\def\UBc{\mathrm{U}_{\Bc}}
	\def\UBt{\mathrm{U}_{\B^t}}
	\def\UBtt{\mathrm{U}_{{\B}^{t+1}}}
	
	\def\E{\mathbb{E}}
	
	\def\Vup{\overline{\textup{V}}}
	\def\Vlow{\underline{\textup{V}}}
	\def\Aup{\overline{\textup{A}}}
	\def\Alow{\underline{\textup{A}}}
	\def\Sup{\overline{\textup{H}}}
	\def\Slow{\underline{\textup{H}}}
	\def\E{\mathbb{E}}
	\def\Hup{\overline{\textup{H}}}
	\def\Hlow{\underline{\textup{H}}}
	\def\Qup{\overline{\textup{Q}}}
	\def \R{\mathbb{R}}

	\section{Optimality Analysis}\label{sec:opt}
	In this section, we focus on two key optimality conditions: coordinate-wise stationary optimality condition and critical point condition.  We discuss the relationship between these two conditions and also provide the theoretical analysis. It is a fundamental aspect of the optimality hierarchy for solving Problem (\ref{problem original}). Optimal point, coordinate-wise point and critical point are denoted as $\bar{\x}$, $\ddot{\x}$ and $\breve{\x}$. We primarily focus on solving Problem (\ref{problem original}), which is formulated as: $\min_{\x \in \Omega}\ F(\x) = f(\x) + h(\x) + g(\x)$.

	Critical point condition is the most general form of optimality, where a solution is considered optimal if the gradient of the function is zero within the feasible set. This criterion frequently serves as a standard for evaluating the effectiveness of optimization algorithms. We define $\mathcal{I}_{\Omega}(\cdot)$ is the indicator function of the feasible set $\Omega$. The critical point is also defined as follows \citep{Toland1979}. 
	\begin{definition}
		\textbf{(Critical Point)}
		\noi  $\breve{\x}$ is a critical point if the following holds: $	\mathbf{0}_n \in  \nabla f (\breve{\x}) + \partial h(\breve{\x}) + \partial g(\breve{\x}) + \partial \mathcal{I}_{\Omega}(\breve{\x})$.
	\end{definition}

	Coordinate-wise stationary point condition is described as a solution where the minimum value of a specific function $\mathcal{M}_{\B}(\ddot{\x},\d_\B)$ is zero for any working set $\B$. In other word, we can not improve the objective function value for $\mathcal{M}_{\B}(\ddot{\x},\d_\B)$ for any working set $\B \in \mathbb{N}^k$. In \citep{yuan2023coordinate}, the coordinate-wise stationary point is defined for the coordinate descent method when $k=1$. For the block-$k$ ($k>1$) coordinate method, the definition of a coordinate-wise stationary point is as follows.
	\begin{definition}\label{defition:cws}
		\textbf{(Coordinate-Wise Stationary Point, \textbf{CWS}-point)}
		\noi  For any working set $\B \in \mathbb{N}^k$, a solution $\ddot{\x}$ is called a coordinate-wise stationary point if the following condition is met: $\mathcal{M}_{\B}(\ddot{\x},\mathbf{0}_k)=\underset{\d_\B}{\min}\ \mathcal{M}_{\B}(\ddot{\x},\d_\B)$.
	\end{definition}

	\noi When $z(\x) + \tfrac{\rho}{2}\|\x - \y\|_2^2$ is convex, we refer to $z(\x)$ as having $\rho$-bounded nonconvexity. We define the locally $\rho$-bounded nonconvexity and give following lemma shows that $z(\ddot{\x}) = -\|\ddot{\x}\|_{[s]}$ is locally $\rho$-bounded nonconvex if $\x$ is the \textbf{CWS}-point $\ddot{\x}$.
	
	\begin{definition}\label{defition:rho}
		\noi  \textbf{(locally $\rho$-bounded nonconvexity)} A function $z(\x)$ is called to be locally $\rho$-bounded nonconvex if: $\forall \x,\ \y,\ z(\x) \leq z(\y) + \la \x - \y ,\ \partial z(\x)\ra + \tfrac{\rho}{2}\|\x - \y\|_2^2 $ with $\rho \leq +\infty$.
	\end{definition}

	\begin{lemma} \label{lemma:rhobounded}
		\textbf{(Proof in Appendix \ref{app:lemma:rhobounded})}
		\noi For any \textbf{CWS}-point $\ddot{\x}$, the function $\\z(\ddot{\x}) = -\|\ddot{\x}\|_{[s]} \triangleq - \sum_{j=1}^{k}|\ddot{\x}_{[j]}|$ is locally $\rho$-bounded nonconvex with $\rho < + \infty$. 
	\end{lemma}
	
	\noi \textbf{Remarks.} This concept is alternatively known as semi-nonconvexity, approximate nonconvexity, or weakly-nonconvexity in the literature. Besides, many functions have the locally $\rho$-bounded nonconvexity, such as $z(\x) = -\|\A\x\|_p$ and $z(\x) = -\|\x\|$.

	\noi Using the definition \ref{defition:rho}, the subsequent theorem introduces a quadratic growth condition for all \textbf{CWS}-point.
	\begin{theorem} \label{theo:descent}
		\textbf{(Proof in Appendix \ref{app:theo:descent})}
		\noi Assuming $g(\x)$ is locally $\rho$-bounded nonconvex, for any $\ddot{\x}$ and $\ddot{\x}+\textbf{d}$ are in feasible set $\Omega$, it holds that: $F(\ddot{\x})\leq F(\ddot{\x}+\textbf{d})  + \tfrac{\Qup + \theta + \rho}{2}\|\d\|_2^2$, where $\Qup = \|\Q\|_2$.
	\end{theorem}

	\noi \textbf{Remarks.} The definition of a local minimum point can be stated as follows: for any $\d$ such that $\|\d\| \leq \epsilon$ and $\ddot{\x} + \d \in \Omega$, we have $F(\ddot{\x}) \leq F(\ddot{\x} + \d) + \sigma$, where $\sigma > 0$ and $\d$ represents the radius of the vicinity around $\ddot{\x}$. When comparing the conditions of the \textbf{CWS}-point and the local minimum point, neither condition is strictly stronger nor weaker than the other.
	
	The following theorem builds the relations between three types of points, including the optimal point $\bar{\x}$, critical point $\breve{\x}$, and \textbf{CWS}-point $\ddot{\x}$. The following theorem establishes their relations. 
	
	\begin{theorem} \label{theo:optimality}
		\textbf{(Optimality Hierarchy between the Optimality Conditions, Proof in Appendix \ref{app:theo:optimality})}
		\noi For any $\bar{\x}$, $\ddot{\x}$ and $\breve{\x}$ are in feasible set $\Omega$, the optimality hierarchy of three points holds: $\{\text{Optimal point}\ \bar{\x}\}\overset{(\textbf{i})}{\subseteq} \{\text{\textbf{CWS}-point}\ \ddot{\x}\}\overset{(\textbf{ii})}{\subseteq}  \{\text{Critical point}\ \breve{\x}\}\nonumber\label{t:relation}$.
	\end{theorem}
	\noi \textbf{Remarks.} We claim that the coordinate-wise optimality condition is stronger than the critical point condition. The coordinate-wise optimality condition requires that for each coordinate direction, the function value cannot be decreased by moving in that direction. For a comprehensive demonstration of these relationships, we provide detailed proof in the Appendix \ref{app:theo:optimality}.

	\section{Convergence Analysis}\label{sec:conv}
	In this section, we conduct the convergence analysis of our BCD methods for solving Problem (\ref{problem original}). Our convergence analysis consists of three parts. We first demonstrate that each iteration of the BCD-g and BCD-l-$k$ methods result in a sufficient decrease $F(\x^t)$ from to $F(\x^{t+1})$, which ensures progress towards the solution. Next, we analyze the iteration complexity of our BCD methods to achieve an $\hat{\epsilon}$-stationary point. This complexity guarantees that the BCD methods reach near-optimality within a finite number of iterations. Finally, we prove the convergence rate that our BCD methods exhibit Q-linear convergence to \textbf{CWS}-point or critical point. 
	
	We assume that the working set $\B^t$ in each iteration is selected randomly and uniformly. Our algorithms generate a random output $\x^t$, which depends on the observed realization of the random variable: $\xi^{t-1} \triangleq \{\B^0,\B^1,...,\B^{t-1}\}$. We denote $\Pi_n^k \triangleq \{\B_{(i)}\}^{C_n^k}_{i=1}$ as the set of all possible combinations of index vectors, selecting k items from n. The following theorem reveals the global convergence on our BCD methods.
	\begin{theorem} \label{theo:global}
		(\textbf{Global Convergence for BCD-g and BCD-l-k method, Proof in Appendix \ref{app:theo:global}})

		\noi $\textbf{(i)}$ \textbf{(Sufficient Decrease Condition)} For both BCD-g and BCD-l-$k$ methods, we have the following results: $F(\x^{t+1}) - F(\x^t) \leq -\tfrac{\theta}{2}\|\x^{t+1}-\x^t\|_2^2$.
		
		\noi $\textbf{(ii)}$When the BCD-g and BCD-l-$k$ methods converge to a \textbf{CWS}-point or a critical point of Problem (\ref{problem original}), the following holds: $\mathbb{E}_{\xi^T}[\|\x^{\bar{t}+1} - \x^{\bar{t}}\|_2^2] \leq \tfrac{2(F(\x^0) - F(\bar{\x}))}{ \theta T}$.
		
		\noi $\textbf{(iii)}$ \textbf{(Iteration Complexity)} The BCD-g and BCD-l-$k$ methods find an $\hat{\epsilon}$-approximate \textbf{CWS}-point or an $\hat{\epsilon}$-approximate critical point in at most $T$ iterations in the sense of expectation, where: $ T \leq \lceil \tfrac{2(F(\x^0) - F(\bar{\x}))}{ \theta \hat{\epsilon}} \rceil = \mathcal{O}(\hat{\epsilon}^{-1})$.
	\end{theorem} 
	\noi \textbf{Remarks.} The convergence rate of $\mathcal{O}(\hat{\epsilon}^{-1})$ for gradient descent is considered the best possible worst-case dimension-free rate of convergence for any optimization algorithm \citep{carmon2020lower}. 
	
	In the remainder of this section, we will discuss the convergence rate of the BCD methods. To facilitate this analysis, we first introduce a definition and a lemma about globally $\rho$-$\epsilon$ bounded nonconvexity. It highlights the generality of the function $g(\cdot)$ in our discussion. 
	
	\begin{definition}
		\noi  \textbf{(globally $\rho$-$\epsilon$ bounded nonconvexity)} A function $z(\x)$ is called to be globally $\rho$-$\epsilon$ bounded nonconvex if: $\forall \x,\ \y,\ z(\x) \leq z(\y) + \la \x - \y ,\ \partial z(\x)\ra + \tfrac{\rho}{2}\|\x - \y\|_2^2 + \epsilon$ with $\rho \leq +\infty$ and $\epsilon \geq 0$.
	\end{definition}
	\begin{lemma}\label{lemma:globallybounded}
		\textbf{(Proof in Appendix \ref{app:lemma:globallybounded})}
		Let $z(\x)$ be any $C$-Lipschitz continuous function. The function $z(\x)$ is also globally $(\tfrac{2C^2}{\epsilon},\epsilon)$-bounded nonconvex.
	\end{lemma}
	
	\noi \textbf{Remarks.} \textbf{(i)} Under this definition, $g(\cdot)$ is not required to be smooth. It only needs to satisfy the condition of being Lipschitz continuous. This relaxation allows for a broader class of functions that are possibly non-smooth while still maintaining controlled convergence rates. \textbf{(ii)} We provide the relationship between $\rho$ and $\epsilon$, which is fundamental to the theoretical analysis presented subsequently.

	Then we establish the following two technical lemmas, which will serve as a foundation for analyzing the convergence rate for our BCD methods.

	\begin{lemma}\label{lemma:iequality}
		
		Let $\H_+ \triangleq \UB \UB\trans \Q  \UB \UB \trans + \theta \I_n \in \mathbb{R}^{n\times n}$. Assume $\B$ is selected randomly and uniformly. We have: $c_1\|\x\|_2^2 \leq \|\x\|^2_{\H_+} \leq c_2\|\x\|_2^2$, where $c_1 = \tfrac{k}{n} \Hlow + \theta$, $c_2=\tfrac{k}{n}\Hup + \theta$, $\Hlow = \min_{\B} \|\UB\trans\Q\UB\| $, and $\Hup=\max_{\B} \|\UB\trans\Q\UB\|$.
	\end{lemma}

	\begin{lemma} \label{lemma:bound:4}
		\textbf{(Proof in Appendix \ref{app:lemma:bound:4})} For any $\x \in \R^n, \d \in \mathbb{R}^n$, we have:
		
		\begin{enumerate}[label=\textbf{(\roman*)}, itemsep=1pt, topsep=1pt, parsep=0pt, partopsep=0pt]

			\item $\mathbb{E}_{\B}[\la \x_{\B}, \d_{\B} \ra] = \tfrac{1}{C_n^k} \sum_{\B \in \Pi_n^k} \la \x_{\B}, \d_{\B} \ra = \tfrac{k}{n}\la \x, \d \ra$.
			
			\item $\tfrac{1}{C_n^k} \sum_{\B \in \Pi_n^k} \la \U_{\B}\x_{\B} , \d \ra = \tfrac{k}{n}\la \x , \d \ra$.
			
			\item $\tfrac{1}{C_n^k}\sum_{\B\in \Omega_n^k} \| \x + \U_\B\d_\B \|_2^2 = (1 - \tfrac{k}{n})\|\x\|_2^2 + \tfrac{k}{n}\|\x + \d\|_2^2$.
			
			\item $ \tfrac{1}{C_n^k} \sum_{\B\in \Omega_n^k} f(\x + \U_\B\d_\B) \leq f(\x) + \la  \tfrac{k}{n} \d,\nabla f(\x)\ra + \E_{\B}[ \tfrac{1}{2} \|\d_\B\|_{\H_+}^2] - \tfrac{\theta}{2}  \E_{\B}[ \tfrac{1}{2} \|\d_\B\|_{2}^2]$.
			
			\item $\tfrac{1}{C_n^k} \sum_{\B\in \Omega_n^k} h(\x + \U_\B\d_\B) = \tfrac{k}{n}h(\x+\d) + (1 - \tfrac{k}{n}) (\x)$.
			
			\item $\tfrac{1}{C_n^k} \sum_{\B\in \Pi_n^k} g(\x + \U_\B\d_\B) \leq g(\x) + \tfrac{k}{n}\la \d, \partial g(\x)\ra$.
			
			\item $\tfrac{1}{C_n^k} \sum_{\B\in \Omega_n^k}  g(\x + \U_\B\d_\B) \leq  g(\x) + \tfrac{k}{n} (g(\x + \d) - g(\x) + \tfrac{\rho}{2}\|\d\|_2^2 + \epsilon )$.
		\end{enumerate}

	\end{lemma}

	To establish a stronger convergence result for Algorithm \ref{alg1}, we utilize the Luo-Tseng error bound assumption, a well-recognized framework in mathematical optimization. The following theorems detail the convergence rates of our BCD methods. Specifically, under the Luo-Tseng error bound assumption stated in Lemma \ref{lemma:luo} \citep{tseng2009coordinate,luo1993error}, both BCD methods are proven to achieve Q-linear convergence to either a \textbf{CWS}-point or a critical point.
	
	\begin{lemma} \label{lemma:luo}
		\textbf{( Luo-Tseng error bound assumption)} We define a residual function as: $\mathcal{R}(\x) = \frac{1}{n} \sum_{i=1}^{n} |\text{dist}(0, \mathcal{M}_i(\x))|$. For any $\varsigma \geq \min_{\x} F(\x)$, there exist constants $\delta > 0$ and $\varpi > 0$ such that: $\forall \x, \ \text{dist}(\x, \mathcal{X}) \leq \delta \mathcal{R}(\x)$, whenever $F(\x) \leq \varsigma, \ \mathcal{R}(\x) \leq \varpi$. $\mathcal{X}$ is the set of stationary points satisfying $\mathcal{R}(\x) = 0$.
	\end{lemma}
	
	\begin{theorem}\label{theo:dec:x:f}  
		\textbf{(Convergence Rate for the BCD-g method, Proof in Appendix \ref{app:theo:dec:x:f})}
		Let $\ddot{\x}$ be any \textbf{CWS}-point. Assume feasible set $\Omega$ is convex and $g(\x)$ is globally $\rho$-$\epsilon$ bounded nonconvex. We first define $\ddot{q}^t \triangleq F(\x^t) - F(\ddot{\x})$, $\ddot{q}^0 \triangleq F(\x^0) - F(\ddot{\x})$, $\ddot{r}^{t} \triangleq \tfrac{1}{2} \|\x^{t} - \ddot{\x}\|_{\H_+^t}^2$, where $\H_+^t$, $c_1$ and $c_2$ are constants defined in Lemma \ref{lemma:iequality}.  With any constant $\epsilon$ in Lemma \ref{lemma:globallybounded}, then we providing following results:
		
		\noi \textbf{(i)} We first define $\beta_1 \triangleq 1 - \tfrac{\rho}{c_1}$, $\beta_2 \triangleq 1 - \tfrac{\rho}{c_2}(1-\tfrac{k}{n})$, $\gamma \triangleq (1 + \tfrac{k\rho}{\theta})$. Then we have: $\beta_1 \E_{\B^t} [\ddot{r}^{t+1}] + \gamma \E_{\B^t}[\ddot{q}^{t+1}] \leq \beta_2 \ddot{r}^{t} + (\gamma - \tfrac{k}{n})\ddot{q}^t + \tfrac{2k}{n}\epsilon$.

		\noi \textbf{(ii)} We define three constants $\kappa_0 \triangleq \tfrac{c_2\delta^2}{\theta}$, $\kappa_1 \triangleq \gamma + \beta_2\kappa_0$ and $\ddot{\kappa} \triangleq 1 -  \tfrac{k}{n\kappa_1}$. If $\theta$ is sufficiently large such that $\tfrac{k}{n} \Hlow + \theta > \rho \Leftrightarrow \beta_1 > 0$, where $\Hlow$ is defined in Lemma \ref{lemma:iequality}, then the following holds: $\E_{\xi^t}[\ddot{q}^{t+1}] \leq \ddot{\kappa}^{t+1} \ddot{q}^0 + 2\epsilon$.
		
		\noi \textbf{(iii)} \textbf{(Iteration Complexity)}  With any $\sigma \leq 2\ddot{q}^0$, when $t+1 \geq \log_{\ddot{\kappa}'} \tfrac{\sigma}{2\ddot{q}^0}$, it holds that $\E_{\xi^t}[\ddot{q}^{t+1}] \leq \sigma$, where $\beta_3 \triangleq 1 - \tfrac{8C^2}{c_2\sigma}(1-\tfrac{k}{n})$ and $\ddot{\kappa}' \triangleq 1 - \tfrac{k}{n(\gamma + \beta_3\kappa_0)}$. 
	\end{theorem}

	\begin{theorem}\label{theo:dec:x:f2}  
		\textbf{(Convergence Rate for the BCD-l-$k$ method, Proof in Appendix \ref{app:theo:dec:x:f2})}
		Let $\breve{\x}$ be any critical point. Assume feasible set $\Omega$ is convex and $g(\x)$ is globally $\rho$-$\epsilon$ bounded nonconvex. We first define $\breve{q}^t \triangleq F(\x^t) - F(\breve{\x})$, $\breve{q}^0 \triangleq F(\x^0) - F(\breve{\x})$, $\breve{r}^{t} \triangleq \tfrac{1}{2} \|\x^{t} - \breve{\x}\|_{\H_+^t}^2$, where $\H_+^t$ and $c_1$ are constants defined in Lemma \ref{lemma:iequality}.  Given any constant $\epsilon$, then we have:
		
		\noi \textbf{(i)} $	\E_{\B^t} [\breve{r}^{t+1}] + \E_{\B^t}[\breve{q}^{t+1}] \leq  (1 + \tfrac{k\rho}{nc_1}) \breve{r}^{t}  + (1 - \tfrac{k}{n}) \breve{q}^t + \tfrac{k}{n}\epsilon$.

		\noi \textbf{(ii)} We define $\kappa_2 \triangleq \kappa_0(1 + \tfrac{k\rho}{nc_1}) + 1$ and $\breve{\kappa} \triangleq 1 -  \tfrac{k}{n\kappa_2}$, where $\kappa_0 = \tfrac{c_2\delta^2}{\theta}$. Then it holds: $\E_{\xi^t}[\breve{q}^{t+1}] \leq \breve{\kappa}^{t+1} \breve{q}^0 + \epsilon$.

		\noi \textbf{(iii)} \textbf{(Iteration Complexity)}  Given any $\sigma \leq 2\breve{q}^0$, when $t+1 \geq \log_{\breve{\kappa}'} \tfrac{\sigma}{2\breve{q}^0}$, it holds that $\E_{\xi^t}[\breve{q}^{t+1}] \leq \sigma$, where  $\kappa_3' \triangleq \kappa_0(1 + \tfrac{4Ck}{nc_1\sigma}) + 1$ and $\breve{\kappa}' \triangleq 1 - \tfrac{k}{n\kappa_3'}$.
	\end{theorem}

	\noi \textbf{Remarks.} \textbf{(i)} This work represents the first investigation into the convergence rate of the BCD methods for the specific class of problems described in Problem (\ref{problem original}). Our analysis provides novel insights into the methods' efficiency in this particular setting. \textbf{(ii)} Unlike prior research on coordinate descent method \citep{yuan2023coordinate}, which typically require the smoothness of the function $g(\x)$, our proposed methods only require the continuity of $g(\x)$. This relaxation of assumptions expands the applicability of the method to a broader range of non-smooth problems. 
	
%

	\section{Global Solutions or Local Solutions for Four Applications} \label{sec:block2}
	In this section, we first outline the detailed breakpoint searching methods used in the BCD-g method for obtaining global solutions when $k=2$ across the four applications discussed in Section \ref{sec:app}. With the intricate design, the subproblems (\ref{plb:subproblem}) in four applications can be solve globally. Then we introduce the details on BCD-l-$k$ method when $k \geq 2$ for solving the subproblem (\ref{plb:subproblemlocal}) locally in sparse index tracking task. Considering the feasible set $\Omega = \{\x | u(\x) = 0\}$, our methods concentrate on generating an optimized step size $\bar{\d}_{\B^t} \in \R^{k}$, which leads to the updated feasible solution $\x^{t+1} = \x^t + \U_{\B^t}\bar{\d}^t_{\B^t}$. Since we focus on the $t$-th iteration, we use $i$, $j$, $r$ and $\B$ instead of $i^t$, $j^t$, $r^t$ and $\B^t$ to simplify notation. With $\r = \UB\d_{\B}$ and $\bar{\Q} = \Q + \theta \I_n$, Problem (\ref{plb:subproblem}) can be reformulated as following optimization problem: $\min_{\d_{\B}} \, \mathcal{J}(\d_{\B}) \triangleq \tfrac{1}{2} \|\r\|_{\bar{\Q}}^2 + \la \r , \nabla f(\x^t) \ra + h(\x^t + \r) + g(\x^t +\r),\ s.t.\ \x^t + \r  \in \Omega$.

	\subsection{Global Solutions for Sparse Index Tracking When $k = 2$}
	The sparse index tracking problem (\ref{plb:app:topk}) is presented as follows: $\min_{\x \in \mathbb{R}^n} \tfrac{1}{2}\|\A\x - \y\|_2^2 +  \mathcal{I}_{\geq 0}(\x)- \lambda \|\x\|_{[s]},\ s.t.\ \|\x\|_1 = 1$. With given two indices $i$ and $j$, we have: $\r = \d_ie_i + \d_je_j$. Letting $f(\x) = \tfrac{1}{2}\|\A\x - \y\|_2^2$, , Problem (\ref{plb:app:topk}) is equivalent to the following problem: 
	\beq
	\min_{\d_{i},\d_{j}} \mathcal{J}(\d_{i},\d_{j}) \triangleq \tfrac{1}{2} \| \r\|_{\bar{\Q}}^2 + \la \nabla f(\x^t), \r\ra - \lambda \| \x^t + \r\|_{[s]},\label{plb:Jfunction:index}\\
	\ s.t.\ \d_{i} \geq -\x_{i}^t,~\d_{j} \geq -\x_{j}^t,~\d_{i} + \d_{j} = 0,\nn
	\eeq
	where $\bar{\Q} = \A^{\top}\A + \theta \I_n$. Using the relation $\d_{i} + \d_{j} = 0$, we transform the Problem (\ref{plb:Jfunction:index}) into a one-variable problem by minimizing only over $\eta$:
	\beq
	\min_{\eta} \mathcal{P} (\eta) \triangleq \mathcal{J} (\eta,-\eta) = \tfrac{1}{2} \alpha \eta^2 + \beta \eta - \lambda \|\x^t + \eta e_{i} - \eta e_{j}\|_{[s]},\ s.t.\ -\x^t_{i} \leq \eta \leq \x^t_{j}, \label{eq:subproblem2}
	\eeq
	\noi where $\alpha\triangleq \bar{\Q}_{ii} + \bar{\Q}_{jj} - 2\bar{\Q}_{ij}$ and $\beta\triangleq \nabla_{i} f(\x^t) + \nabla_{j} f(\x^t)$. By converting the block-2 coordinate descent problem into a block-1 coordinate descent problem, it becomes more straightforward to minimize $ \mathcal{P}(\eta)$ and design the breakpoint searching strategy for Problem (\ref{eq:subproblem2}).
	
	\noi$\bullet$ \textbf{Breakingpoint Searching Strategy.}  
	This part presents a novel breakpoint searching strategy designed to efficiently and directly solve the subproblem (\ref{eq:subproblem2}). This strategy begins by identifying four distinct situations and then locates five potential breakpoints. Finally, it chooses the breakpoint that minimizes $\mathcal{P}(\eta)$ as the optimal solution $\bar{\eta}^t$. The five breakpoints, corresponding to the four situations, are described in detail as follows:

	\begin{enumerate}[label=\textbf{(\roman*)}, itemsep=1pt, topsep=1pt, parsep=0pt, partopsep=0pt]
		\item While $-\x^t_{i} \leq \eta \leq \x^t_{j}$, the first breakpoint and the second breakpoint are $\eta_1$ = $-\x^t_{i}$ and $\eta_2$ = $\x^t_{j}$.
		\item When $\x^t_i + \eta e_{i}$ and $\x^t_j - \eta e_{j}$ are both in top-$s$ set or both not in top-$s$ set, $ \mathcal{P}(\eta)=\tfrac{1}{2} \alpha\eta ^2 + \beta \eta $, the third breakpoint is $\eta_3$ = $-\beta\ /\ \alpha$.
		\item When $\x^t_i + \eta e_{i}$ and is in top-$s$ set but $\x^t_j - \eta e_{j}$ is not in top-$s$ set, $ \mathcal{P}(\eta)=\tfrac{1}{2} \alpha \eta ^2 + \beta \eta - \lambda(\x^t_i + \eta e_{i})$, the fourth breakpoint is $\eta_4$ = $(\lambda - \beta) / \alpha$.
		\item When $\x^t_i + \eta e_{i}$ and is not in top-$s$ set but $\x^t_j -  \eta e_{j}$ is in top-$s$ set, $ \mathcal{P}(\eta)=\tfrac{1}{2} \alpha\eta^2 + \beta\eta - \lambda (\x^t_j - \eta e_{j})$, the fifth breakpoint is $\eta_5$ = $-(\lambda + \beta) / \alpha$.
	\end{enumerate}

	\noi The optimal step size $\bar{\eta}^t$ in $t$ iteration is generated as follows: $\bar{\eta}^t \in \arg\min_{r} \mathcal{P}(r)$, where $r\in\{\eta_1,\eta_2,\mathbb{P}_1(\eta_3),\mathbb{P}_1(\eta_4),\mathbb{P}_1(\eta_5)\}$ and $\mathbb{P}_1(t) \triangleq \max(-\x^t_i,\min(t,\x^t_j))$. Ultimately, $\x^{t+1}$ is updated by $\x^{t+1} = \x^t + \d_{i}e_i + \d_{j}e_j$, with $\d_{i} = \bar{\eta}^t$ and $\d_{j} = -\bar{\eta}^t$.

	\subsection{Global Solutions for Non-negative Sparse PCA When $k = 2$}
	We recap the non-negative sparse PCA problem (\ref{plb:app:nnspca2}): $\min_{\x \in \mathbb{R}^n}\,\tfrac{1}{2}\x^{\top}\hat{\Q}\x + \mathcal{I}_{\geq 0}(\x) + \lambda(\x^{\top}\mathbf{1}_n  - \|\x\|_{[s]},\ s.t.\ \|\x\|^2_2=1$. When $f(\x) = \tfrac{1}{2}\x^{\top}\hat{\Q}\x$, we have following problem with two given indices $i$ and $j$: 
	\beq
	\min_{\d_{i},\d_{j}} \mathcal{J}(\d_{i},\d_{j}) \triangleq \tfrac{1}{2}\|\r\|^2_{\bar{\Q}} + \la (\nabla f(\x^t) + \lambda\mathbf{1}_n), \r\ra - \lambda\|\x^t + \r\|_{[s]},\label{plb:Jfunction:nnspca}\\
	\ s.t.\ (\d_{i}+\x^t_{i})^2 + (\d_{j}+\x^t_{j})^2 = v^2,\nn
	\eeq
	where $\bar{\Q} = \hat{\Q} + \theta \I_n$ and $v^2 = 1 - (\x^t_{i})^2 + (\x^t_{j})^2$. Letting $v\sin(\alpha) = \d_{i}+\x^t_{i}$, and $v\cos(\alpha) = \d_{j}+\x^t_{j}$, we convert the two-variable optimization problem (\ref{plb:Jfunction:nnspca}) into a one-variable optimization problem by minimizing over $\alpha$:
	\beq
	\min_{\alpha}\mathcal{Q}(\alpha) \triangleq \tfrac{1}{2}\|\r\|^2_{\bar{\Q}} + \la (\nabla f(\x^t) + \lambda\mathbf{1}_n), \r\ra - \lambda\|\x^t + \r\|_{[s]},\ s.t.\ \alpha \in [0,\tfrac{\pi}{2}],\label{plb:alpha}
	\eeq
	\noi where $\r =  (v\sin(\alpha) - \x^t_{i})e_{i} + (v\cos(\alpha) - \x^t_{j})e_{j}$. Furthermore, Problem (\ref{plb:alpha}) can be represented as follows:
	\beq
	\min_{\alpha \in [0,\tfrac{\pi}{2}]}\mathcal{Q}(\alpha) =  a \cos^2(\alpha) + b \sin(\alpha) + c \cos(\alpha) + d \sin(\alpha)\cos(\alpha) - \lambda\|\x^t + \r\|_{[s]},  \label{plb:nnspca:break}
	\eeq where $a = \tfrac{1}{2}(\bar{\Q}_{jj}-\bar{\Q}_{ii})v^2$, $b = ([\nabla f(\x^t) + \lambda\mathbf{1}_n]_{i} - \bar{\Q}_{ii}\x^t_{i} - \bar{\Q}_{ij}\x^t_{j})v$, $c = ([\nabla f(\x^t) + \lambda\mathbf{1}_n]_{j} - \bar{\Q}_{ij}\x^t_{i} - \bar{\Q}_{jj}\x^t_{j})v$, $d = \bar{\Q}_{ij}v^2 $.

	\noi$\bullet$ \textbf{Breakpoint Searching Strategy.}  In this part, we introduce the breakpoint searching strategy for solving  Problem (\ref{plb:nnspca:break}). This strategy is designed to identify all breakpoints across five situations. During this strategy, term $- \lambda\|\x^t + \r\|_{[s]}$ in $\mathcal{Q}(\alpha)$ is discussed in the first four situations.
	
	\begin{enumerate}[label=\textbf{(\roman*)}, itemsep=1pt, topsep=1pt, parsep=0pt, partopsep=0pt]
		\item When $\x^t_{i} + \r_{i}$ and $\x^t_{j} + \r_{j}$ are both not in top-$s$ set, then $- \lambda\|\x^t + \r\|_{[s]} = - \lambda\|\x^t\|_{[s]}$ and $\mathcal{Q}(\alpha) = a \cos^2(\alpha) + b \sin(\alpha) + c \cos(\alpha) + d \sin(\alpha)\cos(\alpha) - \lambda\|\x^t\|_{[s]}$. The first breakpoint is $\alpha_1 \in \arg\min_{\alpha \in [0,\tfrac{\pi}{2}] } \mathcal{Q}(\alpha)$.
		
		\item When $\x^t_{i} + \r_{i}$ and $\x^t_{j} + \r_{j}$ are both in top-$s$ set, then $- \lambda\|\x^t + \r\|_{[s]} = - \lambda v \sin(\alpha) - \lambda v \cos(\alpha) - \lambda\|\x^t\|_{[s]} $ and $\mathcal{Q}(\alpha) = a \cos^2(\alpha) + (b - \lambda v) \sin(\alpha) + (c - \lambda v) \cos(\alpha) + d \sin(\alpha)\cos(\alpha) - \lambda\|\x^t\|_{[s]}$. The second breakpoint is $\alpha_2 \in \arg\min_{\alpha \in [0,\tfrac{\pi}{2}] }  \mathcal{Q}(\alpha)$.
		
		\item When $\x^t_{i} + \r_{i}$ and is in top-$s$ set but $\x^t_{j} + \r_{j}$ is not in top-$s$ set, then $- \lambda\|\x^t + \r\|_{[s]} = - \lambda v \sin(\alpha) - \lambda\|\x^t\|_{[s]}$ and $\mathcal{Q}(\alpha) = a \cos^2(\alpha) + (b - \lambda v) \sin(\alpha) + c \cos(\alpha) + d \sin(\alpha)\cos(\alpha) - \lambda\|\x^t\|_{[s]}$. The third breakpoint is $\alpha_3 \in \arg\min_{\alpha \in [0,\tfrac{\pi}{2}] }  \mathcal{Q}(\alpha)$.
		
		\item When $\x^t_{i} + \r_{i}$ and is not in top-$s$ set but $\x^t_{j} + \r_{j}$ is in top-$s$ set, then $- \lambda\|\x^t + \r\|_{[s]} =  - \lambda v \cos(\alpha) - \lambda\|\x^t\|_{[s]}$ and $\mathcal{Q}(\alpha) = a \cos^2(\alpha) + b \sin(\alpha) + (c - \lambda v) \cos(\alpha) + d \sin(\alpha)\cos(\alpha) - \lambda\|\x^t\|_{[s]}$. The fourth breakpoint is $\alpha_4 \in \arg\min_{\alpha \in [0,\tfrac{\pi}{2}] }  \mathcal{Q}(\alpha)$.
		
		\item While $\d_{i}+\x^t_{i} \geq 0$ and $\d_{j}+\x^t_{j} \geq 0$, the fifth breakpoint and the sixth breakpoint are $\alpha_5$ = $\arcsin(0) = 0$ and $\alpha_6$ = $\arccos(0) =\pi / 2$.
	\end{enumerate}

	To solve the subproblem $\min_{\alpha} \mathcal{Q}(\alpha)$ in each first four situations (situation (\textbf{i}) to (\textbf{iv})), we reformulate the Problem (\ref{plb:nnspca:break}) into an equivalent problem $\min_{\alpha \in [0,\tfrac{\pi}{2}] } \mathcal{P}(\alpha) \triangleq a\cos^2(\alpha) + \hat{b} \sin(\alpha) + \hat{c} \cos(\alpha) + d \sin(\alpha)\cos(\alpha)$, where $\hat{b} \triangleq b $ or $b - \lambda v$ and $\hat{c} \triangleq c $ or $c - \lambda v$ according to the situation discussed above. Using the substitution $\beta \triangleq \tan(\alpha)$ with $\beta \geq 0$, we convert the problem $\min_{\alpha \in [0,\tfrac{\pi}{2}] } \mathcal{P}(\alpha)$ into $\min_{\beta \geq 0} \mathcal{P}(\beta) \triangleq \tfrac{a + d\beta}{1+\beta^2} + \tfrac{\hat{c} + \hat{b} \beta}{\sqrt{1+\beta^2}}$. This substitution is based on the trigonometric identities that $\cos(\alpha) = \tfrac{1}{\sqrt{1+\tan^2(\alpha)}}$ and $\sin(\alpha) = \tfrac{ \tan(\alpha)}{\sqrt{1+\tan^2(\alpha)}}$. Setting the gradient $\nabla \mathcal{P}(\beta)$ to zero yields: $\tfrac{[d\beta^o - 2\beta(a+d\beta)] + \sqrt{\beta^o}[\hat{b}\beta^o - \beta(\hat{c}+\hat{b}\beta)]}{\beta^o} = 0 \Rightarrow d\beta^o - 2\beta(a+d\beta)  \overset{\step{1}}{=} (\hat{c}\beta - \hat{b}) \sqrt{\beta^o} \Rightarrow [d\beta^o - 2\beta(a+d\beta)]^2 \overset{\step{2}}{=}  (\hat{c}\beta - \hat{b})^2\beta^o \Rightarrow c_4\beta^4 + c_3\beta^3 + c_2\beta^2 + c_1\beta +c_0 \overset{\step{3}}{=} 0$, where step $\step{1}$ uses the definition $\beta^o \triangleq 1 + \beta^2$; step \step{2} uses the fact that $\beta^o \geq 0$; step $\step{3}$ uses $c_4 = d^2-\hat{c}^2$, $c_3 = 4ad+2\hat{b}\hat{c}$, $c_2 = 4a^2-2d^2-\hat{b}^2-\hat{c}^2$, $c_1 = 2\hat{b}\hat{c} - 4ad$, $c_0 = d^2 - \hat{b}^2$.
	By applying Lodovico Ferrari’s method to solve the equation analytically, we obtain all its real non-negative roots $\{\mathbb{P}_2(\beta_1), \mathbb{P}_2(\beta_2),..., \mathbb{P}_2(\beta_j)\}$, where $1\leq j \leq 4$ and $\mathbb{P}_2(t) \triangleq \max(0,t)$. Then we select the root that leads the smallest objective value of $\mathcal{P}(\alpha)$. Consequently, the optimal solution $\alpha_i$ with $1\leq i \leq 4$ in each first four situation is determined as: $\alpha_i \in \argmin_{\alpha}  \mathcal{P}(\alpha),~\text{where}\ \alpha\in \{\arctan(\mathbb{P}_2(\beta_1)),\arctan(\mathbb{P}_2(\beta_2)),...,\arctan(\mathbb{P}_2(\beta_j))\}$. 
	
	Finally, the optimal solution $\bar{\alpha}$ to Problem (\ref{plb:alpha}) is generated as: $\bar{\alpha} \in \argmin_{\alpha} \mathcal{Q}(\alpha)$, where $\alpha \in \{\alpha_1,\alpha_2,\alpha_3,\alpha_4,\alpha_5,\alpha_6\}$. The step sizes $\d_{i}$ and $\d_{j}$ are determined by $\d_{i} = v\sin(\bar{\alpha}) - \x^t_{i},\ \d_{j} = v\cos(\bar{\alpha}) - \x^t_{j}$. We update $\x^{t+1}$ with $\x^{t+1} = \x^{t} + \d_{i}e_{i} + \d_{j}e_{j}$.

	\subsection{Global Solutions for DC Penalized Binary Optimization When $k = 2$}
	
	We review the DC penalized binary optimization problem (\ref{plb:app:bin2}) as follows: $\min_{\x \in \R^n}  \p^{\top}\x + \mathcal{I}_{[-1,1]}(\x) - \tfrac{1}{2}\x^{\top}\hat{\Q}\x,\ s.t.\ \x^{\top}\textbf{1} = c$. Denoting $\hat{f}(\x) = \p^{\top}\x - \tfrac{1}{2}\x^{\top}\hat{\Q}\x$, we have: 
	\beq
	\min_{\d_{i},\d_{j}} \mathcal{J}(\d_{i},\d_{j}) \triangleq \tfrac{1}{2}\|\r\|^2_{\bar{\Q}} + \la \nabla \hat{f}(\x^t) , \r \ra,\ s.t.\ \d_{i}+\d_{j} = 0,\ l_1 \leq \d_{i} \leq u_1,\ l_2 \leq \d_{j} \leq u_2, \label{plb:Jfunction:bin1}
	\eeq 
	where $\bar{\Q} = -\hat{\Q} + \theta \I_n$, $l_1 = -1 - \x^t_{i}$, $l_2 = -1 - \x^t_{j}$, $u_1 = 1- \x^t_{i}$ and $u_2 = 1- \x^t_{j}$. Using the relation $\d_{i} + \d_{j} = 0$, we transform Problem (\ref{plb:Jfunction:bin1}) into only minimizing over $\eta$:
	\beq
	\min_{\eta} \mathcal{P} (\eta) \triangleq \mathcal{J} (\eta,-\eta) = \tfrac{1}{2} \alpha \eta^2 + \beta \eta,\ s.t.\ l \leq \eta \leq u, \label{plb:bin:sub3}
	\eeq
	\noi where $\alpha =  \bar{\Q}_{ii} + \bar{\Q}_{jj} - 2\bar{\Q}_{ij}$, $\beta = \nabla_{i} \hat{f}(\x^t) - \nabla_{j} \hat{f}(\x^t)$, $l = \max(l_1,-u_2)$ and $u = \min(u_2,-l_2)$. Therefore, Problem (\ref{plb:bin:sub3}) contains 3 breakpoints $\{l,u,\mathbb{P}_3 (-\tfrac{\beta}{\alpha})\}$, where $\mathbb{P}_3(t) \triangleq \max(l,\min(t,u))$. The complexity of solving Problem (\ref{plb:bin:sub3}) is $\mathcal{O}(1)$. Then we pick one of the breakpoints leading $\mathcal{P} (\eta)$ the least as $\bar{\eta}^t$. Finally, $\x^{t+1}$ is generated as $\x^{t+1} = \x^t + \d_{i}e_i + \d_{j}e_j$, with $\d_{i} = \bar{\eta}^t$ and $\d_{j} = -\bar{\eta}^t$.

	We now discuss another variational reformulation of DC penalized binary optimization problem (\ref{plb:app:bin3}): $\min_{\x \in \R^n} \tfrac{1}{2} \|\A\x - \y\|_2^2 + \mathcal{I}_{[-1,1]}(\x) - \lambda\|\x\|_2,\ s.t.\  \x^{\top}\textbf{1} = c$. With $\bar{\Q} = \A^{\top}\A + \theta \I_n$, we let $f(\x) = \tfrac{1}{2} \|\A\x - \y\|_2^2$ and define $\mathcal{H}(\d_{i},\d_{j})  \triangleq   \tfrac{1}{2}\|\r\|^2_{\bar{\Q}} + \nabla_{i} f(\x^t)\d_{i} + \nabla_{j} f(\x^t) \d_{j}$ and $\mathcal{G}(\d_{i},\d_{j}) \triangleq \sqrt{\|\x^t\|_2^2 + \d_{i}^2 + 2\d_{i} \x^t_{i} + \d_{j}^2 + 2\d_{j}\x^t_{j}}$. Then we have following problem: 
	\beq
	\min_{\d_{i},\d_{j}} \mathcal{J}(\d_{i},\d_{j}) \triangleq \mathcal{H}(\d_{i},\d_{j}) - \lambda \mathcal{G}(\d_{i},\d_{j}) ,\ s.t.\ \d_{i}+\d_{j} = 0,\ l_1 \leq \d_{i} \leq u_1,\ l_2 \leq \d_{j} \leq u_2,\label{plb:Jfunction:bin2}
	\eeq 
	where $l_1 = -1 - \x^t_{i}$, $l_2 = -1 - \x^t_{j}$, $u_1 = 1- \x^t_{i}$ and $u_2 = 1- \x^t_{j}$. Using the relation $\d_{i} + \d_{j} = 0$, we convert Problem (\ref{plb:Jfunction:bin2}) into:
	\beq
	\min_{\eta} \mathcal{P} (\eta) \triangleq \mathcal{J} (\eta,-\eta) = \tfrac{1}{2} \alpha \eta^2 + \beta \eta  - \lambda \sqrt{\|\x^t\|_2^2 + 2\eta^2 + 2\x^t_{i}\eta - 2\x^t_{j}\eta},\ s.t.\ l \leq \eta \leq u, \label{plb:bin:sub4}
	\eeq
	where $\alpha = \bar{\Q}_{ii} + \bar{\Q}_{jj} - 2\bar{\Q}_{ij}$, $\beta = \nabla_{i} f(\x^t) - \nabla_{j} f(\x^t)$, $l = \max(l_1,-u_2)$ and $u = \min(u_1,-l_2)$.

	\noi$\bullet$ \textbf{Breakpoint Searching Strategy.} We outline all breakpoints for Problem (\ref{plb:bin:sub4}) during 3 situations presented as following. 
	
	\begin{enumerate}[label=\textbf{(\roman*)}, itemsep=1pt, topsep=1pt, parsep=0pt, partopsep=0pt]
		\item Since $l \leq \eta \leq u$, the first and second breakpoints are $\eta_1 = l$ and $\eta_2 = u$.
		
		\item When $\|\x^t\|_2^2 + 2\eta^2 + 2\x^t_{i}\eta - 2\x^t_{j}\eta = 0$, the third breakpoint is $\eta_3 = -\tfrac{\beta}{\alpha}$.
		
		\item When $\|\x^t\|_2^2 + 2\eta^2 + 2\x^t_{i}\eta - 2\x^t_{j}\eta \neq 0$, we set the gradient $\nabla \mathcal{P} (\eta) = 0$. This leads to the equation: $\alpha \eta + \beta = \lambda(4\eta + 2\x^t_{i} - 2\x^t_{j}) \cdot (\|\x^t\|_2^2 + 2\eta^2 + 2\x^t_{i}\eta - 2\x^t_{j}\eta)^{-\tfrac{1}{2}}$. Squaring both sides, we obtain a quartic polynomial: $(\alpha \eta + \beta)^2 = \lambda^2(4\eta + 2\x^t_{i} - 2\x^t_{j})^2 \cdot (\|\x^t\|_2^2 + 2\eta^2 + 2\x^t_{i}\eta - 2\x^t_{j}\eta)^{-1}$. By analytically solving this quartic equation using Lodovico Ferrari’s method, we find all its real roots $\{\mathbb{P}_3(\eta_4),...,\mathbb{P}_3(\eta_j)\}$ with $4 \leq j \leq 7$.
	\end{enumerate}
	Therefore, there are at most (2 + 1 + 4) breakpoints. Then $\bar{\eta}^t$ is determined as follows: $\bar{\eta}^t \in \arg\min_{r}  \mathcal{P}(r)~\text{where}\ r \in \{\eta_1,\eta_2,\eta_3,\mathbb{P}_3(\eta_4),\mathbb{P}_3(\eta_5),\mathbb{P}_3(\eta_6),\mathbb{P}_3(\eta_7)\}$. Finally, we update $\x^{t+1}$ using the formula: $\x^{t+1} = \x^t + \d_{i}e_i + \d_{j}e_j$, with $\d_{i} = \bar{\eta}^t$ and $\d_{j} = -\bar{\eta}^t$.

	\subsection{Local Solutions for Sparse Index Tracking when $k \geq 2$}
	With given a working set $\B \in \mathbb{N}^k$, the sparse index tracking problem (\ref{plb:app:topk}) is reformulated as following problem: 
	\beq
	\min_{\d_{\B} \in \mathbb{R}^k} \mathcal{J}(\d_{\B}) \triangleq \tfrac{1}{2} \| \r\|_{\bar{\Q}}^2 + \la \nabla f(\x^t), \r\ra - \lambda \| \x^t + \r\|_{[s]},\ s.t.\ \d_{\B} \geq -\x_{\B}^t,\ \d_{\B}^{\top}\mathbf{1} = 0, \label{plb:blockkJfunction:index}
	\eeq
	where $f(\x) = \tfrac{1}{2}\|\A\x - \y\|$ and $\bar{\Q} = \A^{\top}\A + \theta \I_n$. In this subproblem, for any coordinate $i$ in $\B$, every variable $[\x^t_{\B} + \d_{\B}]_{i}$ in working set $\B$ has two possible states: either $[\x^t_{\B} + \d_{\B}]_{i}$ belongs to the top-$s$ set or it does not. We define the variables both in top-$s$ set and working set $\B$ as $\S$, while the remaining variables of $\B$ form the set $\Z$. Consequently, Problem (\ref{plb:blockkJfunction:index}) can be converted into:
	\beq
	\min_{\d_{\B} \in \mathbb{R}^k} \mathcal{J}_{\S}(\d_{\B}) \triangleq \tfrac{1}{2} \| \d_{\B}\|_{\bar{\Q}_{\B\B}}^2 + \la [\nabla f(\x^t)]_{\B}, \d_{\B} \ra - \lambda \| \x^t_{\S} + \d_{\S}\|_{[s]},\ s.t.\ \d_{\B} \geq -\x_{\B}^t,\ \d_{\B}^{\top}\mathbf{1} = 0. \label{plb:blockkJfunction:index2}
	\eeq
	The subgradient of the above objective function (\ref{plb:blockkJfunction:index2}) is given by: $\partial \mathcal{J}_{\S}(\d_{\B}) = \bar{\Q}_{\B\B} \d_{\B} + [\nabla f(\x^t)]_{\B} - \lambda \mathbf{1}_{\S}$, where $\mathbf{1}_{\S} \in \mathbb{N}^k$ is defied as: $
	[\mathbf{1}_{\S}]_{i} = \left\{
	\begin{aligned}
		&1, && i \in \S; \\
		&0, && i \in \Z.
	\end{aligned}
	\right.
	$. Using the subgradient, we can solve the subproblem (\ref{plb:blockkJfunction:index2}) to obtain a local minimum as one of the breakpoints by other methods such as projected subgradient method or proximal gradient method. Since there are $2^k$ states for step size $\d_{\B}$, we have $2^k$ candidate breakpoints for $\mathcal{J}(\d_{\B})$. Finally, we select the solution that results in the lowest objective value as the optimal step size $\bar{\d}_{\B} \in \argmin_{\d_{\B}} \mathcal{J}(\d_{\B})$ and update $\x^{t+1}$ by $\x^{t+1} = \x^{t} + \UB\bar{\d}_{\B}$.

	\section{Semi-greedy Strategies for Working Set Selection}\label{sec:greedy}
	In this section, we introduce two novel semi-greedy strategies designed to improve the computational efficiency for BCD-g method in sparse index tracking task (\ref{plb:app:topk}) and non-negative sparse PCA problem (\ref{plb:app:nnspca2}) when $k = 2$. Given that there are $n(n-1)$ possible coordinate pair combinations, traditional cyclic or random sampling approaches become infeasible as $n$ increases. The semi-greedy strategies we propose combine random and greedy selection methods. These strategies are heuristic methods for selecting the coordinate pairs that are expected to result in the largest decrease in the objective function.  While they also ensuring that every index has the potential to be selected through random sampling.
	
	\subsection{Semi-greedy Strategies for Sparse Index Tracking Task} 
	The sparse index tracking problem (\ref{plb:app:topk}) can be formulated as: $\min_{\x \in \Omega_1} F(\x) = f(\x) +h(\x) + g(\x)$, where $f(\x) = \tfrac{1}{2}\|\A\x - \y\|_2^2$, $h(\x) = \mathcal{I}_{\geq 0}(\x)$, $g(\x) = - \lambda\|\x\|_{[s]}$, and $\Omega_1 = \{\x | \|\x\|_1 = 1\}$. We define the gradient $\nabla f (\x) + \partial h(\x) + \partial g(\x)$ at $\x$, $\x^t$ and $\breve{\x}$ as $\textbf{g}$, $\textbf{g}^t$ and $\breve{\textbf{g}}$. We consider the following optimality measure: $S_1(i) \triangleq \max_{i \neq \tilde{j}}\{\sqrt{L_{i,\tilde{j}}} \min\{\tfrac{1}{L_{i,\tilde{j}}} (\textbf{g}^t_i - \textbf{g}^t_{\tilde{j}}),\x^t_i\}$, where $S_1(i) \in \R$ and the index $\tilde{j}$ is defined as $\tilde{j} \in \argmin_{j = 1,2,...,n}\textbf{g}_j^t$ \citep{beck20142}. The measure $S_1(i)$ is non-negative and equal to $0$ only at critical points. Based on the measure $S_1(i)$, we propose the following greedy index selection strategy. The index $j^t$ corresponds to the variable with the smallest gradient value in $\textbf{g}^t$: $j^t = \argmin_{j} \textbf{g}^t_j$, while the index $i^t$ is chosen to maximize $S_1(i)$: $i^t = \argmax_{i \neq j^t} S_1(i)$. These two chosen indices are considered as the coordinate pair that leads to the largest violation of optimality concerning the optimality measure $S_1(i)$. Therefore, we propose a semi-greedy index selection strategy, which combines greedy selection with random selection, as outlined in Algorithm \ref{alg:semigreedyindex}.
	
	
	\begin{algorithm} [!h]
	\DontPrintSemicolon
	\caption{Semi-greedy Index Selection Strategy for Sparse Index Tracking.}\label{alg:semigreedyindex}
	Give an current feasible solution $\x^t$ and the gradient $\textbf{g}^t$. \\
	\If{$t$ is odd}{
		$i^t$ and $j^t$ are selected randomly.\\
	}
	\Else{
		$j^t = \argmin_{j} \textbf{g}^t_j$\\
		$i^t = \argmax_{i}S_1(i) = \max_{i \neq j^t}\{\sqrt{L_{i,j^t}} \min\{\tfrac{1}{L_{i,j^t}} (\textbf{g}^t_i - \textbf{g}^t_{j^t}),\x^t_i\}$\\
	}
	
	\end{algorithm}

	\subsection{Semi-greedy Strategies for Non-negative Sparse PCA Problem}
	For the non-negative sparse PCA problem, we also propose a semi-greedy strategy to accelerate the BCD-g method. The non-negative sparse PCA problem (\ref{plb:app:nnspca2}) can be rewritten as: 
	\beq
	\min_{\x \in \Omega_2} F(\x) = f(\x) + h(\x) + g(\x), \label{pbl:greedy:nnpca}
	\eeq
	\noi where $f(\x) = -\tfrac{1}{2}\|\A\x\|_2^2$, $h(\x) = \lambda ( \mathcal{I}_{\geq 0}(\x) + \x^{\top}\mathbf{1}^n)$, $g(\x) = - \lambda\|\x\|_{[s]}$, and $\Omega_2 = \{\x | \x^{\top}\x = 1\}$.
	\begin{theorem} \label{theo: greedy2}
		\noi When $\breve{\x}$ is a critical point in Problem (\ref{pbl:greedy:nnpca}), it holds: $\breve{\textbf{g}}\odot\breve{\x} = (\breve{\x}^{\top} \breve{\textbf{g}}\breve{\x})\odot\breve{\x}$, where $\breve{\textbf{g}}$ is denoted as the gradient $ \nabla f(\breve{\x}) + \partial h(\breve{\x}) + \partial g(\breve{\x})$ at $\breve{\x}$.
	\end{theorem}
	\begin{proof}
		We first define the Lagrangian function as: $L(\x, \nu, \mu) \triangleq F(\x) + \nu(\x^{\top}\x - 1) + \langle \mu, -\x \rangle$, where $\nu \in \R$ and $\mu \in \R^n$ are the Lagrange multipliers. The corresponding Karush-Kuhn-Tucker (KKT) conditions for this problem are:
		\beq
		\left\{
		\begin{aligned}
			& \nabla L(\x,\nu,\mu) = \textbf{g} + \nu\x - \mu = 0, \, &(a)\\
			& \x^{\top}\x - 1 = 0, \, &(b)\\
			& -\x \leq 0, \, &(c)\\
			& \nu \geq 0, \mu \geq 0, \, &(d)\\
			& \mu \odot \x = \mathbf{0}_n. \, &(e)
		\end{aligned}
		\right.
		\eeq

		\noi Combining (a), (b), and (e), we obtain $\nu = -\x^{\top}\textbf{g}$. Substituting $\nu$ into (a), we get $\textbf{g} - \x^{\top}\textbf{g}\x - \mu = 0$. Finally, using (e), we have $\textbf{g} \odot \x = (\x^{\top} \textbf{g} \x) \odot \x$.
	\end{proof}
	\noi When $\breve{\x}$ is a critical point, the following condition holds: $\breve{\textbf{g}}\odot\breve{\x} = (\breve{\x}^{\top} \breve{\textbf{g}}\breve{\x})\odot\breve{\x}$. Therefore, we introduce the following optimality measure: $S_2(\x,\textbf{g}) \triangleq |\textbf{g}\odot\x - (\x^{\top} \textbf{g}\x)\odot\x|$. For any critical point $\breve{\x}$, $S_2(\breve{\x},\breve{\textbf{g}})$ is equal to a zero vector $\mathbf{0}_n$. We identify the coordinate pair that most and least violates the optimality measure as the candidate coordinate pair. Letting $\z = S_2(\x^t,\textbf{g}^t)$, the greedy index selection strategy is introduced as following. The index $i^t$ is selected as the coordinate corresponding to the largest value in $\z$: $i^t = \argmax_{i} \z_i$ representing the coordinate that violates the optimality measure the most. In contrast, the index $j^t$ is chosen as the coordinate corresponding to the smallest value in $\z$: $j^t = \argmin_{j} \z_j$, indicating the coordinate that violates the optimality measure the least. Finally, we integrate the semi-greedy and random index selection strategy for the non-negative sparse PCA problem, as shown in Algorithm \ref{alg:semigreedypca}.
	
	\begin{algorithm} [!h]
		\DontPrintSemicolon
		\caption{Semi-greedy Index Selection Strategy for Non-negative Sparse PCA.}\label{alg:semigreedypca}
		Give an current feasible solution $\x^t$ and the gradient $\textbf{g}^t$. \\
		\If{$t$ is odd}{
			$i^t$ and $j^t$ are selected randomly.\\
		}
		\Else{
			$\z = S_2(\x^t,\textbf{g}^t) = |\textbf{g}^t\odot\x^t - ({\x^t}^{\top} {\textbf{g}^t}\x^t)\odot\x^t|$\\
			$i^t = \argmax_{i} \z_i$\\
			$j^t = \argmin_{j} \z_j$
		}
	\end{algorithm}

	\section{Geometric Analyses of CWS-Point for Different Applications}\label{sec:geo}
	This section presents a geometric analysis of \textbf{CWS}-points for the three applications outlined in Section \ref{sec:app}. The analysis primarily focuses on utilizing the necessary optimality conditions of \textbf{CWS}-points to derive insights into the solution properties. Specifically, we investigate two key characteristics: the exactness property and the extreme point property, which provide deeper understanding of the solution behavior under the considered framework.
	\subsection{Sparse Index Tracking Task: An Exactness Property}
	
	We provide an exactness property of our algorithm for the following sparse index tracking problem: 
	\beq
	\min_{\x}\, f(\x) +  \mathcal{I}_{\geq 0}(\x) - \lambda\|\x\|_{[s]},\ s.t.\ \|\x\|_1 = 1,\label{plb:geo:topk}
	\eeq
	where $f(\x) = \tfrac{1}{2}\|\A\x-\y\|_2^2$. We introduce the theorem concerning the exactness property of the sparse index tracking problem.
	\begin{theorem}\label{theo:excat:index}
		\textbf{(Proof in Appendix \ref{app:theo:excat:index})}
		We let $\ddot{\x}$ be any \textbf{CWS}-point for Problem (\ref{plb:geo:topk}). When $\lambda$ is sufficient large such that $\lambda > 2\|\nabla f(\ddot{\x})\|_{\infty}$, we have:
		
		\begin{enumerate}[label=\textbf{(\roman*)}, itemsep=1pt, topsep=1pt, parsep=0pt, partopsep=0pt]
			\item $\ddot{\x}_i = 0, \,\forall i\notin \mathcal{T}(\x)$, i.e., $\|\ddot{\x}\|_1 = \sum_{i=1}^s |\ddot{\x}_{[i]}|$.
			
			\item $\nabla_i f(\ddot{\x}) = \nabla_j f(\ddot{\x}), \, \forall i,\ j\in \mathcal{T}(\ddot{\x})$.
		\end{enumerate}
		
		\noi Here, $\mathcal{T}(\ddot{\x})$ is the indices of non-zero components of $\ddot{\x}$ in magnitude. 
	\end{theorem}
	
	\noi \textbf{Remarks.} When $\lambda$ exceeds a specific threshold $2\|\nabla f(\ddot{\x})\|_{\infty}$ in Problem (\ref{plb:geo:topk}), any \textbf{CWS}-point of the nonconvex problem satisfies to the sparse constraint $\|\ddot{\x}\|_1 - \|\ddot{\x}\|_{[s]} = 0$. Nevertheless, existing methods such as PSG, MSCR, and PDCA only identify the critical points of Problem (\ref{plb:geo:topk}), without guaranteeing satisfaction of the sparse constraint for all critical points. Thus, the results we establish in Theorem \ref{theo:excat:index} demonstrate superior performance compared to these methods. Moreover, we observe that when $\lambda$ surpasses the threshold, all elements in $\nabla_{\mathcal{T}(\ddot{\x})} f(\x)$ are equal, which  $\mathcal{T}(\ddot{\x})$ represents the set of non-zero elements.

	\subsection{DC Penalized Binary Optimization: An Extreme Point Property}
	
	This subsection provides an extreme point property of \textbf{CWS}-point for the DC penalized binary problem:
	\beq
	\min_{\x}\,  \hat{f}(\x) + \mathcal{I}_{[-1,1]}(\x),\ s.t.\ \x^{\top}\textbf{1} = c, \label{plb:geo:bin}
	\eeq
	\noi where $\hat{f}(\x) = \p^{\top}\x - \tfrac{1}{2}\x^{\top}\hat{\Q}\x$ and $\hat{\Q}\succeq 0$. Our main results are outlined in the following theorem.

	\begin{theorem}\label{theo:extreme:bin}
		\textbf{(Proof in Appendix \ref{app:theo:extreme:bin})} We let $\bar{\Q} = -\hat{\Q} + \theta \I_n$ and assume $\bar{\Q}_{ii} + \bar{\Q}_{jj} - 2\bar{\Q}_{ij} \leq 0 , \forall i, j \in [n]$. We let $\ddot{\x}$ be any \textbf{CWS}-point for Problem (\ref{plb:geo:bin}). It holds that $|\ddot{\x}_r| = 1$ for all $r \in [n]$.
	\end{theorem}

	\noi \textbf{Remarks.} The extreme point property indicates that the solutions $\x$ of problem (\ref{plb:geo:bin}) lie on the boundary of the feasible region $\x \in \{-1, +1\}^n$. The analysis demonstrates that the \textbf{CWS}-points $\ddot{\x}$ consistently satisfy these boundary conditions $\ddot{\x} \in \{-1, +1\}^n$, a feature that is not guaranteed for all critical points.
	
	\subsection{DC Penalized Binary Optimization: An Exactness Property}
	This subsection establishes the exactness property of \textbf{CWS}-points for the DC-penalized binary optimization problem defined as follows:
	\beq
	\min_{\x}\, f(\x) + \mathcal{I}_{[-1,1]}(\x) - \lambda\|\x\|_2,\ s.t.\ \x^{\top}\textbf{1} = c,\label{plb:geo:bin2}
	\eeq
	where $f(\x) = \tfrac{1}{2}\|\A\x-\y\|_2^2$.  We now present the theorem of the exactness property.
	\begin{theorem}\label{theo:exact:bin}
		\textbf{(Proof in Appendix \ref{app:theo:exact:bin})}
		We let $\bar{\Q} = \A^{\top}\A + \theta \I_n$ and define $\phi \triangleq \max(\bar{\Q}) - \min(\bar{\Q})$. When $\ddot{\x}$ is the \textbf{CWS}-point of Problem (\ref{plb:geo:bin2}), we have:
		\begin{enumerate}[label=\textbf{(\roman*)}, itemsep=1pt, topsep=1pt, parsep=0pt, partopsep=0pt]
			\item When $\lambda > \phi\sqrt{n+2}$, it holds that $\ddot{\x} \in \{-1, +1\}^n$ and $\sqrt{n} = \|\ddot{\x}\|$.
			
			\item When $\lambda > \phi\sqrt{n}$, it holds that $\ddot{\x}_i \neq 0$ for all $i \in [n]$.
		\end{enumerate}
	\end{theorem}
	\noi \textbf{Remarks.} When $\lambda$ surpasses a certain threshold, any \textbf{CWS}-point $\ddot{\x}$ of Problem (\ref{plb:geo:bin2}) achieves exactness by satisfying the binary constraint $\ddot{\x} \in \{-1, +1\}^n$ and condition $\ddot{\x}_i \neq 0,\ \forall i$. In contrast, existing methods such as PSG, MSCR, and PDCA lack the ability to guarantee the binary constraints. This results highlight the robustness and precision of the proposed methods.
	

	\section{Experiments}\label{sec:exp}
	In this section, we first evaluate the performance of our BCD method for obtaining global solutions under different penalty parameters $\lambda$, sparsity $s$, and sum $c$ across the four applications in Section \ref{sec:app}. We denote the BCD method that globally solves the subproblem (\ref{plb:subproblem}) as BCD-g. By varying these parameters, we analyze their influence on solution quality and convergence speed, and assess the sensitivity and robustness of the BCD-g method. Additionally, we perform an experiment on our BCD method aimed at achieving local solutions for the subproblem (\ref{plb:subproblemlocal}) in the SIT problem, which we denote as BCD-l-$k$, with $k$ representing the number of elements optimized in each iteration. We compare the objective values generated by BCD-l-$k$ and other methods with varying $k$.

	\subsection{Experimental Settings}
	
	\textbf{Datasets}. For the sparse index tracking task, we collect daily returns for the NASDAQ 100 and the Standard \& Poor's 500 indices from 2016 to 2020. We consider two types of datasets for the matrix $\A \in \R^{m\times n}$: 'NAS-year' and 'SP-year' (e.g., 'NAS-16'). The detailed information, includes the number of components $n$ and the number of trading days $m$, is shown in Table \ref{indexdatasets}. To construct the data matrix $\A \in \R^{m\times n}$ in non-negative sparse PCA problem and the DC penalized binary optimizations, we consider five public datasets, which consist of random and real-world datasets: 'randn', 'TDT2', '20News', 'Cifar', and 'MNIST'. For random datasets, we generate a matrix $\A$ using the function randn($m$, $n$), which produces a standard Gaussian random matrix of size $m \times n$. For real-world datasets, we randomly and uniformly select a subset of $m \times n$ elements for analysis. We denote the different types of datasets for varying size as 'TDT2-a', 'TDT2-b' for example. The type-a datasets are sampled from size $(m, n) = (256, 256)$, while the type-b datasets are chosen from sizes $(m, n) = (512, 512)$.


	\begin{table}[h]
		\centering
		\resizebox{\textwidth}{!}{%
			\begin{tabular}{|c|c|c|c|c|c|c|c|c|c|c|} 
				\hline
				& NAS-16 & NAS-17 & NAS-18 & NAS-19 & NAS-20 & SP-16 & SP-17 & SP-18 & SP-19 & SP-20 \\ \hline
				$m$ & 252 & 251 & 251 & 252 & 253 & 252 & 251 & 251  & 252  & 253 \\ \hline
				$n$ & 84 & 78 & 77 & 80 & 88 & 419  & 431  & 447  & 458  & 472 \\ \hline
			\end{tabular}%
		}
		\caption{Sizes of sparse index tracking datasets.}
		\label{indexdatasets}
	\end{table}

	\noi\textbf{Implementations}.
	All experiments are implemented in MATLAB on an Intel 3.40 GHz CPU with 32 GB RAM. Due to the large element-wise loops requirement in the breakpoint searching strategy, we develop the procedure in C++ for higher efficiency. 
	
	\subsection{Experiment Results on BCD-g Method} \label{sec:exp:subsec:indextracking}
	
	We compare our BCD-g method with the following methods to solve the sparse index tracking (SIT) problem (\ref{plb:app:topk}), non-negative spare PCA (NNSPCA) problem (\ref{plb:app:nnspca2}) and two DC penalized binary (DCPB1, DCPB2) optimization problems when $k = 2$. \textbf{(i)} Projected Sub-gradient Descent (PSG). It considers the following iteration: $\x^{t+1} = P_{\Omega}(\x^{t} - \alpha_{t}G^t$) where $P_{\Omega}(\cdot)$ is the projection function. \textbf{(ii)} Multi-Stage Convex Relaxation (MSCR). It solves the following subproblem: $\z^{t+1} =\text{arg}\ \underset{\z}{\text{min}}\ f(\z) + \la \z-\z^{t},\partial g(\z^t)\ra,\ \x^{t+1}=P_{\Omega}(\z^{t+1}),\ f(\z)=\tfrac{1}{2}\|\A\z-\textbf{y}\|^2_2,\ g(\z)=\lambda\|\z\|_{[s]}$. \textbf{(iii)} Proximal DC Algorithm (PDCA). $\x^{t+1}$ is updated by: $\z^{t+1} =\text{arg}\ \underset{\z}{\text{min}}\ \tfrac{L}{2}\|\z-\z^{t}\|^2_2+\la\nabla f(\z), \z-\z^{t} \ra + \la \z-\z^{t},\partial g(\z^t)\ra,\ \x^{t+1}=P_{\Omega}(\z^{t+1})$.
	In our comparative analysis, we evaluate the performance of the BCD-g method against PSG, MSCR, and PDCA by varying the penalty $\lambda$ over a range from $1$ to $10^{4}$ to assess the effect of different penalty terms.

	Since Problem (\ref{plb:app:bin2}) requires a sufficiently large $\lambda$ to ensure $\hat{\Q} \succ 0$, we do not compare performance across varying $\lambda$ in DCPB1. The step size $\alpha_{t}$ for PSG is set to $0.01 \cdot \sqrt{t^{-1}}$ at $t$-th iteration. For MSCR, each subproblem is solved 10 times per iteration. The sparsity $s$ of $\x$ is fixed at 30 across all datasets. All the results are averaged over ten executions with different initializations of $\x$. The best results are highlighted in red in Table \ref{tab:lambdaindex}. A smaller objective value indicates a better local minimum. Therefore, our observations from the comparative analysis are as follows: \textbf{(i)} Our BCD-g method consistently achieves the lowest objective function value when converging to a coordinate-wise stationary point in most cases. \textbf{(ii)} The objective values increase with larger $\lambda$, indicating that the stricter constraint leads to higher objective values. Additionally, The BCD-g method performs better under stricter constraint. \textbf{(iii)} MSCR and PDCA consistently yield the same values, suggesting that they converge to the same critical point.

	In our performance comparison, we also investigate the effect of varying sparsity $s$ and sum $c$ across 20 datasets. We set $\lambda = 1000$ in SIT while $\lambda = 10000$ in NNSPCA, DCPB1 and DCPB2 to ensure that the solutions $\x$ do not violate the sparsity constraint due to the heavy penalties. We evaluate performance over a range of the sparsity $s = \{5, 10, 15, \dots, 55, 60\}$. For the DC penalized binary optimizations, we vary $c = \{10, 20, 30, \dots, 90, 100\}$. The objective values across datasets with varying sparsity $s$ or sum $c$ are visualized in Figure \ref{fig:diff1} and Figure \ref{fig:diff2}. Our BCD-g method consistently achieves the lowest objective value across different values of $s$ or $c$. 
	
	\subsection{Experiment Results on BCD-l-$k$ Method}
	
	In the SIT task, we vary $k = \{2, 3, 5, 10, 15, 20, 50\}$ to compare the objective values with BCD-g, BCD-l-$k$, PSG, MSCR, and PDCA. For a certain state of $\d_{\B}$, we employ projected subgradient method in BCD-l-$k$ method to solve Problem (\ref{plb:blockkJfunction:index2}). We set $\lambda = 1000$ and $s = 10$. As shown in Table \ref{tbl:blockk}, the objective values in bold indicate that two BCD methods outperform other three methods. We observe that the BCD-g method always yields the best objective value, since it solves the subproblem (\ref{plb:subproblem}) globally. The results also demonstrate that BCD-l-$k$ method performs better and consistently finds better local minima compared to PSG, MSCR and PDCA when $k$ is 10 or 20. However, when $k=50$, the BCD-l-$50$ sometimes fails to provide better objective values, as the convergence time increases significantly.
	
	\newcommand{\textbr}[1]{\textbf{\textcolor{red}{#1}}}
	
	\begin{table}[htbp]
		\centering
		\resizebox{\textwidth}{!}{ 
			\begin{tabular}{|c|c|c|c|c|c|c|c|c|c|c|} 
				\hline
				& PSG & MSCR & PDCA & BCD-g & BCD-l-2 & BCD-l-3 & BCD-l-5 & BCD-l-10 & BCD-l-20 & BCD-l-50  \\ \hline
				NAS-16 & 1.22e+01 & 1.84e+01 & 1.84e+01  & \textbr{7.17e+00} & 3.08e+01 & 1.30e+01 & \textbr{1.10e+01} & \textbr{8.85e+00} & \textbr{9.01e+00} & 1.27e+01 \\ \hline
				NAS-17 & 1.01e+01 & 1.64e+01 & 1.64e+01 & \textbr{5.14e+00} & 2.61e+01 & \textbr{9.13e+00} & \textbr{7.46e+00} & \textbr{5.90e+00} & \textbr{5.98e+00} & 1.09e+01   \\ \hline
				NAS-18 & 1.73e+01 & 2.52e+01 & 2.52e+01 & \textbr{7.06e+00} & 3.74e+01 & 1.94e+01 & \textbr{1.20e+01} & \textbr{1.03e+01} & \textbr{1.08e+01} & \textbr{1.56e+01}    \\ \hline
				NAS-19 & 1.29e+01 & 2.19e+01 & 2.19e+01 & \textbr{6.16e+00} & 3.62e+01 & \textbr{1.17e+01} & \textbr{9.25e+00} & \textbr{7.57e+00} & \textbr{8.30e+00} & \textbr{1.11e+01}    \\ \hline
				NAS-20 & 2.43e+01 & 4.87e+01 & 4.87e+01 & \textbr{1.24e+01} & 8.64e+01 & 3.09e+01 & \textbr{2.42e+01} & \textbr{1.93e+01} & \textbr{2.11e+01} & 2.54e+01   \\ \hline

				SP-16 & 1.49e+01 & 1.80e+01 & 1.80e+01 & \textbr{6.50e+00} & 1.66e+02 & 2.93e+01 & 1.20e+01 & \textbr{9.99e+00} & \textbr{8.83e+00} & \textbr{1.04e+01}  \\ \hline
				SP-17 & 1.12e+01 & 1.38e+01 & 1.38e+01 & \textbr{5.05e+00} & 1.24e+02 & 1.01e+01 & \textbr{7.24e+00} & \textbr{6.22e+00} & \textbr{5.74e+00} & \textbr{9.65e+00}    \\ \hline
				SP-18 & 1.37e+01 & 2.22e+01 & 2.22e+01 & \textbr{7.81e+00} & 1.61e+02 &  4.19e+01 & 1.59e+01 & \textbr{1.32e+01} & \textbr{9.91e+00} & \textbr{1.22e+01}     \\ \hline
				SP-19 & 1.37e+01 & 1.29e+01 & 2.19e+01  & \textbr{7.33e+00} & 1.78e+02 &  3.88e+01 & \textbr{1.19e+01} & \textbr{1.03e+01} & \textbr{9.74e+00} & 1.43e+01   \\ \hline
				SP-20 & 3.47e+01 & 3.55e+01 & 3.60e+01 & \textbr{1.64e+01} & 4.42e+02 & 1.21e+02 & 5.88e+01 & \textbr{2.57e+01} & \textbr{2.11e+01} & \textbr{2.74e+01}    \\ \hline
			\end{tabular}
		}
		\caption{Comparisons of objective values of five methods with varying $k$.}	\label{tbl:blockk}
	\end{table}

	\begin{table}[htbp]
		\tiny
		\renewcommand{\arraystretch}{0.4}
		\centering
		\begin{tabular}{|c|cccc|cccc|}
			\hline
			&  PSG  & MSCR & PDCA & BCD-g &  PSG  & MSCR & PDCA & BCD-g\\ 
			\hline
			
			$\lambda$&\multicolumn{4}{c|}{results of SIT in NAS-16} & \multicolumn{4}{c|}{results of SIT in SP-16}\\
			\hline
			$10^0 $ & \textbr{1.31e+00} & 1.32e+00 & 1.32e+00 & 1.32e+00 & 3.97e-01 & 4.00e-01 & 4.15e-01 & \textbr{3.82e-01} \\ 
			$10^1 $ & 2.40e+00 & 2.36e+00 & 2.35e+00 & \textbr{1.83e+00} & 1.62e+00 & 2.03e+00 & 2.00e+00 & \textbr{1.25e+00} \\ 
			$10^2 $ & 4.39e+00 & 5.19e+00 & 5.19e+00 & \textbr{1.79e+00} & 4.16e+00 & 4.33e+00 & 4.43e+00 & \textbr{1.34e+00} \\ 
			$10^3 $ & 4.38e+00 & 5.19e+00 & 5.19e+00 & \textbr{1.91e+00} & 4.21e+00 & 4.37e+00 & 4.37e+00 & \textbr{1.36e+00} \\ 
			$10^4 $ & 4.38e+00 & 5.19e+00 & 5.19e+00 & \textbr{1.75e+00} & 4.21e+00 & 4.37e+00 & 4.37e+00 & \textbr{1.36e+00} \\ 
			\hline
			\hline
			
			$\lambda$&\multicolumn{4}{c|}{results of SIT in NAS-17}&\multicolumn{4}{c|}{results of SIT in SP-17}\\
			\hline
			$10^0 $ & 1.24e+00 & 1.24e+00 & 1.24e+00 & \textbr{1.23e+00} & 3.95e-01 & 3.95e-01 & 3.99e-01 & \textbr{3.48e-01} \\ 
			$10^1 $ & 1.78e+00 & 2.24e+00 & 2.14e+00 & \textbr{1.47e+00} & 1.57e+00 & 1.89e+00 & 1.84e+00 & \textbr{1.00e+00} \\ 
			$10^2 $ & 3.04e+00 & 3.60e+00 & 3.67e+00 & \textbr{1.42e+00} & 3.55e+00 & 3.61e+00 & 3.53e+00 & \textbr{1.01e+00} \\ 
			$10^3 $ & 3.04e+00 & 3.67e+00 & 3.67e+00 & \textbr{1.44e+00} & 3.43e+00 & 3.61e+00 & 3.61e+00 & \textbr{1.13e+00} \\ 
			$10^4 $ & 3.04e+00 & 3.67e+00 & 3.67e+00 & \textbr{1.44e+00} & 3.47e+00 & 3.61e+00 & 3.61e+00 & \textbr{1.02e+00} \\ 
			\hline
			\hline
			
			$\lambda$&\multicolumn{4}{c|}{results of SIT in NAS-18}&\multicolumn{4}{c|}{results of SIT in SP-18}\\
			\hline
			$10^0 $ & \textbr{1.33e+00} & \textbr{1.33e+00} & \textbr{1.33e+00} & \textbr{1.33e+00} & \textbr{3.71e-01} & 4.19e-01 & 4.25e-01 & 3.81e-01 \\ 
			$10^1 $ & 2.11e+00 & 2.40e+00 & 2.44e+00 & \textbr{1.78e+00} & 2.07e+00 & 2.23e+00 & 2.18e+00 & \textbr{1.34e+00} \\ 
			$10^2 $ & 3.91e+00 & 5.92e+00 & 5.92e+00 & \textbr{1.81e+00} & 5.67e+00 & 4.60e+00 & 4.60e+00 & \textbr{1.64e+00} \\ 
			$10^3 $ & 3.91e+00 & 5.91e+00 & 5.92e+00 & \textbr{1.78e+00} & 4.99e+00 & 4.60e+00 & 4.60e+00 & \textbr{1.62e+00} \\ 
			$10^4 $ & 3.91e+00 & 5.92e+00 & 5.92e+00 & \textbr{1.88e+00} & 4.99e+00 & 4.51e+00 & 4.60e+00 & \textbr{1.53e+00} \\ 
			\hline
			\hline
			
			$\lambda$&\multicolumn{4}{c|}{results of SIT in NAS-19}&\multicolumn{4}{c|}{results of SIT in SP-19}\\
			\hline
			$10^0 $ & 1.06e+00 & 1.06e+00 & 1.06e+00 & \textbr{1.05e+00} & \textbr{3.48e-01} & 3.90e-01 & 3.94e-01 & 3.65e-01 \\ 
			$10^1 $ & 2.06e+00 & 2.14e+00 & 2.11e+00 & \textbr{1.46e+00} & 1.75e+00 & 2.05e+00 & 2.02e+00 & \textbr{1.32e+00} \\ 
			$10^2 $ & 3.04e+00 & 4.92e+00 & 4.92e+00 & \textbr{1.39e+00} & 4.85e+00 & 4.79e+00 & 4.79e+00 & \textbr{1.34e+00} \\ 
			$10^3 $ & 3.04e+00 & 4.92e+00 & 4.92e+00 & \textbr{1.49e+00} & 4.87e+00 & 4.87e+00 & 4.79e+00 & \textbr{1.33e+00} \\ 
			$10^4 $ & 3.04e+00 & 4.92e+00 & 4.92e+00 & \textbr{1.48e+00} & 4.87e+00 & 4.79e+00 & 4.79e+00 & \textbr{1.35e+00} \\ 
			\hline
			\hline
			
			$\lambda$&\multicolumn{4}{c|}{results of SIT in NAS-20}&\multicolumn{4}{c|}{results of SIT in SP-20}\\
			\hline
			$10^0 $ & 2.29e+00 & 2.30e+00 & 2.30e+00 & \textbr{2.28e+00} & 2.53e+01 & 5.28e-01 & 5.47e-01 & \textbr{5.16e-01} \\ 
			$10^1 $ & 3.38e+00 & 3.58e+00 & 3.47e+00 & \textbr{3.23e+00} & 3.36e+01 & 2.45e+00 & 2.56e+00 & \textbr{2.24e+00} \\ 
			$10^2 $ & 6.13e+00 & 7.02e+00 & 7.14e+00 & \textbr{3.46e+00} & 1.02e+02 & 7.59e+00 & 7.75e+00 & \textbr{3.37e+00} \\ 
			$10^3 $ & 8.17e+00 & 7.67e+00 & 7.67e+00 & \textbr{3.35e+00} & 7.46e+00 & 7.52e+00 & 7.53e+00 & \textbr{2.84e+00} \\ 
			$10^4 $ & 8.17e+00 & 7.67e+00 & 7.67e+00 & \textbr{3.41e+00} & 7.46e+00 & 7.50e+00 & 7.53e+00 & \textbr{3.21e+00} \\ 
			\hline
			\hline

			$\lambda$&\multicolumn{4}{c|}{results of NNSPCA randn-a}&\multicolumn{4}{|c|}{results of NNSPCA in randn-b}\\
			\hline
			$10^0 $ & \textbr{-3.59e+02} & \textbr{-3.59e+02} & \textbr{-3.59e+02} & -3.55e+02 & \textbr{-7.27e+02} & \textbr{-7.27e+02} & \textbr{-7.27e+02} & -7.26e+02 \\ 
			$10^1 $ & -3.23e+02 & -3.23e+02 & -3.23e+02 & \textbr{-3.24e+02} & -6.56e+02 & -6.56e+02 & -6.56e+02 & \textbr{-6.58e+02} \\ 
			$10^2 $ & -2.06e+02 & -2.06e+02 & -2.06e+02 & \textbr{-2.95e+02} & -3.75e+02 & -3.75e+02 & -3.75e+02 & \textbr{-4.94e+02} \\ 
			$10^3 $ & -2.06e+02 & -2.06e+02 & -2.06e+02 & \textbr{-2.92e+02} & -3.72e+02 & -3.72e+02 & -3.72e+02 & \textbr{-4.90e+02} \\ 
			$10^4 $ & -2.06e+02 & -2.06e+02 & -2.06e+02 & \textbr{-2.91e+02} & -3.72e+02 & -3.72e+02 & -3.72e+02 & \textbr{-4.93e+02} \\ 
			\hline
			\hline
			
			$\lambda$&\multicolumn{4}{c|}{results of NNSPCA in TDT2-a}&\multicolumn{4}{|c|}{results of NNSPCA in TDT2-b}\\
			\hline
			$10^0 $ & -3.90e-01 & -3.94e-01 & -3.94e-01 & \textbr{-1.11e+00} & -9.30e-02 & -1.04e-01 & -1.04e-01 & \textbr{-1.90e+00} \\ 
			$10^1 $ & -3.90e-01 & -3.94e-01 & -3.94e-01 & \textbr{-1.11e+00} & -9.18e-02 & -1.04e-01 & -1.04e-01 & \textbr{-1.90e+00} \\ 
			$10^2 $ & -3.91e-01 & -3.94e-01 & -3.94e-01 & \textbr{-1.11e+00} & -9.30e-02 & -1.04e-01 & -1.04e-01 & \textbr{-1.90e+00} \\ 
			$10^3 $ & -3.90e-01 & -3.94e-01 & -3.94e-01 & \textbr{-1.11e+00} & -9.30e-02 & -1.04e-01 & -1.04e-01 & \textbr{-1.90e+00} \\ 
			$10^4 $ & -3.89e-01 & -3.94e-01 & -3.94e-01 & \textbr{-1.11e+00} & -9.30e-02 & -1.04e-01 & -1.04e-01 & \textbr{-1.90e+00} \\ 
			\hline
			\hline
			
			$\lambda$&\multicolumn{4}{c|}{results of NNSPCA in 20News-a}&\multicolumn{4}{|c|}{results in of NNSPCA 20News-b}\\
			\hline
			$10^0 $ & -6.60e-02 & -8.36e-02 & -8.36e-02 & \textbr{-6.00e-01} & -1.05e-04 & -1.47e-03 & -1.47e-03 & \textbr{-5.00e-01} \\ 
			$10^1 $ & -6.57e-02 & -8.36e-02 & -8.36e-02 & \textbr{-6.00e-01} & -1.05e-04 & -1.47e-03 & -1.47e-03 & \textbr{-5.00e-01} \\ 
			$10^2 $ & -6.59e-02 & -8.36e-02 & -8.36e-02 & \textbr{-6.00e-01} & -1.05e-04 & -1.47e-03 & -1.47e-03 & \textbr{-4.51e-01} \\ 
			$10^3 $ & -6.58e-02 & -8.36e-02 & -8.36e-02 & \textbr{-6.00e-01} & -1.05e-04 & -1.47e-03 & -1.47e-03 & \textbr{-5.00e-01} \\ 
			$10^4 $ & -6.58e-02 & -8.36e-02 & -8.36e-02 & \textbr{-6.00e-01} & -1.05e-04 & -1.47e-03 & -1.47e-03 & \textbr{-5.00e-01} \\ 
			\hline
			\hline
			
			$\lambda$&\multicolumn{4}{c|}{results of NNSPCA in Cifar-a}&\multicolumn{4}{|c|}{results of NNSPCA in Cifar-b}\\
			\hline
			$10^0 $ & \textbr{-8.88e+03} & \textbr{-8.88e+03} & \textbr{-8.88e+03} & \textbr{-8.88e+03} & \textbr{-3.53e+04} & \textbr{-3.53e+04} & \textbr{-3.53e+04} & \textbr{-3.53e+04} \\ 
			$10^1 $ & \textbr{-8.76e+03} & \textbr{-8.76e+03} & \textbr{-8.76e+03} & \textbr{-8.76e+03} & \textbr{-3.51e+04} & \textbr{-3.51e+04} & \textbr{-3.51e+04} & \textbr{-3.51e+04} \\ 
			$10^2 $ & \textbr{-7.53e+03} & \textbr{-7.53e+03} & \textbr{-7.53e+03} & \textbr{-7.53e+03} & \textbr{-3.32e+04} & \textbr{-3.32e+04} & \textbr{-3.32e+04} & \textbr{-3.32e+04} \\ 
			$10^3 $ & -1.03e+03 & -1.03e+03 & -1.03e+03 & \textbr{-1.41e+03} & -1.45e+04 & -1.45e+04 & -1.45e+04 & \textbr{-1.47e+04} \\ 
			$10^4 $ & -1.03e+03 & -1.03e+03 & -1.03e+03 & \textbr{-1.42e+03} & -2.08e+03 & -2.08e+03 & -2.08e+03 & \textbr{-2.81e+03} \\ 
			\hline
			\hline
			
			$\lambda$&\multicolumn{4}{c|}{results of NNSPCA in MNIST-a}&\multicolumn{4}{|c|}{results of NNSPCA in MNIST-b}\\
			\hline
			$10^0 $ & \textbr{-1.83e+03} & \textbr{-1.83e+03} & \textbr{-1.83e+03} & \textbr{-1.83e+03} & \textbr{-7.06e+03} & \textbr{-7.06e+03} & \textbr{-7.06e+03} & \textbr{-7.06e+03} \\ 
			$10^1 $ & \textbr{-1.78e+03} & \textbr{-1.78e+03} & \textbr{-1.78e+03} & \textbr{-1.78e+03} & \textbr{-6.96e+03} & \textbr{-6.96e+03} & \textbr{-6.96e+03} & \textbr{-6.96e+03} \\ 
			$10^2 $ & \textbr{-1.35e+03} & \textbr{-1.35e+03} & \textbr{-1.35e+03} & \textbr{-1.35e+03} & \textbr{-6.01e+03} & \textbr{-6.01e+03} & -6.00e+03 & \textbr{-6.01e+03} \\ 
			$10^3 $ & -2.43e+02 & -2.43e+02 & -2.43e+02 & \textbr{-9.87e+02} & -4.71e+02 & -4.71e+02 & -4.71e+02 & \textbr{-2.17e+03} \\ 
			$10^4 $ & -2.43e+02 & -2.43e+02 & -2.43e+02 & \textbr{-9.93e+02} & -4.71e+02 & -4.71e+02 & -4.71e+02 & \textbr{-2.14e+03} \\ 
			\hline
			\hline

			$\lambda$&\multicolumn{4}{c|}{results of DCPB2 in randn-a}&\multicolumn{4}{|c|}{results of DCPB2 in randn-b}\\
			\hline
			$10^0 $ & 6.05e+01 & 4.96e+02 & 5.49e+01 & \textbr{3.97e+01} & 3.95e+01 & 5.21e+02 & \textbr{3.58e+01} & 3.65e+01 \\ 
			$10^1 $ & 6.24e+02 & 6.24e+02 & 2.01e+02 & \textbr{1.13e+02} & 7.12e+02 & 7.12e+02 & 2.36e+02 & \textbr{1.85e+02} \\ 
			$10^2 $ & 1.89e+03 & 1.89e+03 & 1.51e+03 & \textbr{6.72e+02} & 2.63e+03 & 2.63e+03 & 2.13e+03 & \textbr{1.07e+03} \\ 
			$10^3 $ & 1.45e+04 & 1.45e+04 & 1.37e+04 & \textbr{5.07e+03} & 2.18e+04 & 2.18e+04 & 2.07e+04 & \textbr{1.80e+04} \\ 
			$10^4 $ & 3.23e+04 & 3.23e+04 & 3.21e+04 & \textbr{7.02e+03} & 1.26e+05 & 1.27e+05 & 1.25e+05 & \textbr{4.60e+04} \\ 
			\hline
			\hline
			
			$\lambda$&\multicolumn{4}{c|}{results of DCPB2 in TDT2-a}&\multicolumn{4}{|c|}{results of DCPB2 in TDT2-b}\\
			\hline
			$10^0 $ & 8.28e+00 & 9.40e+00 & 8.56e+00 & \textbr{9.05e-01} & 3.88e+00 & 3.46e+00 & 3.91e+00 & \textbr{5.63e-01} \\ 
			$10^1 $ & 5.90e+00 & 6.02e+00 & 5.99e+00 & \textbr{2.46e+00} & 3.51e+00 & 3.34e+00 & 3.34e+00 & \textbr{1.25e+00} \\ 
			$10^2 $ & 5.18e+00 & 5.18e+00 & 5.18e+00 & \textbr{2.62e+00} & 3.51e+00 & 3.51e+00 & 3.51e+00 & \textbr{1.77e+00} \\ 
			$10^3 $ & 5.32e+00 & 5.32e+00 & 5.32e+00 & \textbr{2.62e+00} & 3.53e+00 & 3.53e+00 & 3.53e+00 & \textbr{1.77e+00} \\ 
			$10^4 $ & 5.32e+00 & 5.32e+00 & 5.32e+00 & \textbr{2.62e+00} & 3.54e+00 & 3.54e+00 & 3.54e+00 & \textbr{1.77e+00} \\ 
			\hline
			\hline
			
			$\lambda$&\multicolumn{4}{c|}{results of DCPB2 in 20News-a}&\multicolumn{4}{|c|}{results of DCPB2 in 20News-b}\\
			\hline 
			$10^0 $ & 1.59e+00 & 2.30e+00 & 1.88e+00 & \textbr{3.03e-01} & 5.64e-01 & 5.66e-01 & 5.64e-01 & \textbr{6.73e-02} \\ 
			$10^1 $ & 1.29e+00 & 1.02e+00 & 9.91e-01 & \textbr{2.92e-01} & 7.73e-01 & 5.60e-01 & 5.60e-01 & \textbr{4.70e-01} \\ 
			$10^2 $ & 1.05e+00 & 1.05e+00 & 1.05e+00 & \textbr{2.92e-01} & 5.62e-01 & 5.62e-01 & 5.62e-01 & \textbr{4.93e-01} \\ 
			$10^3 $ & 1.04e+00 & 1.04e+00 & 1.04e+00 & \textbr{2.92e-01} & 5.62e-01 & 5.62e-01 & 5.62e-01 & \textbr{4.93e-01} \\ 
			$10^4 $ & 1.04e+00 & 1.04e+00 & 1.04e+00 & \textbr{2.92e-01} & 5.62e-01 & 5.62e-01 & 5.62e-01 & \textbr{4.93e-01} \\ 
			\hline
			\hline
			
			$\lambda$&\multicolumn{4}{c|}{results of DCPB2 in Cifar-a}&\multicolumn{4}{|c|}{results of DCPB2 in Cifar-b}\\
			\hline
			$10^0 $ & \textbr{1.48e+04} & 1.53e+04 & 1.53e+04 & 1.66e+04 & \textbr{2.49e+04} & 2.55e+04 & 2.55e+04 & 2.73e+04 \\ 
			$10^1 $ & 1.56e+04 & \textbr{1.45e+04} & 1.46e+04 & 1.66e+04 & 2.85e+04 & \textbr{2.62e+04} & \textbr{2.62e+04} & 2.75e+04 \\ 
			$10^2 $ & 3.15e+04 & 2.00e+04 & 1.96e+04 & \textbr{1.72e+04} & 6.36e+04 & 3.70e+04 & 3.65e+04 & \textbr{2.90e+04} \\ 
			$10^3 $ & 3.23e+04 & 2.98e+04 & 2.76e+04 & \textbr{1.54e+04} & 8.27e+04 & 5.67e+04 & 5.53e+04 & \textbr{2.59e+04} \\ 
			$10^4 $ & 2.31e+04 & 2.29e+04 & 2.22e+04 & \textbr{1.54e+04} & 2.05e+05 & 1.83e+05 & 1.67e+05 & \textbr{2.59e+04} \\ 
			\hline
			\hline
			
			$\lambda$&\multicolumn{4}{c|}{results of DCPB2 in MNIST-a}&\multicolumn{4}{|c|}{results of DCPB2 in MNIST-b}\\
			\hline
			$10^0 $ & 3.74e+01 & 2.45e+03 & 1.36e+01 & \textbr{6.66e+00} & 2.28e+02 & 5.00e+03 & 2.06e+01 & \textbr{1.25e+01} \\ 
			$10^1 $ & 5.56e+02 & 2.57e+03 & 1.37e+02 & \textbr{4.35e+01} & 1.70e+03 & 5.20e+03 & 2.06e+02 & \textbr{7.26e+01} \\ 
			$10^2 $ & 3.85e+03 & 3.85e+03 & 1.39e+03 & \textbr{2.13e+02} & 7.11e+03 & 7.11e+03 & 2.04e+03 & \textbr{3.43e+02} \\ 
			$10^3 $ & 1.66e+04 & 1.66e+04 & 1.36e+04 & \textbr{3.50e+02} & 2.63e+04 & 2.63e+04 & 2.08e+04 & \textbr{1.50e+03} \\ 
			$10^4 $ & 9.30e+04 & 8.88e+04 & 7.89e+04 & \textbr{3.50e+02} & 2.18e+05 & 2.18e+05 & 2.05e+05 & \textbr{9.73e+02} \\ 
			\hline
		\end{tabular}
		\caption{Comparisons of objective values of four the methods with varying $\lambda$ on NASDAQ 100 and S\&P 500 datasets.}	\label{tab:lambdaindex}
		
	\end{table}

	\captionsetup[subfigure]{font=tiny}
	
	\newpage
	\begin{figure}[htbp]
		\centering
		\begin{subfigure}{0.2\textwidth}
			\includegraphics[width=\textwidth]{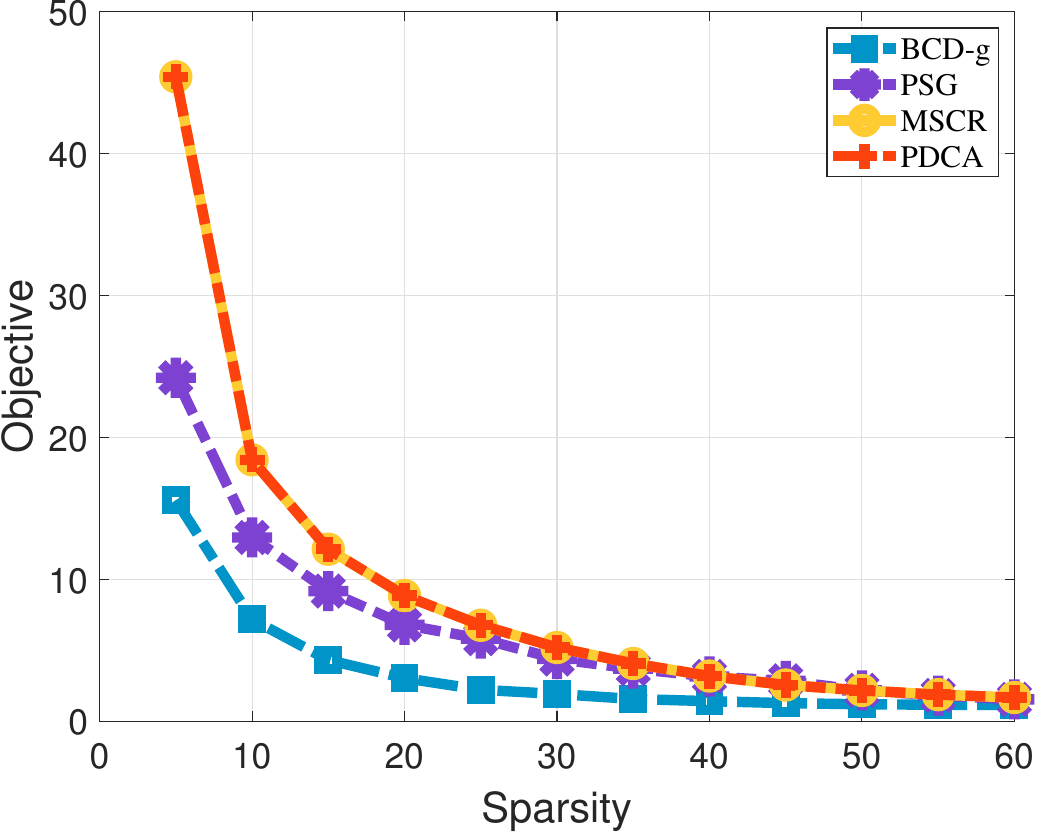}
			\caption{SIT in NAS-16}
		\end{subfigure}
		\begin{subfigure}{0.2\textwidth}
			\includegraphics[width=\textwidth]{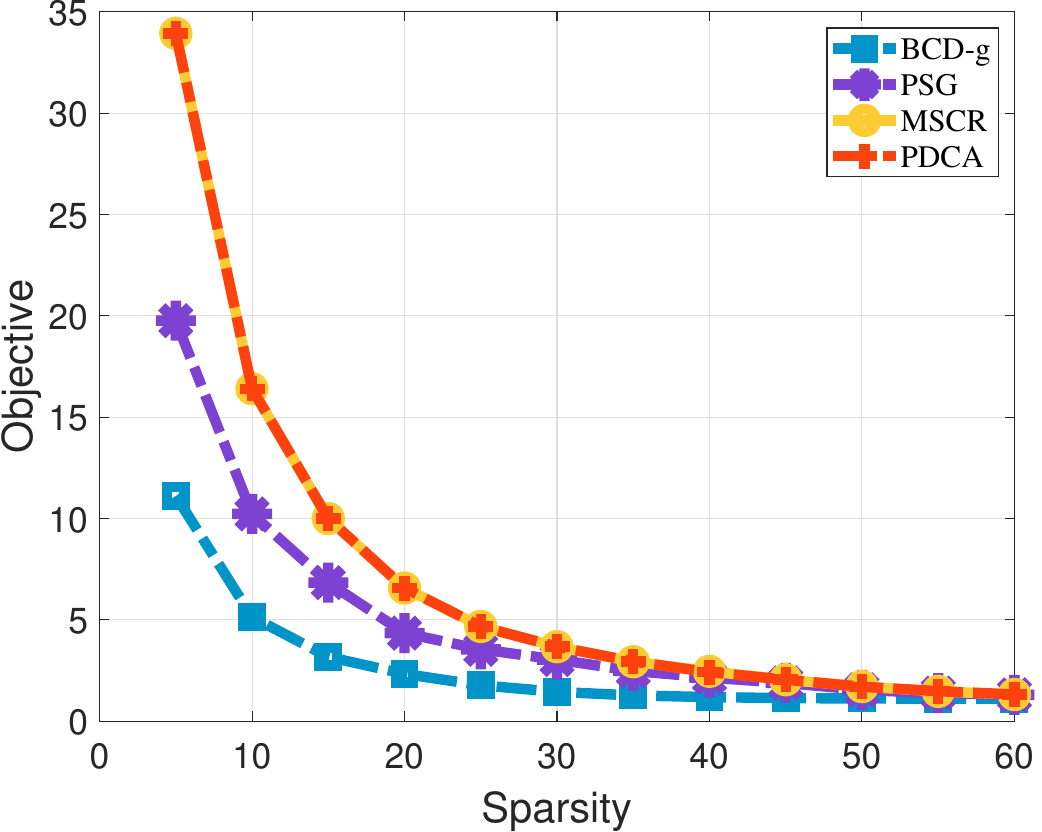}
			\caption{SIT in NAS-17}
		\end{subfigure}
		\begin{subfigure}{0.2\textwidth}
			\includegraphics[width=\textwidth]{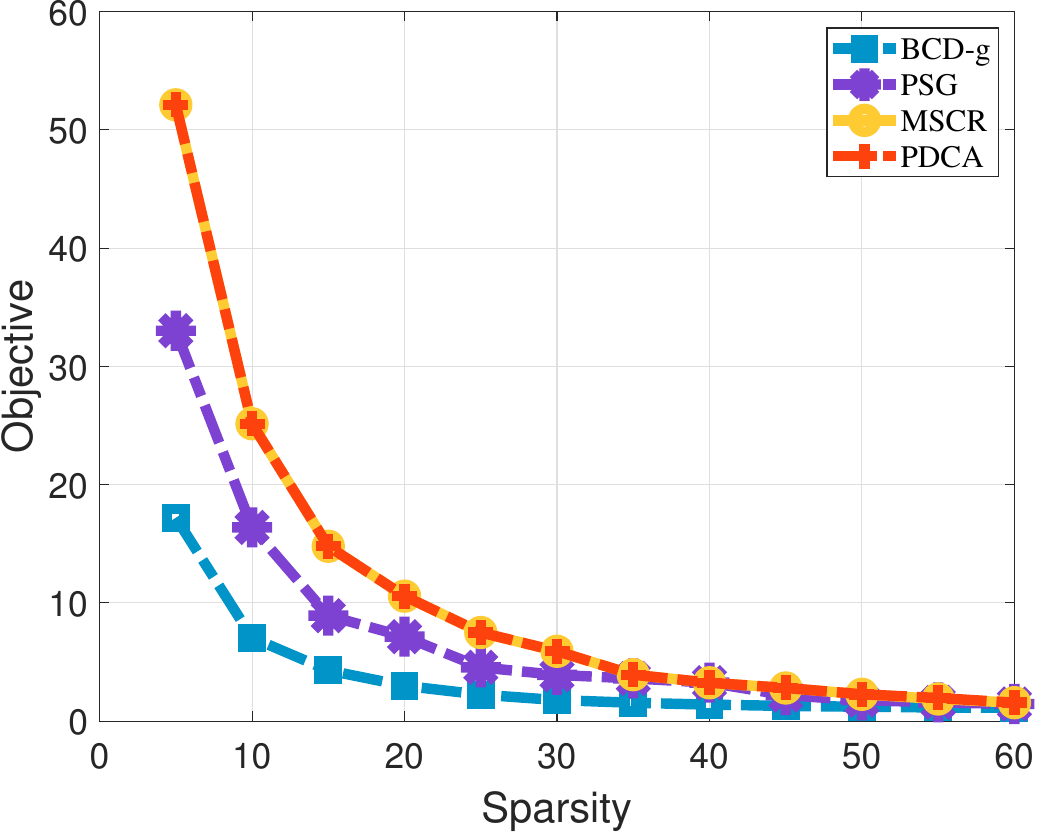}
			\caption{SIT in NAS-18}
		\end{subfigure}
		\begin{subfigure}{0.2\textwidth}
			\includegraphics[width=\textwidth]{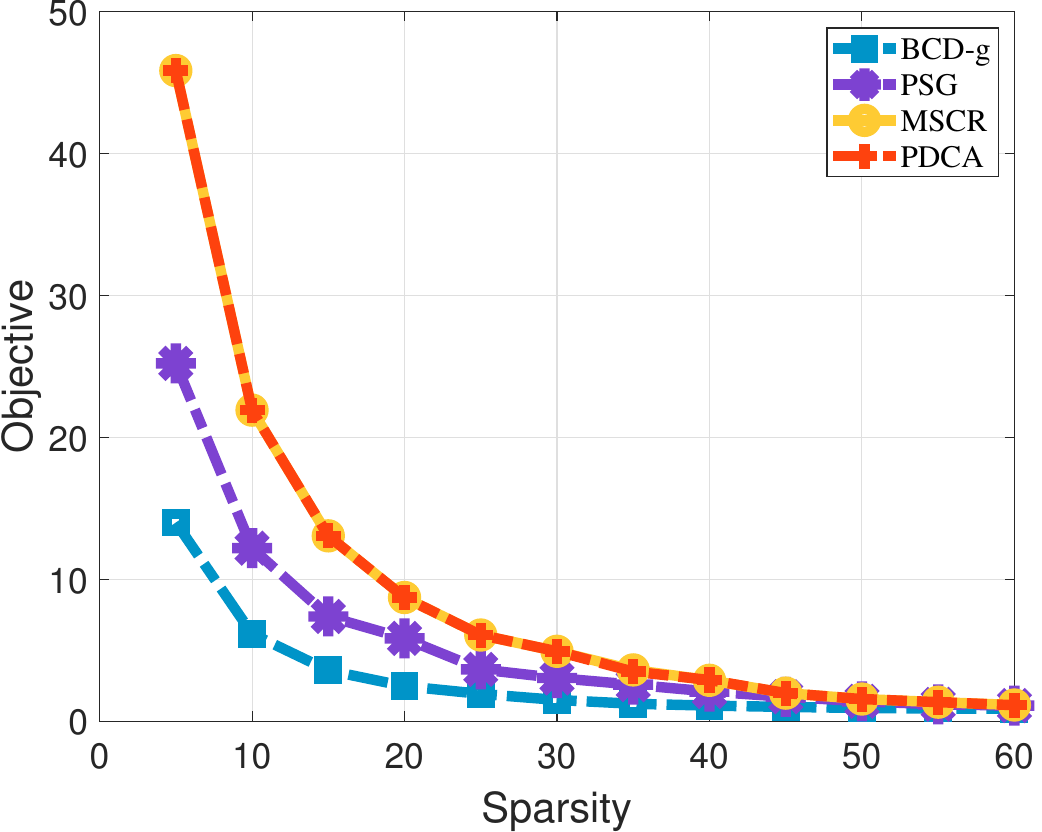}
			\caption{SIT in NAS-19}
		\end{subfigure}
		\begin{subfigure}{0.2\textwidth}
			\includegraphics[width=\textwidth]{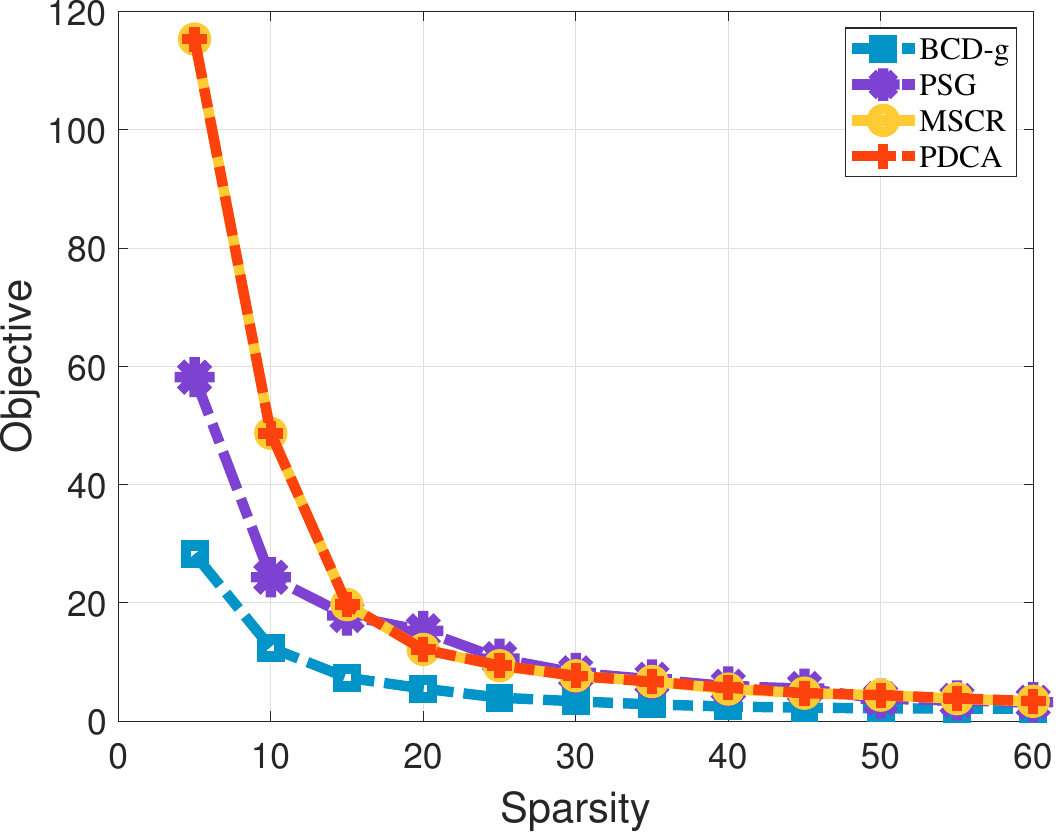}
			\caption{SIT in NAS-20}
		\end{subfigure}
		\begin{subfigure}{0.2\textwidth}
			\includegraphics[width=\textwidth]{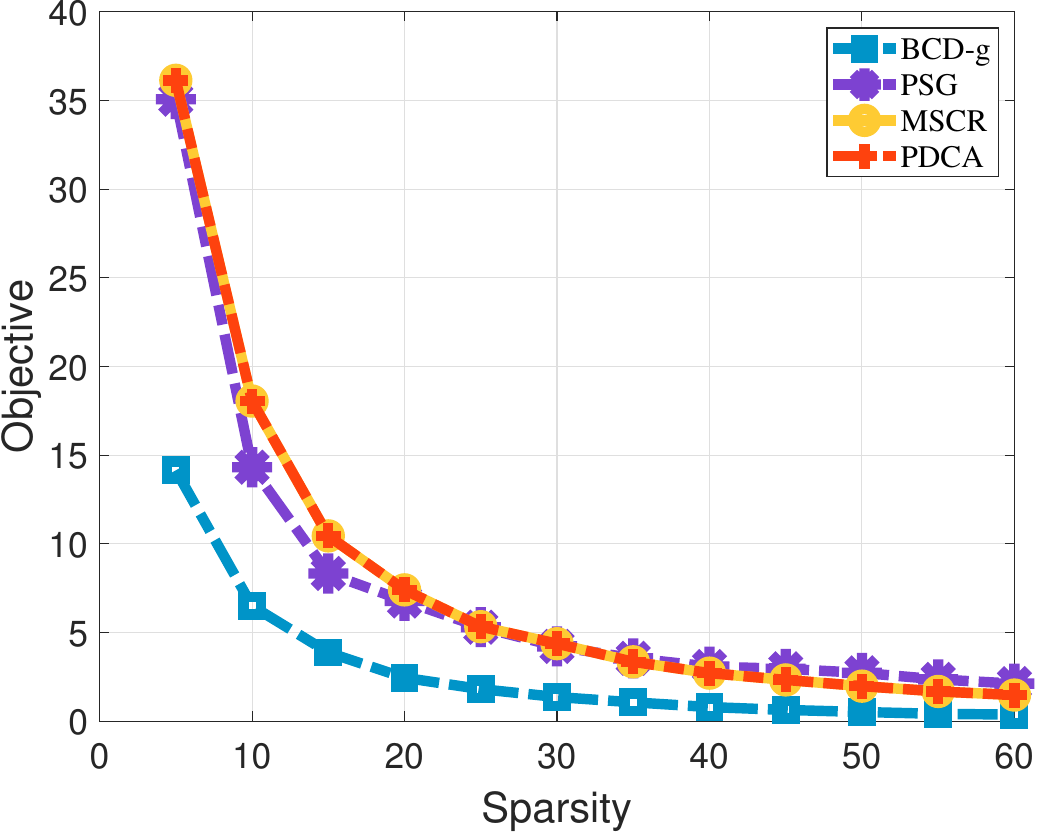}
			\caption{SIT in SP-16}
		\end{subfigure}
		\begin{subfigure}{0.2\textwidth}
			\includegraphics[width=\textwidth]{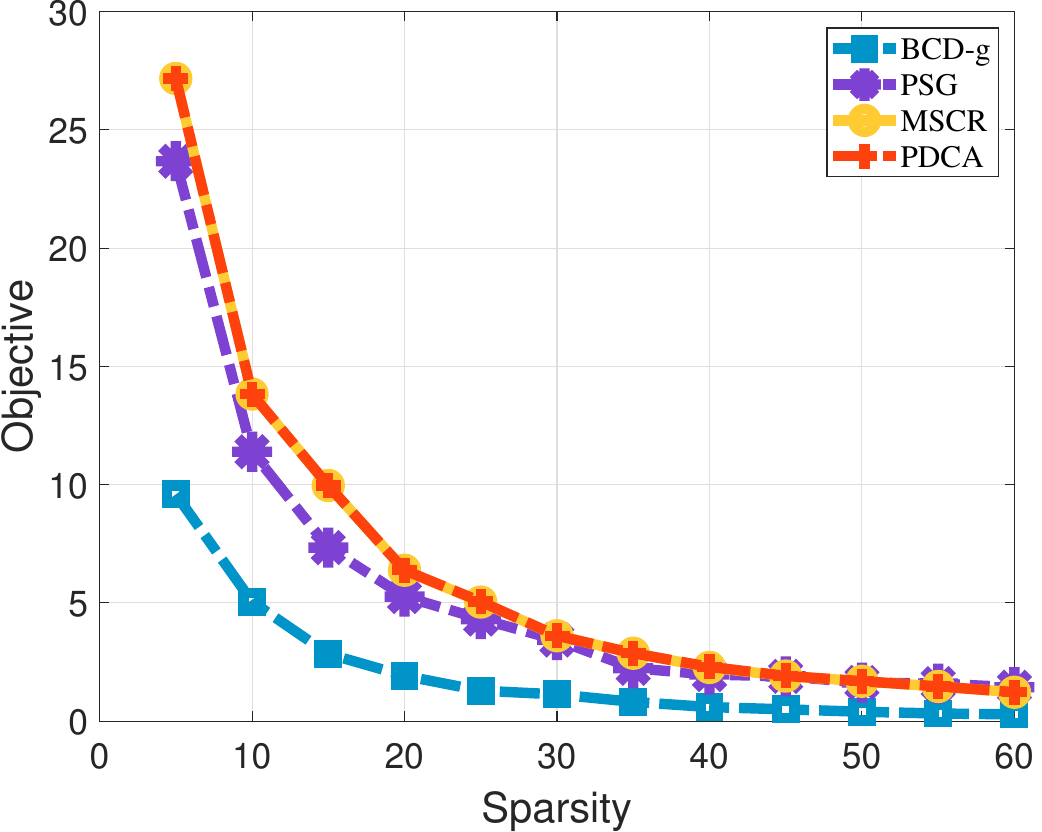}
			\caption{SIT in SP-17}
		\end{subfigure}
		\begin{subfigure}{0.2\textwidth}
			\includegraphics[width=\textwidth]{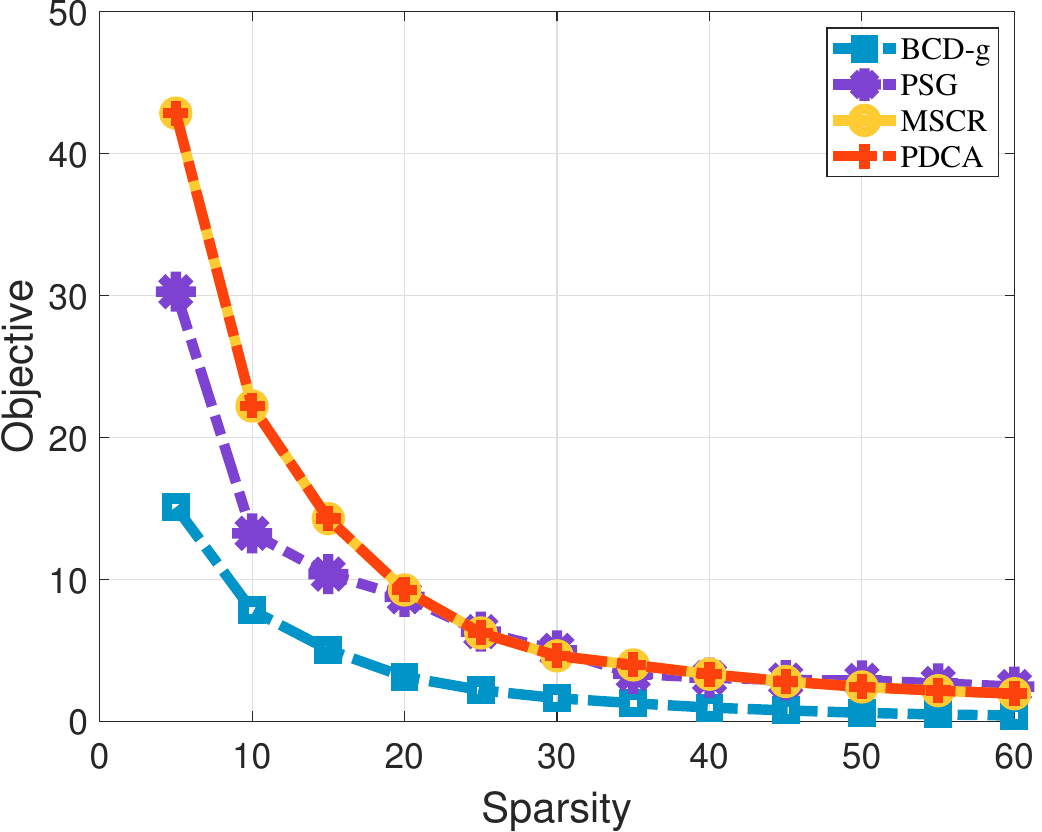}
			\caption{SIT in SP-18}
		\end{subfigure}
		\begin{subfigure}{0.2\textwidth}
			\includegraphics[width=\textwidth]{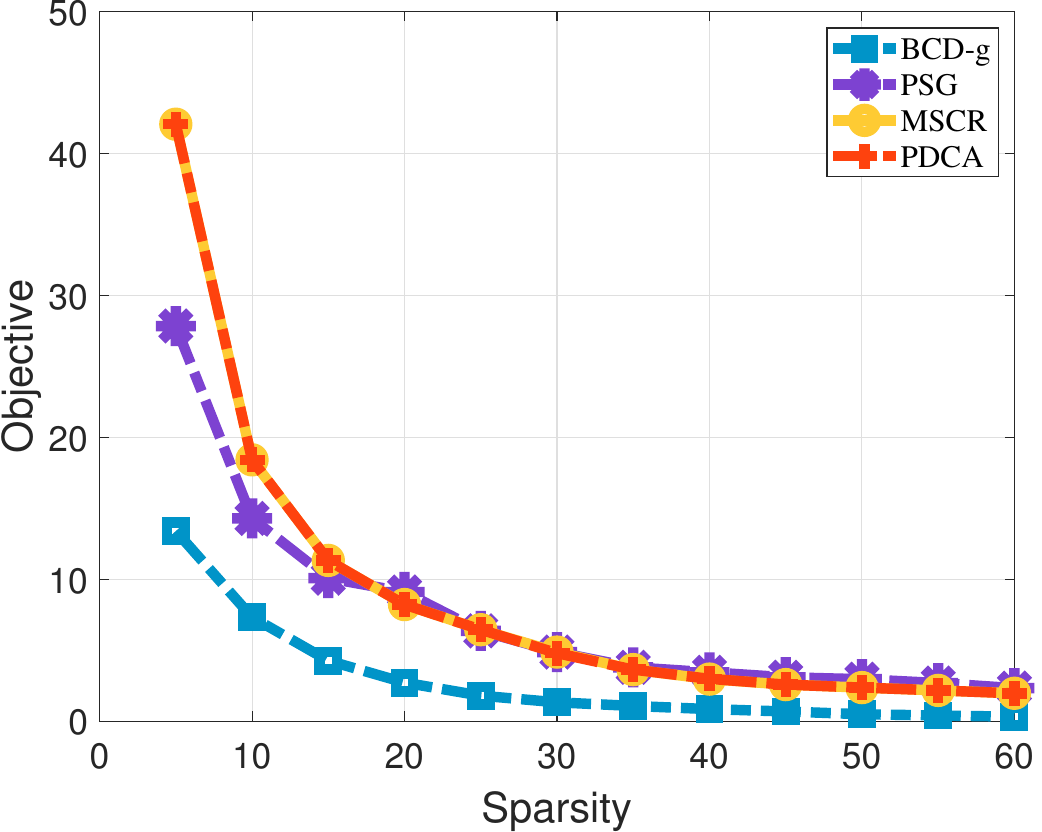}
			\caption{SIT in SP-19}
		\end{subfigure}
		\begin{subfigure}{0.2\textwidth}
			\includegraphics[width=\textwidth]{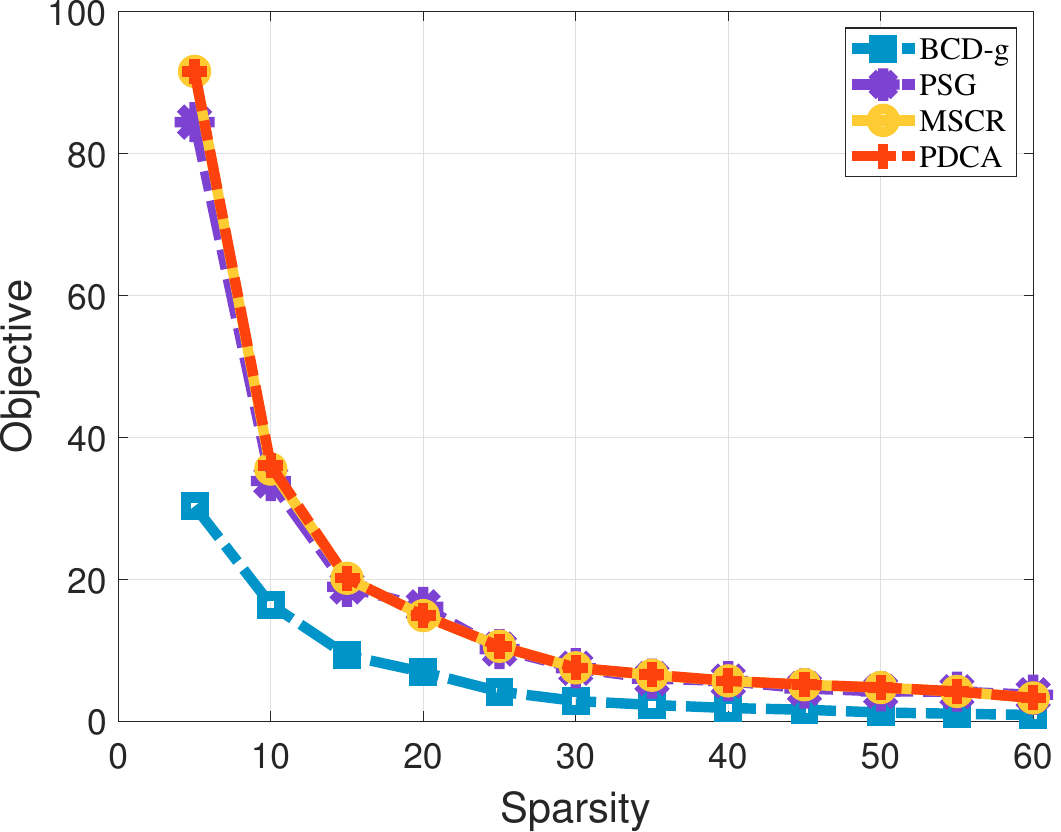}
			\caption{SIT in SP-20}
		\end{subfigure}
		\begin{subfigure}{0.2\textwidth}
			\includegraphics[width=\textwidth]{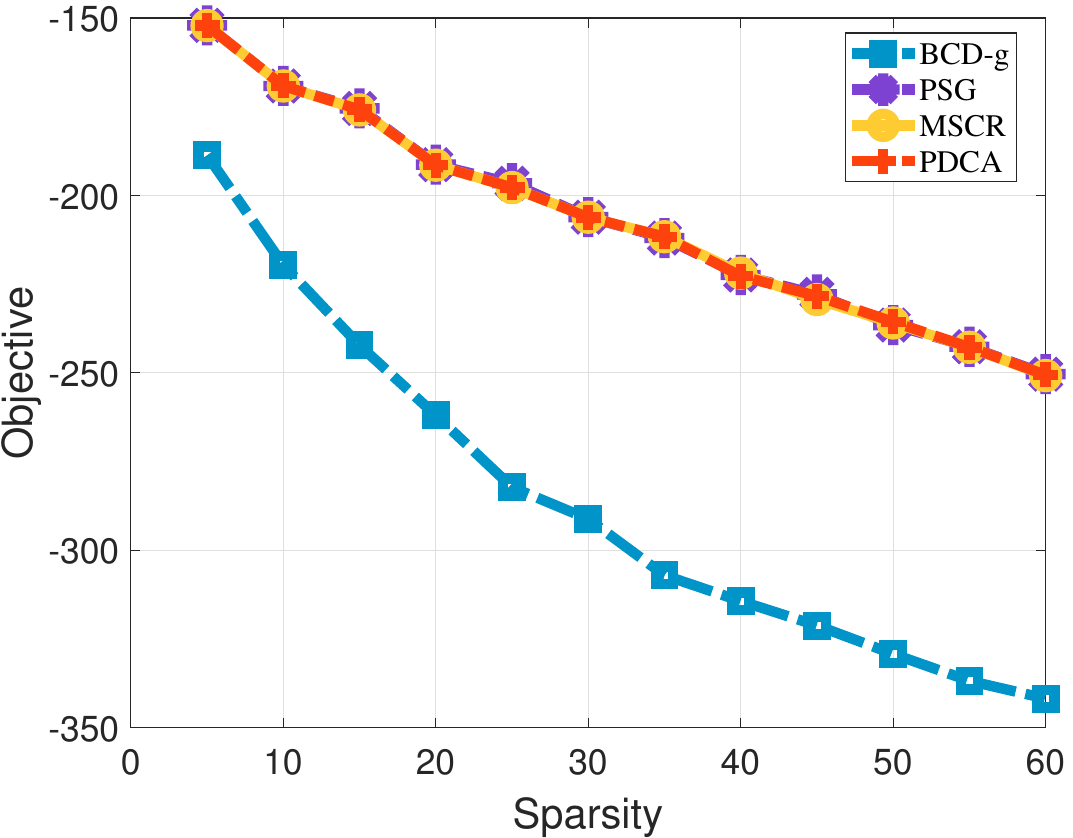}
			\caption{NNSPCA in randn-a}
		\end{subfigure}
		\begin{subfigure}{0.2\textwidth}
			\includegraphics[width=\textwidth]{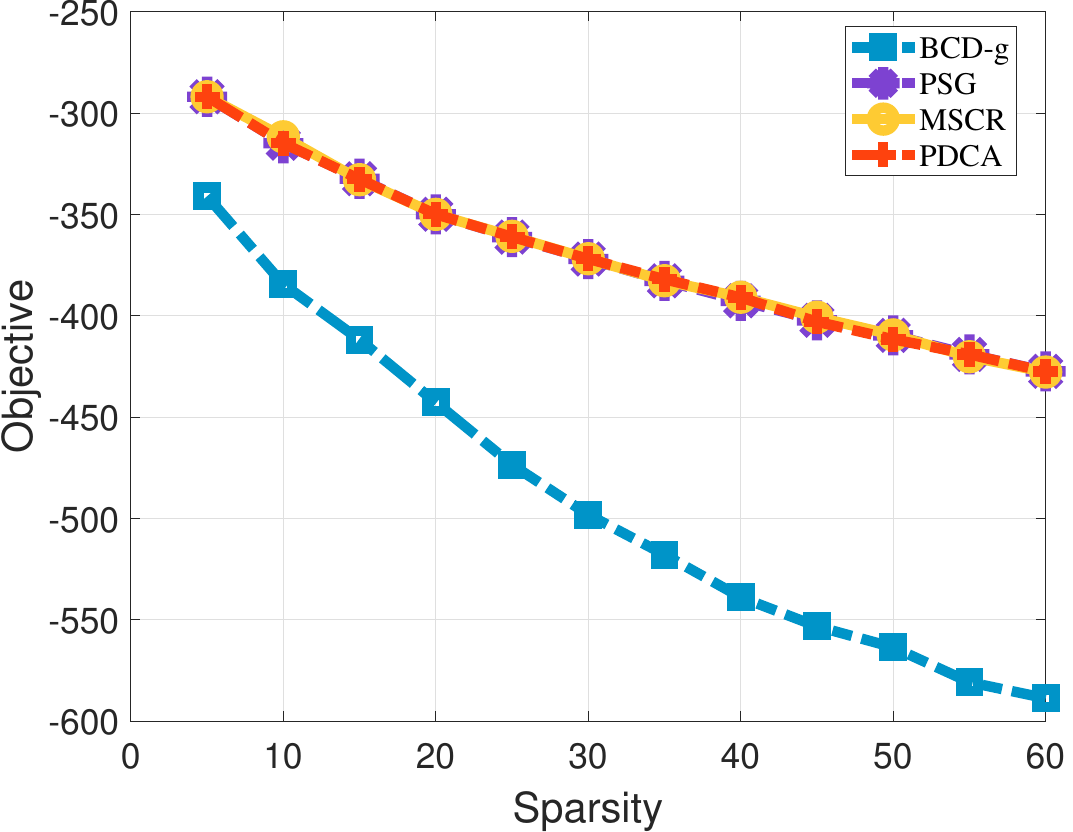}
			\caption{NNSPCA in randn-b} 
		\end{subfigure}
		\begin{subfigure}{0.2\textwidth}
			\includegraphics[width=\textwidth]{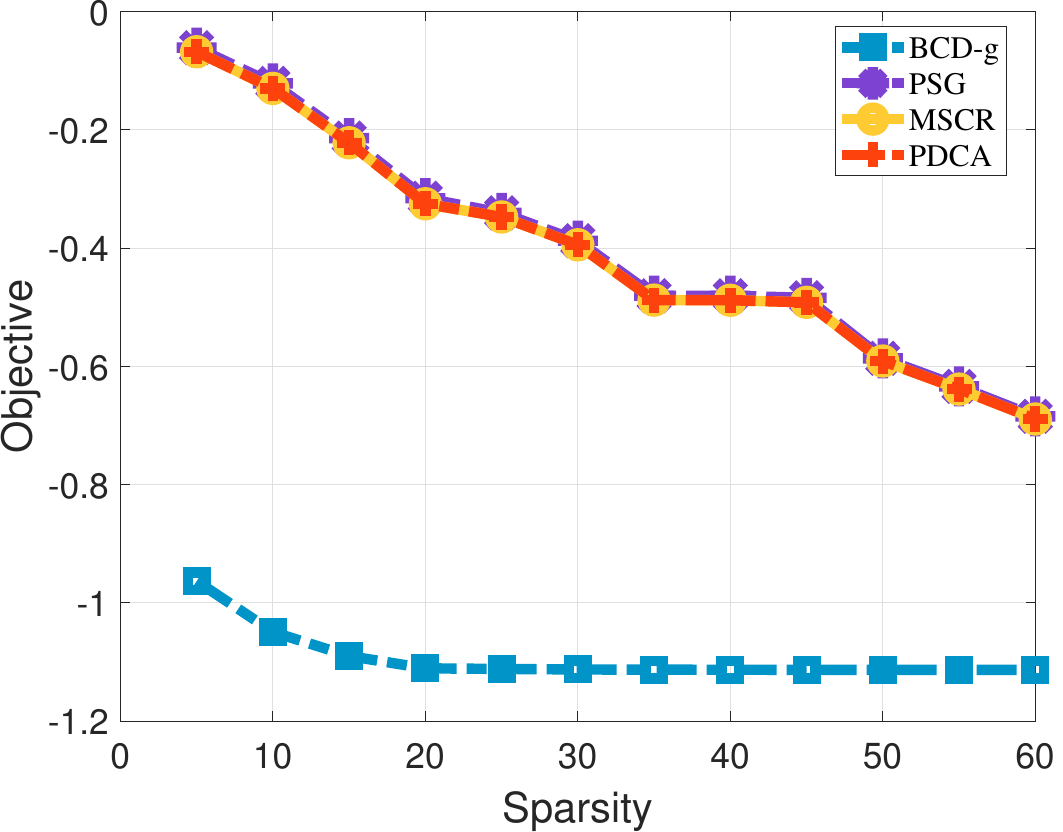}
			\caption{NNSPCA in TDT2-a}
		\end{subfigure}
		\begin{subfigure}{0.2\textwidth}
			\includegraphics[width=\textwidth]{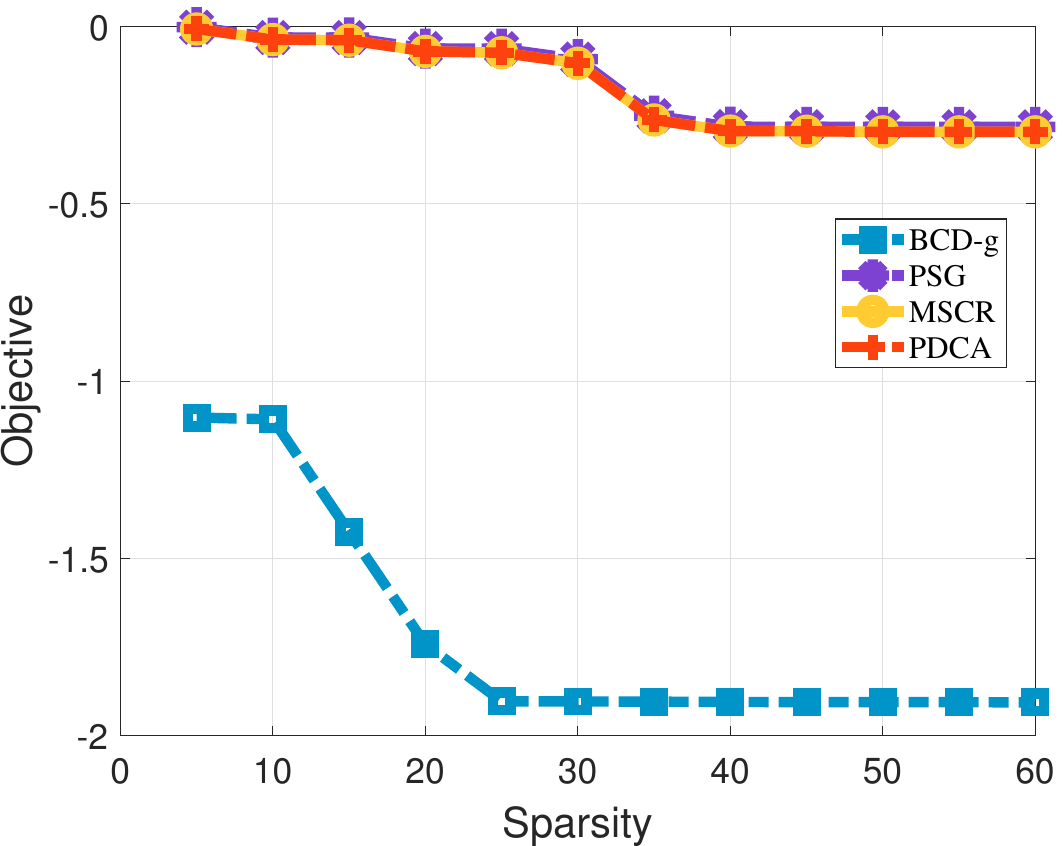}
			\caption{NNSPCA in TDT2-b}
		\end{subfigure}
		\begin{subfigure}{0.2\textwidth}
			\includegraphics[width=\textwidth]{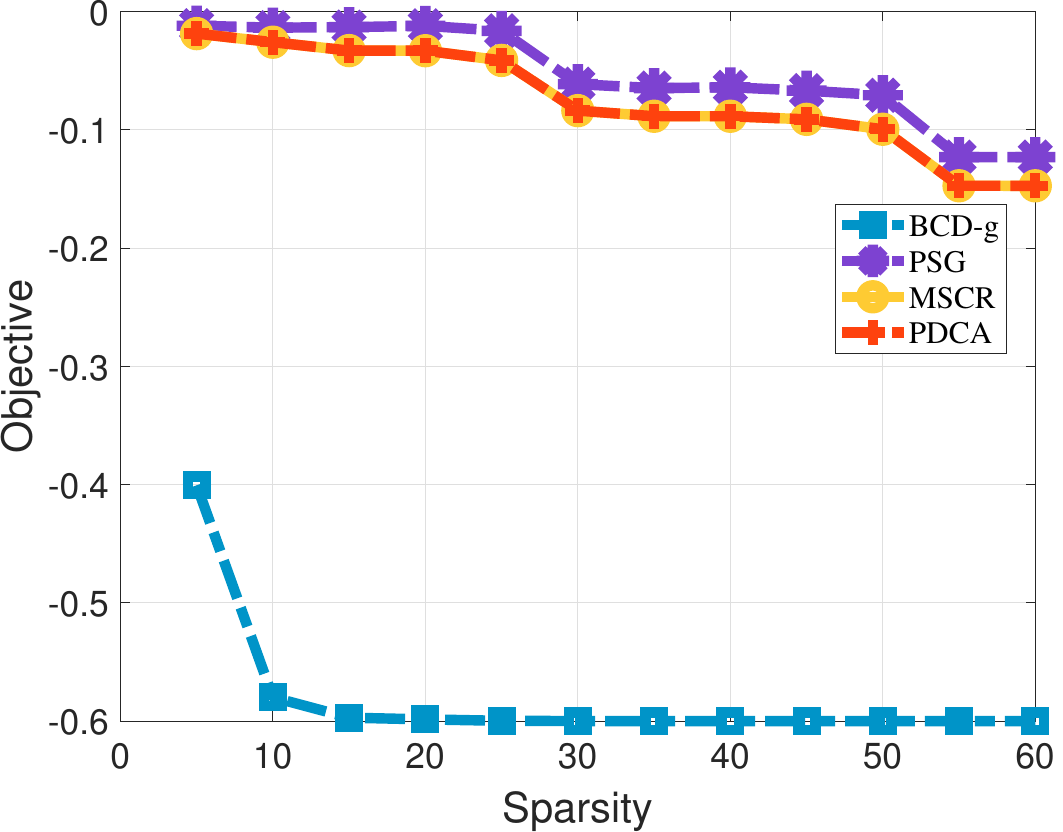}
			\caption{NNSPCA in 20News-a}
		\end{subfigure}
		\begin{subfigure}{0.2\textwidth}
			\includegraphics[width=\textwidth]{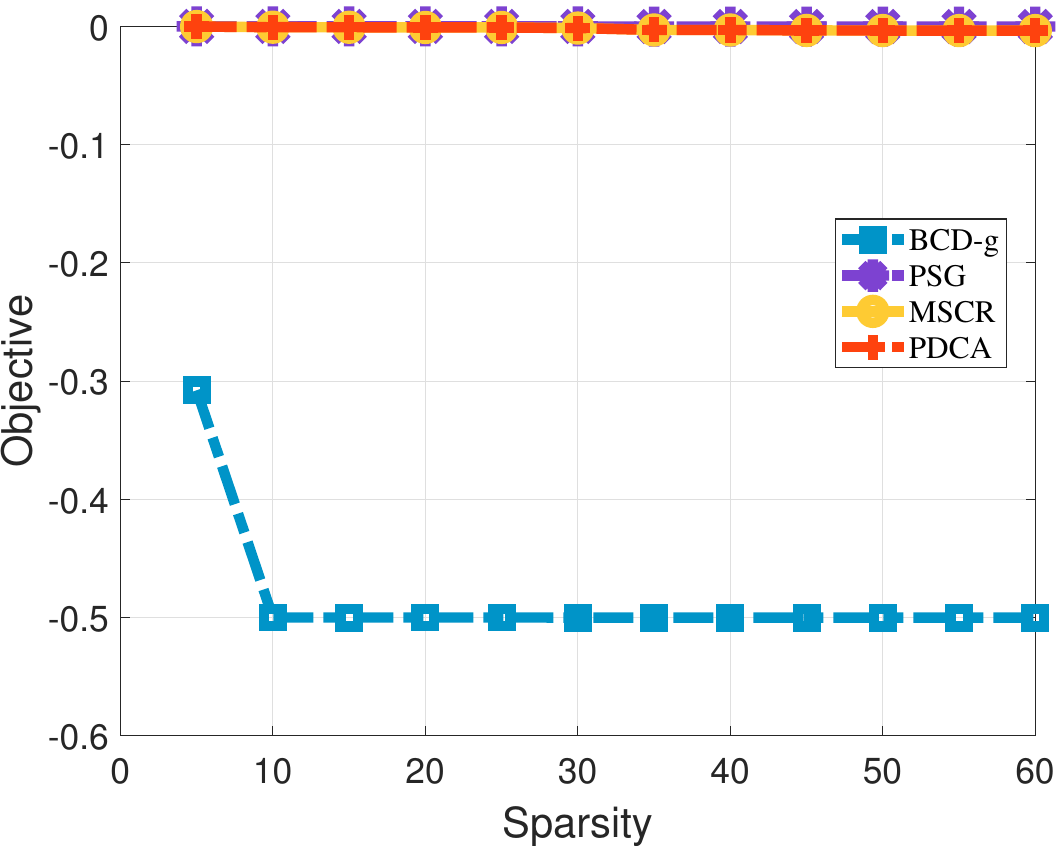}
			\caption{NNSPCA in 20News-b}
		\end{subfigure}
		\begin{subfigure}{0.2\textwidth}
			\includegraphics[width=\textwidth]{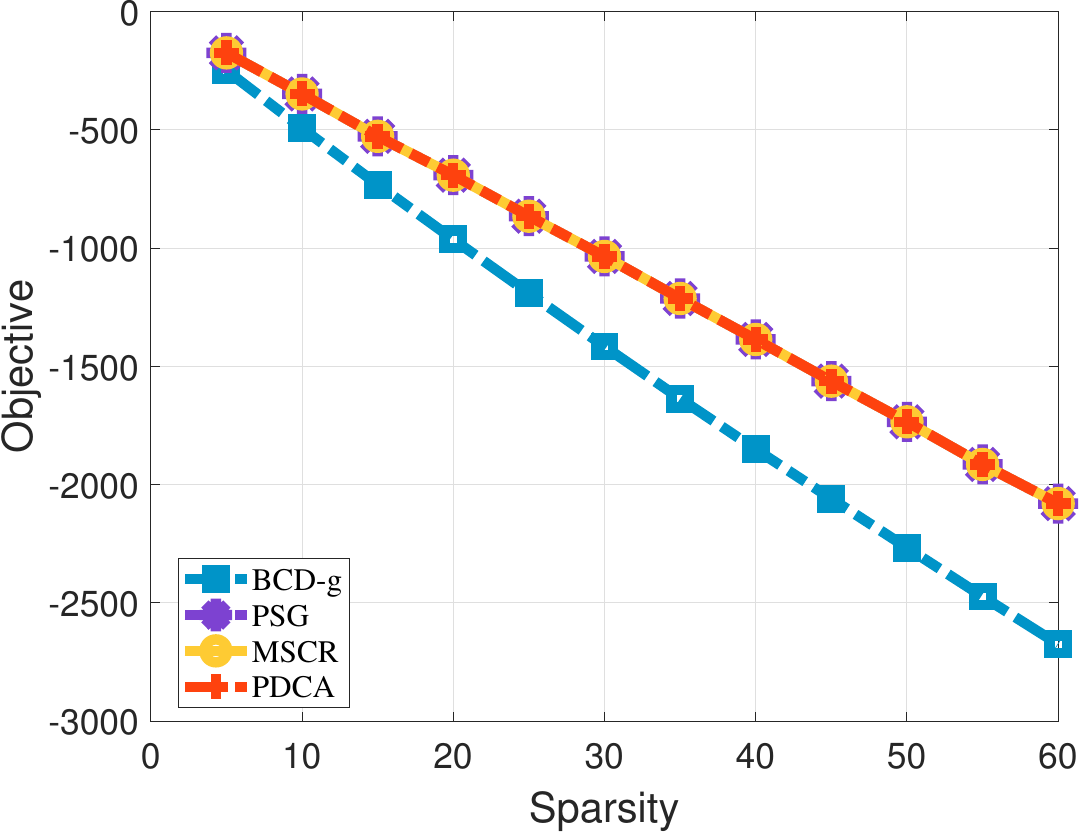}
			\caption{NNSPCA in Cifar-a}
		\end{subfigure}
		\begin{subfigure}{0.2\textwidth}
			\includegraphics[width=\textwidth]{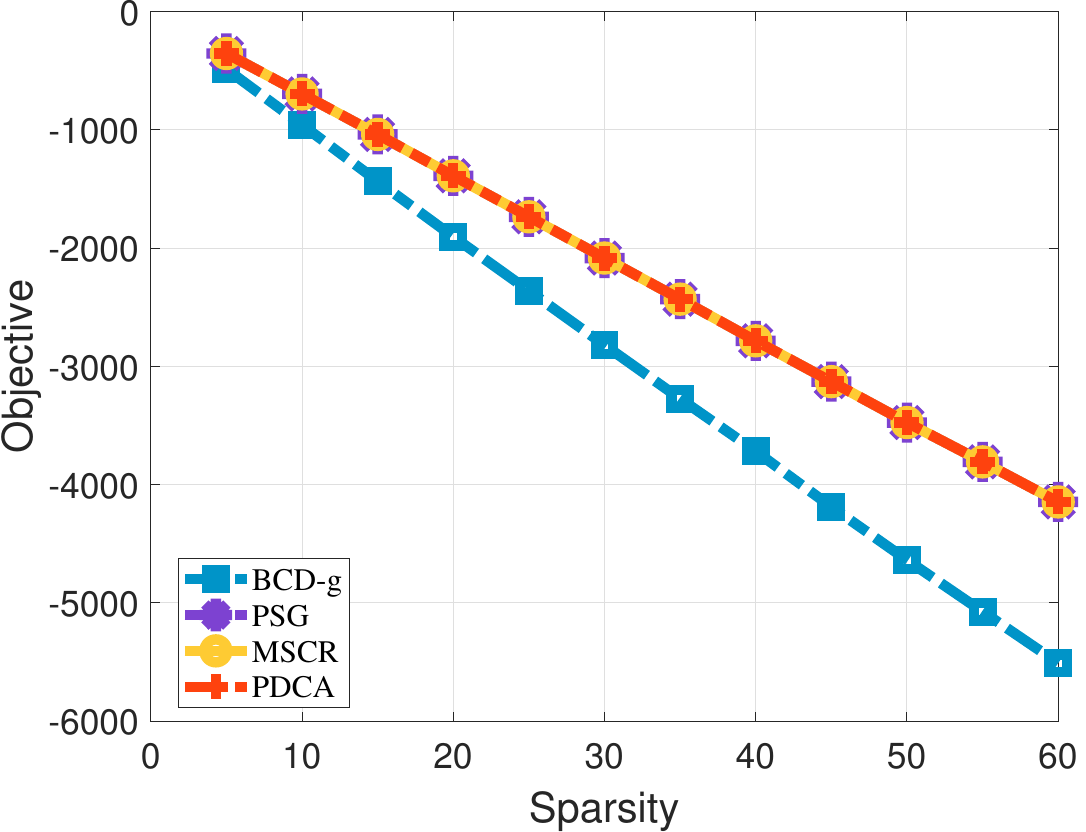}
			\caption{NNSPCA in Cifar-b}
		\end{subfigure}
		\begin{subfigure}{0.2\textwidth}
			\includegraphics[width=\textwidth]{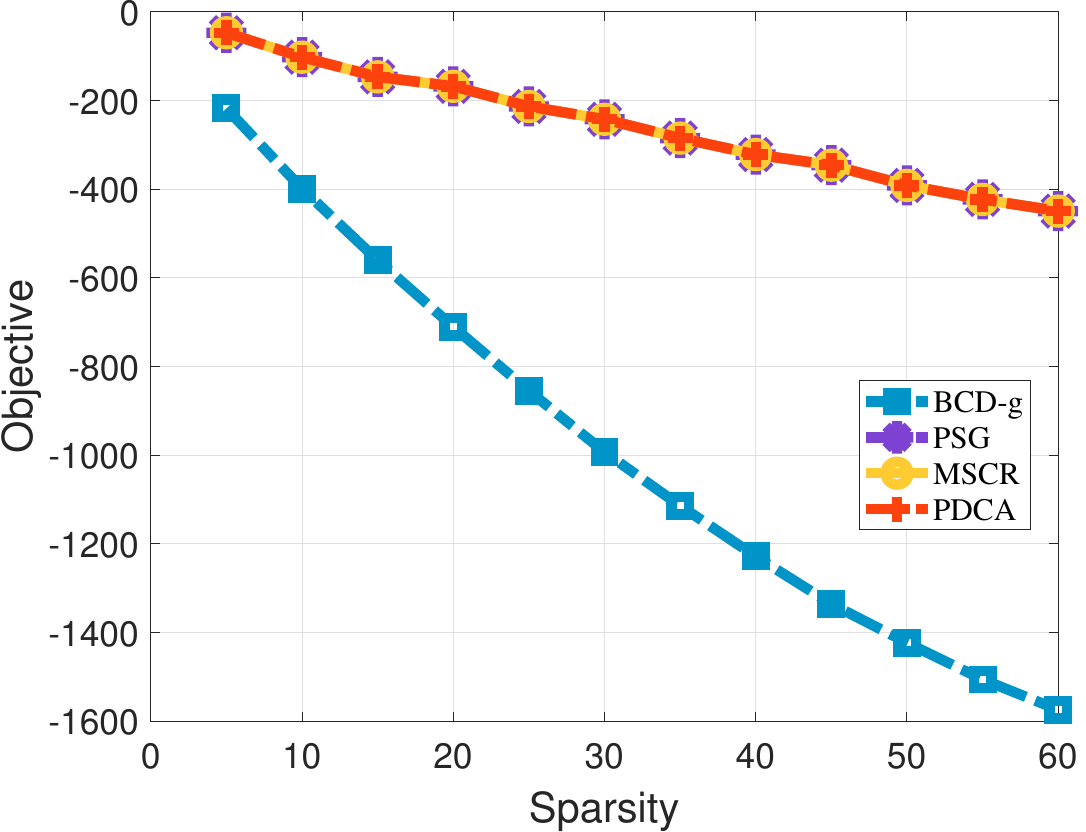}
			\caption{NNSPCA in MNIST-a}
		\end{subfigure}
		\begin{subfigure}{0.2\textwidth}
			\includegraphics[width=\textwidth]{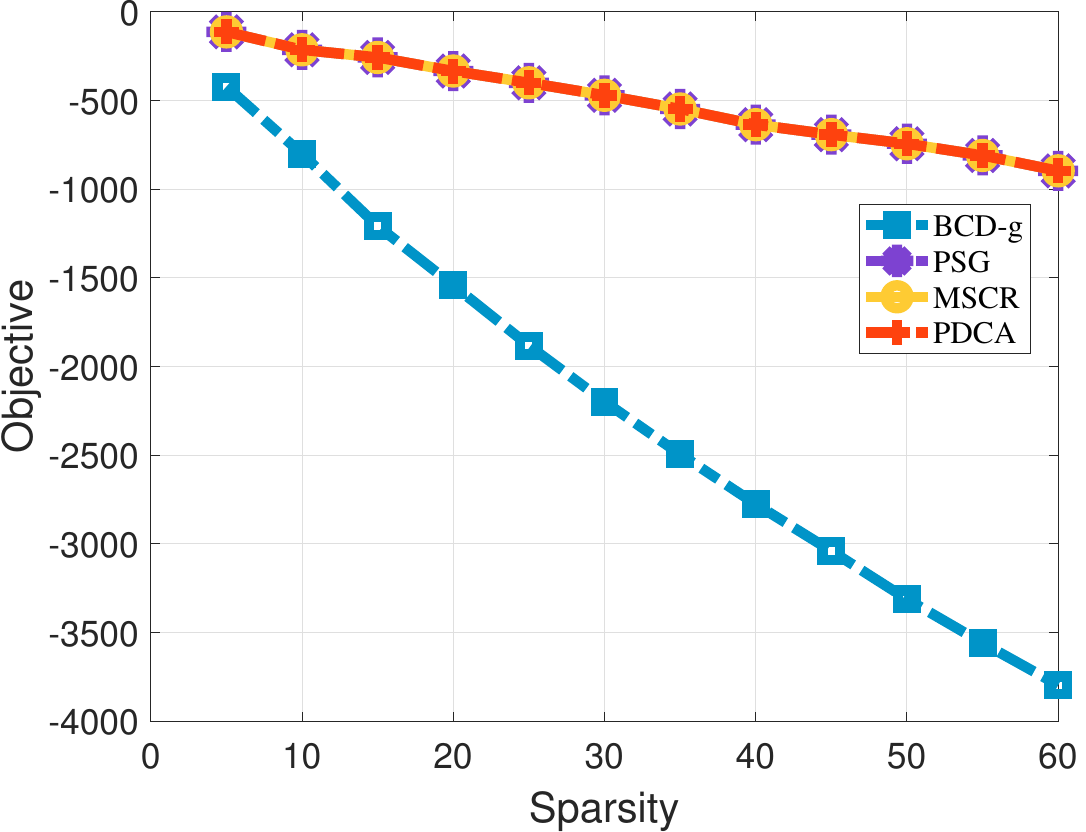}
			\caption{NNSPCA in MNIST-b}
		\end{subfigure}
		\caption{Objective values for the sparse index tracking problem and non-negative sparse PCA problem with varying sparsity $s$ across 20 datasets.}
		\label{fig:diff1}
		
	\end{figure}

	\captionsetup[subfigure]{font=tiny}
	\begin{figure}[htbp]
		\centering
		\begin{subfigure}{0.2\textwidth}
			\includegraphics[width=\textwidth]{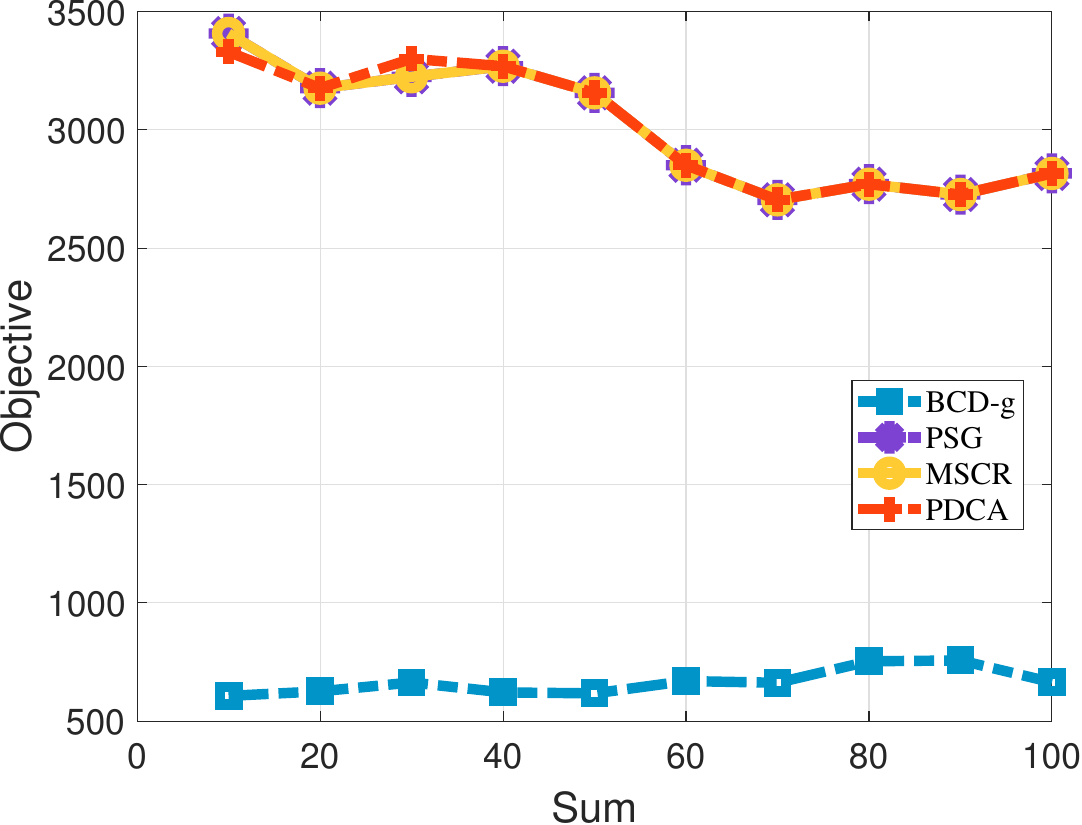}
			\caption{DCPB1 in randn-a}
		\end{subfigure}
		\begin{subfigure}{0.2\textwidth}
			\includegraphics[width=\textwidth]{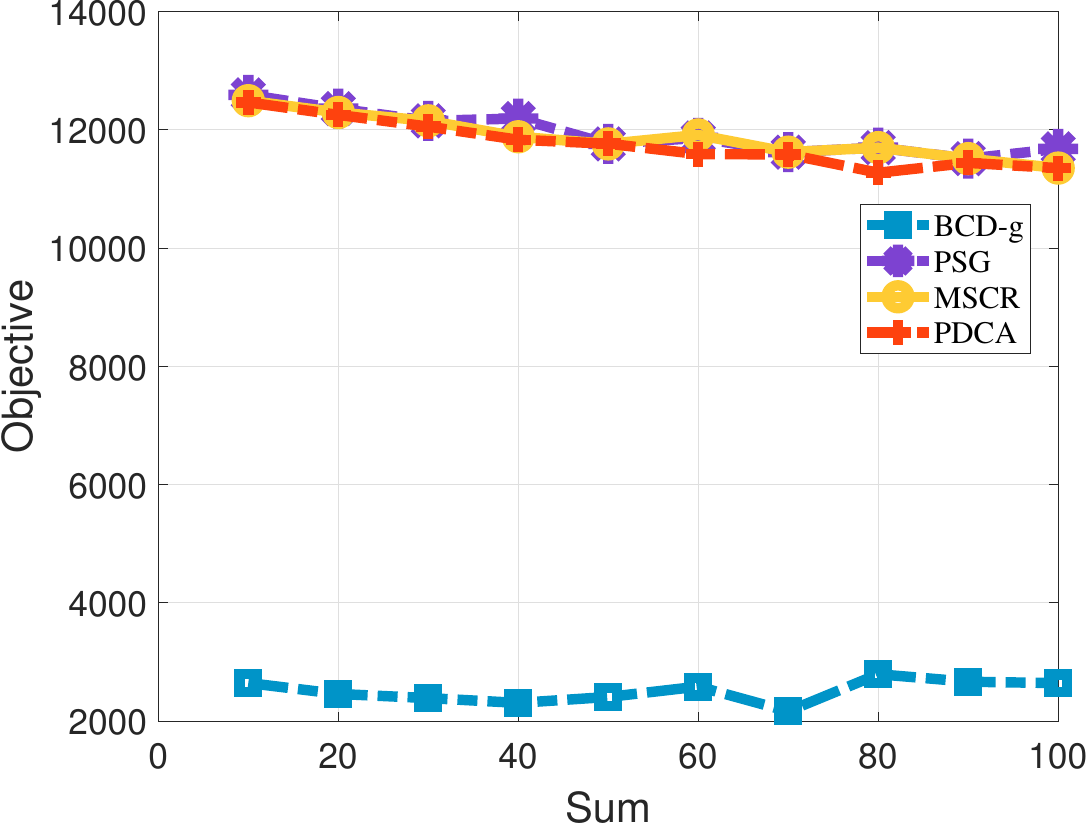}
			\caption{DCPB1 in randn-b}
		\end{subfigure}
		\begin{subfigure}{0.2\textwidth}
			\includegraphics[width=\textwidth]{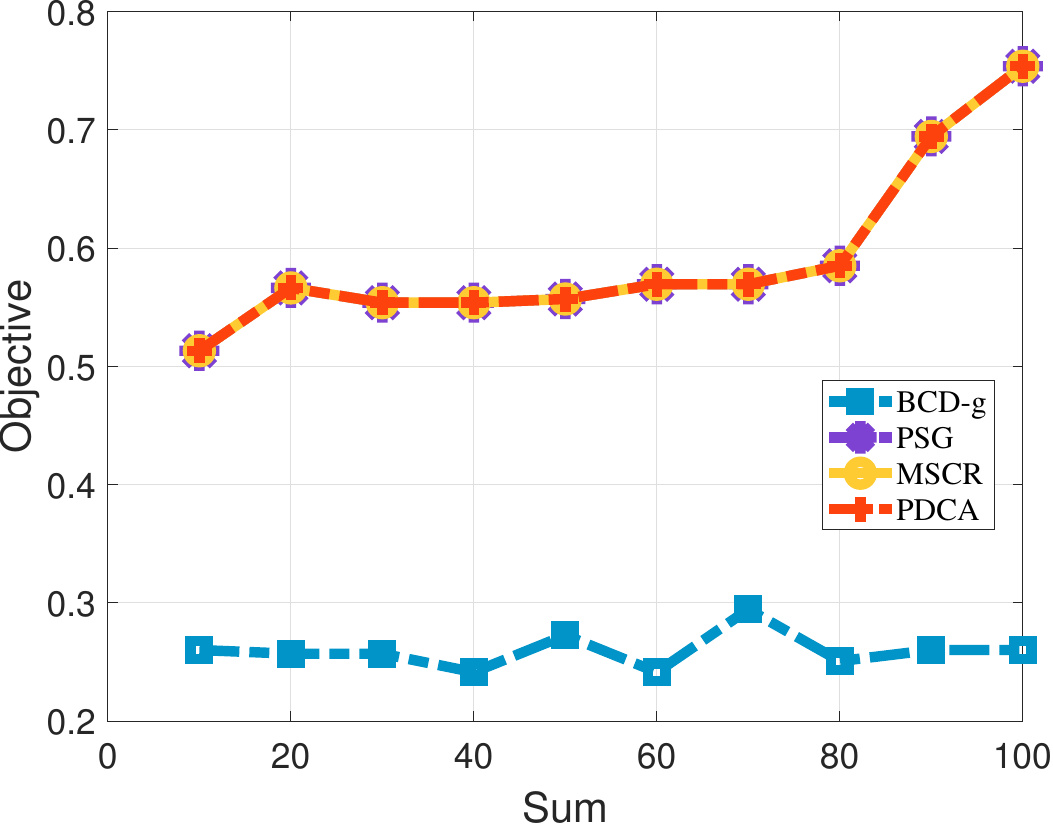}
			\caption{DCPB1 in TDT2-a}
		\end{subfigure}
		\begin{subfigure}{0.2\textwidth}
			\includegraphics[width=\textwidth]{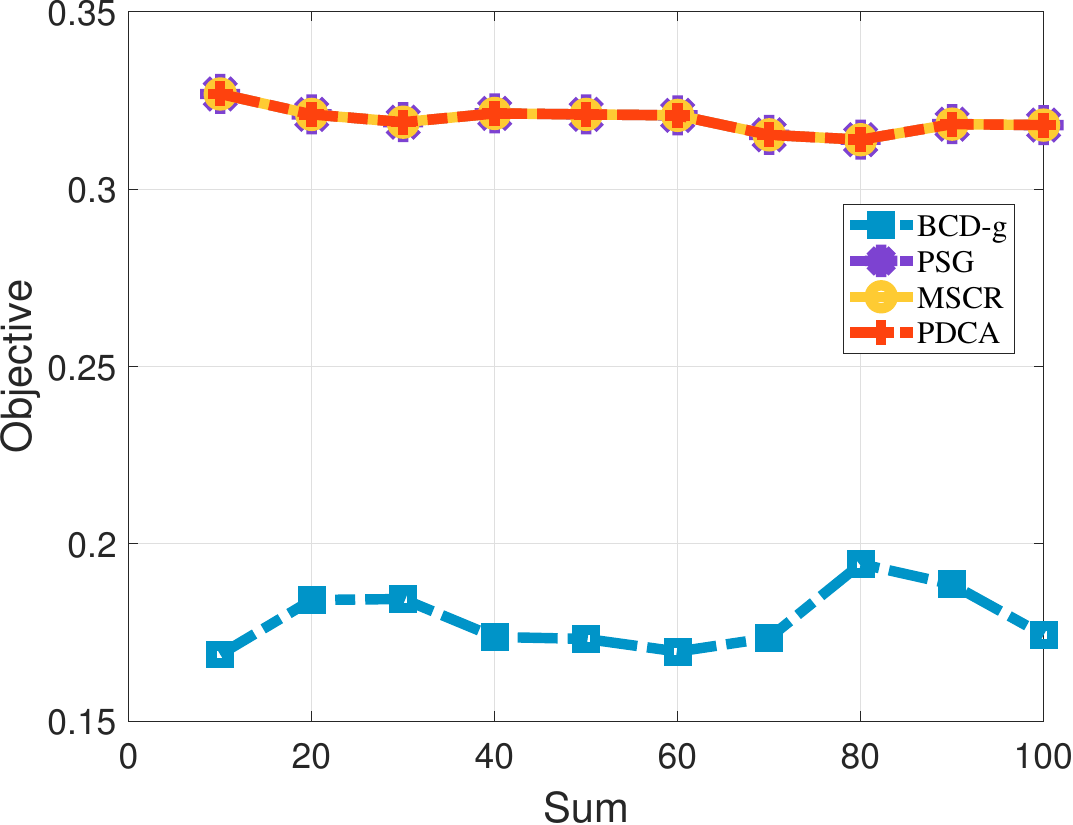}
			\caption{DCPB1 in TDT2-b}
		\end{subfigure}
		\begin{subfigure}{0.2\textwidth}
			\includegraphics[width=\textwidth]{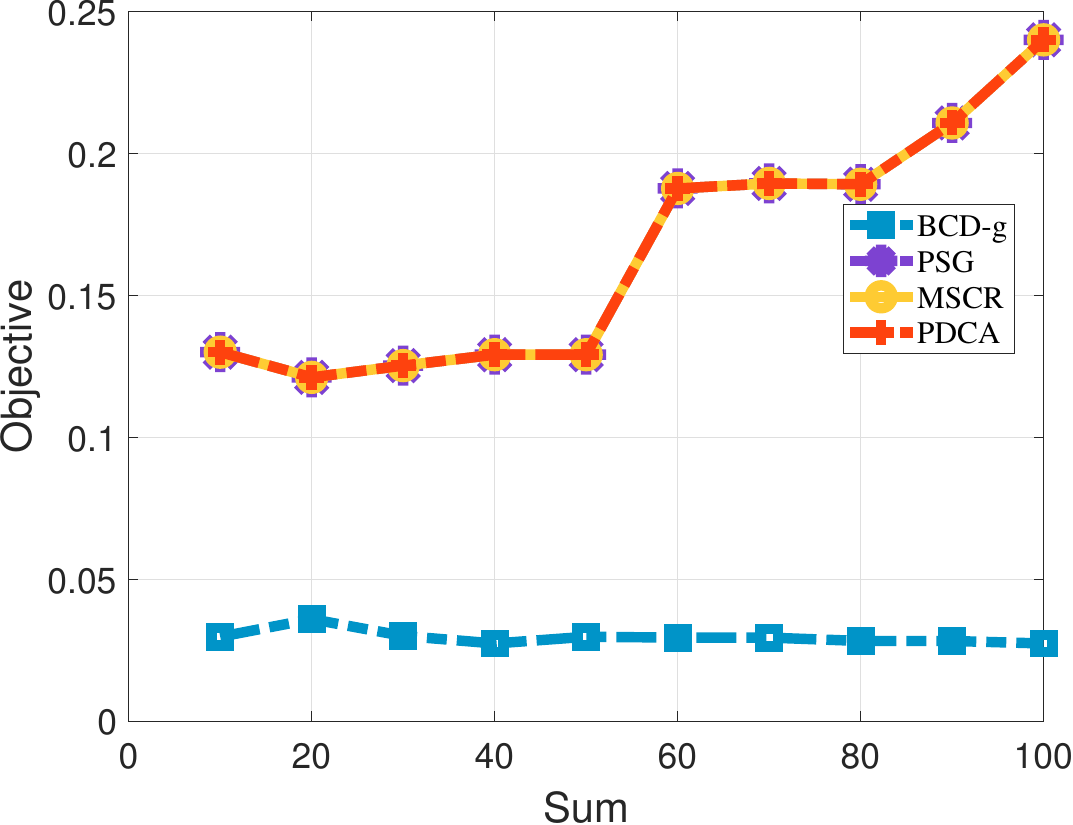}
			\caption{DCPB1 in 20News-a}
		\end{subfigure}
		\begin{subfigure}{0.2\textwidth}
			\includegraphics[width=\textwidth]{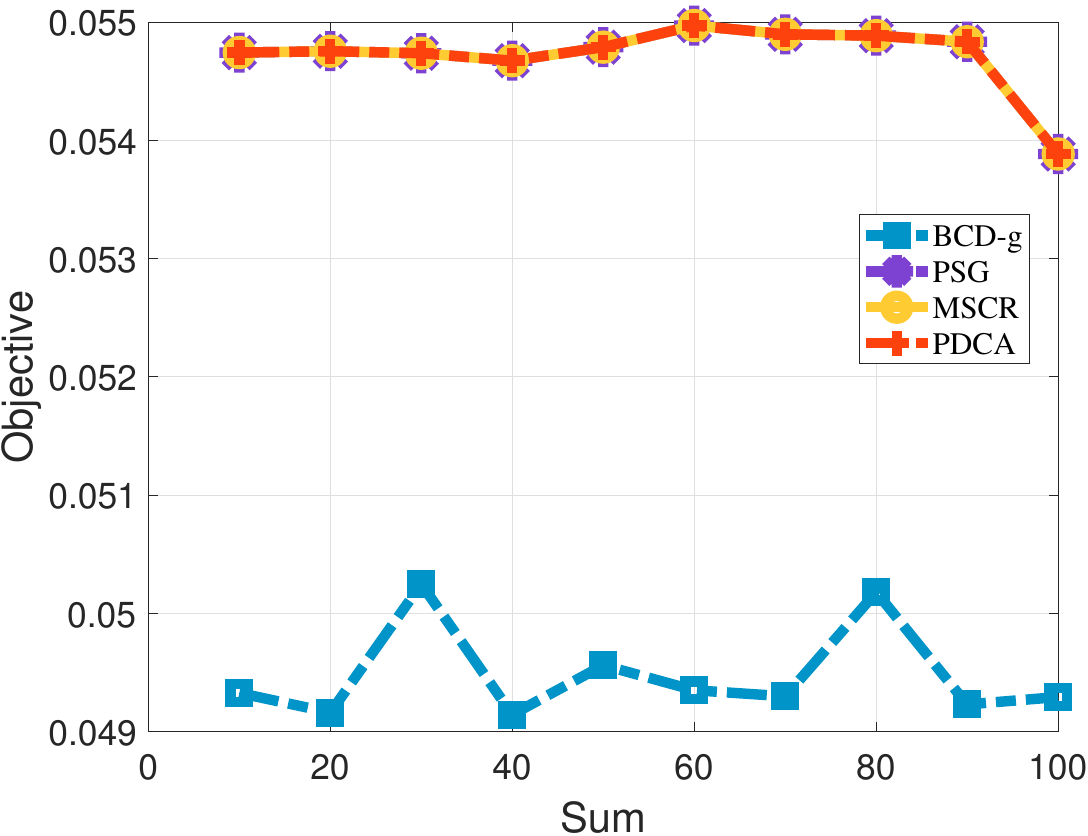}
			\caption{DCPB1 in 20News-b}
		\end{subfigure}
		\begin{subfigure}{0.2\textwidth}
			\includegraphics[width=\textwidth]{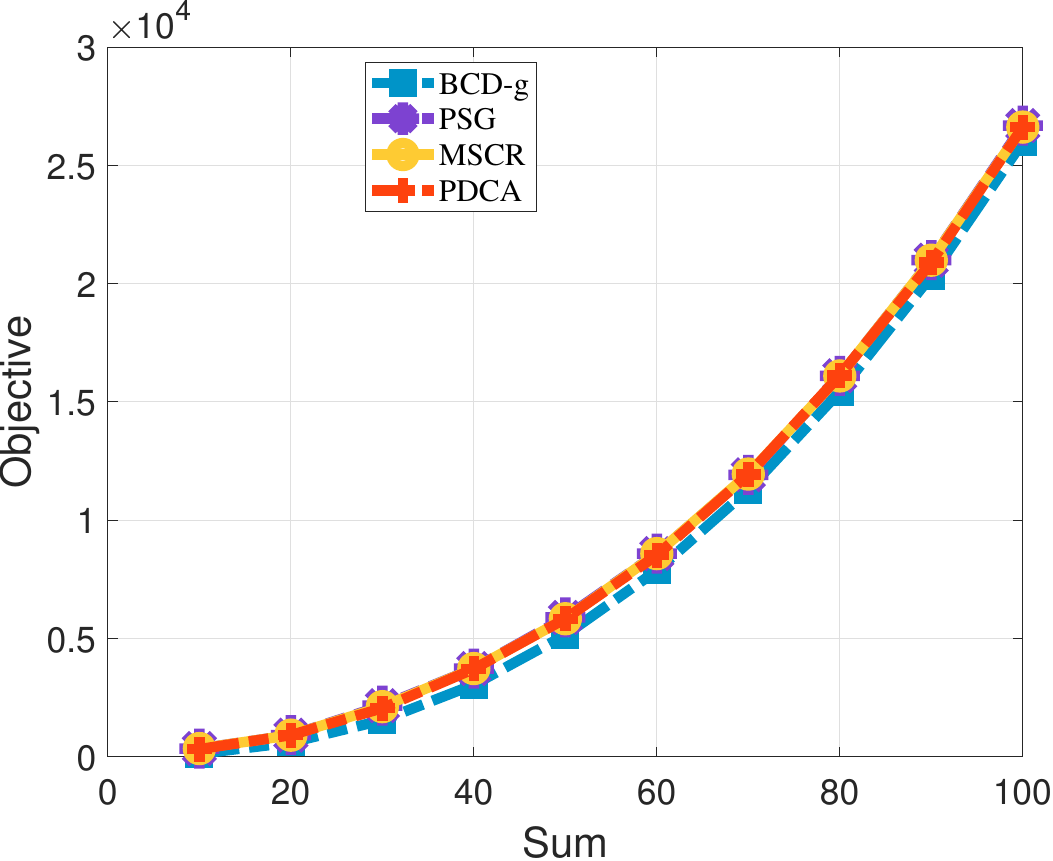}
			\caption{DCPB1 in Cifar-a}
		\end{subfigure}
		\begin{subfigure}{0.2\textwidth}
			\includegraphics[width=\textwidth]{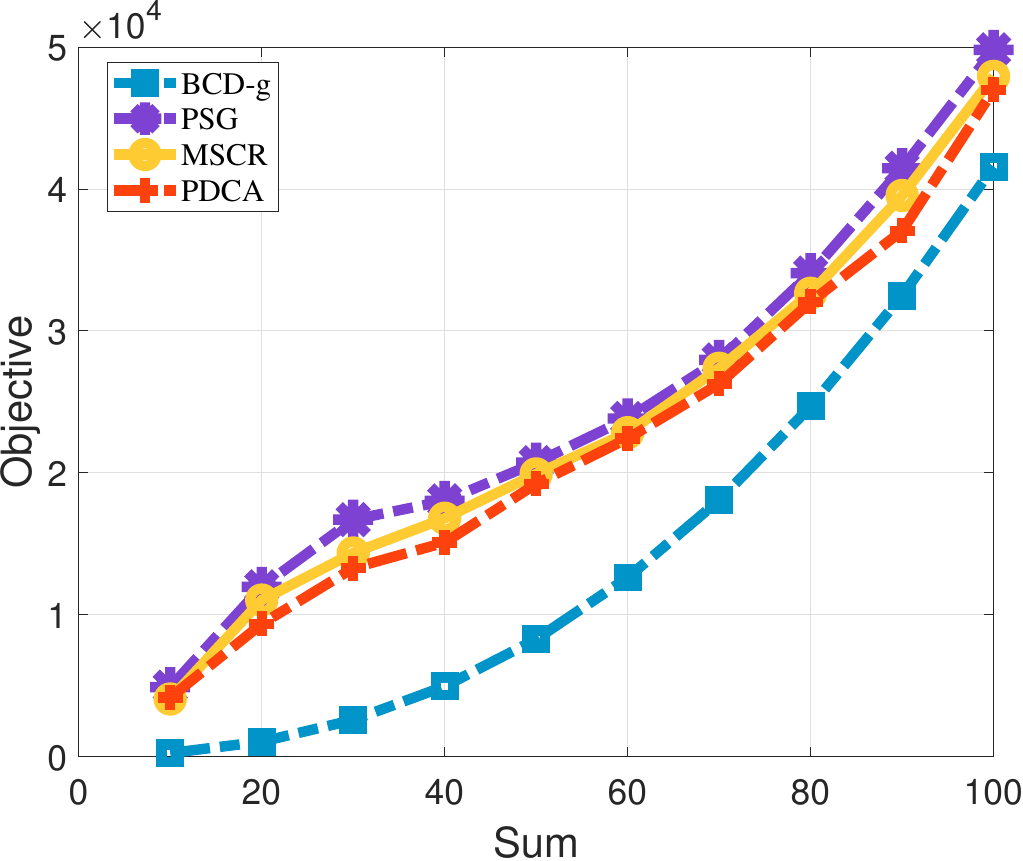}
			\caption{DCPB1 in Cifar-b}
		\end{subfigure}
		\begin{subfigure}{0.2\textwidth}
			\includegraphics[width=\textwidth]{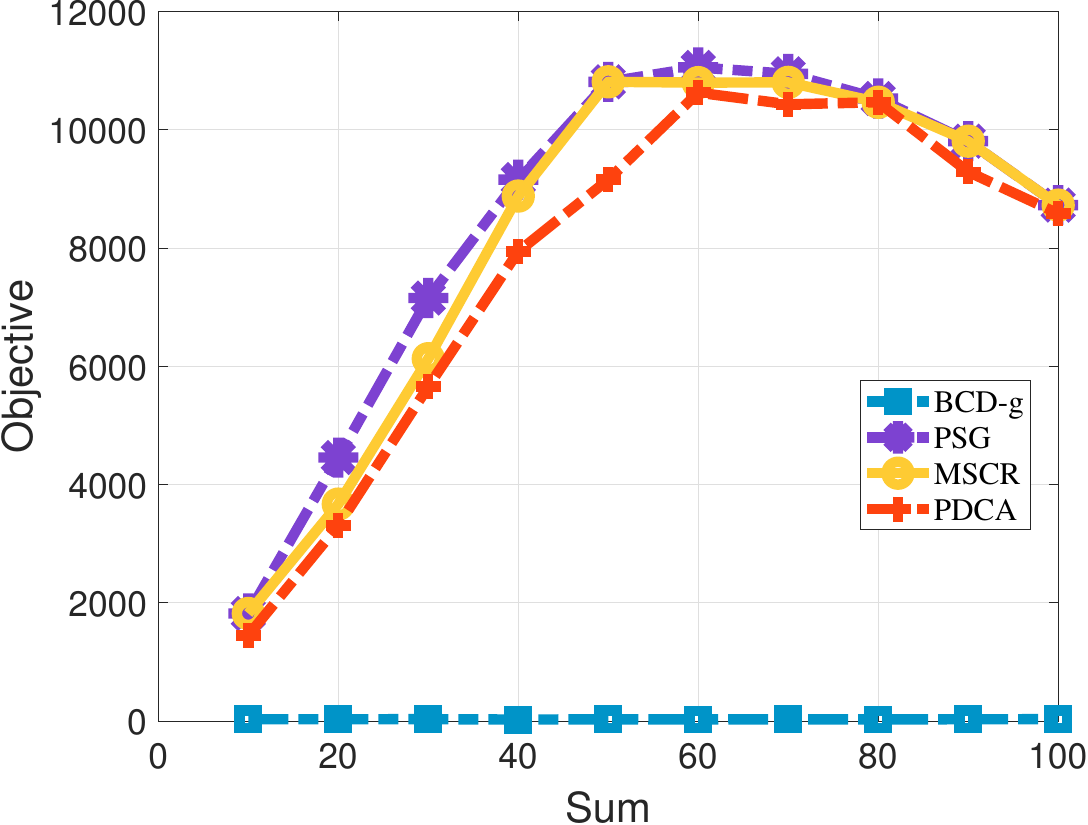}
			\caption{DCPB1 in MNIST-a}
		\end{subfigure}
		\begin{subfigure}{0.2\textwidth}
			\includegraphics[width=\textwidth]{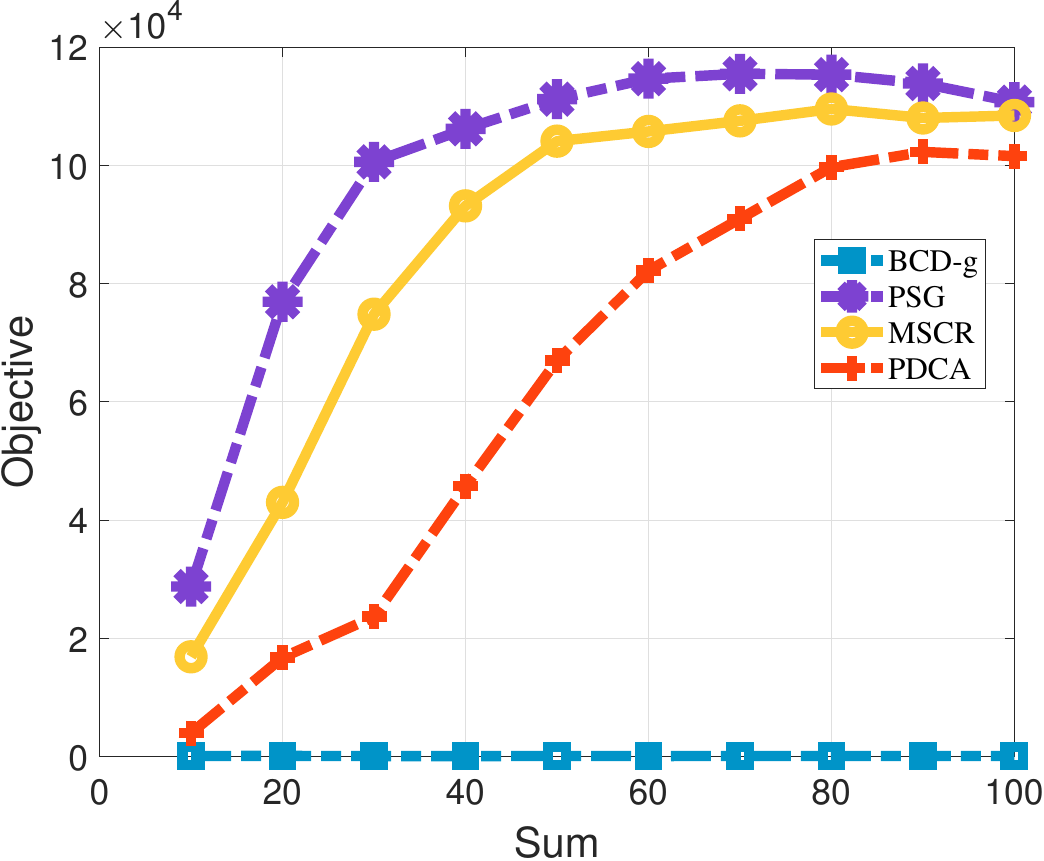}
			\caption{DCPB1 in MNIST-b}
		\end{subfigure}
		\begin{subfigure}{0.2\textwidth}
			\includegraphics[width=\textwidth]{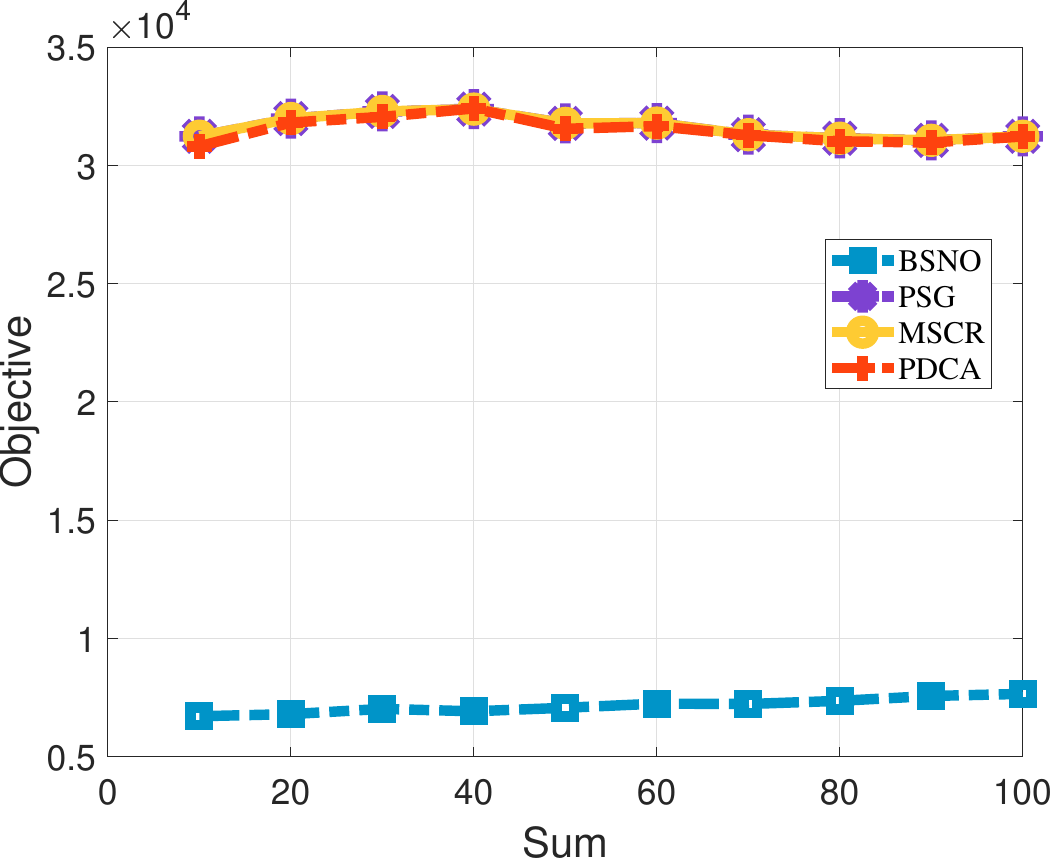}
			\caption{DCPB2 in randn-a}
		\end{subfigure}
		\begin{subfigure}{0.2\textwidth}
			\includegraphics[width=\textwidth]{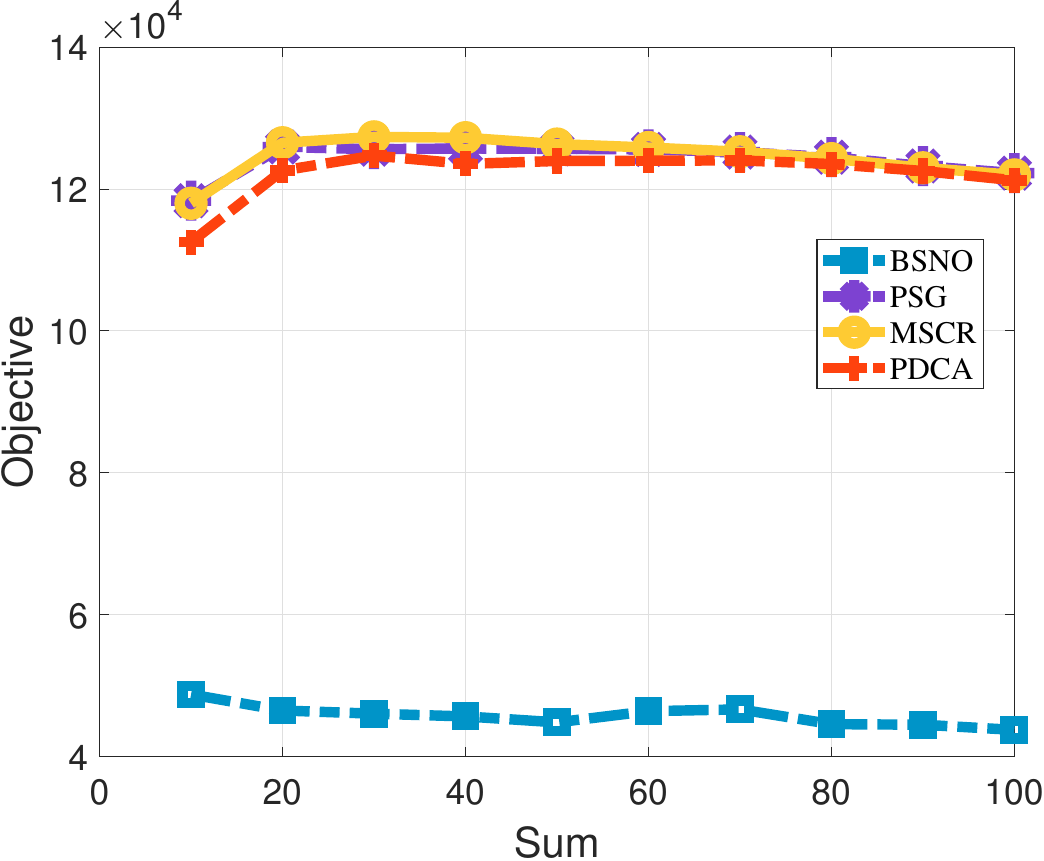}
			\caption{DCPB2 in randn-b}
		\end{subfigure}
		\begin{subfigure}{0.2\textwidth}
			\includegraphics[width=\textwidth]{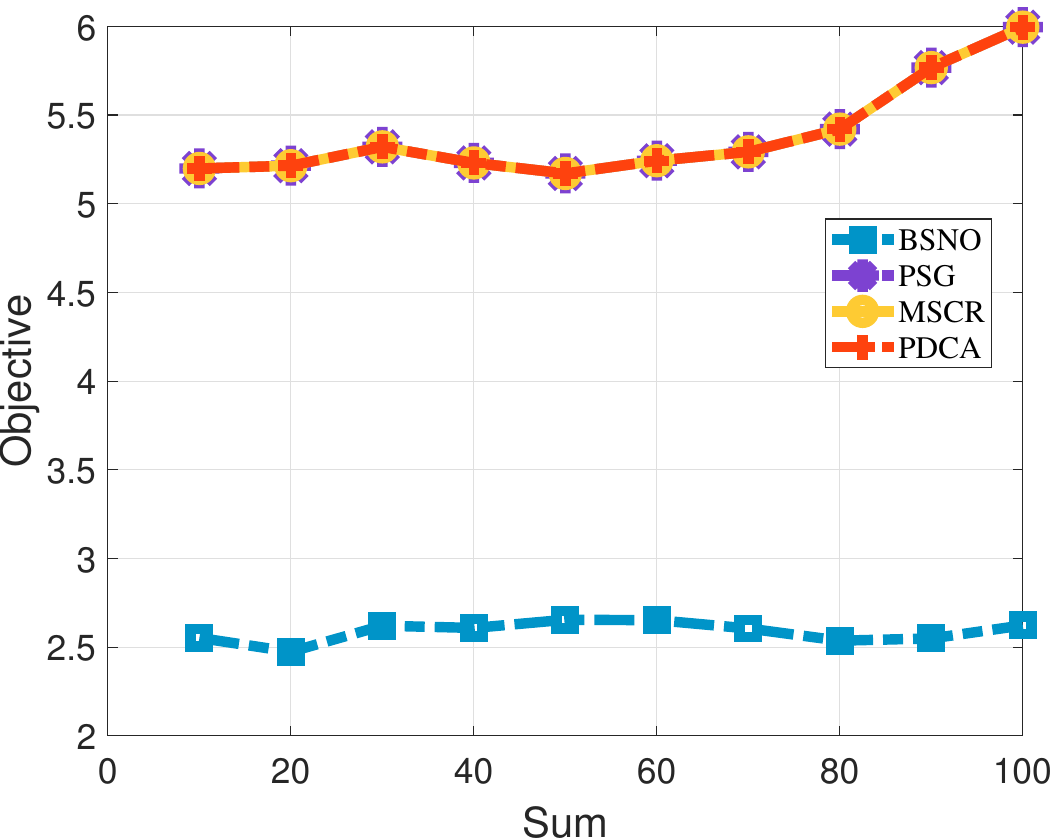}
			\caption{DCPB2 in TDT2-a}
		\end{subfigure}
		\begin{subfigure}{0.2\textwidth}
			\includegraphics[width=\textwidth]{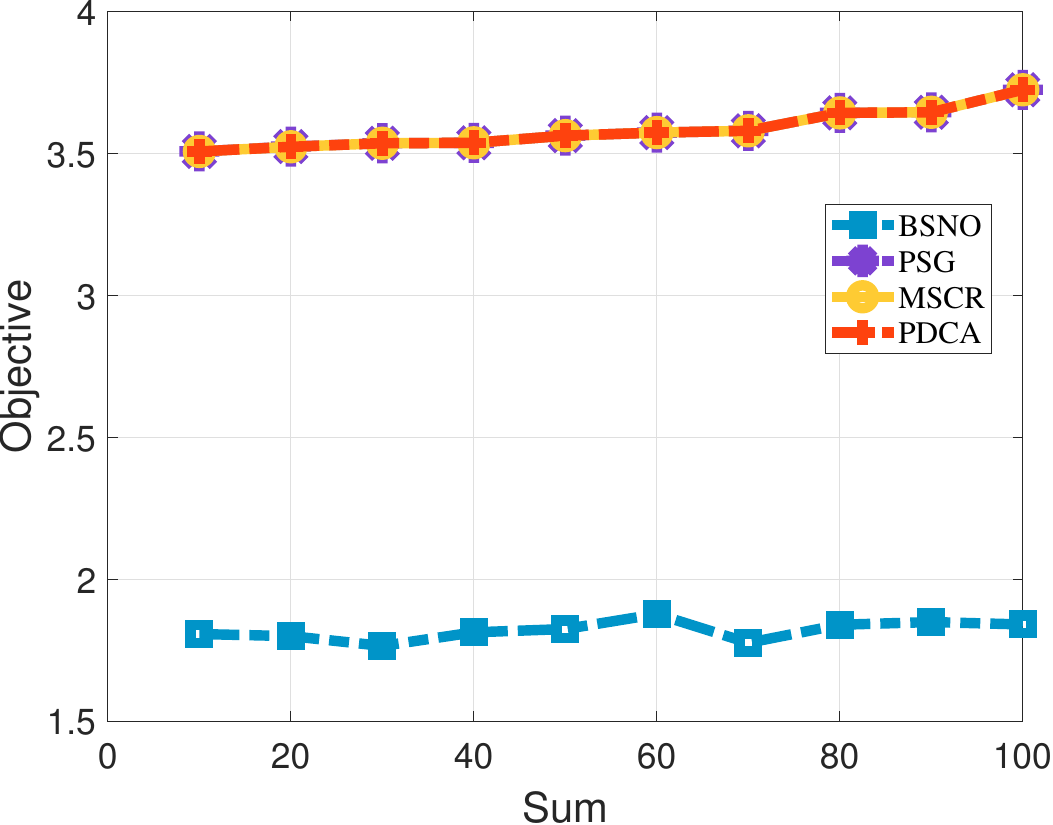}
			\caption{DCPB2 in TDT2-b}
		\end{subfigure}
		\begin{subfigure}{0.2\textwidth}
			\includegraphics[width=\textwidth]{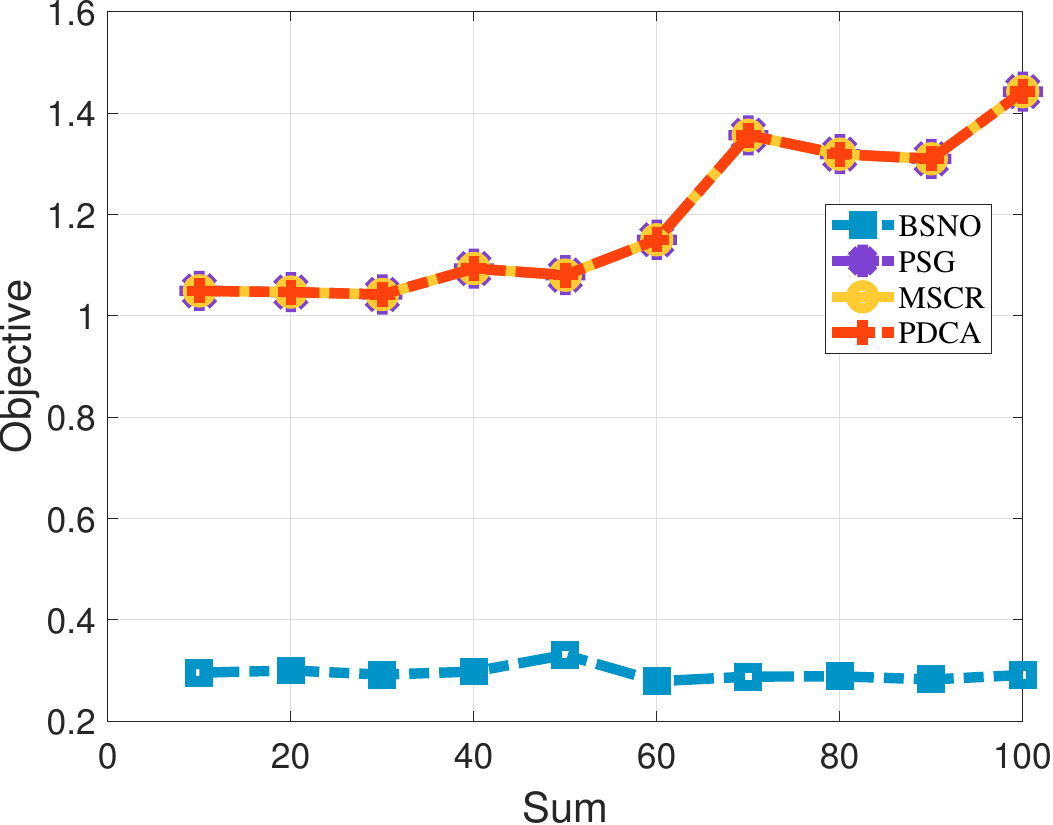}
			\caption{DCPB2 in 20News-a}
		\end{subfigure}
		\begin{subfigure}{0.2\textwidth}
			\includegraphics[width=\textwidth]{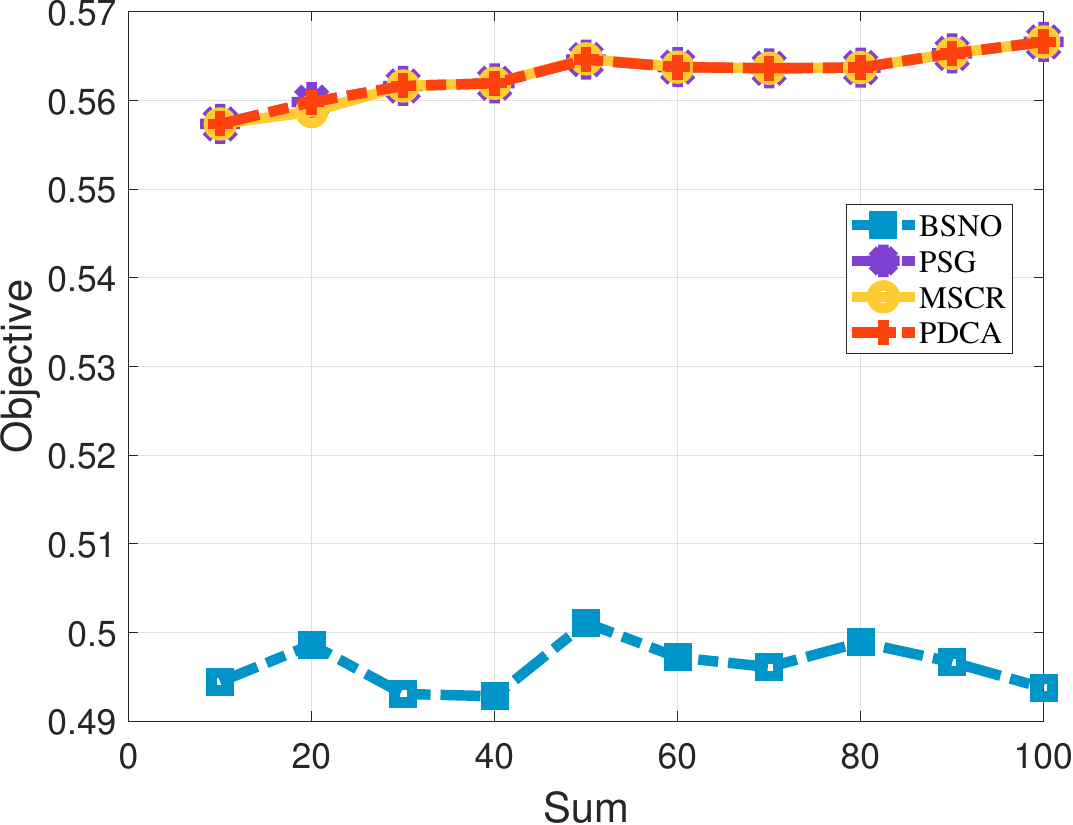}
			\caption{DCPB2 in 20News-b}
		\end{subfigure}
		\begin{subfigure}{0.2\textwidth}
			\includegraphics[width=\textwidth]{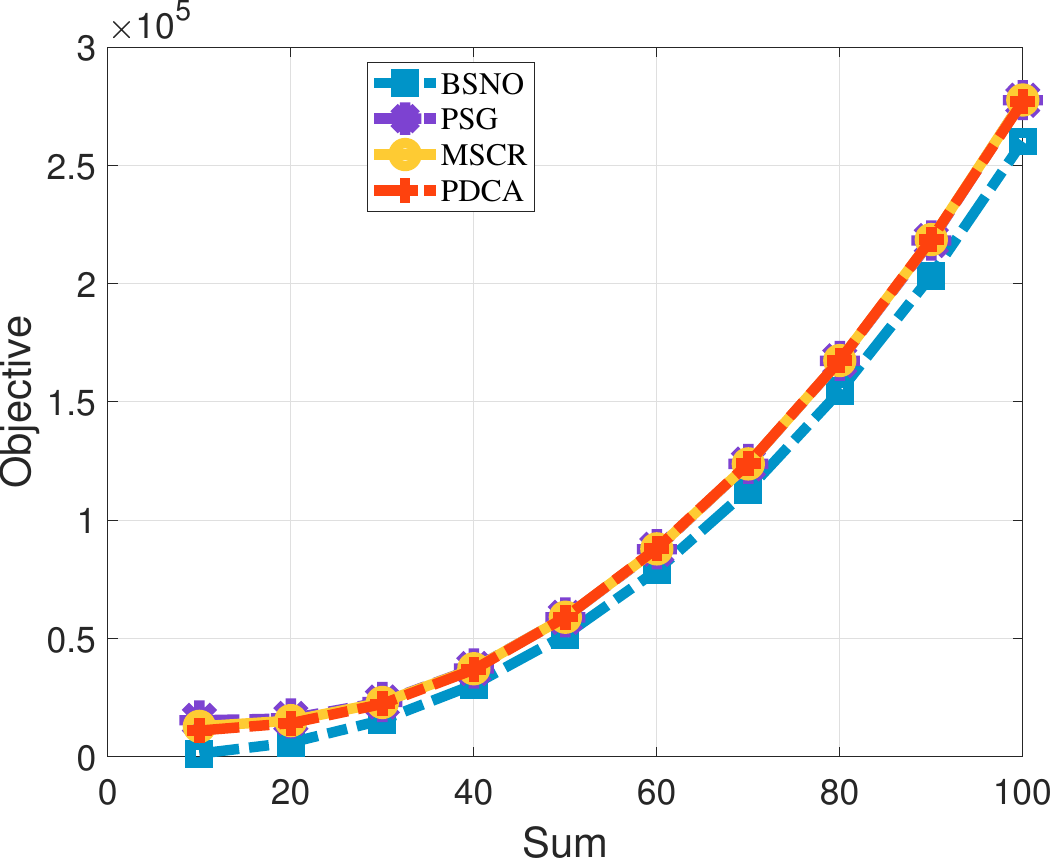}
			\caption{DCPB2 in Cifar-a}
		\end{subfigure}
		\begin{subfigure}{0.2\textwidth}
			\includegraphics[width=\textwidth]{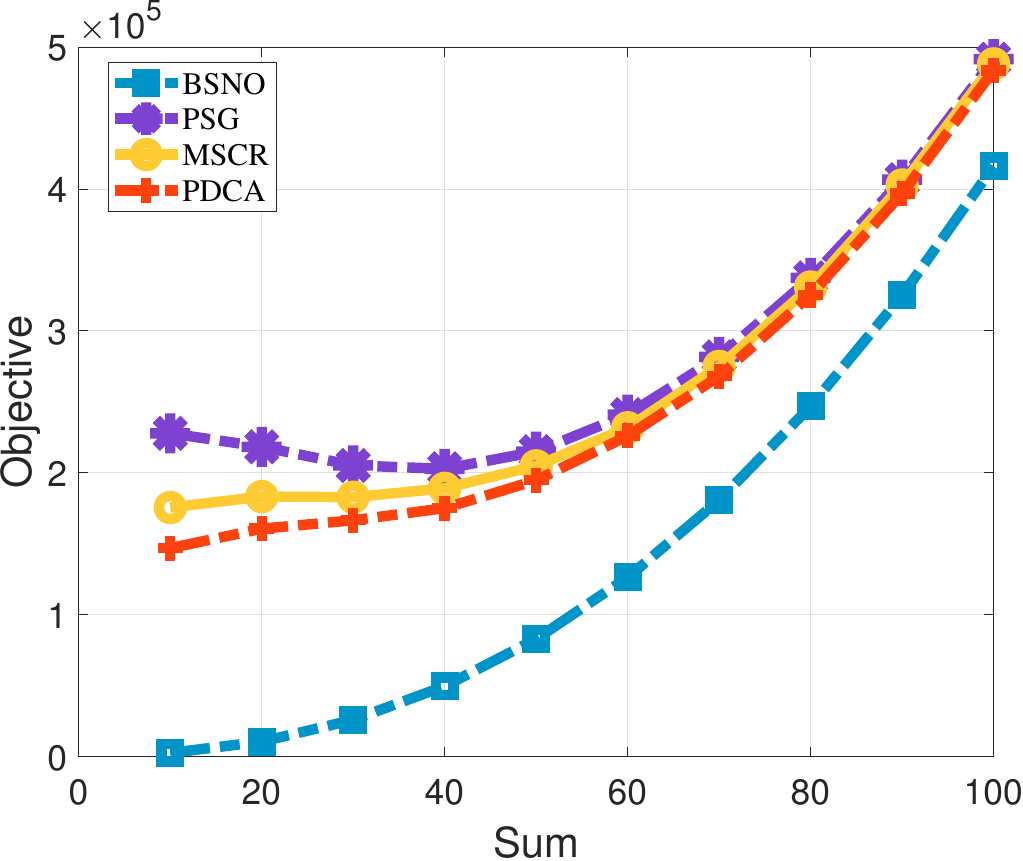}
			\caption{DCPB2 in Cifar-b}
		\end{subfigure}
		\begin{subfigure}{0.2\textwidth}
			\includegraphics[width=\textwidth]{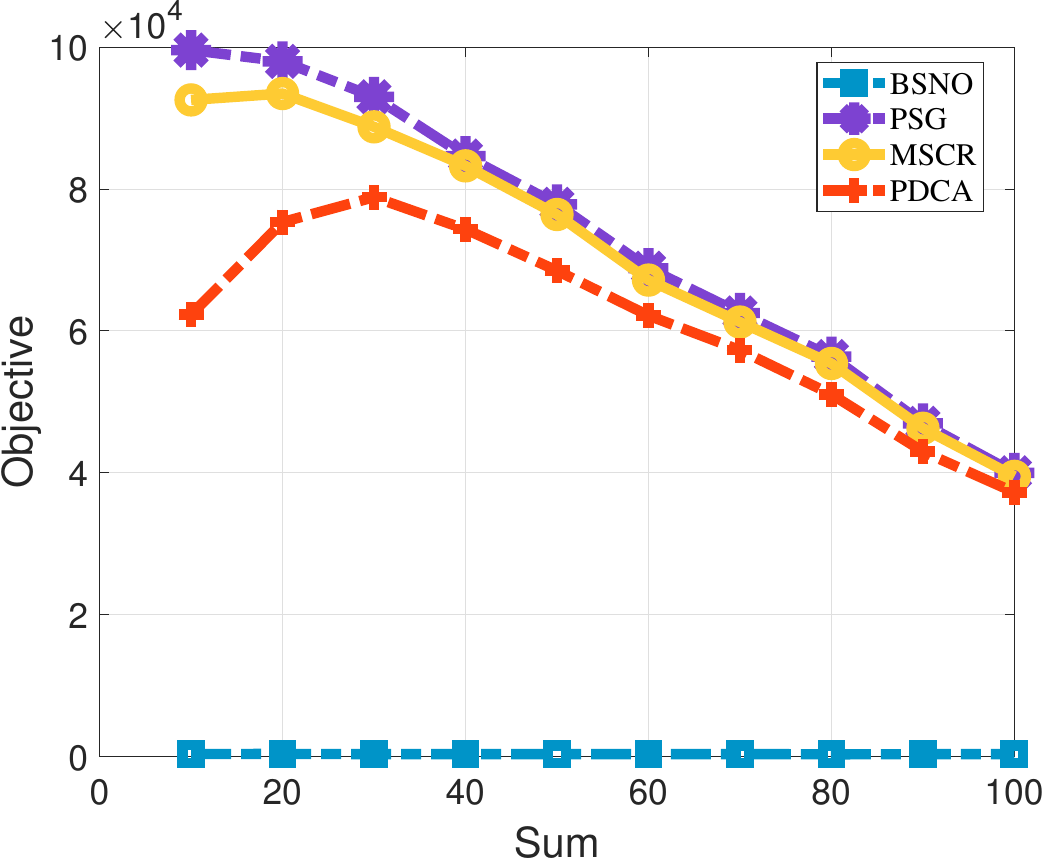}
			\caption{DCPB2 in MNIST-a}
		\end{subfigure}
		\begin{subfigure}{0.2\textwidth}
			\includegraphics[width=\textwidth]{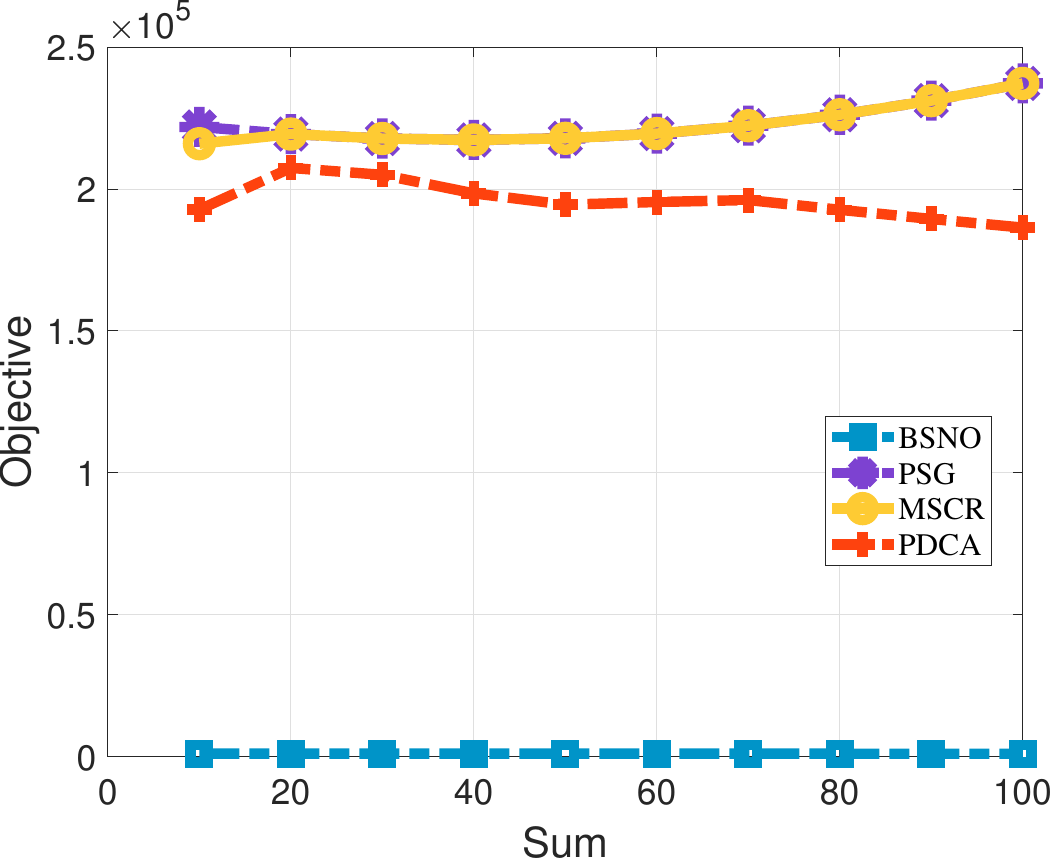}
			\caption{DCPB2 in MNIST-b}
		\end{subfigure}
		\caption{Objective values for the two DC penalized binary optimizations with varying sum $c$ across 20 datasets.}
		\label{fig:diff2}
	\end{figure}

	\section{Conclusion}\label{sec:conclusion}

	We introduce BCD-g and BCD-l-$k$ for addressing a class of structured nonconvex optimization problems with nonseparable constraints. We offer an optimality hierarchy analysis of the proposed method. This method leverages specific function structures to avoid bad local minima and achieve better stationary points. Furthermore, we demonstrate that the developed method can converge to coordinate-wise stationary points with linear convergence rates. Our method has exhibited superior performance compared to existing approaches, both theoretically and experimentally.
	 d

	\normalem

	\bibliography{mybib}

\begin{thebibliography}{39}
\providecommand{\natexlab}[1]{#1}
\providecommand{\url}[1]{\texttt{#1}}
\expandafter\ifx\csname urlstyle\endcsname\relax
  \providecommand{\doi}[1]{doi: #1}\else
  \providecommand{\doi}{doi: \begingroup \urlstyle{rm}\Url}\fi

\bibitem[Alber et~al.(1998)Alber, Iusem, and Solodov]{alber1998projected}
Ya~I Alber, Alfredo~N Iusem, and Mikhail~V Solodov.
\newblock On the projected subgradient method for nonsmooth convex optimization
  in a hilbert space.
\newblock \emph{Mathematical Programming}, 81:\penalty0 23--35, 1998.

\bibitem[Asadi \& Povh(2021)Asadi and Povh]{asadi2021block}
Soodabeh Asadi and Janez Povh.
\newblock A block coordinate descent-based projected gradient algorithm for
  orthogonal non-negative matrix factorization.
\newblock \emph{Mathematics}, 9\penalty0 (5):\penalty0 540, 2021.

\bibitem[Asteris et~al.(2014)Asteris, Papailiopoulos, and
  Dimakis]{asteris2014nonnegative}
Megasthenis Asteris, Dimitris Papailiopoulos, and Alexandros Dimakis.
\newblock Nonnegative sparse pca with provable guarantees.
\newblock In \emph{International Conference on Machine Learning}, pp.\
  1728--1736. PMLR, 2014.

\bibitem[Beck(2014)]{beck20142}
Amir Beck.
\newblock The 2-coordinate descent method for solving double-sided simplex
  constrained minimization problems.
\newblock \emph{Journal of Optimization Theory and Applications}, 162:\penalty0
  892--919, 2014.

\bibitem[Beck \& Teboulle(2003)Beck and Teboulle]{beck2003mirror}
Amir Beck and Marc Teboulle.
\newblock Mirror descent and nonlinear projected subgradient methods for convex
  optimization.
\newblock \emph{Operations Research Letters}, 31\penalty0 (3):\penalty0
  167--175, 2003.

\bibitem[Beck \& Tetruashvili(2013)Beck and Tetruashvili]{beck2013convergence}
Amir Beck and Luba Tetruashvili.
\newblock On the convergence of block coordinate descent type methods.
\newblock \emph{SIAM journal on Optimization}, 23\penalty0 (4):\penalty0
  2037--2060, 2013.

\bibitem[Benidis et~al.(2017)Benidis, Feng, and Palomar]{benidis2017sparse}
Konstantinos Benidis, Yiyong Feng, and Daniel~P Palomar.
\newblock Sparse portfolios for high-dimensional financial index tracking.
\newblock \emph{IEEE Transactions on signal processing}, 66\penalty0
  (1):\penalty0 155--170, 2017.

\bibitem[Blondel et~al.(2013)Blondel, Seki, and Uehara]{blondel2013block}
Mathieu Blondel, Kazuhiro Seki, and Kuniaki Uehara.
\newblock Block coordinate descent algorithms for large-scale sparse multiclass
  classification.
\newblock \emph{Machine learning}, 93:\penalty0 31--52, 2013.

\bibitem[Carmon et~al.(2020)Carmon, Duchi, Hinder, and
  Sidford]{carmon2020lower}
Yair Carmon, John~C Duchi, Oliver Hinder, and Aaron Sidford.
\newblock Lower bounds for finding stationary points i.
\newblock \emph{Mathematical Programming}, 184\penalty0 (1):\penalty0 71--120,
  2020.

\bibitem[Drikvandi \& Lawal(2023)Drikvandi and Lawal]{drikvandi2023sparse}
Reza Drikvandi and Olamide Lawal.
\newblock Sparse principal component analysis for natural language processing.
\newblock \emph{Annals of data science}, 10\penalty0 (1):\penalty0 25--41,
  2023.

\bibitem[Eltved(2021)]{eltved2021convex}
Anders Eltved.
\newblock Convex relaxation techniques for nonlinear optimization.
\newblock 2021.

\bibitem[Escribe et~al.(2021)Escribe, Lu, Keller-Baruch, Forgetta, Xiao,
  Richards, Bhatnagar, Oualkacha, and Greenwood]{escribe2021block}
C{\'e}lia Escribe, Tianyuan Lu, Julyan Keller-Baruch, Vincenzo Forgetta, Bowei
  Xiao, J~Brent Richards, Sahir Bhatnagar, Karim Oualkacha, and Celia~MT
  Greenwood.
\newblock Block coordinate descent algorithm improves variable selection and
  estimation in error-in-variables regression.
\newblock \emph{Genetic Epidemiology}, 45\penalty0 (8):\penalty0 874--890,
  2021.

\bibitem[Gong et~al.(2013)Gong, Zhang, Lu, Huang, and Ye]{gong2013general}
Pinghua Gong, Changshui Zhang, Zhaosong Lu, Jianhua Huang, and Jieping Ye.
\newblock A general iterative shrinkage and thresholding algorithm for
  non-convex regularized optimization problems.
\newblock In \emph{international conference on machine learning}, pp.\  37--45.
  PMLR, 2013.

\bibitem[Gotoh et~al.(2018)Gotoh, Takeda, and Tono]{gotoh2018dc}
Jun-ya Gotoh, Akiko Takeda, and Katsuya Tono.
\newblock Dc formulations and algorithms for sparse optimization problems.
\newblock \emph{Mathematical Programming}, 169:\penalty0 141--176, 2018.

\bibitem[Hauswirth et~al.(2016)Hauswirth, Bolognani, Hug, and
  D{\"o}rfler]{hauswirth2016projected}
Adrian Hauswirth, Saverio Bolognani, Gabriela Hug, and Florian D{\"o}rfler.
\newblock Projected gradient descent on riemannian manifolds with applications
  to online power system optimization.
\newblock In \emph{2016 54th Annual Allerton Conference on Communication,
  Control, and Computing (Allerton)}, pp.\  225--232. IEEE, 2016.

\bibitem[Hong et~al.(2017)Hong, Wang, Razaviyayn, and Luo]{hong2017iteration}
Mingyi Hong, Xiangfeng Wang, Meisam Razaviyayn, and Zhi-Quan Luo.
\newblock Iteration complexity analysis of block coordinate descent methods.
\newblock \emph{Mathematical Programming}, 163:\penalty0 85--114, 2017.

\bibitem[Hsieh et~al.(2008)Hsieh, Chang, Lin, Keerthi, and
  Sundararajan]{hsieh2008dual}
Cho-Jui Hsieh, Kai-Wei Chang, Chih-Jen Lin, S~Sathiya Keerthi, and
  Sellamanickam Sundararajan.
\newblock A dual coordinate descent method for large-scale linear svm.
\newblock In \emph{Proceedings of the 25th international conference on Machine
  learning}, pp.\  408--415, 2008.

\bibitem[Jain et~al.(2014)Jain, Tewari, and Kar]{jain2014iterative}
Prateek Jain, Ambuj Tewari, and Purushottam Kar.
\newblock On iterative hard thresholding methods for high-dimensional
  m-estimation.
\newblock \emph{Advances in neural information processing systems}, 27, 2014.

\bibitem[Jain et~al.(2017)Jain, Kar, et~al.]{jain2017non}
Prateek Jain, Purushottam Kar, et~al.
\newblock Non-convex optimization for machine learning.
\newblock \emph{Foundations and Trends{\textregistered} in Machine Learning},
  10\penalty0 (3-4):\penalty0 142--363, 2017.

\bibitem[Jenatton et~al.(2010)Jenatton, Obozinski, and
  Bach]{jenatton2010structured}
Rodolphe Jenatton, Guillaume Obozinski, and Francis Bach.
\newblock Structured sparse principal component analysis.
\newblock In \emph{Proceedings of the Thirteenth International Conference on
  Artificial Intelligence and Statistics}, pp.\  366--373. JMLR Workshop and
  Conference Proceedings, 2010.

\bibitem[Liu et~al.(2014{\natexlab{a}})Liu, Wang, Shen, Zhang, and
  Wang]{liu2014blockwise}
Bao-Di Liu, Yu-Xiong Wang, Bin Shen, Yu-Jin Zhang, and Yan-Jiang Wang.
\newblock Blockwise coordinate descent schemes for sparse representation.
\newblock In \emph{2014 IEEE International Conference on Acoustics, Speech and
  Signal Processing (ICASSP)}, pp.\  5267--5271. IEEE, 2014{\natexlab{a}}.

\bibitem[Liu et~al.(2014{\natexlab{b}})Liu, Wright, R{\'e}, Bittorf, and
  Sridhar]{liu2014asynchronous}
Ji~Liu, Steve Wright, Christopher R{\'e}, Victor Bittorf, and Srikrishna
  Sridhar.
\newblock An asynchronous parallel stochastic coordinate descent algorithm.
\newblock In \emph{International Conference on Machine Learning}, pp.\
  469--477. PMLR, 2014{\natexlab{b}}.

\bibitem[Luo \& Tseng(1993)Luo and Tseng]{luo1993error}
Zhi-Quan Luo and Paul Tseng.
\newblock Error bounds and convergence analysis of feasible descent methods: a
  general approach.
\newblock \emph{Annals of Operations Research}, 46\penalty0 (1):\penalty0
  157--178, 1993.

\bibitem[Maing{\'e}(2008)]{mainge2008strong}
Paul-Emile Maing{\'e}.
\newblock Strong convergence of projected subgradient methods for nonsmooth and
  nonstrictly convex minimization.
\newblock \emph{Set-valued analysis}, 16\penalty0 (7):\penalty0 899--912, 2008.

\bibitem[Massart \& Abrol(2022)Massart and Abrol]{massart2022coordinate}
Estelle Massart and Vinayak Abrol.
\newblock Coordinate descent on the orthogonal group for recurrent neural
  network training.
\newblock In \emph{Proceedings of the AAAI Conference on Artificial
  Intelligence}, volume~36, pp.\  7744--7751, 2022.

\bibitem[Nesterov(2012)]{nesterov2012efficiency}
Yu~Nesterov.
\newblock Efficiency of coordinate descent methods on huge-scale optimization
  problems.
\newblock \emph{SIAM Journal on Optimization}, 22\penalty0 (2):\penalty0
  341--362, 2012.

\bibitem[Nie et~al.(2021)Nie, Xue, Wu, Wang, Li, and Li]{nie2021coordinate}
Feiping Nie, Jingjing Xue, Danyang Wu, Rong Wang, Hui Li, and Xuelong Li.
\newblock Coordinate descent method for k-means.
\newblock \emph{IEEE Transactions on Pattern Analysis and Machine
  Intelligence}, 44\penalty0 (5):\penalty0 2371--2385, 2021.

\bibitem[Patrascu \& Necoara(2015{\natexlab{a}})Patrascu and
  Necoara]{patrascu2015efficient}
Andrei Patrascu and Ion Necoara.
\newblock Efficient random coordinate descent algorithms for large-scale
  structured nonconvex optimization.
\newblock \emph{Journal of Global Optimization}, 61\penalty0 (1):\penalty0
  19--46, 2015{\natexlab{a}}.

\bibitem[Patrascu \& Necoara(2015{\natexlab{b}})Patrascu and
  Necoara]{patrascu2015random}
Andrei Patrascu and Ion Necoara.
\newblock Random coordinate descent methods for sparse optimization:
  application to sparse control.
\newblock In \emph{2015 20th International Conference on Control Systems and
  Computer Science}, pp.\  909--914. IEEE, 2015{\natexlab{b}}.

\bibitem[Ruspini et~al.(2019)Ruspini, Bezdek, and Keller]{ruspini2019fuzzy}
Enrique~H Ruspini, James~C Bezdek, and James~M Keller.
\newblock Fuzzy clustering: A historical perspective.
\newblock \emph{IEEE Computational Intelligence Magazine}, 14\penalty0
  (1):\penalty0 45--55, 2019.

\bibitem[Toland(1979)]{Toland1979}
J.~F. Toland.
\newblock A duality principle for non-convex optimisation and the calculus of
  variations.
\newblock \emph{Archive for Rational Mechanics and Analysis}, 71\penalty0
  (1):\penalty0 41--61, May 1979.

\bibitem[Tseng \& Yun(2009)Tseng and Yun]{tseng2009coordinate}
Paul Tseng and Sangwoon Yun.
\newblock A coordinate gradient descent method for nonsmooth separable
  minimization.
\newblock \emph{Mathematical Programming}, 117:\penalty0 387--423, 2009.

\bibitem[Yan et~al.(2024)Yan, Zhong, Jin, Ke, Xie, and Huang]{yan2024binary}
Xueming Yan, Guo Zhong, Yaochu Jin, Xiaohua Ke, Fenfang Xie, and Guoheng Huang.
\newblock Binary spectral clustering for multi-view data.
\newblock \emph{Information Sciences}, pp.\  120899, 2024.

\bibitem[Yuan(2023)]{yuan2023coordinate}
Ganzhao Yuan.
\newblock Coordinate descent methods for dc minimization: Optimality conditions
  and global convergence.
\newblock In \emph{Proceedings of the AAAI Conference on Artificial
  Intelligence}, volume~37, pp.\  11034--11042, 2023.

\bibitem[Yuan(2024)]{pmlr-v235-yuan24a}
Ganzhao Yuan.
\newblock Smoothing proximal gradient methods for nonsmooth sparsity
  constrained optimization: Optimality conditions and global convergence.
\newblock In Ruslan Salakhutdinov, Zico Kolter, Katherine Heller, Adrian
  Weller, Nuria Oliver, Jonathan Scarlett, and Felix Berkenkamp (eds.),
  \emph{Proceedings of the 41st International Conference on Machine Learning},
  volume 235 of \emph{Proceedings of Machine Learning Research}, pp.\
  57842--57870. PMLR, 21--27 Jul 2024.

\bibitem[Yuan et~al.(2019)Yuan, Shen, and Zheng]{yuan2019decomposition}
Ganzhao Yuan, Li~Shen, and Wei-Shi Zheng.
\newblock A decomposition algorithm for the sparse generalized eigenvalue
  problem.
\newblock In \emph{Proceedings of the IEEE/CVF Conference on Computer Vision
  and Pattern Recognition}, pp.\  6113--6122, 2019.

\bibitem[Yuan et~al.(2020)Yuan, Shen, and Zheng]{yuan2020block}
Ganzhao Yuan, Li~Shen, and Wei-Shi Zheng.
\newblock A block decomposition algorithm for sparse optimization.
\newblock In \emph{Proceedings of the 26th ACM SIGKDD International Conference
  on Knowledge Discovery \& Data Mining}, pp.\  275--285, 2020.

\bibitem[Zhang(2010)]{zhang2010analysis}
Tong Zhang.
\newblock Analysis of multi-stage convex relaxation for sparse regularization.
\newblock \emph{Journal of Machine Learning Research}, 11\penalty0 (3), 2010.

\bibitem[Zhang(2013)]{zhang2013multi}
Tong Zhang.
\newblock Multi-stage convex relaxation for feature selection.
\newblock \emph{Bernoulli}, 19\penalty0 (5B):\penalty0 2277--2293, 2013.

\end{thebibliography}
	\bibliographystyle{plain}

	\appendix
	\onecolumn
	{\huge \textbf{Appendix}}
	
	\section{Proofs of Results in Section \ref{sec:opt}}
	
	\begin{lemma} 
		\noi For any coordinate-wise stationary point $\ddot{\x}$, the function $g(\ddot{\x}) = -\|\ddot{\x}\|_{[s]} \triangleq - \sum_{j=1}^{k}|\ddot{\x}_{[j]}|$ is locally $\rho$-bounded nonconvex with $\rho < + \infty$. 
	\end{lemma} 
	
	\subsection{Proof of Lemma \ref{lemma:rhobounded}}\label{app:lemma:rhobounded}
	\begin{proof}
		\noi For any $\ddot{\x}$ and a given $k$, the subgradient of $\|\ddot{\x}\|_{[s]}$ can be computed as $\partial \|\ddot{\x}\|_{[s]} = \left\{
		\begin{aligned}
			&\text{sign}(\ddot{\x}_i),\ i \in \vartriangle_k(\ddot{\x})\ \text{and}\ \ddot{\x}_i \neq 0\\
			&[-1,1],\ \text{else}
		\end{aligned}
		\right.$, where $\vartriangle_k(\ddot{\x})$ is the index of the largest $k$ components of $\ddot{\x}$. Given two reference points $\ddot{\x} \neq \y$, we assume that there exists a constant $\epsilon > 0$ satisfying $\|\ddot{\x}-\y\|_2\geq \epsilon$. We have: $ z(\ddot{\x}) - z(\y) - \la \ddot{\x}- \y, \partial g(\ddot{\x}) \ra = - \|\ddot{\x}\|_{[s]} + \|\y\|_{[s]} - \la \ddot{\x}- \y, \partial (-\|\ddot{\x}\|_{[s]}) \ra 	\overset{\step{1}}{\leq} |\y-\ddot{\x}\|_{[s]}+ \|\y-\ddot{\x}\|_2 \cdot \|\partial \|\ddot{\x}\|_{[s]}\|_2 \overset{\step{2}}{\leq}  \|\y-\ddot{\x}\|_1+\|\y-\ddot{\x}\|_2 \cdot \sqrt{n}   \overset{\step{3}}{\leq}  2\sqrt{n} \|\ddot{\x}-\y\|_2 	\overset{\step{4}}{\leq}  \tfrac{2\sqrt{n}}{\epsilon} \|\ddot{\x}-\y\|_2^2$, where step $\step{1}$ adapts the triangle inequality that $\|\y\|_{[s]} - \|\ddot{\x}\|_{[s]} \leq \|\y-\ddot{\x}\|_{[s]}$ as $\|\cdot\|_{[s]}$ is a norm; step $\step{2}$ uses the fact that $\|\partial \|\ddot{\x}\|_{[s]}\|_2 \leq \sqrt{n}$; step $\step{3}$ uses the fact that $\|\ddot{\x}\|_2\leq \|\ddot{\x}\|_1 \leq \sqrt{n}\|\ddot{\x}\|_2$ for all $\ddot{\x}$; step $\step{4}$ uses $\|\ddot{\x}-\y\|_2\geq \epsilon$. Hence, $z(\x)$ is a $\rho$-bounded nonconvex function with $\rho < + \infty$.
	\end{proof}
	\subsection{Proof of Theorem \ref{theo:descent}}\label{app:theo:descent}
	
	\begin{proof}
		\noi We derive the following inequalities:	
		\beq
		\sum_{\B\in \Pi_n^k} g(\x + \U_\B\d_\B)
		&\overset{\step{1}}{\leq} &  \sum_{\B\in \Pi_n^k} g(\x) + \la \U_\B\d_\B, \partial g(\x) \ra  
		\overset{\step{2}}{=}  C_n^k g(\x) + C_n^k \tfrac{k}{n}\la \d, \partial g(\x)\ra \nonumber \\ 
		&\overset{\step{3}}{\leq} &  C_n^k g(\x) + C_n^k \tfrac{k}{n} \left(g(\x + \d) - g(\x) + \tfrac{\rho}{2}\|\d\|_2^2 \right),\label{eq:gx}
		\eeq
		\noi where step \step{1} uses the fact that $-g(\x)$ is convex with $g(\y)\leq g(\x) + \la \y-\x,\ \partial g(\x)\ra$; step \step{2} uses the equality in \textbf{Part} \textbf{(ii)} in Lemma \ref{lemma:bound:4}; step \step{3} uses the fact that $g(\x)$ is locally $\rho$-bounded nonconvex and it holds that $g(\x)\leq g(\y) + \la \x-\y,\partial g(\x)\ra+\tfrac{\rho}{2}\|\x-\y\|_2^2$ with $\y = \x + \d$.
		\noi For any $\ddot{\x}$ and $\ddot{\x}+\d$ are in feasible set $\Omega$, we obtain follow inequalities:
		\beq
		f(\ddot{\x}) + h(\ddot{\x}) + g(\ddot{\x})
		&\overset{\step{1}}{\leq} & f(\ddot{\x}) + \la \U_\B\d_\B,\nabla f(\ddot{\x})  \ra + h(\ddot{\x} + \U_\B\d_\B) + g(\ddot{\x} + \U_\B\d_\B) + \tfrac{1}{2} \|\d_\B\|_{\Q_+}^2 \nonumber\\
		\Rightarrow h(\ddot{\x}) + g(\ddot{\x})
		&\leq & \la\U_\B\d_\B ,\nabla f(\ddot{\x})  \ra + h(\ddot{\x} + \U_\B\d_\B) + g(\ddot{\x} + \U_\B\d_\B)  + 
		\tfrac{1}{2} \|\d_\B\|_{\Q_+}^2, \label{ieq:therom1}
		\eeq
		where step \step{1} uses the optimality of \textbf{CSW}-point and the definition $\Q_+ \triangleq \Q_{\B\B} + \theta\I_k$.
		\noi Summing the inequality (\ref{ieq:therom1}) over $\B\in \Pi_n^k$, we have:
		\beq
		&& C_n^k \left(h(\ddot{\x}) + g(\ddot{\x})\right)\nonumber\\
		&\leq& \sum_{\B\in \Pi_n^k} \bigg ( \la \U_\B\d_\B,\nabla f(\ddot{\x})  \ra + \tfrac{1}{2} \|\d_\B\|_{\Q_+}^2 + h(\ddot{\x} + \U_\B\d_\B) + g(\ddot{\x} + \U_\B\d_\B) \bigg )  \nonumber\\
		&\overset{\step{1}}{\leq} & \sum_{\B\in \Pi_n^k} \bigg ( \la \U_\B\d_\B , \nabla f(\ddot{\x})  \ra + h(\ddot{\x} +\U_\B\d_\B) + g(\ddot{\x} + \U_\B\d_\B) \bigg ) + \sum_{\B\in \Pi_n^k} \tfrac{\Qup + \theta}{2} \|\d_\B\|_2^2  \nonumber\\
		&\overset{\step{2}}{=} & \sum_{\B\in \Pi_n^k} \bigg ( \la \U_\B\d_\B , \nabla f(\ddot{\x})  \ra + h(\ddot{\x} +\U_\B\d_\B) + g(\ddot{\x} + \U_\B\d_\B) \bigg ) + \tfrac{kC_n^k(\Qup + \theta)}{2n} \|\d\|_2^2 \nonumber\\
		&\overset{\step{3}}{\leq} &  \sum_{\B\in \Pi_n^k} \bigg ( \la \U_\B\d_\B,\nabla f(\ddot{\x})  \ra + h(\ddot{\x} +\U_\B\d_\B) \bigg )  + \tfrac{kC_n^k(\Qup + \theta)}{2n} \|\d\|_2^2 + C_n^k g(\ddot{\x}) \nonumber\\
		&& +\ C_n^k \tfrac{k}{n} \left(g(\ddot{\x} + \d) - g(\ddot{\x}) + \tfrac{\rho}{2}\|\d\|_2^2 \right) \nonumber\\
		&\overset{\step{4}}{\leq} &  \sum_{\B\in \Pi_n^k} \la \U_\B\d_\B,\nabla f(\ddot{\x})  \ra   + C_n^k h(\ddot{\x}) + C_n^k g(\ddot{\x}) +  \tfrac{kC_n^k}{n} \psi + \tfrac{kC_n^k(\Qup + \theta + \rho)}{2n} \|\d\|_2^2, \label{ieq:hxgx}
		\eeq

		\noi step \step{1} uses the definition: $\Qup \triangleq \|\Q\|_2$ and the inequality $\sum_{\B\in \Pi_n^k} \tfrac{1}{2} \|\d_\B\|_{\Q_+}^2 \leq \sum_{\B\in \Pi_n^k} \tfrac{\Qup + \theta}{2} \|\d_\B\|_2^2 $; step \step{2} uses the \textbf{Part} \textbf{(i)} in Lemma \ref{lemma:bound:4}; step \step{3} uses (\ref{eq:gx}); step \step{4} uses the \textbf{Part} \textbf{(iv)} in Lemma \ref{lemma:bound:4} and the definition $\psi \triangleq h(\ddot{\x} + \d) - h(\ddot{\x}) + g(\ddot{\x} + \d) - g(\ddot{\x})$. According to (\ref{ieq:hxgx}), we obtain following inequalities:
		\beq
		0\ &\leq &  \sum_{\B\in \Pi_n^k}  \la \U_\B\d_\B , \nabla f(\ddot{\x})  \ra + \tfrac{kC_n^k }{n} \psi + \tfrac{kC_n^k(\Qup + \theta + \rho)}{2n} \|\d\|_2^2 \nonumber\\
		&\overset{\step{1}}{=} & \la \tfrac{kC_n^k }{n} \d,\ \nabla f(\ddot{\x})\ra + \tfrac{kC_n^k }{n} \psi + \tfrac{kC_n^k(\Qup + \theta + \rho)}{2n} \|\d\|_2^2 \nonumber\\ 
		& = & \la \d,\ \nabla f(\ddot{\x})\ra + \psi + \tfrac{\Qup + \theta + \rho}{2}\|\d\|_2^2  \nonumber\\ 
		&\overset{\step{2}}{\leq} & f(\ddot{\x} + \textbf{d}) - f(\ddot{\x}) + \psi + \tfrac{\Qup + \theta + \rho}{2}\|\d\|_2^2 \nonumber\\ 
		&\overset{\step{3}}{=} & F(\ddot{\x}+\textbf{d})-F(\ddot{\x})  + \tfrac{\Qup + \theta + \rho}{2}\|\d\|_2^2 \nonumber\\ 
		\Rightarrow F(\ddot{\x}) &\leq& F(\ddot{\x}+\textbf{d}) + \tfrac{\Qup + \theta + \rho}{2}\|\d\|_2^2, \label{eq:theorem1}
		\eeq 
		where step \step{1} uses the \textbf{Part} \textbf{(ii)} in Lemma \ref{lemma:bound:4}; step \step{2} uses the convexity of $f(\cdot)$: $\la \textbf{d},\ \nabla f(\x)\ra \leq f(\x + \textbf{d}) - f(\x)$; step \step{3} uses the definition of $F(\x) \triangleq f(\x) + h(\x) + g(\x)$ and $\psi = h(\ddot{\x} + \d) - h(\ddot{\x}) + g(\ddot{\x} + \d) - g(\ddot{\x})$.
	\end{proof}
	
	\subsection{Proof of Theorem \ref{theo:optimality}}\label{app:theo:optimality}

	\begin{proof} 
		\noi (\textbf{i}) \{Optimal Point $\bar{\x}$\} $\subseteq$ \{\textbf{CWS}-Point $\ddot{\x}$\}
		
		\noi Since $f(\cdot)$ is convex and continuously differentiable, it holds as:
		\beq
		f(\x + \U_\B\d_\B)	&\leq& f(\x) + \la \d_\B, [\nabla f(\x)]_\B \ra  + \tfrac{1}{2}\|\d_\B\|_{\Q_{\B\B}}^2,\label{eq:lipschitz}
		\eeq
		\noi where $\nabla f(\cdot)$ is coordinate-wise Lipschitz continuous with $\Q_{\B\B} \geq 0$.
		
		\noi By the optimality of $\bar{\x}$, for any $\x$, it follows that: $f(\bar{\x}) + h(\bar{\x}) + g(\bar{\x})\leq f(\x) + h(\x) + g(\x)$. Letting $\x = \bar{\x} + \U_\B\d_\B$, since the gradient of $f(\cdot)$ is coordinate-wise Lipschitz continuous (\ref{eq:lipschitz}), we have:
		\beq
		&& f(\bar{\x}) + h(\bar{\x}) + g(\bar{\x}) \nn\\
		&\leq& f(\bar{\x}+\U_\B\d_\B) + h(\bar{\x}+\U_\B\d_\B) + g(\bar{\x} + \U_\B\d_\B) \nonumber\\
		&\overset{\step{1}}{\leq}& f(\bar{\x})+ \la  \U_\B\d_\B, \nabla f(\bar{\x})\ra  + \tfrac{1}{2} \|\d_\B\|_{\Q_{\B\B}}^2 + h(\bar{\x}+\U_\B\d_\B) + g(\bar{\x} + \U_\B\d_\B), \label{ieq:F(bar{x})}
		\eeq
		where step \step{1} uses the coordinate-wise Lipschitz continuity of $\nabla f(\cdot)$.
		Rearrange terms for (\ref{ieq:F(bar{x})}) and using the fact that $\theta\geq 0$, we have:
		\beq
		h(\bar{\x}) + g(\bar{\x}) &\leq &  h(\bar{\x} +\U_\B\d_\B) + g(\bar{\x} + \U_\B\d_\B)) + \la  \U_\B\d_\B , \nabla f(\bar{\x})\ra + \tfrac{1}{2} \|\d_\B\|_{\Q_{\B\B}}^2 \nonumber\\
		&\leq & h(\bar{\x} + \U_{\B}\d_{\B}) + g(\bar{\x} +\U_{\B}\d_{\B}) + \la \U_{\B}\d_{\B}, \nabla f(\bar{\x})\ra  + \tfrac{1}{2} \|\d_{\B}\|_{\Q_+}^2\nonumber\\
		\Rightarrow \mathcal{M}_{\B}(\bar{\x},\mathbf{0}_k) &\overset{\step{1}}{\leq} & \underset{\d_\B}{\text{min}}\ \mathcal{M}_{\B}(\bar{\x}, \d_\B),\ \forall \B\nonumber\\
		\Rightarrow \mathbf{0}_k &\in & \bar{\mathcal{M}}_{\B}(\bar{\x}),\ \forall \B,
		\eeq
		where the step \step{1} uses the fact that $\mathcal{M}_{\B}(\bar{\x},\mathbf{0}_k) = h(\bar{\x}) + g(\bar{\x})$ and the definition $\underset{\d_\B}{\text{min}}\ \mathcal{M}_{\B}(\bar{\x}, \d_\B)$. It suggests that the optimal point is a \textbf{CWS}-point.

		~\\
		\noi $(\textbf{ii})$ \{\textbf{CWS}-Point $\ddot{\x}$\} $\subseteq$ \{Critical Point\ $\breve{\x}$\}
		
		For any solution $\d_{\B}$ of Problem (\ref{plb:subproblemalg}), we have the following first-order optimality condition with any $\x$ and working set $\B$: $	\mathbf{0}_k \in  [\nabla f(\x)]_{\B} + (\Q_{\B\B} + \theta\I_k) \d_{\B} + [ \partial h(\x  + \U_{\B}\d_{\B})]_{\B} + [\partial g(\x  + \U_{\B}\d_{\B})]_{\B} +  [\partial \mathcal{I}_{\Omega}(\x  + \U_{\B}\d_{\B})]_{\B}$. When $\x = \ddot{\x}$ the first-order optimality condition, $\d_{\B} = \mathbf{0}_k$ according to the Definition \ref{defition:cws} of \textbf{CWS}-point. Subsequently, we have following results for any working set $\B$: $\mathbf{0}_k \in  [\nabla f(\x)]_{\B} + [ \partial h(\x)]_{\B} + [\partial g(\x)]_{\B} +  [\partial \mathcal{I}_{\Omega}(\x)]_{\B}$. Therefore, any \textbf{CWS}-point is also a critical point. 
	\end{proof}

	\section{Proofs of Results in Section \ref{sec:conv}}
	
	\subsection{Proof of Theorem \ref{theo:global}}\label{app:theo:global}
	
	\begin{proof}
		\noi \textbf{(i)} When $\bar{\d}_{\B^t}$ is a global solution searching in working set $\B^t$ at $t$ iteration by BCD-g method, we have: $\mathcal{M}(\x^t,\bar{\d}_{\B^t}) \leq \mathcal{M}(\x^t,\mathbf{0}_k)$. Using the definition $\Q_+ \triangleq \Q_{\B^t\B^t} + \theta \I_k$ and $\bar{\r}^t = \U_{\B^t}\bar{\d}_{\B^t}$, then we provide following inequalities:
		\beq
		f(\x^t) + \la \bar{\r}^t, \nabla f(\x^t) \ra + h(\x^t + \bar{\r}^t) + g(\x^t +  \bar{\r}^t) + \tfrac{1}{2} \|\bar{\d}_{\B^t}\|_{\Q_+}^2 & \leq & f(\x^t) + h(\x^t) + g(\x^t)\nn\\
		h(\x^t + \bar{\r}^t) - h(\x^t) + g(\x^t  + \bar{\r}^t) - g(\x^t) + \tfrac{1}{2} \|\bar{\d}_{\B^t}\|_{\Q_+}^2 & \leq & -\la \bar{\r}^t, \nabla f(\x^t)\ra \nn\\
		h(\x^t + \bar{\r}^t) - h(\x^t) + g(\x^t  + \bar{\r}^t) - g(\x^t) + \tfrac{\theta}{2} \|\bar{\d}_{\B^t}\|_2^2 & \overset{\step{1}}{\leq} & f(\x^t) - f(\x^t + \bar{\r}^t) \nn\\
		\Rightarrow F(\x^{t+1})-F(\x^t)  & \overset{\step{2}}{\leq} & -\tfrac{\theta}{2} \|\bar{\d}_{\B^t}\|_2^2\nn\\
		& \overset{\step{3}}{=} & - \tfrac{\theta}{2}\|\x^{t+1} - \x^t\|_2^2,\label{ineq:optd}
		\eeq
		
		\noi where step \step{1} uses the fact that $\nabla f(\cdot)$ is coordinate-wise Lipschitz continuous; step \step{2} uses the definition $F(\x^{t+1}) \triangleq f(\x^{t+1})+h(\x^{t+1}) + g(\x^{t+1})$ and step \step{3} use the fact that $\x^{t+1} = \x^{t} + \U_{\B^t}\bar{\d}_{\B^t}$.
		
		When $\bar{\d}_{\B^t}$ is a solution given by BCD-l-$k$ method at $t$ iteration, we have: $\mathcal{N}(\x^t,\bar{\d}_{\B^t}) \leq \mathcal{N}(\x^t,\mathbf{0}_k)$. We define $\Q_+ \triangleq \Q_{\B^t\B^t} + \theta \I_k$ and $\bar{\r}^t = \U_{\B^t}\bar{\d}_{\B^t}$, then have following inequalities:
		\beq
		f(\x^t) + h(\x^t + \bar{\r}^t) + g(\x^t) + \la \bar{\r}^t, \nabla f(\x^t) + \partial g(\x^t)\ra  + \tfrac{1}{2} \|\bar{\d}_{\B^t}\|_{\Q_+}^2 & \leq & f(\x^t) + h(\x^t) + g(\x^t)\nn\\
		h(\x^t + \bar{\r}^t) - h(\x^t) + \tfrac{1}{2} \|\bar{\d}_{\B^t}\|_{\Q_+}^2 + \la \bar{\r}^t, \nabla f(\x^t) + \partial g(\x^t)\ra& \leq &  0 \nn\\
		f(\x^t + \bar{\r}^t) - f(\x^t) + g(\x^t + \bar{\r}^t) - g(\x^t) + h(\x^t + \bar{\r}^t) - h(\x^t) & \overset{\step{1}}{\leq} & - \tfrac{\theta}{2}\|\bar{\d}_{\B^t}\|_2^2\nn\\
		\Rightarrow F(\x^{t+1})-F(\x^t)  & \overset{\step{2}}{\leq} & - \tfrac{\theta }{2}\|\bar{\d}_{\B^t}\|_2^2\nn\\
		& \overset{}{=} & - \tfrac{\theta}{2}\|\x^{t+1} - \x^t\|_2^2,\nn
		\eeq
		
		\noi where step \step{1} uses the fact that $\nabla f(\cdot)$ is coordinate-wise Lipschitz continuous and $-g(\x)$ is convex; step \step{2} uses the definition $F(\x^{t+1}) \triangleq f(\x^{t+1})+h(\x^{t+1}) + g(\x^{t+1})$.
		
		
		%
		
		\noi \textbf{(ii)} We obtain a lower bound on the progress made by each iteration for our two BCD methods: $\mathbb{E}_{\B^t}[F(\x^{t+1})]-F(\x^t) \leq  -\mathbb{E}_{\B^t} [\tfrac{\theta}{2}\|\x^{t+1}-\x^t\|^2_2]$. Summing up the inequality above over $t=0,1,...,T-1$, we have: $\mathbb{E}_{\xi^T}[\tfrac{\theta}{2}  \sum_{t=0}^{T-1}\|\x^{t+1} - \x^t\|^2_2]  \leq  \mathbb{E}_{\xi^T}[F(\x^0) - F(\x^{T})]\leq  \mathbb{E}_{\xi^T}[F(\x^0) - F(\bar{\x})]$. As a result, there exists an index $\bar{t}$ with $0 \leq \bar{t} \leq T-1$ such that:
		\beq
		\mathbb{E}_{\xi^T}[\|\x^{\bar{t}+1} - \x^{\bar{t}}\|_2^2] \leq \tfrac{2(F(\x^0) - F(\bar{\x}))}{ \theta T}\label{eq:expT}
		\eeq
		\textbf{(iii)} Besides, for any $\bar{t}$, we have:
		\beq
		\mathbb{E}_{\xi^{\bar{t}}}[\|\x^{\bar{t}+1} - \x^{\bar{t}}\|_2^2] = \tfrac{1}{C_n^k} \sum_{\B\in \Pi_n^k} \|\bar{\d}_{\B^{\bar{t}}}\|_2^2.\label{eq:expM}
		\eeq
		Combining (\ref{eq:expT}) and (\ref{eq:expM}), we have the following result: $	\tfrac{1}{C_n^k} \sum_{\B\in \Pi_n^k}  \|\bar{\d}_{\B^{\bar{t}}}\|_2^2\leq \tfrac{2(F(\x^0) - F(\bar{\x}))}{ \theta T}$. Therefore, we conclude that the BCD-g and BCD-l-$k$ method find a $\hat{\epsilon}$-approximate \textbf{CWS}-point or a $\hat{\epsilon}$-approximate critical point in at most $T$ iterations in the
		sense of expectation, where $T \leq \lceil \tfrac{2(F(\x^0) - F(\bar{\x}))}{ \theta \hat{\epsilon}} \rceil = \mathcal{O}(\hat{\epsilon}^{-1})$.
	\end{proof}

	\subsection{Proof of Lemma \ref{lemma:globallybounded}}\label{app:lemma:globallybounded}
	\begin{proof}
		First, we prove that, for all $\alpha \geq 0$, it holds that:
		\beq
		h(\alpha) \triangleq  4 C \alpha - 2\epsilon \alpha^2 \leq  \tfrac{2C^2}{\epsilon} . \label{eq:h:maximizer}
		\eeq
		\noi We notice that $h(\alpha)$ is concave, and maximum always exists. Setting its derivative of $h(\alpha)$ in Equation (\ref{eq:h:maximizer}) to zero yields $0= \nabla h(\alpha) = 4 C - 4 \epsilon \alpha $. The maximizer can be computed as: $\bar{\alpha} = \tfrac{C}{\epsilon}$. Therefore, we have: $h(\alpha) \leq h(\bar{\alpha}) = 4  \tfrac{C^2}{\epsilon} - 2\epsilon (\tfrac{C}{\epsilon})^2 = \tfrac{2C^2}{\epsilon}$.
		We have the following inequalities: $\tfrac{z(\x) - z(\y) - \la \x-\y,\partial z(\x)  \ra - \epsilon}{ \tfrac{1}{2} \|\x-\y\|_2^2} \overset{\step{1}}{\leq} \tfrac{  2C\|\x-\y\| - \epsilon}{ \tfrac{1}{2} \|\x-\y\|_2^2} = \tfrac{  4C }{ \|\x-\y\|} - \tfrac{ 2\epsilon}{ \|\x-\y\|_2^2} \overset{\step{2}}{\leq} \tfrac{2C^2}{\epsilon}$, where step \step{1} uses the fact that $z(\x)$ is $C$-Lipschitz continuous; step \step{2} uses Inequality (\ref{eq:h:maximizer}) with $\alpha = \tfrac{1}{\|\x-\y\|_2}$.
	\end{proof}
	
	\subsection{Proof of Lemma \ref{lemma:bound:4}}
	\label{app:lemma:bound:4}

	\begin{proof}
		\noi \textbf{(i)}  $\B$ is selected from $\Pi_n^k$ randomly and uniformly. For any $\x \in \mathbb{R}^n$ and $\z \in \mathbb{R}^n$, we have: $\mathbb{E}_{\B}[\la \x_{\B}, \d_{\B} \ra] = \tfrac{1}{C_n^k} \sum_{\B \in \Pi_n^k} \la \x_{\B}, \d_{\B} \ra \overset{\step{1}}{=} \tfrac{k}{n}\la \x, \d \ra$ where step \step{1} uses the fact that each entry $(\mathbf{x}_i \cdot \mathbf{z}_i)$ appears for $\left( C_n^k \cdot \tfrac{k}{n} \right)$ times for all $i \in [n]$.

		\noi \textbf{(ii)} For any $\x \in \mathbb{R}^n$ and $\z \in \mathbb{R}^n$, we derive: $ \sum_{\B \in \Pi_n^k} \la \U_{\B}\x_{\B} , \z \ra \overset{\step{1}}{=} \sum_{\B \in \Pi_n^k} \la \x_{\B} , \z_{\B} \ra \overset{\step{2}}{=} C_n^k \tfrac{k}{n} \la \x, \z \ra$,
		where step \step{1} uses $[\U_{\B}\x_{\B}]_{\B^c} = \textbf{0}^{n-k}$ and step \step{2} use Lemma \ref{lemma:bound:4} \textbf{(i)}.

		\noi \textbf{(iii)} For any $\x \in \mathbb{R}^n$, $\d \in \mathbb{R}^n$, we derive the following equalities:
		\beq
		\textstyle\sum_{\B\in \Pi_n^k} \| \x + \U_\B\d_\B \|_2^2 & = &  \textstyle\sum_{\B\in \Pi_n^k} \left( \|\x\|_2^2 + 2\la \d_\B, \x_\B \ra + \|\d_\B\|_2^2 \right) \nn\\
		& \overset{\step{1}}{=} & \textstyle C_n^k \left( \|\x\|_2^2 + C_n^k\tfrac{2k}{n}\la\d , \x\ra + C_n^k\tfrac{k}{n} \|\d\|_2^2 \right)  \nn\\
		& = &\textstyle C_n^k \left(  \|\x\|_2^2 + \tfrac{2k}{n} \la \d , \x\ra + \tfrac{k}{n}\|\d\|_2^2 \right) \nn\\
		& = & \textstyle C_n^k \left( (1 - \tfrac{k}{n})\|\x\|_2^2 + \tfrac{k}{n}\|\x\|_2^2 + \tfrac{2k}{n} \la \d , \x\ra + \tfrac{k}{n}\|\d\|_2^2 \right) \nn\\
		& = & \textstyle C_n^k \left( (1 - \tfrac{k}{n})\|\x\|_2^2 + \tfrac{k}{n}\|\x + \d\|_2^2 \right),
		\eeq
		\noi step \step{1} uses Lemma \ref{lemma:bound:4} \textbf{(i)} and \textbf{(ii)}. Dividing both sides by $C_n^k$, we swiftly conclude this lemma.

		\noi \textbf{(iv)} We obtain the following results:
		\beq
		\textstyle  \sum_{\B\in \Pi_n^k} f(\x + \U_\B\d_\B) & \overset{\step{1}}{\leq} &  \textstyle  \sum_{\B\in \Pi_n^k}\left( f(\x) + \la \U_\B\d_\B,\nabla f(\x) \ra + \tfrac{1}{2} \|\d_\B\|_{\Q_{\B\B}}^2 \right)\nn \\ 
		& \overset{\step{2}}{=} & \textstyle C_n^k f(\x) + \la C_n^k \tfrac{k}{n} \d,\nabla f(\ddot{\x})\ra +  \sum_{\B\in \Pi_n^k} \tfrac{1}{2} \|\d_\B\|_{\Q_+}^2 - \sum_{\B\in \Pi_n^k} \tfrac{\theta}{2} \|\d_\B\|_{2}^2 \nn \\ 
		& \overset{\step{3}}{=} & \textstyle C_n^k f(\x) + \la C_n^k \tfrac{k}{n} \d,\nabla f(\ddot{\x})\ra +  C_n^k  \E_{\B}[ \tfrac{1}{2} \|\d_\B\|_{\Q_+}^2 ] -  C_n^k  \E_{\B}[ \tfrac{\theta}{2} \|\d_\B\|_{2}^2 ] \nn \\ 
		& \overset{}{\leq} &  \textstyle C_n^k f(\x) + \la C_n^k \tfrac{k}{n} \d,\nabla f(\ddot{\x})\ra +  C_n^k  \E_{\B}[ \tfrac{1}{2} \|\d_\B\|_{\Q_+}^2], 
		\eeq
		\noi where step \step{1} uses the coordinate-wise Lipschitz continuity of $\nabla f(\x)$; step \step{2} uses claim \textbf{(ii)} in Lemma \ref{lemma:bound:4}; step \step{3} uses the fact that $\sum_{\B\in \Pi_n^k}  \|\d_\B\|_{\Q_+}^2 = C_n^k  \E_{\B}[ \|\d_\B\|_{\Q_+}^2 ]$.

		\noi \textbf{(v)} We have the following results:
		\beq
		\textstyle \sum_{\B\in \Pi_n^k} h(\x + \U_\B\d_\B )	 &=& \textstyle \sum_{\B\in \Pi_n^k}\left( \sum_{i \in \B} h_{i}(\x_{i} + \d_{i}) + \sum_{j \in \B^c} h_{j}(\x_{j})\right)\nn\\
		& = & \textstyle \sum_{\B\in \Pi_n^k} \sum_{i \in \B} h_{i}(\x_{i} + \d_{i}) + \sum_{\B\in \Pi_n^k} \sum_{j \in \B^c} h_{j}(\x_{j})\nn\\
		& = & \textstyle C_n^k(\tfrac{k}{n}h(\x+\d) + (1 - \tfrac{k}{n}) h(\x)).
		\eeq

		\noi \textbf{(vi)} We have the following inequalities according step \step{2} in (\ref{eq:gx}):
		\beq
		\sum_{\B\in \Pi_n^k} g(\x + \U_\B\d_\B)
		\overset{\step{1}}{\leq}  \sum_{\B\in \Pi_n^k} g(\x) + \la \U_\B\d_\B, \partial g(\x) \ra  \overset{\step{2}}{=}  C_n^k g(\x) + C_n^k \tfrac{k}{n}\la \d, \partial g(\x)\ra \label{ieq:lemma:g(x)}
		\eeq
		\noi where step \step{1} uses the fact that $-g(\x)$ is convex with $g(\y)\leq g(\x) + \la \y-\x,\ \partial g(\x)\ra$; step \step{2} uses \textbf{(ii)} in Lemma \ref{lemma:bound:4}.
		
		\noi \textbf{(vii)} According to (\ref{ieq:lemma:g(x)}), we derive following results:
		\beq	
		\sum_{\B\in \Pi_n^k} g(\x + \U_\B\d_\B) & \leq & C_n^k g(\x) + C_n^k \tfrac{k}{n}\la \d, \partial g(\x)\ra \nn\\
		&\overset{\step{1}}{\leq} & C_n^k g(\x) + C_n^k \tfrac{k}{n} \left(g(\x + \d) - g(\x) + \tfrac{\rho}{2}\|\d\|_2^2 + \epsilon \right),
		\eeq
		\noi step \step{1} uses the fact that $-g(\x)$ is $\rho$-$\epsilon$ bounded nonconvex and it holds that $g(\x)\leq g(\y) + \la \x-\y,\partial g(\x)\ra+\tfrac{\rho}{2}\|\x-\y\|_2^2 + \epsilon$ with $\y = \x + \d$.
	\end{proof}
	
	\subsection{Proof of Theorem \ref{theo:dec:x:f}}
	\label{app:theo:dec:x:f}
	Before the proof, we give following lemma.
	\begin{lemma}\label{lemma:boundedterm}
		For any $\w \in \mathbb{R}^n$, we have following result: $\tfrac{k}{n}\la \w - \x^{t+1} , \nabla f(\x^t) + \partial h(\x^{t+1})\ra \leq \tfrac{k}{n}f(\w)  +  \E_{\B^t}[ \tfrac{1}{2} \|\x^{t+1}-\x^{t} \|_{\H^t_+}^2 - f(\x^{t+1}) - h(\x^{t+1})]  + (1 - \tfrac{k}{n})(f(\x^t) + h(\x^t)) + \tfrac{k}{n}h(\w)$, where $\x^t$ and $\x^{t+1}$ are generated by BCD-g or BCD-l-$k$.
	\end{lemma}
	
	\begin{proof}
		Applying \textbf{Part} \textbf{(iv)} and \textbf{Part} \textbf{(v)} as presented in Lemma \ref{lemma:bound:4} with $\x = \x^t$ and $\d = \x^{t+1} - \x^t$, we obtain:
		\beq
		\E_{\B^t}[f(\x^{t+1})] & \leq&  \textstyle f(\x^{t}) + \la  \tfrac{k}{n}( \x^{t+1}-\x^{t}),\ \nabla f(\x^{t})\ra + \sum_{\B\in \Pi_n^k} \tfrac{1}{2C_n^k} \|\d_\B\|_{\Q_+}^2\label{expctation of fxt+1}\\
		\E_{\B^t}[h(\x^{t+1})] &=& \textstyle \tfrac{1}{n}h(\x^{t+1}) + (1 - \tfrac{1}{n}) h(\x^t)\label{expctation of hxt+1}
		\eeq

		\noi We bound $\tfrac{k}{n}\la \w - \x^{t+1} , \nabla f(\x^t) + \partial h(\x^{t+1})\ra$ by following inequalities:
		\beq
		&& \textstyle  \tfrac{k}{n} \la \w - \x^{t+1} , \nabla f(\x^t) + \partial h(\x^{t+1})\ra \nonumber\\
		&= & \textstyle  \tfrac{k}{n}\la \w - \x^{t} , \nabla f(\x^t)\ra + \tfrac{k}{n}\la \x^t - \x^{t+1} , \nabla f(\x^t)\ra + \tfrac{k}{n}\la \w - \x^{t+1} , \partial h(\x^{t+1})\ra \nonumber\\
		&\overset{ \step{1}}{\leq} & \textstyle \tfrac{k}{n}(f(\w) - f(\x^t)) - \tfrac{k}{n}\la \x^{t+1} - \x^t , \nabla f(\x^t)\ra + \tfrac{k}{n}(h(\w) - h(\x^{t+1}))\nonumber\\
		&\overset{\step{2}}{\leq} & \textstyle  \tfrac{k}{n}f(\w)  +  \E_{\B^t}[ \tfrac{1}{2} \|\x^{t+1}-\x^{t} \|_{\H^t_+}^2 - f(\x^{t+1})]  + (1 - \tfrac{k}{n})f(\x^t)) + \tfrac{k}{n}(h(\w) - h(\x^{t+1}))\nonumber\\
		&\overset{\step{3}}{\leq} & \textstyle  \tfrac{k}{n}f(\w)  +  \E_{\B^t}[ \tfrac{1}{2} \|\x^{t+1}-\x^{t} \|_{\H^t_+}^2 - f(\x^{t+1}) - h(\x^{t+1})]  + (1 - \tfrac{k}{n})(f(\x^t) + h(\x^t)) + \tfrac{k}{n}h(\w) ,\nn\\
		\eeq
		where step \step{1} uses the convexity of $f(\cdot)$ and $h(\cdot)$ that: $\la\w - \x^t, \nabla f(\x^t)\ra \leq  f(\w) - f(\x^t),
		\la\w - \x^{t+1}, \partial h(\x^{t+1})\ra \leq  h(\w) - h(\x^{t+1})$; step \step{2} uses (\ref{expctation of fxt+1}); step \step{3} uses (\ref{expctation of hxt+1}).
	\end{proof}
	
	\noi We now give the proof of Theorem \ref{theo:dec:x:f}.
	\begin{proof}
		\textbf{(i)} \noi Let $\ddot{\x}$ be any \textbf{CWS}-point, $\x^{t}$ and $\x^{t+1}$ be any solution generated by BCD-g. Based on Lemma \ref{lemma:iequality}, we have $\H^t \triangleq \UBt \UBt\trans \Q  \UBt \UBt \trans \in \mathbb{R}^{n\times n}$, $\H_+^t = \H^t + \theta \I_n$, $\Hlow = \min_{\B} \| \UB\trans \Q  \UB \|  \in \mathbb{R}$, and $\Hup = \max_{\B} \| \UB\trans \Q  \UB \|  \in \mathbb{R}$.
		
		First, for any solution $\bar{\d}_{\B^{t}}$ in subproblem (\ref{plb:subproblem}), we have the following first-order optimality condition:
		\beq \label{eq:first:order:opt}
		\mathbf{0}_k \in  [\nabla f(\x^t)]_{\B} + (\Q_{\B\B} + \theta\I_k) (\x^{t+1}-\x^{t})_{\B} + [ \partial h(\x^{t+1})]_{\B} + [\partial g(\x^{t+1})]_{\B} +  [\partial \mathcal{I}_{\Omega}(\x^{t+1})]_{\B}. \label{eq:optimality}
		\eeq
		
		Second, given the set $\Omega$ is convex, we conclude that the function $\mathcal{I}_{\Omega}(\x)$ is also convex. We have the following inequality for all $\x$ and $\x'$ when $z(\x)$ is a convex function: $\la \x-\x',\partial z(\x') \ra\leq z(\x) - z(\x')$. Applying this inequality with $\x= \ddot{\x}$, $\x'= \x^{t+1}$ and $z(\x) = \mathcal{I}_{\Omega}(\x)$, we have:
		\beq \label{eq:delta:upper:bound}
		\la \ddot{\x} - \x^{t+1} , \partial \mathcal{I}_{\Omega}(\x^{t+1}) \ra \leq \mathcal{I}_{\Omega}(\ddot{\x}) - \mathcal{I}_{\Omega}(\x^{t+1}) = 0 \label{ieq:indicator},
		\eeq
		\noi where the equality holds due to the fact that $\ddot{\x} \in \Omega$ and $\x^{t+1} \in \Omega$ for all $t$.

		Third, applying \textbf{Part} \textbf{(vi)} in Lemma \ref{lemma:bound:4} with $\x = \x^t$ and $\d = \x^{t+1} - \x^t$, we obtain:
		\beq
		\E_{\B^t}[g(\x^{t+1})] &\leq& \textstyle g(\x^t) + \tfrac{k}{n}g(\x^{t+1}) - \tfrac{k}{n}g(\x^{t}) +  \tfrac{k\rho}{2n}\|\x^{t+1} - \x^t\|_2^2 + \tfrac{k}{n}\epsilon. \label{expctation of gxt+1}
		\eeq

		Fourth, for any $\x^t$, $\x^{t+1}$, and $\ddot{\x}$, due to the uniform and random selection of the working set $\B^t$, we observe that:
		\beq
		\textstyle \E_{\B^t}[\|\x^{t+1} - \ddot{\x}\|_2^2] = 	\tfrac{1}{C_n^k} \sum_{\B\in \Pi_n^k} \| \x^t + \UB \d_{\B} - \ddot{\x}\|_2^2 \label{eq:expctation of xt+1}.
		\eeq
		\noi Applying the equality in \textbf{Part} \textbf{(iii)} in Lemma \ref{lemma:bound:4} with $\x = \x^t - \ddot{\x}$ and $\d = \x^{t+1} - \x^t$, we have:
		\beq
		\textstyle \tfrac{k}{n}\|\x^{t+1} - \ddot{\x} \|_2^2 + (1 - \tfrac{k}{n})\|\x^{t} - \ddot{\x}\|_2^2  = \tfrac{1}{C_n^k} \sum_{\B\in \Pi_n^k} \| \x^t + \U_{\B}\d_{\B} - \ddot{\x}\|_2^2. \label{eq:sum of xt+1}
		\eeq
		\noi Combining the two inequalities(\ref{eq:expctation of xt+1}) and (\ref{eq:sum of xt+1}), we have:
		\beq
		\E_{\B^t}[\|\x^{t+1} - \ddot{\x}\|_2^2] = \tfrac{k}{n}\|\x^{t+1} - \ddot{\x} \|_2^2 + (1 - \tfrac{k}{n})\|\x^{t} - \ddot{\x}\|_2^2. \label{eq:expectation and sum for xt+1}
		\eeq
		
		%

		We derive the following results:
		\beq \label{eq:Pythagoras}
		&& \E_{\B^t}[\tfrac{1}{2} \|\x^{t+1} - \ddot{\x}\|_{\H_+^t}^2] - \tfrac{1}{2} \|\x^{t} - \ddot{\x}\|_{\H_+^t}^2\nn\\
		&\overset{\step{1}}{=} & \E_{\B^t}[\la \x^{t+1} - \ddot{\x}, \H_+^t(\x^{t+1} - \x^t) \ra] - \E_{\B^t}[\tfrac{1}{2}\|\x^{t+1} - \x^t\|_{\H_+^t}^2]\nn\\
		&\overset{}{=} & \E_{\B^t}[\la \x^{t+1}_{\B^t} - \ddot{\x}_{\B^t},(\Q_{\B\B} + \theta \I_k)(\x^{t+1}_{\B^t} - \x^t_{\B^t})\ra] - \E_{\B^t}[\tfrac{1}{2}\|\x^{t+1} - \x^t\|_{\H_+^t}^2]\nn\\
		&\overset{\step{2}}{=} & -\E_{\B^t}[\la \x^{t+1}_{\B^t} - \ddot{\x}_{\B^t},  \omega_{\B^t} \ra ]  - \E_{\B^t}[\tfrac{1}{2}\|\x^{t+1} - \x^t\|_{\H_+^t}^2]\nn\\
		&\overset{\step{3}}{\leq}  & \tfrac{k}{n}\la \ddot{\x} - \x^{t+1} , \nabla f(\x^t) + \partial h(\x^{t+1}) + \partial g(\x^{t+1}) \ra - \E_{\B^t}[\tfrac{1}{2}\|\x^{t+1} - \x^t\|_{\H_+^t}^2],
		\eeq
		\noi where step \step{1} uses the Pythagoras relation that: $\forall \x, \y, \z, \tfrac{1}{2}\|\y - \z\|_2^2 - \tfrac{1}{2}\|\x - \z\|_2^2 = \la \y-\z, \y-\x \ra - \tfrac{1}{2}\|\x - \y\|_2^2$; step \step{2} uses the optimality condition in (\ref{eq:optimality}) and the definition: $\omega \triangleq \nabla f(\x^{t}) + \partial h(\x^{t+1}) + \partial g(\x^{t+1}) + \partial \mathcal{I}_{\Omega} (\x^{t+1})$; step \step{3} uses the inequity (\ref{ieq:indicator}) and \textbf{Part (i)} in Lemma \ref{lemma:bound:4}.

		\noi We now bound the term $\tfrac{k}{n}\la \ddot{\x} - \x^{t+1} ,  \partial g(\x^{t+1})\ra$ in (\ref{eq:Pythagoras}) by following inequalities:
		\beq
		&& \textstyle \tfrac{k}{n}\la \ddot{\x} - \x^{t+1} ,  \partial g(\x^{t+1})\ra\nn\\
		&\overset{ \step{1}}{\leq} &  \textstyle \tfrac{k}{n}g(\ddot{\x}) - \tfrac{k}{n}g(\x^{t+1}) + \tfrac{k\rho}{2n}\|\ddot{\x} - \x^{t+1}\|_2^2 + \tfrac{k}{n}\epsilon\nonumber\\
		&\overset{ \step{2}}{\leq} &  \textstyle \tfrac{k}{n}g(\ddot{\x}) - \tfrac{k}{n}g(\x^{t+1}) + \tfrac{\rho}{2} \{ \E_{\B^t}[\|\ddot{\x} - \x^{t+1}\|_2^2] - (1-\tfrac{k}{n}) \|\ddot{\x} - \x^{t}\|_2^2\} + \tfrac{k}{n}\epsilon\nonumber\\
		&\overset{\step{3}}{\leq} &  \textstyle \tfrac{k}{n}g(\ddot{\x}) - \tfrac{k}{n}g(\x^{t+1}) + \tfrac{\rho}{2 c_1 }   \E_{\B^t}[\|\ddot{\x} - \x^{t+1}\|_{\H^t_+}^2] - \tfrac{\rho}{2 c_2} (1 - \tfrac{k}{n})\|\ddot{\x} - \x^t\|_{\H^t_+}^2 + \tfrac{k}{n}\epsilon\nonumber\\
		&\overset{\step{4}}{\leq} &  \textstyle \tfrac{k}{n}g(\ddot{\x}) - \E_{\B^t}[g(\x^{t+1})] + (1 - \tfrac{k}{n})  g(\x^t) +  \tfrac{k\rho}{2n}\|\x^{t+1} - \x^t\|_2^2 + \tfrac{\rho}{2 c_1 }   \E_{\B^t}[\|\ddot{\x} - \x^{t+1}\|_{\H^t_+}^2]\nn\\
		&&- \tfrac{\rho}{2 c_2} (1 - \tfrac{k}{n})\|\ddot{\x} - \x^t\|_{\H^t_+}^2 + \tfrac{2k}{n}\epsilon,\label{bounded term 1}
		\eeq
		\noi where step \step{1} uses the globally $\rho$-$\epsilon$ bounded non-convexity of $g(\cdot)$; step \step{2} uses (\ref{eq:expectation and sum for xt+1}); step \step{3} uses Lemma \ref{lemma:iequality}; step \step{4} uses (\ref{expctation of gxt+1}).
		
		\noi Letting $\w = \ddot{\x}$ in Lemma \ref{lemma:boundedterm}, we now bound the term $\tfrac{k}{n}\la \ddot{\x} - \x^{t+1} , \nabla f(\x^t) + \partial h(\x^{t+1})\ra$ in (\ref{eq:Pythagoras}) as following:
		\beq
		&& \textstyle  \tfrac{k}{n} \la \ddot{\x} - \x^{t+1} , \nabla f(\x^t) + \partial h(\x^{t+1})\ra \nonumber\\
		&\overset{}{\leq} & \textstyle  \tfrac{k}{n}f(\ddot{\x})  +  \E_{\B^t}[ \tfrac{1}{2} \|\x^{t+1}-\x^{t} \|_{\H^t_+}^2 - f(\x^{t+1}) - h(\x^{t+1})]  + (1 - \tfrac{k}{n})(f(\x^t) + h(\x^t)) + \tfrac{k}{n}h(\ddot{\x}) ,\nn\\
		\label{bounded term 2}
		\eeq
		
		\noi Combining (\ref{eq:Pythagoras}), (\ref{bounded term 1}) and (\ref{bounded term 2}), and using the fact that $F(\x) = f(\x) + h(\x) + g(\x)$, we derive:
		\beq 
		&& \E_{\B^t}[\tfrac{c_1 - \rho}{2c_1} \|\x^{t+1} - \ddot{\x}\|_{\H_+^t}^2] - \tfrac{c_2 - \rho(1-\tfrac{k}{n})}{2c_2} \|\x^{t} - \ddot{\x}\|_{\H_+^t}^2\nn\\
		&\leq & \tfrac{k}{n} F(\ddot{\x}) - \tfrac{k}{n} F(\x^t)  - \E_{\B^t}[F(\x^{t+1})] + F(\x^t)  +  \tfrac{k\rho}{2n}\|\x^{t+1} - \x^t\|_2^2 + \tfrac{2k}{n}\epsilon  \nn\\
		& \overset{\step{1}}{\leq} & \tfrac{k}{n} F(\ddot{\x}) - \tfrac{k}{n} F(\x^t)  - \E_{\B^t}[F(\x^{t+1})] + F(\x^t)  +  \tfrac{k\rho}{\theta}(F(\x^t) - \E_{\B^t}[F(\x^{t+1})]) + \tfrac{2k}{n}\epsilon \nn\\
		& \overset{\step{2}}{=} & - \tfrac{k}{n} \ddot{q}^t +(1 + \tfrac{k\rho}{\theta}) (\ddot{q}^t-\E_{\B^t}[\ddot{q}^{t+1}])+ \tfrac{2k}{n}\epsilon,
		\eeq
		where step \step{1} uses the sufficient decrease condition that $\tfrac{1}{2n}\|\x^{t+1} - \x^t\|_2^2 = \E_{\B^t}[\tfrac{1}{2}\|\x^{t+1} - \x^t\|_2^2] \leq \tfrac{1}{\theta}\E_{\B^t} [F(\x^t) - F(\x^{t+1})]$; step \step{2} uses the
		definition of $\ddot{q}^t \triangleq F(\x^t) - F(\ddot{\x})$ and the fact that $F(\x^t) - F(\x^{t+1}) = \ddot{q}^{t} - \ddot{q}^{t+1}$. Using the definitions that $\ddot{r}^{t+1} \triangleq \tfrac{1}{2} \|\x^{t+1} - \ddot{\x}\|_{\H_+^t}^2$, $\beta_1 \triangleq 1 - \tfrac{\rho}{c_1}$, $\beta_2 \triangleq 1 - \tfrac{\rho}{c_2}(1-\tfrac{k}{n})$ and $\gamma \triangleq (1 + \tfrac{k\rho}{\theta})$, we rewrite the inequality above as:
		\beq
		(1 - \tfrac{\rho}{c_1}) \E_{\B^t} [\ddot{r}^{t+1}] - (1 - \tfrac{\rho}{c_2}(1-\tfrac{k}{n})) \ddot{r}^{t} & \leq & - \tfrac{k}{n} \ddot{q}^t + (1 + \tfrac{k\rho}{\theta}) (\ddot{q}^t-\E_{\B^t}[\ddot{q}^{t+1}]) + \tfrac{2k}{n}\epsilon\nn\\
		\Rightarrow \beta_1 \E_{\B^t} [\ddot{r}^{t+1}] + \gamma \E_{\B^t}[\ddot{q}^{t+1}] & \leq & \beta_2 \ddot{r}^{t} + (\gamma - \tfrac{k}{n})\ddot{q}^t + \tfrac{2k}{n}\epsilon.
		\eeq
		\textbf{(ii)} We now derive following results when $\beta_1 > 0$. The function $\mathcal{S}_{\B}(\x,\d_\B) + h(\x + \U_{\B}\d_{\B}) + \tfrac{\theta}{2}\|\x + \U_{\B}\d_{\B} - \x \|_2^2$ is considered to be ($\text{min}(\Q) + \theta$)-strongly convex with respect to $\d_\B$, and the term $g(\x + \U_{\B}\d_{\B})$ is globally $\rho$-$\epsilon$ bounded nonconvex. Therefore, $\mathcal{M}_{\B}(\x,\d_\B)$ is convex if: $\text{min}(\Q) + \theta - \rho  > \epsilon \Leftrightarrow \beta_1 > 0$. When $F(\cdot)$ satisfies the Luo-Tseng error bound assumption, we bound the term $\ddot{r}^t$ as: 
		\beq
		\ddot{r}^t  &\triangleq &  \tfrac{c_2}{2}\|\x^t - \ddot{\x}\|_2^2
		\overset{\step{1}}{\leq} \tfrac{c_2}{2}\left(\tfrac{\delta}{C_n^k}\right)^2 \left(\sum_{\B\in \Pi_n^k} |\bar{\mathcal{M}}_{\B}(\x^t)| \right)^2
		\overset{\step{2}}{\leq} \tfrac{c_2}{2}\left(\tfrac{\delta}{C_n^k}\right)^2 C_n^k \left(\sum_{\B\in \Pi_n^k} |\bar{\mathcal{M}}_{\B}(\x^t)|^2 \right)\nn\\
		&\overset{\step{3}}{\leq}& \tfrac{c_2}{2}\left(\tfrac{\delta}{C_n^k}\right)^2 \left( C_n^k \right)^2 \left( \E_{\B^t}[\|\x^{t+1} - \x^t\|_2^2] \right)
		\overset{}{=} \tfrac{c_2\delta^2}{\theta} \cdot \tfrac{\theta}{2}(\E_{\B^t}[\|\x^{t+1} - \x^t\|_2^2])\nn\\
		&\overset{\step{4}}{\leq}& \tfrac{c_2\delta^2}{\theta}( F(\x^t) - \E_{\B^t}[F(\x^{t+1})])
		\overset{}{=} \tfrac{c_2\delta^2}{\theta}\left( \ddot{q}^t - \E_{\B^t}[\ddot{q}^{t+1}] \right)
		\overset{\step{5}}{=} \kappa_0\left( \ddot{q}^t - \E_{\B^t}[\ddot{q}^{t+1}] \right),
		\eeq
		where step \step{1} uses the Luo-Tseng error bound assumption that $\|\x^t - \ddot{\x}\|_2^2 \leq \delta^2(\tfrac{1}{C_n^k} \sum_{\B\in \Pi_n^k} \\ |\text{dist} (\mathbf{0},\bar{\mathcal{M}}_{\B}(\x)|)^2$ or any \textbf{CWS}-point $\ddot{\x}$; step \step{2} uses $\forall \x \in \R^n,\ \|\x\|_1^2 \leq n\|\x\|_2^2$; step \step{3} uses the fact that $\E_{\B^t}[\|\x^{t+1} - \x^t\|_2^2] =  \E_{\B^t}[\|\x^t + \bar{\mathcal{M}}_{\B}(\x^t) - \x^t\|_2^2] = \E_{\B^t}[ |\bar{\mathcal{M}}_{\B}(\x^t)|^2 ] = \tfrac{1}{C_n^k}\sum_{\B\in \Pi_n^k} |\bar{\mathcal{M}}_{\B}(\x^t)|^2$; step \step{4} uses $\E_{\B^t}[\tfrac{1}{2}\|\x^{t+1} - \x^t\|_2^2] \leq \tfrac{1}{\theta} \E_{\B^t}[F(\x^t) - F(\x^{t+1})]$; step \step{5} uses the definition of $\kappa_0 \triangleq \tfrac{c_2\delta^2}{\theta}$. Since $\beta_1 > 0$, we derive $\beta_2 > 0$. Then we have:
		\beq
		&&\beta_1 \E_{\B^t} [\ddot{r}^{t+1}] + \gamma \E_{\B^t}[\ddot{q}^{t+1}] \leq \beta_2 \ddot{r}^{t} + (\gamma - \tfrac{k}{n})\ddot{q}^t + \tfrac{2k}{n}\epsilon\nn\\
		&\Rightarrow&\gamma \E_{\B^t}[\ddot{q}^{t+1}] \leq \beta_2 \kappa_0   \left(\ddot{q}^t - \E_{\B^t}[\ddot{q}^{t+1}]\right) + (\gamma - \tfrac{k}{n}) \ddot{q}^t + \tfrac{2k}{n}\epsilon\nonumber\\
		&\Rightarrow& (\gamma + \beta_2\kappa_0) \E_{\B^t}[\ddot{q}^{t+1}] \leq (\gamma + \beta_2\kappa_0 - \tfrac{k}{n})\ddot{q}^t + \tfrac{2k}{n}\epsilon\nonumber\\
		&\Rightarrow& \E_{\B^t}[\ddot{q}^{t+1}] \overset{\step{1}}{\leq} (1 -  \tfrac{k}{n\kappa_1}) \ddot{q}^t + \tfrac{2k}{n\kappa_1}\epsilon \nonumber\\
		&\Rightarrow& \E_{\xi^t}[\ddot{q}^{t+1}] \overset{\step{2}}{\leq} \ddot{\kappa}^{t+1} \ddot{q}^0 + \tfrac{2k(\ddot{\kappa}^{t+2} -1)}{n\kappa_1(\ddot{\kappa} - 1)}\epsilon = \ddot{\kappa}^{t+1} \ddot{q}^0 + 2(1 - \ddot{\kappa}^{t+2})\epsilon \leq \ddot{\kappa}^{t+1} \ddot{q}^0 + 2\epsilon,
		\eeq 
		
		\noi where step \step{1} uses the definition $\kappa_1 \triangleq \gamma + \beta_2\kappa_0$; step \step{2} uses the definition $\ddot{\kappa} \triangleq 1 -  \tfrac{k}{n\kappa_1}$. 
		
		\noi \textbf{(iii)} Considering the $\sigma$ with $\sigma \leq 2\ddot{q}^0$, we now discuss the situation where the inequality $\E_{\xi^t}[\ddot{q}^{t+1}] \leq \sigma$ holds. We define: $\beta_2' \triangleq 1 - \tfrac{2C^2}{c_2\epsilon}(1-\tfrac{k}{n})$,  $\kappa_1' \triangleq 1 - \tfrac{k}{n(\gamma + \beta_2'\kappa_0)}$. Based on Lemma \ref{lemma:globallybounded}, we have: $\ddot{\kappa} = \kappa_1'$. Since $\kappa_1 > 1$, $0 < \ddot{\kappa} < 1$, we derive following result:
		\beq
		\left\{
		\begin{aligned}
			\ddot{\kappa}^{t+1} \ddot{q}^0 \leq \tfrac{\sigma}{2}\\
			2\epsilon \leq \tfrac{\sigma}{2}
		\end{aligned}
		\right.
		\Rightarrow \left\{
		\begin{aligned}
			t + 1 &\geq \log_{\kappa_1'}\tfrac{\sigma}{2\ddot{q}^0} \\
			\epsilon &\leq \tfrac{\sigma}{4}
		\end{aligned}
		\right.
		\Rightarrow t+1 \geq \log_{\ddot{\kappa}'} \tfrac{\sigma}{2\ddot{q}^0},
		\eeq
		
		\noi where $\beta_3 \triangleq 1 - \tfrac{8C^2}{c_2\sigma}(1-\tfrac{k}{n})$ and $\ddot{\kappa}'$ is defined as $1 - \tfrac{k}{n(\gamma + \beta_3 \kappa_0)}$. Hence, when $t + 1 \geq log_{\ddot{\kappa}'}\tfrac{\sigma}{2\ddot{q}^0}$, we obtain that $\E_{\xi^t}[\ddot{q}^{t+1}] \leq \sigma$. Thus, we finish the proof of this theorem. 
	\end{proof}

	\subsection{Proof of Theorem \ref{theo:dec:x:f2}}
	\label{app:theo:dec:x:f2}
	
	\begin{proof}
		\textbf{(i)} \noi Let $\breve{\x}$ be any critical point. For any solution $\x^{t}$ and $\x^{t+1}$ generated by BCD-l-$k$, we have the following first-order optimality condition:
		\beq 
		\mathbf{0}_k \in  [\nabla f(\x^t)]_{\B} + (\Q_{\B\B} + \theta\I_k) (\x^{t+1}-\x^{t})_{\B} + [ \partial h(\x^{t+1})]_{\B} + [\partial g(\x^{t})]_{\B} +  [\partial \mathcal{I}_{\Omega}(\x^{t+1})]_{\B}. \label{eq:first:order:optlocal}
		\eeq

		\noi Using \textbf{Part (vi)} in Lemma \ref{lemma:bound:4} and letting $\d = \x^{t+1} - \x^t$, we obtain:
		\beq
		\E_{\B^t}[g(\x^{t+1})] \leq \textstyle g(\x^t) + \tfrac{k}{n}\la \partial g(\x^t), \x^{t+1} - \x^t \ra. \label{expctation of gxt}
		\eeq

		We derive the following results:
		\beq \label{eq:Pythagoras:local}
		&& \E_{\B^t}[\tfrac{1}{2} \|\x^{t+1} - \breve{\x}\|_{\H_+^t}^2] - \tfrac{1}{2} \|\x^{t} - \breve{\x}\|_{\H_+^t}^2\nn\\
		&\overset{\step{1}}{=} & \E_{\B^t}[\la \x^{t+1} - \breve{\x}, \H_+^t(\x^{t+1} - \x^t) \ra] - \E_{\B^t}[\tfrac{1}{2}\|\x^{t+1} - \x^t\|_{\H_+^t}^2]\nn\\
		&\overset{}{=} & \E_{\B^t}[\la \x^{t+1}_{\B^t} - \breve{\x}_{\B^t},(\Q_{\B\B} + \theta \I_k)(\x^{t+1}_{\B^t} - \x^t_{\B^t})\ra] - \E_{\B^t}[\tfrac{1}{2}\|\x^{t+1} - \x^t\|_{\H_+^t}^2]\nn\\
		&\overset{\step{2}}{=} & -\E_{\B^t}[\la \x^{t+1}_{\B^t} - \breve{\x}_{\B^t},  \omega'_{\B^t} \ra ]  - \E_{\B^t}[\tfrac{1}{2}\|\x^{t+1} - \x^t\|_{\H_+^t}^2]\nn\\
		&\overset{\step{3}}{\leq}  & \tfrac{k}{n}\la \breve{\x} - \x^{t+1} , \nabla f(\x^t) + \partial h(\x^{t+1}) + \partial g(\x^{t}) \ra - \E_{\B^t}[\tfrac{1}{2}\|\x^{t+1} - \x^t\|_{\H_+^t}^2],
		\eeq
		\noi where step \step{1} applies the Pythagoras relation that: $\forall \x, \y, \z, \tfrac{1}{2}\|\y - \z\|_2^2 - \tfrac{1}{2}\|\x - \z\|_2^2 = \la \y-\z, \y-\x \ra - \tfrac{1}{2}\|\x - \y\|_2^2$; step \step{2} uses the optimality condition in (\ref{eq:first:order:optlocal}) and the definition: $\omega' \triangleq \nabla f(\x^{t}) + \partial h(\x^{t+1}) + \partial g(\x^{t}) + \partial \mathcal{I}_{\Omega} (\x^{t+1})$; step \step{3} uses the inequity (\ref{ieq:indicator}) and \textbf{Part (i)} in Lemma \ref{lemma:bound:4}.

		\noi We bound the term $\tfrac{k}{n}\la \breve{\x} - \x^{t+1} ,  \partial g(\x^{t})\ra$ in (\ref{eq:Pythagoras:local}) as follows:
		\beq
		&& \textstyle \tfrac{k}{n}\la \breve{\x} - \x^{t+1}, \partial g(\x^{t})\ra\nn\\
		& = & \tfrac{k}{n} \la \breve{\x} - \x^t, \partial g(\x^{t})\ra +  \tfrac{k}{n} \la \x^t - \x^{t+1} , \partial g(\x^{t})\ra\nn\\
		&\textstyle \overset{\step{1}}{\leq}& \tfrac{k}{n} g(\breve{\x}) - \tfrac{k}{n} g(\x^t) + \tfrac{k\rho}{2n} \|\breve{\x}- \x^t\|_2^2 + \tfrac{k\epsilon}{n}+ \tfrac{k}{n} \la \x^t - \x^{t+1}, \partial g(\x^t) \ra \nn\\
		&\textstyle \overset{\step{2}}{\leq}& \tfrac{k}{n} g(\breve{\x}) - \tfrac{k}{n} g(\x^t) + \tfrac{k\rho}{2n} \|\breve{\x}- \x^t\|_2^2 + \tfrac{k\epsilon}{n} + g(\x^t) -	\E_{\B^t}[g(\x^{t+1})],\label{bounded term 1:local}
		\eeq
		where step \step{1} uses the globally $\rho$-$\epsilon$ bounded non-convexity of $g(\cdot)$; step \step{2} uses (\ref{eq:expectation and sum for xt+1}); step \step{3} uses Lemma \ref{lemma:iequality}; step \step{4} uses (\ref{expctation of gxt}).
		
		\noi By applying Lemma \ref{lemma:boundedterm} with $\w = \breve{\x}$, we derive following inequalities:
		\beq
		&& \textstyle  \tfrac{k}{n} \la \breve{\x} - \x^{t+1} , \nabla f(\x^t) + \partial h(\x^{t+1})\ra \nn\\
		& \leq & \textstyle  \tfrac{k}{n}f(\breve{\x})  +  \E_{\B^t}[ \tfrac{1}{2} \|\x^{t+1}-\x^{t} \|_{\H^t_+}^2 - f(\x^{t+1}) - h(\x^{t+1})]  + (1 - \tfrac{k}{n})(f(\x^t) + h(\x^t)) + \tfrac{k}{n}h(\breve{\x}).\nn\\
		\label{bounded term 2:local}
		\eeq
		
		\noi Combining (\ref{eq:Pythagoras:local}), (\ref{bounded term 1:local}) and (\ref{bounded term 2:local}), and using $F(\x) = f(\x) + h(\x) + g(\x)$, we have:
		\beq 
		&& \E_{\B^t}[\tfrac{1}{2} \|\x^{t+1} - \breve{\x}\|_{\H_+^t}^2] - \tfrac{1}{2} \|\x^{t} - \breve{\x}\|_{\H_+^t}^2\nn\\
		& \leq & \tfrac{k}{n} F(\breve{\x}) + (1 - \tfrac{k}{n}) F(\x^t) - \E_{\B^t}[F(\x^{t+1})] + \tfrac{k\rho}{2n}\|\breve{\x} - \x^t\|_2^2 + \tfrac{k}{n}\epsilon  \nn\\
		& \overset{\step{1}}{\leq} & \tfrac{k}{n} F(\breve{\x}) + (1 - \tfrac{k}{n}) F(\x^t) - \E_{\B^t}[F(\x^{t+1})] + \tfrac{k\rho}{2nc_1}\|\breve{\x} - \x^t\|_{\H_+^t}^2 + \tfrac{k}{n}\epsilon  \nn\\
		& \overset{\step{2}}{=} & (1 - \tfrac{k}{n}) \breve{q}^t - \E_{\B^t}[\breve{q}^{t+1}] + \tfrac{k\rho}{2nc_1}\|\breve{\x} - \x^t\|_{\H_+^t}^2 + \tfrac{k}{n}\epsilon,
		\eeq
		where step \step{1} uses the fact that $c_1\|\x\|_2^2 \leq \|\x\|^2_{\H_+}$ in Lemma \ref{lemma:iequality}; step \step{2} uses the definition of $\breve{q}^t \triangleq F(\x^t) - F(\breve{\x})$ and $F(\x^t) - F(\x^{t+1}) = \breve{q}^{t} - \breve{q}^{t+1}$. With the definition $\breve{r}^{t+1} \triangleq \tfrac{1}{2} \|\x^{t+1} - \breve{\x}\|_{\H_+^t}^2$, we rewrite the inequality above as:
		\beq
		\E_{\B^t} [\breve{r}^{t+1}] + \E_{\B^t}[\breve{q}^{t+1}]  & \leq &  (1 + \tfrac{k\rho}{nc_1}) \breve{r}^{t}  + (1 - \tfrac{k}{n}) \breve{q}^t + \tfrac{k}{n}\epsilon.
		\eeq
		\textbf{(ii)} We now consider the case when $F(\cdot)$ satisfies the Luo-Tseng error bound assumption for critical point $\breve{\x}$. We bound the term $\breve{r}^t$ using the same strategy in \ref{app:theo:dec:x:f} by following inequalities: 
		\beq
		\breve{r}^t & \triangleq & \tfrac{c_2}{2}\|\x^t - \breve{\x}\|_2^2
		\overset{\step{1}}{\leq} \tfrac{c_2}{2}\left(\tfrac{\delta}{C_n^k}\right)^2 \left(\sum_{\B\in \Pi_n^k} |\bar{\mathcal{N}}_{\B}(\x^t)| \right)^2
		\overset{\step{2}}{\leq} \tfrac{c_2}{2}\left(\tfrac{\delta}{C_n^k}\right)^2 C_n^k \left(\sum_{\B\in \Pi_n^k} |\bar{\mathcal{N}}_{\B}(\x^t)|^2 \right)\nn\\
		&\overset{\step{3}}{\leq} &\tfrac{c_2}{2}\left(\tfrac{\delta}{C_n^k}\right)^2 \left( C_n^k \right)^2 \left( \E_{\B^t}[\|\x^{t+1} - \x^t\|_2^2] \right)
		\overset{}{=} \tfrac{c_2\delta^2}{\theta} \cdot \tfrac{\theta}{2}\left( \E_{\B^t}[\|\x^{t+1} - \x^t\|_2^2] \right) \nn\\
		&\overset{\step{4}}{\leq} &\tfrac{c_2\delta^2}{\theta}( F(\x^t) -  \E_{\B^t}[F(\x^{t+1})] ) 
		\overset{}{=} \tfrac{c_2\delta^2}{\theta}\left( \breve{q}^t - \E_{\B^t}[\breve{q}^{t+1}] \right)
		\overset{\step{5}}{=} \kappa_0\left( \breve{q}^t - \E_{\B^t}[\breve{q}^{t+1}] \right),
		\eeq
		where step \step{1} uses the assumption that $\|\x^t - \breve{\x}\|_2^2 \leq \delta^2(\tfrac{1}{C_n^k} \sum_{\B\in \Pi_n^k} |\text{dist}(\mathbf{0},\bar{\mathcal{N}}_{\B}(\x)|)^2$ for any coordinate-wise stationary point $\breve{\x}$; step \step{2} uses $\|\x\|_1^2 \leq n\|\x\|_2^2$; step \step{3} uses $\E_{\B^t}[\|\x^{t+1} - \x^t\|_2^2] =  \E_{\B^t}[\|\x^t + \bar{\mathcal{N}}_{\B}(\x^t) - \x^t\|_2^2] = \E_{\B^t}[ |\bar{\mathcal{N}}_{\B}(\x^t)|^2 ] = \tfrac{1}{C_n^k}\sum_{\B\in \Pi_n^k} |\bar{\mathcal{N}}_{\B}(\x^t)|^2$; step \step{4} uses the fact that  $\E_{\B^t}[\tfrac{1}{2}\|\x^{t+1} - \x^t\|_2^2] \leq \tfrac{1}{\theta} \E_{\B^t}[F(\x^t) - F(\x^{t+1})]$; step \step{5} uses $\kappa_0 = \tfrac{c_2\delta^2}{\theta}$.

		\noi Then we have following inequalities:
		\beq
		&& \E_{\B^t} [\breve{r}^{t+1}] + \E_{\B^t}[\breve{q}^{t+1}] \leq (1 + \tfrac{k\rho}{nc_1}) \breve{r}^{t} + (1 - \tfrac{k}{n})\breve{q}^t + \tfrac{k}{n}\epsilon\nn\\
		&\Rightarrow& \E_{\B^t}[\breve{q}^{t+1}] \leq \kappa_0(1 + \tfrac{k\rho}{nc_1}) \left(\breve{q}^t - \E_{\B^t}[\breve{q}^{t+1}]\right) + (1 - \tfrac{k}{n}) \breve{q}^t + \tfrac{k}{n}\epsilon\nonumber\\
		&\Rightarrow& (\kappa_0(1 + \tfrac{k\rho}{nc_1}) + 1)\E_{\B^t}[\breve{q}^{t+1}] \leq (\kappa_0(1 + \tfrac{k\rho}{nc_1}) + 1 - \tfrac{k}{n})\breve{q}^t + \tfrac{k}{n}\epsilon\nonumber\\
		&\Rightarrow& \E_{\B^t}[\breve{q}^{t+1}] \overset{\step{1}}{\leq} (1 - \tfrac{k}{n\kappa_2}) \breve{q}^t + \tfrac{k}{n\kappa_2}\epsilon \nonumber\\
		&\Rightarrow& \E_{\xi^t}[\breve{q}^{t+1}] \overset{\step{2}}{\leq} \breve{\kappa}^{t+1} \breve{q}^0 + \tfrac{k(\breve{\kappa}^{t+2} - 1)}{n\kappa_2(\breve{\kappa} - 1)}\epsilon = \breve{\kappa}^{t+1} \breve{q}^0 + (1 - \breve{\kappa}^{t+2})\epsilon \leq \breve{\kappa}^{t+1} \breve{q}^0 + \epsilon,
		\eeq 
		
		\noi where step \step{1} uses the definition $\kappa_2 \triangleq \kappa_0(1 + \tfrac{k\rho}{nc_1}) + 1$; step \step{2} uses the definition $\breve{\kappa} \triangleq 1 -  \tfrac{k}{n\kappa_2}$. 
		
		\noi \textbf{(iii)} Given the $\sigma$ with $\sigma \leq 2\breve{q}^0$, we now consider the situation where the inequality $\E_{\xi^t}[\breve{q}^{t+1}] \leq \sigma$ is satisfied. We first define: $\kappa_3 \triangleq \kappa_0(1 + \tfrac{2Ck}{nc_1\epsilon}) + 1$ and $\kappa_4 \triangleq 1 - \tfrac{k}{n\kappa_3}$. Then we have: $\breve{\kappa} = \kappa_4$ according to Lemma \ref{lemma:globallybounded}. Since $\kappa_2 > 1$, $0 < \breve{\kappa} < 1$, we have:
		\beq
		\left\{
		\begin{aligned}
			\breve{\kappa}^{t+1} \breve{q}^0 \leq \tfrac{\sigma}{2}\\
			\epsilon \leq \tfrac{\sigma}{2}
		\end{aligned}
		\right.
		\Rightarrow \left\{
		\begin{aligned}
			t + 1 &\geq \log_{\kappa_4}\tfrac{\sigma}{2\breve{q}^0} \\
			\epsilon &\leq \tfrac{\sigma}{2}
		\end{aligned}
		\right.
		\overset{\step{1}}{\Rightarrow}t+1 \geq \log_{\breve{\kappa}'} \tfrac{\sigma}{2\breve{q}^0},
		\eeq
		
		\noi where step \step{1} uses the definition $\kappa_3' \triangleq \kappa_0(1 + \tfrac{4Ck}{nc_1\sigma}) + 1$ and $\breve{\kappa}' \triangleq 1 - \tfrac{k}{n\kappa_3'}$. Therefore, $\E_{\xi^t}[\breve{q}^{t+1}] \leq \sigma$ when $t + 1 \geq log_{\breve{\kappa}'}\tfrac{\sigma}{2\breve{q}^0}$
	\end{proof}

	\section{Proofs of Results in Section \ref{sec:geo}}
	
	\subsection{Proof of Theorem \ref{theo:excat:index}}\label{app:theo:excat:index}
	\begin{proof}
		Let $\ddot{\x}$ be any coordinate-wise stationary point of Problem (\ref{plb:geo:topk}). We have the following result for all $i\in[n]$ and $j\in[n]$ with $i\neq j$:
		\beq \label{eq:sparse:1:opt}
		0 \in \arg \min_{\eta} \,  \mathcal{P}(\eta)\triangleq \tfrac{\alpha}{2} \eta^2 + \beta \eta   - \lambda \|\ddot{\x} + \eta e_i - \eta e_j \|_{[s]},
		\eeq
		\noi where $\alpha = \bar{\Q}_{ii} + \bar{\Q}_{jj} - 2\bar{\Q}_{ij}$ and $\beta= \nabla_i f(\ddot{\x}) - \nabla_j f(\ddot{\x})$. Using the results in our previous discussions, we have:
		\begin{align} \label{eq:dfdfopdf}
			\begin{split}
				0 \in \arg \min_{\eta \in \{\eta_1,\eta_2,\eta_3,\eta_4,\eta_5\}} \mathcal{P}(\eta),
			\end{split}
		\end{align}
		\noi $\text{where}~\eta_1 = -\x_{i},\eta_2 = \x_j,\eta_3 =  \tfrac{-\beta}{\alpha},\eta_4 =  (\lambda - \beta) / \alpha,\eta_5 =  -(\lambda + \beta) / \alpha .$ Note that $(\eta_1,\eta_2,\eta_3,\eta_4,\eta_5)$ are the five breakpoints of Problem (\ref{eq:sparse:1:opt}). 
		
		\noi Since $\lambda > 2\| \nabla f(\ddot{\x})\|_{\infty}$, we have: $\forall i,\,\lambda - \nabla_i f(\ddot{\x}) + \nabla_j f(\ddot{\x}) >0,\,-\nabla_i f(\ddot{\x}) + \nabla_j f(\ddot{\x}) - \lambda <0$, leading to $\eta_4 \neq 0$, $\eta_5\neq 0$. The optimality condition in (\ref{eq:dfdfopdf}) reduces to
		\beq\label{eq:dfdfopdf2}
		0 \in \arg \min_{\eta \in \{\eta_1,\eta_2,\eta_3\}} \,\mathcal{P}(\eta) .
		\eeq
		
		\noi We denote $\mathcal{T}(\ddot{\x})$ as the indices of non-zero components of $\ddot{\x}$ in magnitude and analyze two cases for (\ref{eq:dfdfopdf2}). 
		
		\text{Case \textbf{(i)}}: $0 = \eta_1$ or $0=\eta_2$. It implies $-\ddot{\x}_i=0$ or $-\ddot{\x}_j = 0$, leading to $i \notin \mathcal{T}(\ddot{\x})$ or $j \notin \mathcal{T}(\ddot{\x})$. Hence $\|\ddot{\x}\|_1 = \sum_{i \in \mathcal{T}(\ddot{\x})} \ddot{\x}_i$ with $\ddot{\x} \geq 0$. Note that $\mathcal{T}(\ddot{\x})$ is a subset of top-$s$ set. Therefore, we have $\|\ddot{\x}\|_1 = \sum_{i=1}^s |\ddot{\x}_{[i]}|$. 
		
		\text{Case \textbf{(ii)}}: $0 = \eta_3$. When $i,j \notin \mathcal{T}(\ddot{\x})$, this situation is discussed in \text{Case \textbf{(i)}}. When $i,j \in \mathcal{T}(\ddot{\x})$, it implies that $0 = \tfrac{\nabla_i f(\ddot{\x})}{\alpha} - \tfrac{\nabla_j f(\ddot{\x})}{\alpha}  \Rightarrow \nabla_i f(\ddot{\x}) = \nabla_j f(\ddot{\x})$.
	\end{proof}

	\subsection{Proof of Theorem \ref{theo:extreme:bin}}\label{app:theo:extreme:bin}
	We first present the following useful lemma.
	\begin{lemma}\label{lemma:pmbounded}
		For all $\alpha \in \mathbb{R}$, $\beta \in \mathbb{R}$, $x \in \mathbb{R}$, $y \in \mathbb{R}$ with $\alpha < 0$, we bound: $\tfrac{\alpha}{2}(x \pm 1)^2 + \beta(x \pm 1) \leq \tfrac{\alpha}{2} - \tfrac{\beta^2}{2\alpha}.$
	\end{lemma}
	\begin{proof}
		$\tfrac{\alpha}{2}(x \pm 1)^2 + \beta(x \pm 1) + \tfrac{\beta^2}{2\alpha} = \tfrac{\alpha}{2}x^2 + \beta x + \tfrac{\beta^2}{2\alpha} + \tfrac{\alpha}{2} + (\pm \alpha x \pm \beta) \leq  \tfrac{\alpha}{2}x^2 + \beta x + \tfrac{\beta^2}{2\alpha} + \tfrac{\alpha}{2} + \max(\pm \alpha x \pm \beta) \overset{\step{1}}{=} \tfrac{1}{2\alpha}(\alpha x + \beta)^2 + \tfrac{\alpha}{2} + |\alpha x + \beta| \overset{}{=}  \tfrac{1}{2\alpha}(|\alpha x + \beta| + \alpha)^2 \overset{\step{2}}{\leq}  \tfrac{\alpha}{2}$, where \step{1} uses $\max(\pm x \pm c) = |x+c|$, \step{2} uses the fact that: $|\alpha x + \beta| \geq 0$ and $\alpha < 0$.
	\end{proof}
	We now present the proof of Theorem \ref{theo:extreme:bin}.
	\begin{proof}
		\noi Let $\ddot{\x} \in \mathbb{R}^n$ be any \textbf{CWS}-point of Problem (\ref{plb:geo:bin}). We have following subproblem for any $\ddot{\x}$:
		\beq
		0 \in \argmin_{\eta} \mathcal{P}(\eta) \triangleq \tfrac{\alpha}{2} \eta^2 + \beta \eta,\ s.t.\ l \leq \eta \leq u,\label{plb:geo:subbin}
		\eeq
		where $l = \max(-\ddot{\x}_i-1,\ddot{\x}_j+1)$, $u = \min(1-\ddot{\x}_i,\ddot{\x}_j-1)$, $\alpha = \bar{\Q}_{ii} + \bar{\Q}_{jj} - 2\bar{\Q}_{ij}$, and $\beta = \nabla_{i} \hat{f}(\x) - \nabla_{j} \hat{f}(\x)$. Hence, Problem (\ref{plb:geo:subbin}) has three breakpoints $\{l,u,e\}$, with $e = -\tfrac{\beta}{\alpha}$, leading to: $	\min(\mathcal{P}(l),\mathcal{P}(u)) = \min(\tfrac{\alpha}{2} l^2 + \beta l,\tfrac{\alpha}{2} u^2 + \beta u)\overset{\step{1}}{\leq}-\tfrac{\beta^2}{2\alpha} = \mathcal{P}(e)$, where step \step{1} uses Lemma \ref{lemma:pmbounded} and $\alpha < 0$. We obtain: $l = 0$ or $u = 0$, resulting in either $l = -\ddot{\x}_i-1 = 0$ or $l = \ddot{\x}_j+1 = 0$ or $u = 1-\ddot{\x}_i = 0$ or $u = \ddot{\x}_j-1 = 0$. We conclude that $|\ddot{\x}_r| = 1$ for all $r \in [n]$.
	\end{proof}

	\subsection{Proof of Theorem \ref{theo:exact:bin}}\label{app:theo:exact:bin}
	We first give following lemmas.
	\begin{lemma}\label{lemma:xi}
		For any $\alpha \in \mathbb{R} $, $ c \in \mathbb{R} $ and $ \xi \in \mathbb{R} $ with $ \alpha > 0 $ and $ \xi \in (-1, 1) $, we have: $\min\left(\tfrac{\alpha}{2}(\xi \mp 1)^2 + c(\xi \mp 1)\right) \leq \tfrac{\alpha}{2}(1 - \xi^2)$.
	\end{lemma}
	\begin{proof}
		We have: $\min\left(\tfrac{\alpha}{2}(\xi \mp 1)^2 + c(\xi \mp 1)\right) - \tfrac{\alpha}{2} + \tfrac{\alpha}{2}\xi^2 = c \xi + \alpha\xi^2 + \min(\mp \alpha \xi \mp c) \overset{\step{1}}{=}\xi \cdot (\alpha \xi + c) - |\alpha \xi + c| \overset{\step{2}}{\leq} |\xi| \cdot |\alpha \xi + c| - |\alpha \xi + c| = |\alpha \xi + c| \cdot (|\xi| - 1) \overset{\step{3}}{\leq} 0$, where step \step{1} uses $ \min(\mp x) = \min(-x, x) = -|x| $; step \step{2} uses $ xy \leq |x| \cdot |y| $; step \step{3} uses $ \xi \in (-1, 1)$.
	\end{proof}
	\begin{lemma}\label{lemma:xu}
		For any $ x \in \mathbb{R} $ and $u \in \mathbb{R}$ with $x \geq 0$ and $u \geq 0$, we have: $\sqrt{x} - \sqrt{x+u} \leq - \tfrac{u}{2\sqrt{x+u}}$.
	\end{lemma}
	\begin{proof}
		We have: $\sqrt{x+u} - \sqrt{x} =  \tfrac{u}{\sqrt{x+u}+\sqrt{x}} \geq \tfrac{u}{\sqrt{x+u}+\sqrt{x+u}}$.
	\end{proof}
	\noi Then we provide the proof of this theorem.
	\begin{proof}
		We first define $\mathcal{H}(\eta) \triangleq \tfrac{\alpha}{2} \eta^2 + \beta \eta$ and $\mathcal{G}(\eta) \triangleq - \lambda \sqrt{\| \ddot{\x} \|_2^2 + 2\eta^2 + 2 \ddot{\x} _{i}\eta - 2 \ddot{\x} _{j }\eta}$. Let $\ddot{\x}$ be any \textbf{CWS}-point of Problem (\ref{plb:geo:bin2}). Then the following result holds for any $ i ,j \in [n] $:
		\beq
		0 \in \arg \min_{l \leq \eta \leq u} \mathcal{P}(\eta) \triangleq \mathcal{H}(\eta) + \mathcal{G}(\eta),\label{condition1}
		\eeq
		where $l = \max(-1 -  \ddot{\x}_{i}, \ddot{\x}_{j} - 1)$, $u = \min(1-  \ddot{\x}_{i}, \ddot{\x}_{j}+1)$, $\alpha = \bar{\Q}_{ii} + \bar{\Q}_{jj} - 2\bar{\Q}_{ij}$, $\beta = \nabla_{i} f(\ddot{\x}) - \nabla_{j} f( \ddot{\x})$.

		\noindent
		\textbf{(i)} We now prove that $\ddot{\x} = 1 $ if $ \lambda > \phi\sqrt{n+2}$, concluding via contradiction. Suppose $1 - \ddot{\x}_i^2 \neq 0$ for certain $i \in [n]$. Given $-1 \leq \ddot{\x} \leq 1$, we derive:
		\beq
		\min(\mathcal{H}(\pm 1 - \ddot{\x}_i)) = \min( \tfrac{\alpha}{2} (\ddot{\x}_i \mp 1)^2 - \beta (\ddot{\x}_i \mp 1)) \overset{\step{1}}{\leq} \tfrac{\alpha}{2} \cdot (1- \ddot{\x}_i^2),\label{hf}
		\eeq
		where step \step{1} uses Lemma \ref{lemma:xi}. When $- \ddot{\x}_{i} > \ddot{\x}_{j}$, $l = -1 -  \ddot{\x}_{i}$, we have following inequality:
		\beq
		\mathcal{G}(- 1 - \ddot{\x}_i) =  - \lambda \sqrt{\| \ddot{\x} \|_2^2 + 2\ddot{\x}_i + 2 \ddot{\x} _{j} + 2\ddot{\x}_i\ddot{\x}_j+ 2 } \leq - \lambda \sqrt{\| \ddot{\x} \|_2^2 +2 - 2\ddot{\x}_i^2}.\label{gf1}
		\eeq
		\noindent When $- \ddot{\x}_{i} < \ddot{\x}_{j}$, $u = -1 -  \ddot{\x}_{i}$, we have:
		\beq
		\mathcal{G}(1 - \ddot{\x}_i) =  - \lambda \sqrt{\| \ddot{\x} \|_2^2 - 2\ddot{\x}_i  - 2 \ddot{\x} _{j} + 2\ddot{\x}_i\ddot{\x}_j} + 2 \leq - \lambda \sqrt{\| \ddot{\x} \|_2^2 +2 - 2\ddot{\x}_i^2}.\label{gf2}
		\eeq
		\noi
		\noindent Combining (\ref{hf}), (\ref{gf1}) and (\ref{gf2}), we have following: $	\min(\mathcal{P}(\pm 1 - \ddot{\x}_i)) - \mathcal{P}(0) \leq \tfrac{\alpha}{2} (1- \ddot{\x}_i^2) - \lambda \sqrt{\| \ddot{\x} \|_2^2 + 2 - 2\ddot{\x}_i^2} + \lambda \|\ddot{\x}\|_2^2 
		\overset{\step{1}}{\leq} \tfrac{\alpha}{2} (1- \ddot{\x}_i^2) - \tfrac{\lambda(1-\ddot{\x}_i^2)}{\sqrt{\| \ddot{\x} \|_2^2 + 2 - 2\ddot{\x}_i^2}}
		\overset{}{\leq}\tfrac{\alpha}{2} (1- \ddot{\x}_i^2) - \tfrac{\lambda(1-\ddot{\x}_i^2)}{\sqrt{n + 2}}
		\overset{\step{2}}{<} 0$, where step \step{1} uses Lemma \ref{lemma:xu} with $x = \| \ddot{\x} \|_2^2 \geq 0$ and $u = 2-2\ddot{\x}_i^2 \geq 0$; step \step{2} uses $\lambda > \phi\sqrt{n+2}$ and $\tfrac{\alpha}{2} \leq \phi$. Inequality $\min(\mathcal{P}(\pm 1 - \ddot{\x}_i)) \leq \mathcal{P}(0)$ contradicts the optimality condition in (\ref{condition1}). Therefore, we have $\forall i,\ 1-\ddot{\x}_i^2 = 0$, leading to $\forall i,\ |\ddot{\x}_i| = 1$ and $\|\ddot{\x}\|= \sqrt{n}$.
	\end{proof}

	\noindent \textbf{(ii)} We now prove that $\ddot{\x}_i \neq 0$ for all $i \in [n]$ when $\lambda > \phi\sqrt{n}$, concluding by contradiction. Assume $\ddot{\x}_i = 0$ for some $i \in [n]$. Problem (\ref{condition1}) reduces to
	\beq
	0 \in \arg\min_{l \leq \eta \leq u} \mathcal{P}(\eta) \triangleq \tfrac{\alpha}{2} \eta^2 + \beta \eta - \lambda \sqrt{\|\ddot{x}\|_2^2 + 2\eta^2}. \label{condition2}
	\eeq
	We consider two cases for $\beta$. \text{Case \textbf{(i)}}: $\beta \neq 0$. If we set $\eta = -\tfrac{\beta}{\alpha}$, we have: $\mathcal{P}(\eta) = -\tfrac{\beta^2}{2\alpha} - \lambda \sqrt{\|\ddot{\x}\|_2^2 + 2(\beta/\alpha)^2} < -\lambda\|\ddot{\x}\| = \mathcal{P}(0)$, which contradicts Equation (\ref{condition2}). \text{Case \textbf{(ii)}}: $\beta = 0$. Given $\lambda > \phi\sqrt{n}$, we choose any constant $\delta > 0$ with $\lambda > \phi\sqrt{n + 2\delta}$. We have: 
	\beq
	\mathcal{P}(\delta) - \mathcal{P}(0) \overset{}{=} \tfrac{\alpha}{2} \delta^2 - \lambda \sqrt{\|\ddot{\x}\|_2^2 + 2\delta^2} + \lambda \|\ddot{\x}\| \overset{\step{1}}{\leq} \tfrac{\alpha}{2} \delta^2 - \tfrac{\lambda \delta^2}{ \sqrt{n + 2\delta^2}} \overset{\step{2}}{<} 0,
	\eeq
	where step \step{1} uses Lemma \ref{lemma:xu} and $\|\ddot{\x}\|_2^2 \leq n$; step \step{2} uses $\lambda > \phi\sqrt{n + 2\delta}$ and $\tfrac{\alpha}{2} \leq \phi$. Therefore, the inequality $\mathcal{P}(\delta) < \mathcal{P}(0)$ contradicts Equation (\ref{condition2}) with $\beta = 0$, and we conclude that $\ddot{\x}_i \neq 0$ for all $i \in [n]$.

	\section{More Experiment Results on BCD-g and BCD-l-$k$ methods}
	
	We also provide a detailed performance evaluation to assess the computational efficiency of four different methods including BCD-g, PSG, MSCR, and PDCA. These four methods are applied to four problems: SIT, NNSPCA, DCPB1, and DCPB2. Figure \ref{fig:cputimeindex} and Figure \ref{fig:cputimebin} illustrate the convergence curves of these methods, offering two insights into their performance across various datasets: (\textbf{i}) The BCD-g method consistently yields the lowest objective values compared to the other methods in all scenarios. (\textbf{ii}) The BCD-g method requires more time to converge to the optimal solution due to updating only two components in each subproblem. In conclusion, despite the increasing computational time BCD-g method takes, its performance is better than other methods significantly. This trade-off is justified by its ability to consistently find the best solutions. Moreover, we compare the performance of BCD-g, BCD-l-$k$, PSG, MSCR, and PDCA on the SIT task. As shown in Figure \ref{fig:cputimeblock}, BCD-l-10 and BCD-l-20 outperform PSG, MSCR, and PDCA in most cases. When $k$ is too small, the BCD method struggles to find an better local solution. On the other hand, when $k$ is too large, the BCD-l-$k$ method sometimes requires significantly more time to converge.

	Finally, we conduct a series of experiments to verify that the \textbf{CWS}-point condition is stronger than the critical point condition. This theoretical analysis is further elaborated in Section \ref{sec:opt}. The experiments involve applying BCD-g method after other three methods (PSG, MSCR, and PDCA). The primary goal of these experiments is to verify whether the BCD-g method is able to jump out of the local minimum identified by other methods and further reduce the objective value. The hybrid methods, which combine BCD-g with other approaches, are denoted as PSG+BCD-g, MSCR+BCD-g, and PDCA+BCD-g. The results, shown in Figure \ref{fig:compareindex} and Figure \ref{fig:comparebin}, demonstrate that our BCD-g method consistently achieves better local minima and leads to a more significant decrease in the objective function value compared to the other methods. This improvement is attributed to the BCD-g method converging to a \textbf{CWS}-point, while PSG, MSCR, and PDCA converge to critical points. In conclusion, these experiments provide compelling evidence that the \textbf{CWS}-point condition is stronger than the critical point condition.

	\begin{figure}[htbp]
		\centering
		\begin{subfigure}{0.2\textwidth}
			\includegraphics[width=\textwidth]{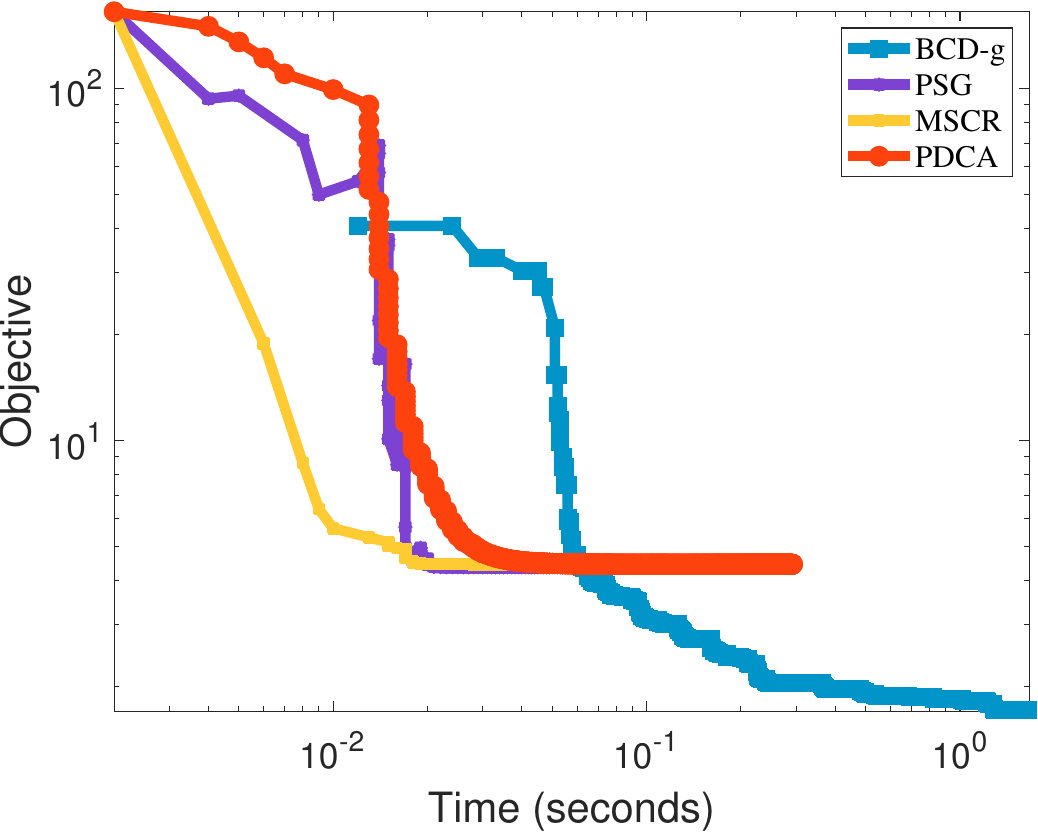}
			\caption{SIT in NAS-16}
		\end{subfigure}
		\begin{subfigure}{0.2\textwidth}
			\includegraphics[width=\textwidth]{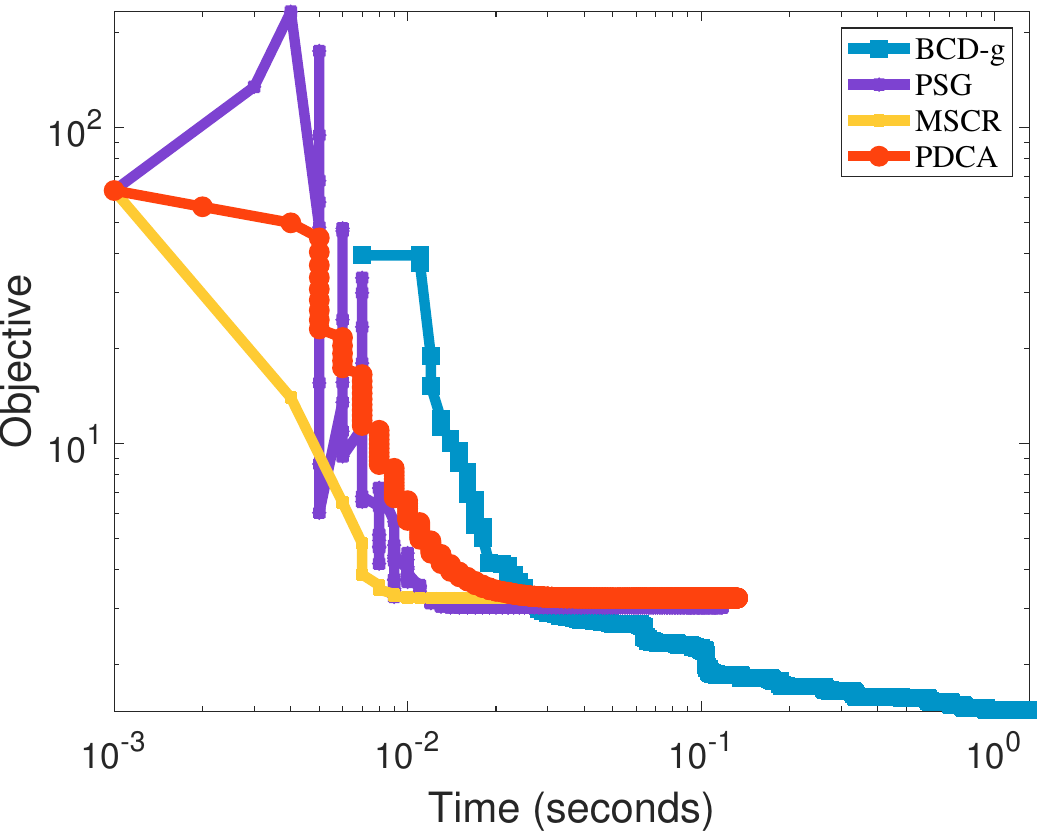}
			\caption{SIT in NAS-17}
		\end{subfigure}
		\begin{subfigure}{0.2\textwidth}
			\includegraphics[width=\textwidth]{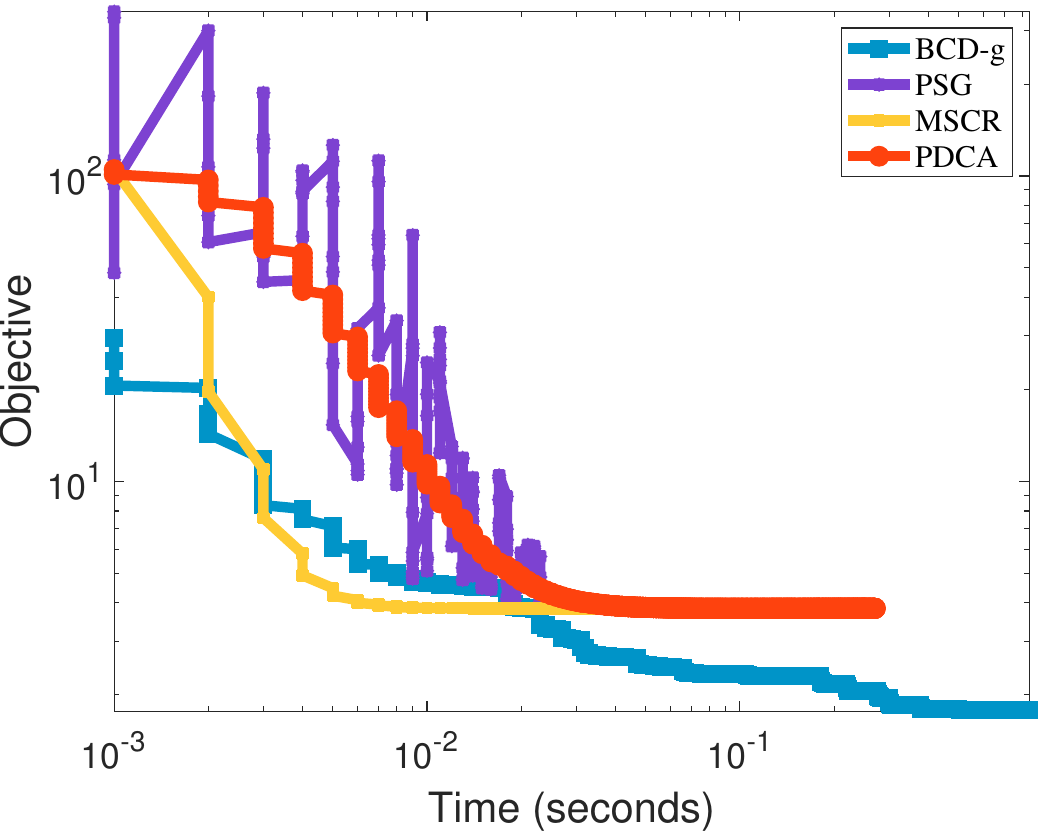}
			\caption{SIT in NAS-18}
		\end{subfigure}
		\begin{subfigure}{0.2\textwidth}
			\includegraphics[width=\textwidth]{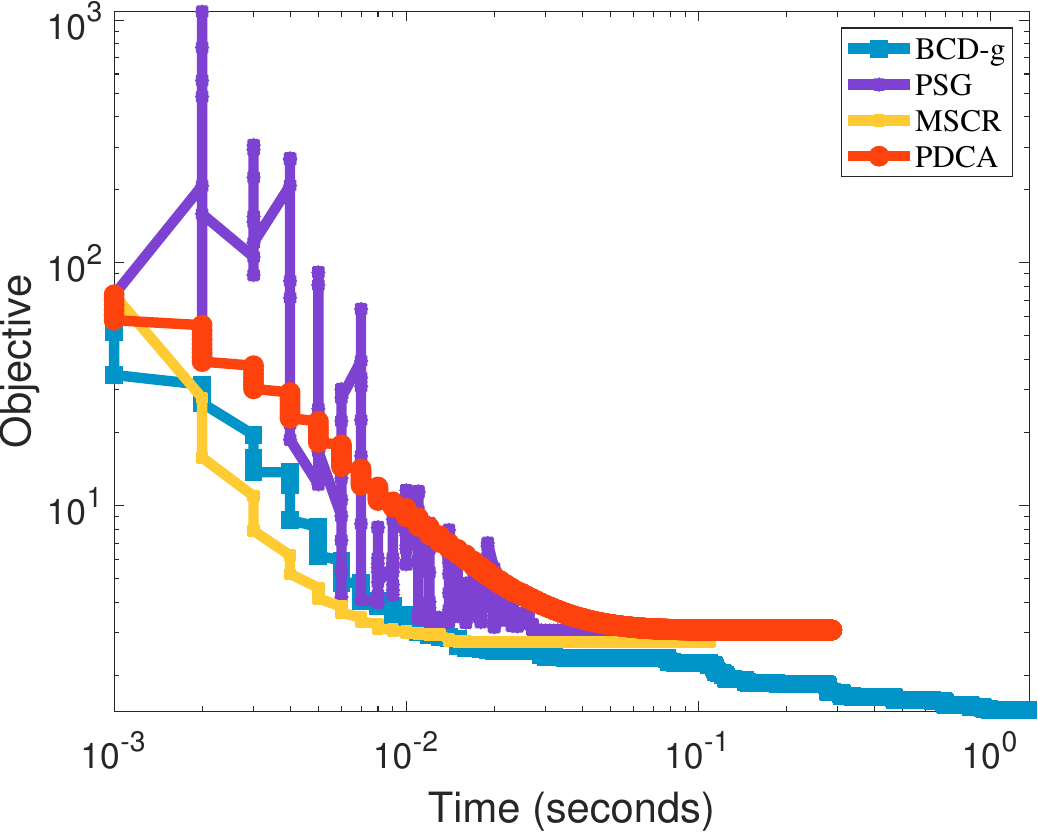}
			\caption{SIT in NAS-19}
		\end{subfigure}
		\begin{subfigure}{0.2\textwidth}
			\includegraphics[width=\textwidth]{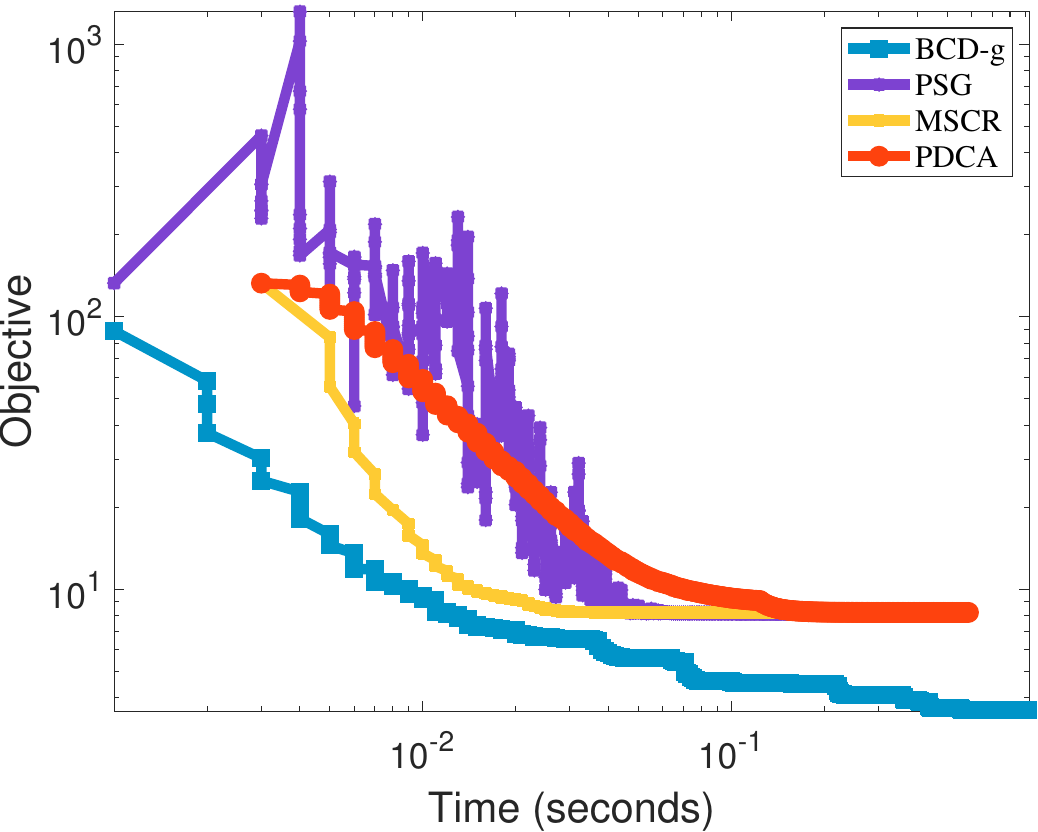}
			\caption{SIT in NAS-20}
		\end{subfigure}
		\begin{subfigure}{0.2\textwidth}
			\includegraphics[width=\textwidth]{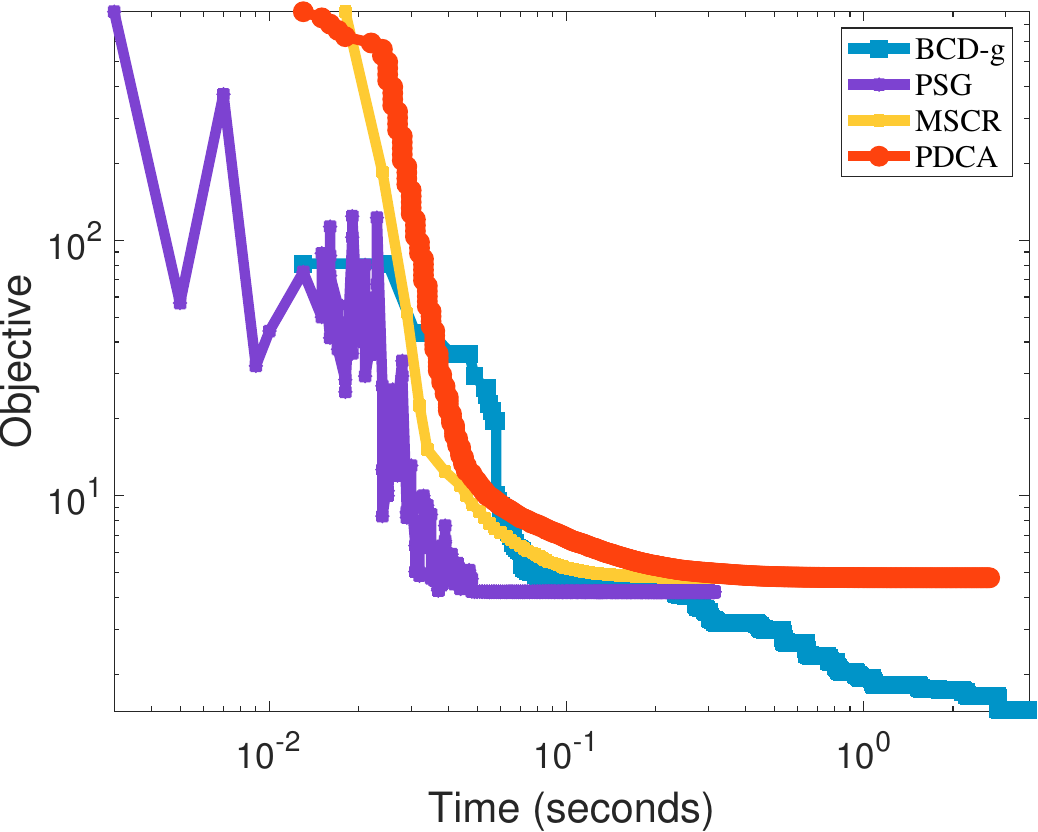}
			\caption{SIT in SP-16}
		\end{subfigure}
		\begin{subfigure}{0.2\textwidth}
			\includegraphics[width=\textwidth]{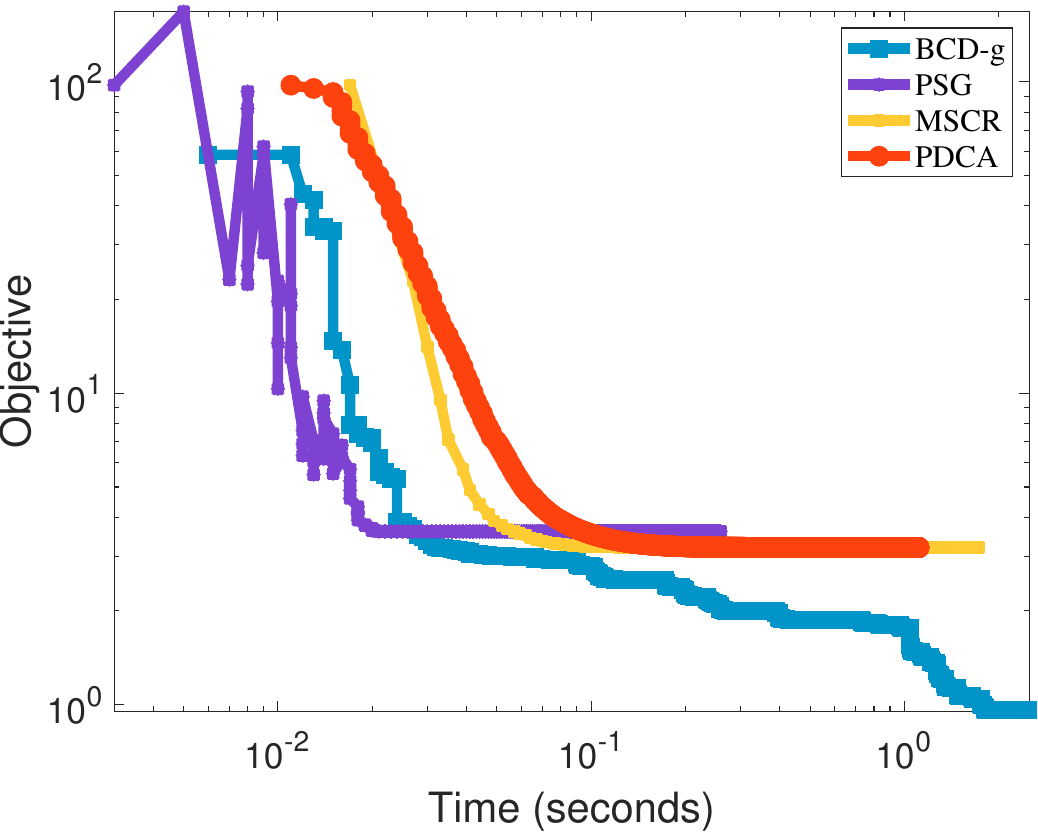}
			\caption{SIT in SP-17}
		\end{subfigure}
		\begin{subfigure}{0.2\textwidth}
			\includegraphics[width=\textwidth]{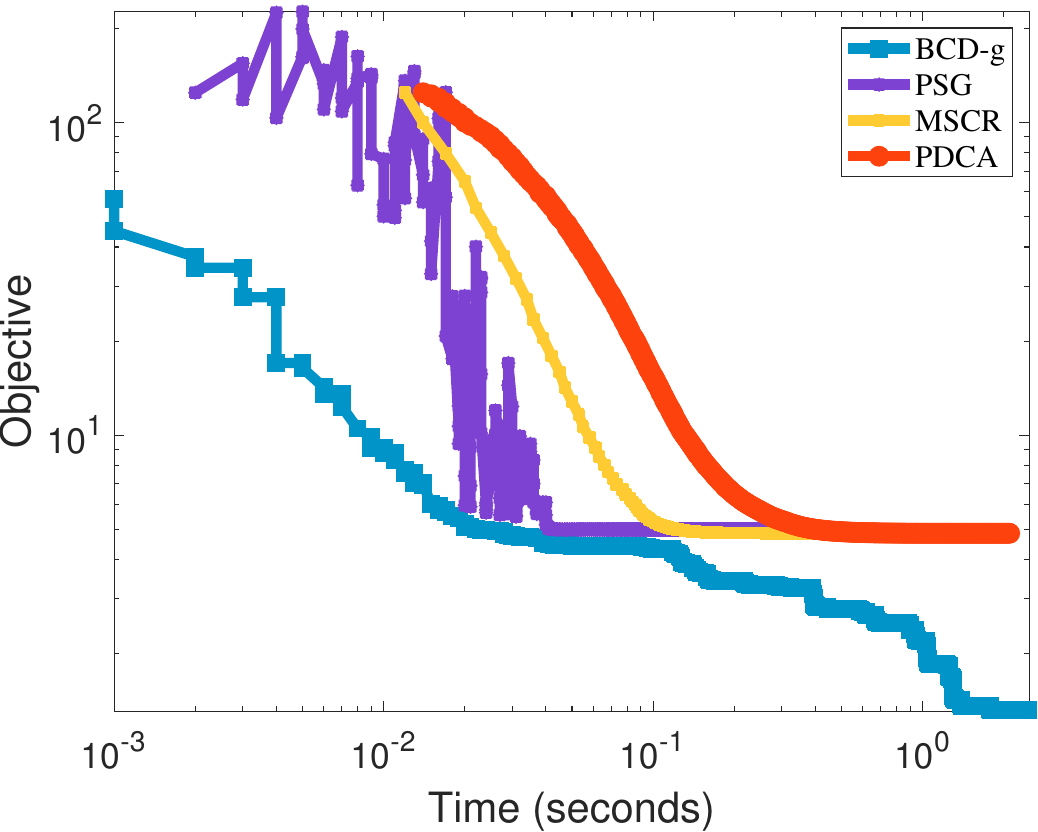}
			\caption{SIT in SP-18}
		\end{subfigure}
		\begin{subfigure}{0.2\textwidth}
			\includegraphics[width=\textwidth]{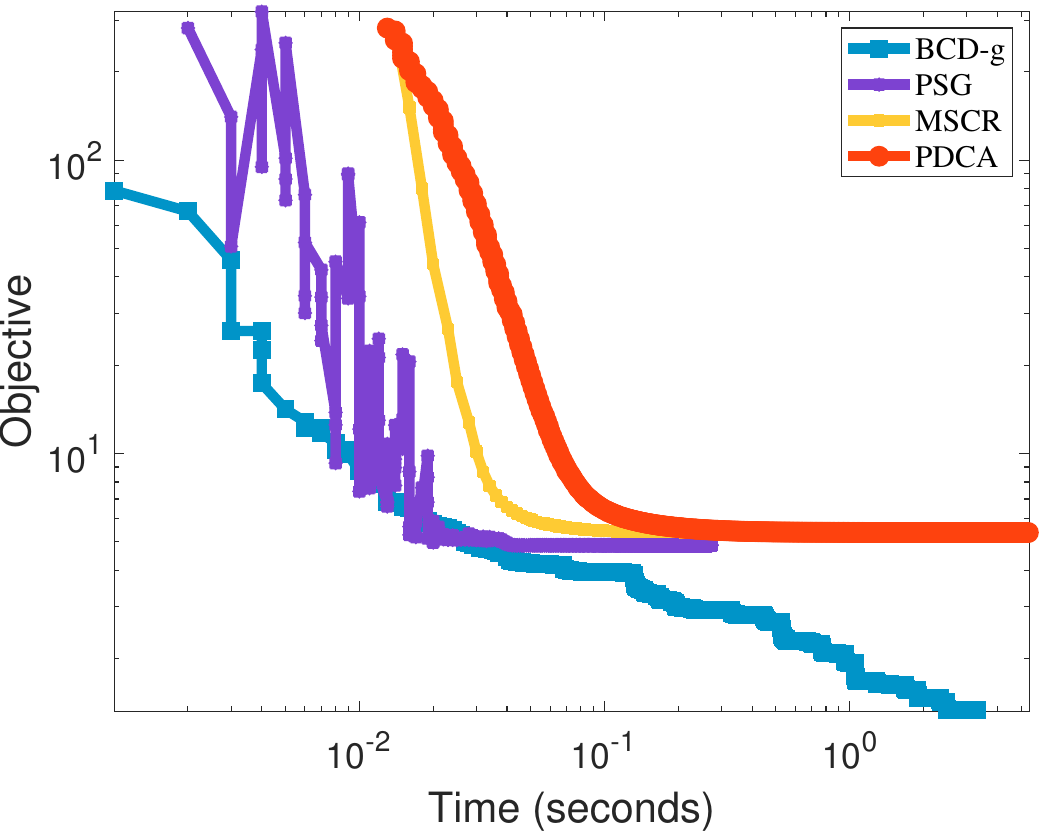}
			\caption{SIT in SP-19}
		\end{subfigure}
		\begin{subfigure}{0.2\textwidth}
			\includegraphics[width=\textwidth]{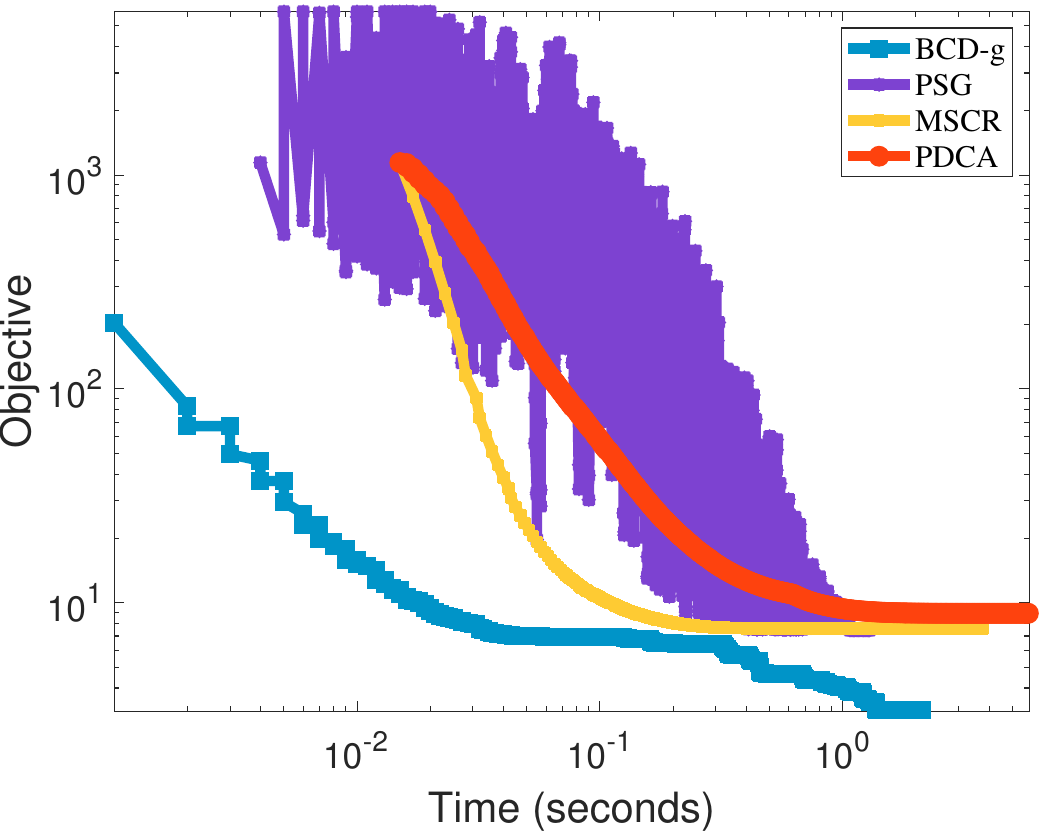}
			\caption{SIT in SP-20}
		\end{subfigure}
		\begin{subfigure}{0.2\textwidth}
			\includegraphics[width=\textwidth]{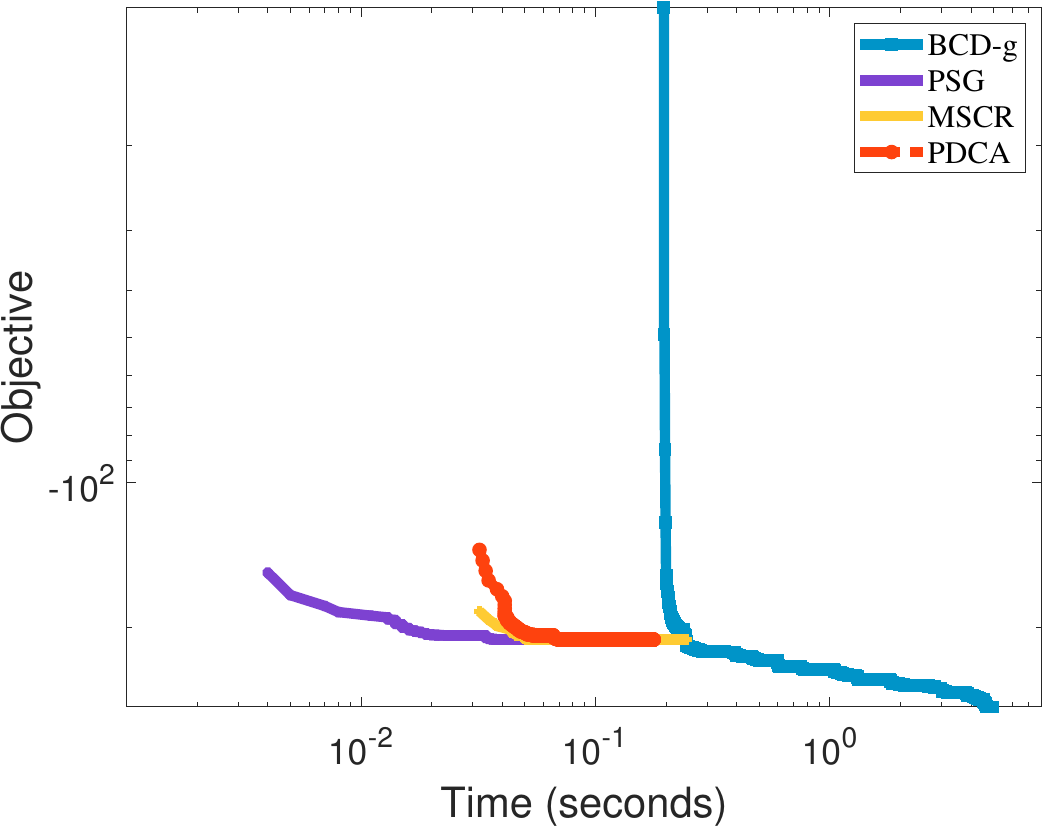}
			\caption{NNSPCA in randn-a}
		\end{subfigure}
		\begin{subfigure}{0.2\textwidth}
			\includegraphics[width=\textwidth]{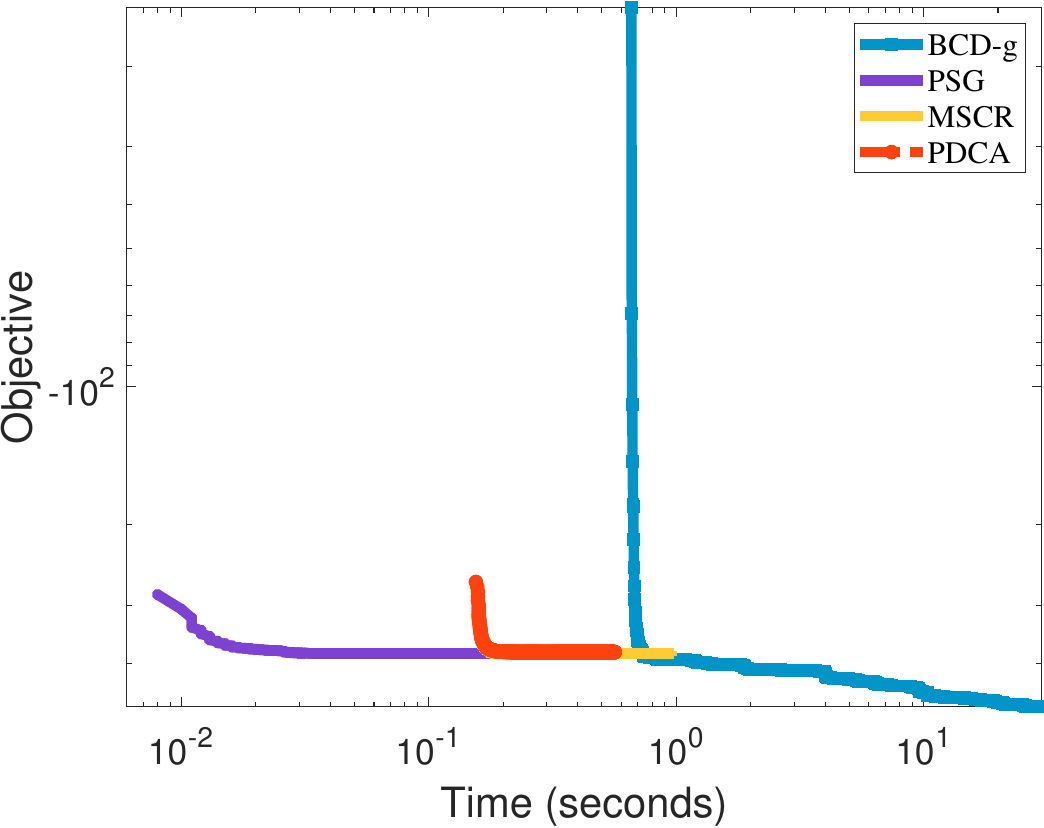}
			\caption{NNSPCA in randn-b}
		\end{subfigure}
		\begin{subfigure}{0.2\textwidth}
			\includegraphics[width=\textwidth]{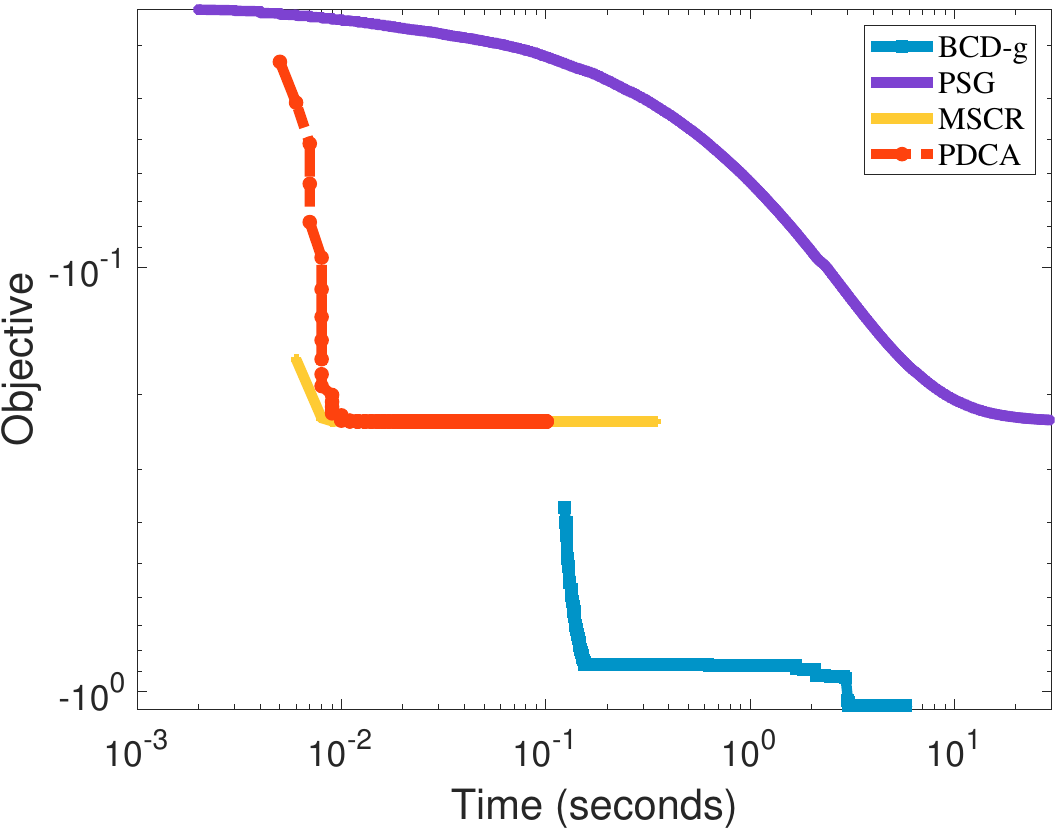}
			\caption{NNSPCA in TDT2-a}
		\end{subfigure}
		\begin{subfigure}{0.2\textwidth}
			\includegraphics[width=\textwidth]{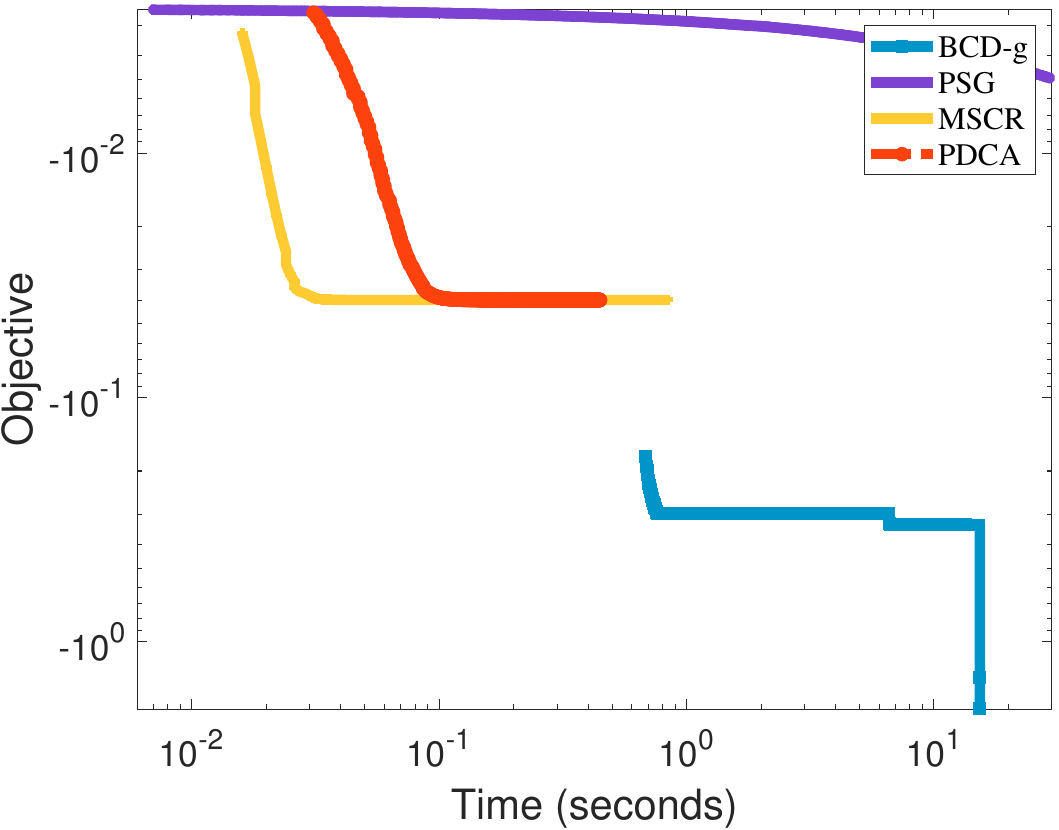}
			\caption{NNSPCA in TDT2-b}
		\end{subfigure}
		\begin{subfigure}{0.2\textwidth}
			\includegraphics[width=\textwidth]{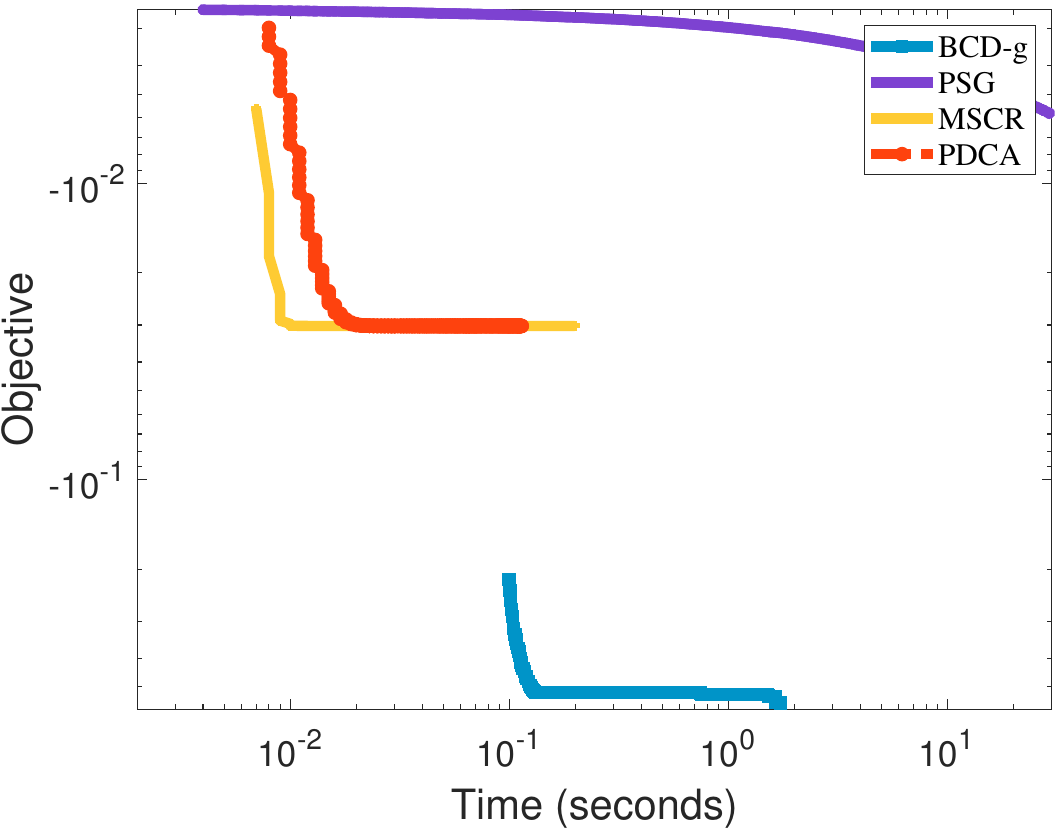}
			\caption{NNSPCA in 20News-a}
		\end{subfigure}
		\begin{subfigure}{0.2\textwidth}
			\includegraphics[width=\textwidth]{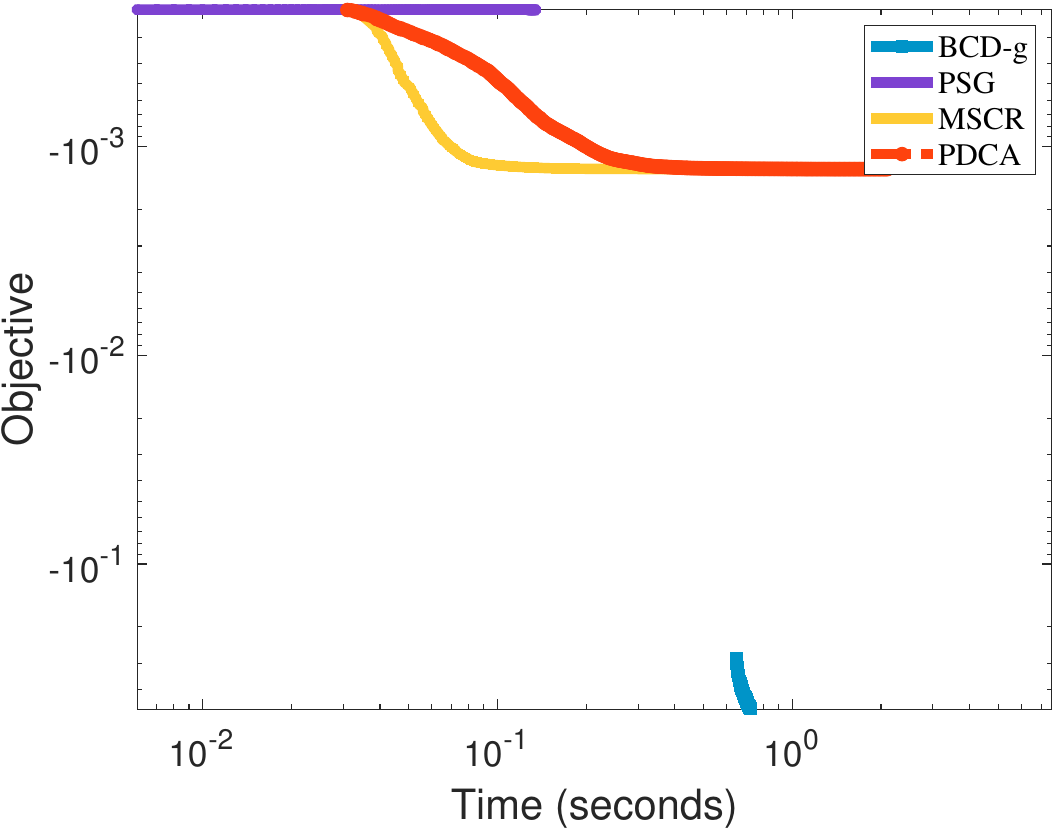}
			\caption{NNSPCA in 20News-b}
		\end{subfigure}
		\begin{subfigure}{0.2\textwidth}
			\includegraphics[width=\textwidth]{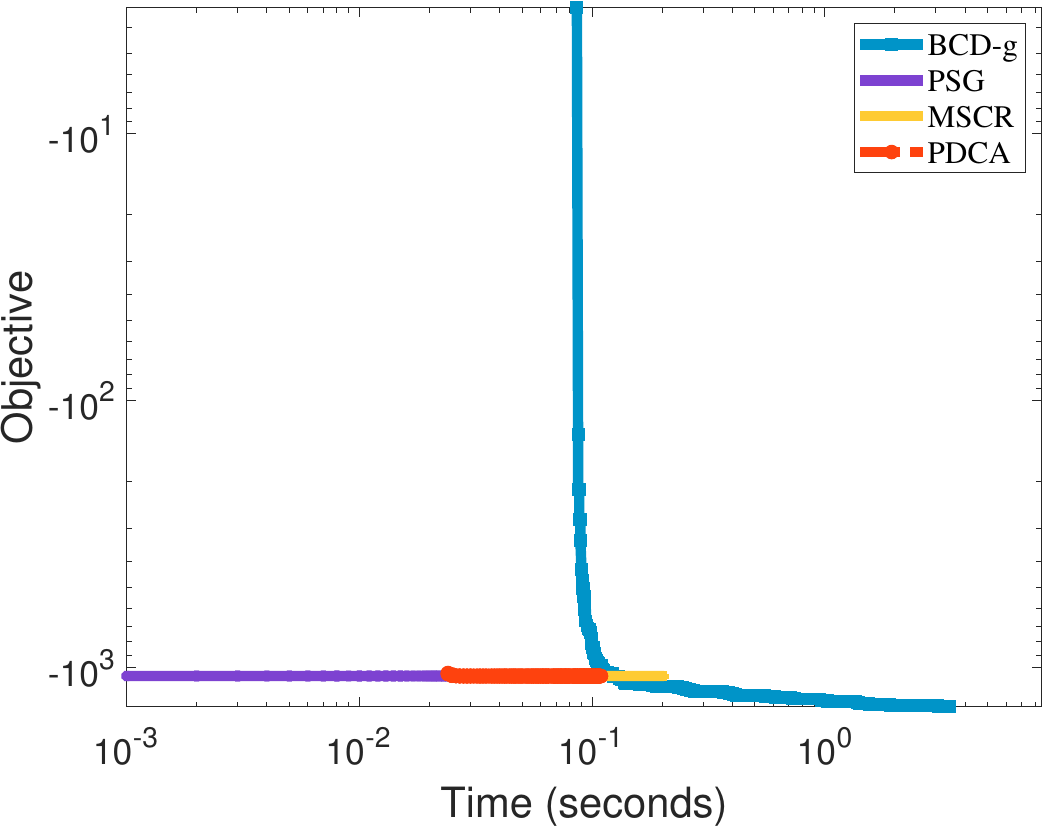}
			\caption{NNSPCA in Cifar-a}
		\end{subfigure}
		\begin{subfigure}{0.2\textwidth}
			\includegraphics[width=\textwidth]{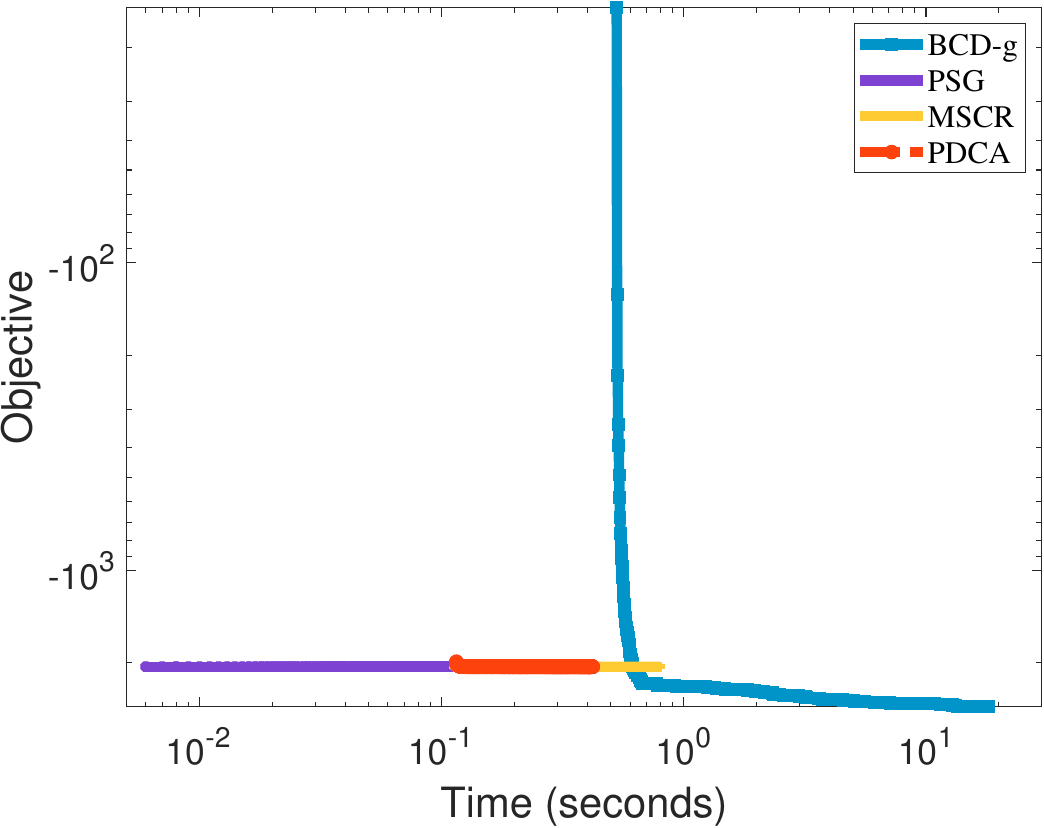}
			\caption{NNSPCA in Cifar-b}
		\end{subfigure}
		\begin{subfigure}{0.2\textwidth}
			\includegraphics[width=\textwidth]{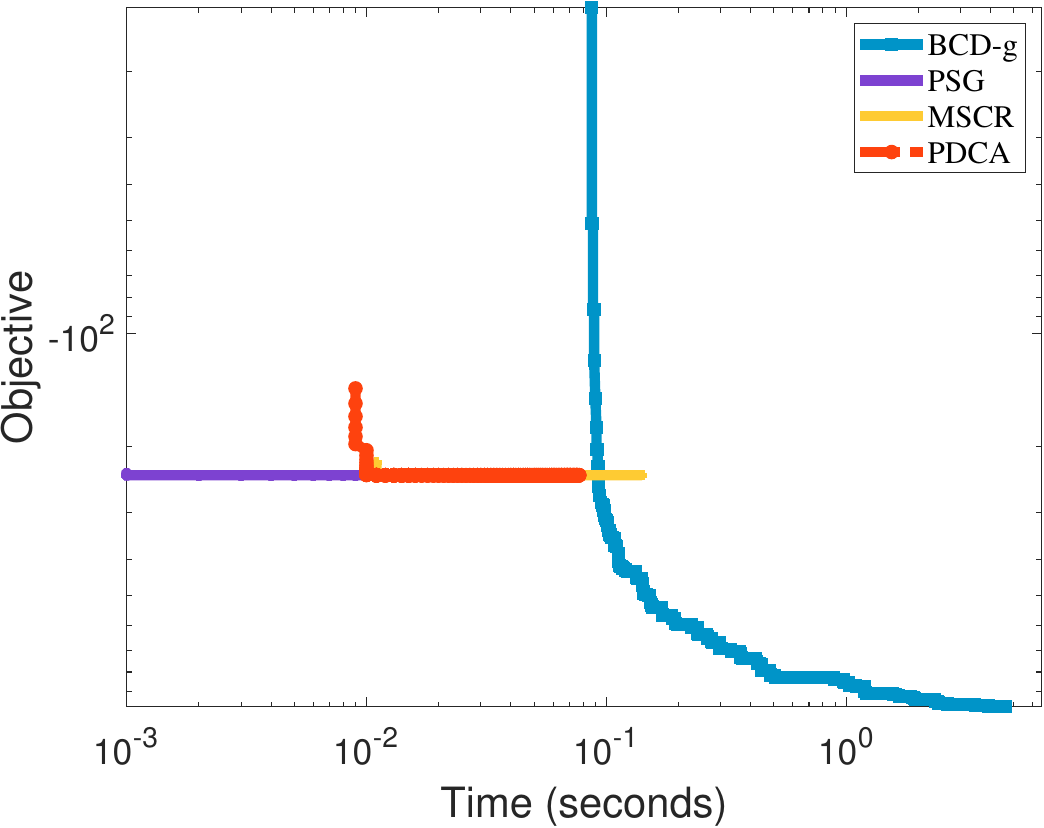}
			\caption{NNSPCA in MNIST-a}
		\end{subfigure}
		\begin{subfigure}{0.2\textwidth}
			\includegraphics[width=\textwidth]{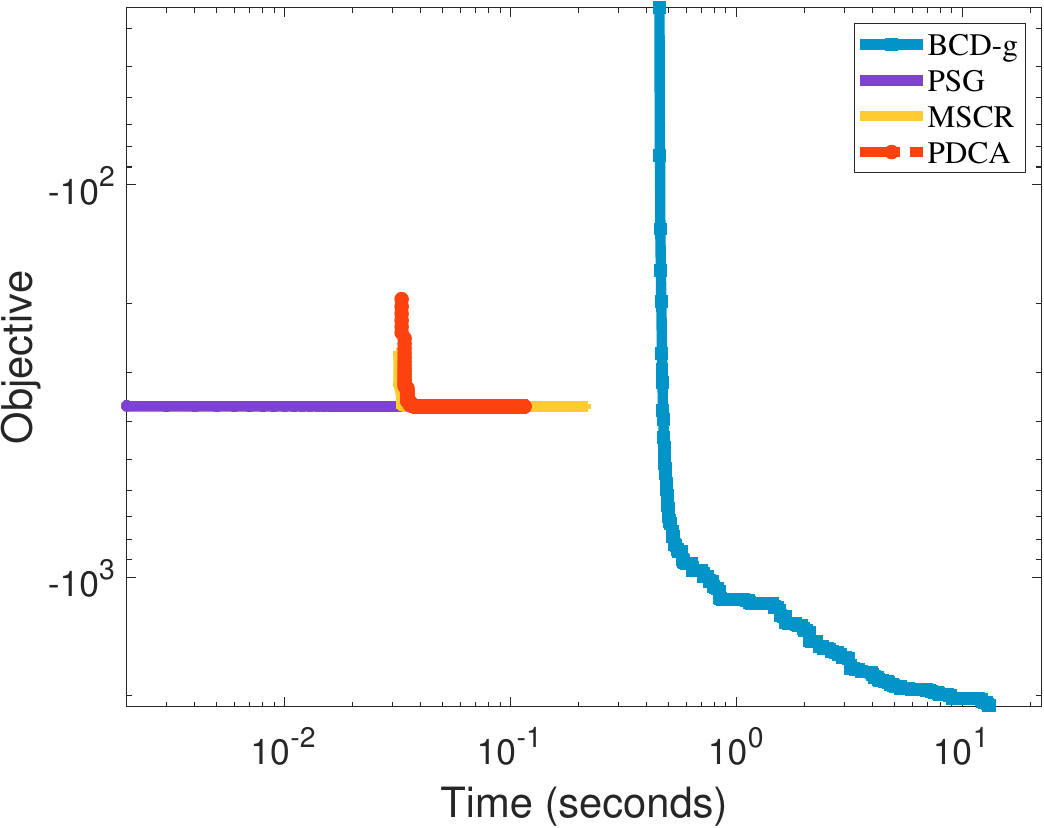}
			\caption{NNSPCA in MNIST-b}
		\end{subfigure}
		\caption{The convergence curve of the compared methods for solving index tracking problem and non-negative sparse PCA problem on 20 different datasets.}
		\label{fig:cputimeindex}
	\end{figure}

	\captionsetup[subfigure]{font=tiny}
	\begin{figure}[htbp]
		\centering
		\begin{subfigure}{0.2\textwidth}
			\includegraphics[width=\textwidth]{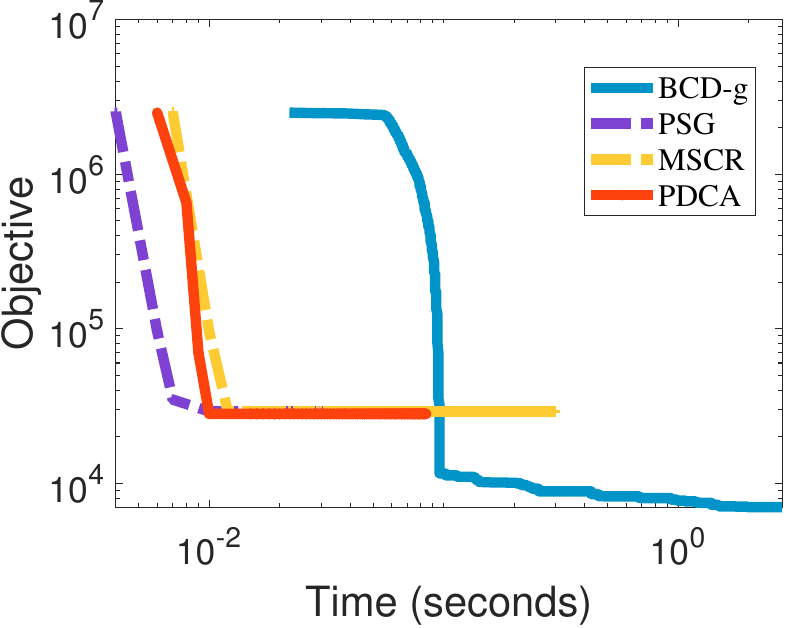}
			\caption{DCPB1 in randn-a}
		\end{subfigure}
		\begin{subfigure}{0.2\textwidth}
			\includegraphics[width=\textwidth]{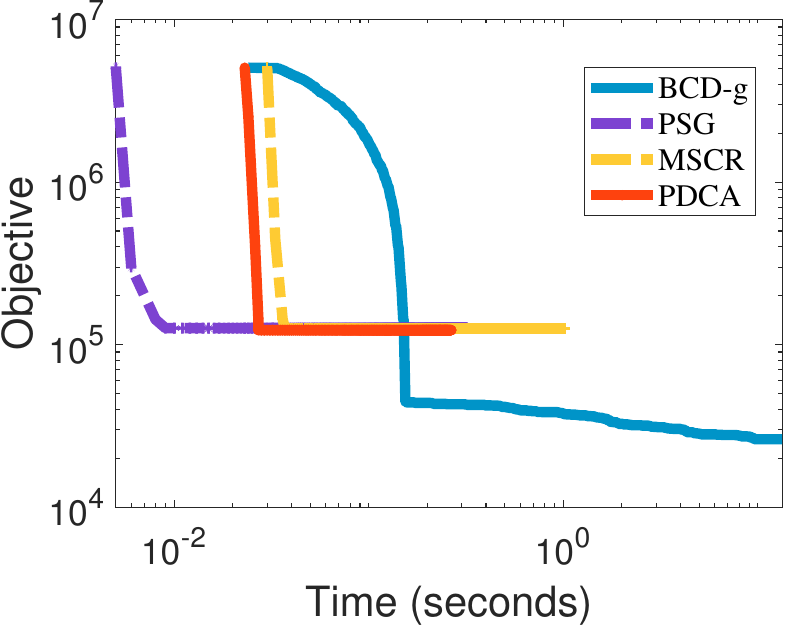}
			\caption{DCPB1 in randn-b}
		\end{subfigure}
		\begin{subfigure}{0.2\textwidth}
			\includegraphics[width=\textwidth]{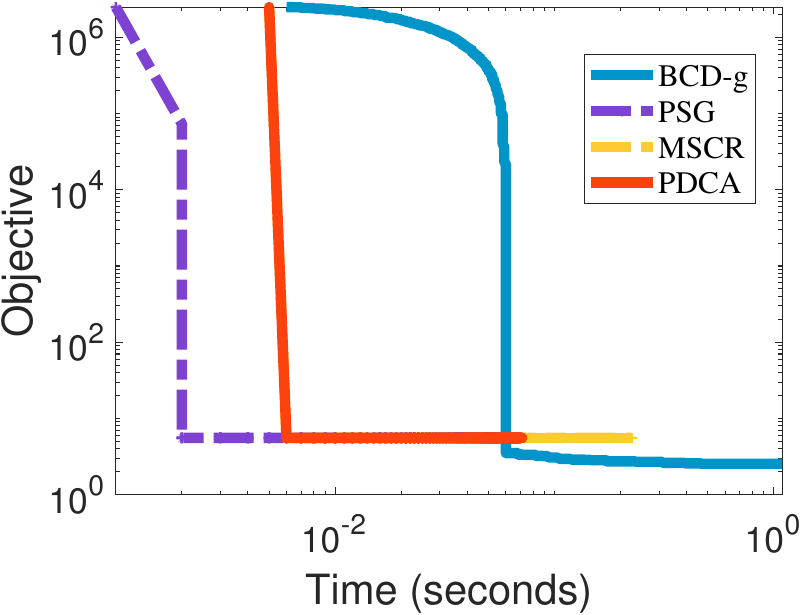}
			\caption{DCPB1 in TDT2-a}
		\end{subfigure}
		\begin{subfigure}{0.2\textwidth}
			\includegraphics[width=\textwidth]{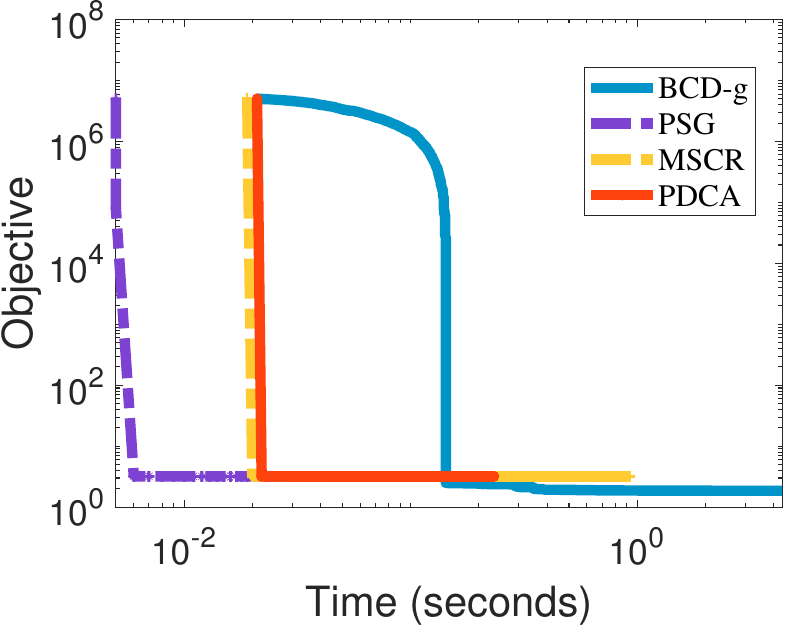}
			\caption{DCPB1 in TDT2-b}
		\end{subfigure}
		\begin{subfigure}{0.2\textwidth}
			\includegraphics[width=\textwidth]{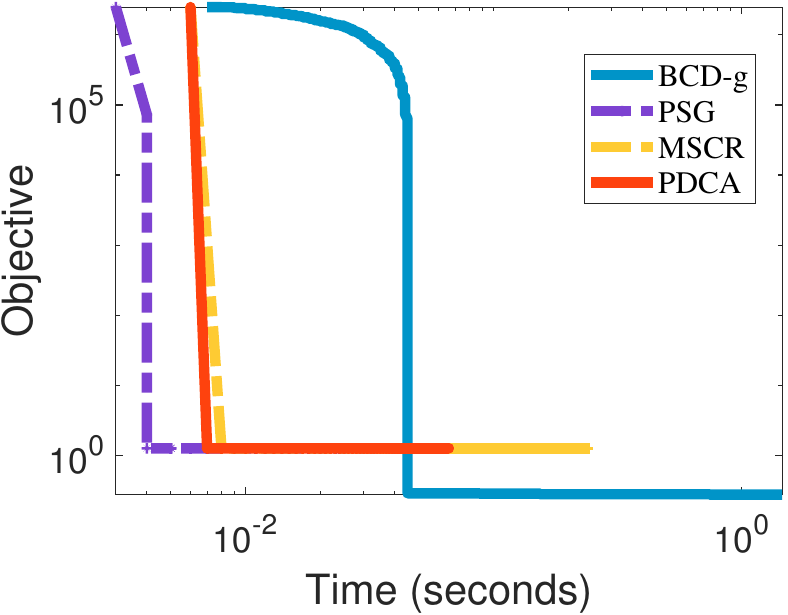}
			\caption{DCPB1 in 20News-a}
		\end{subfigure}
		\begin{subfigure}{0.2\textwidth}
			\includegraphics[width=\textwidth]{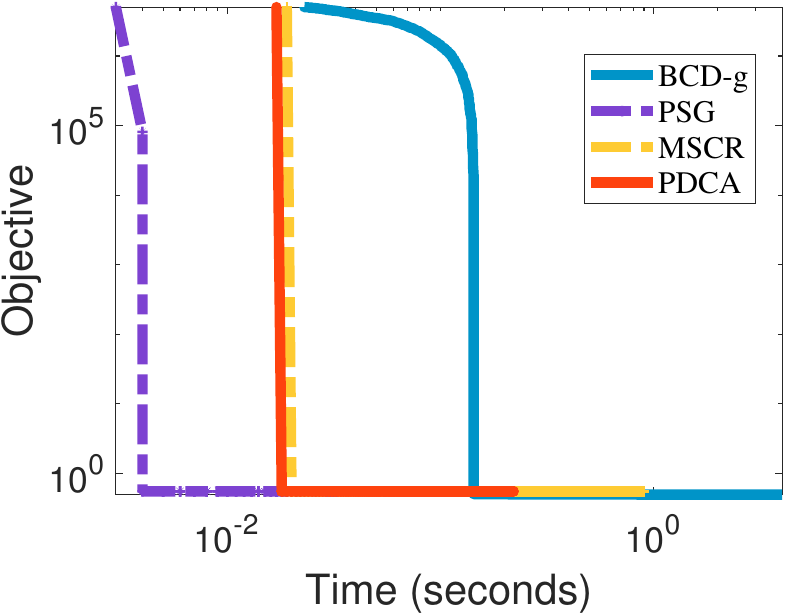}
			\caption{DCPB1 in 20News-b}
		\end{subfigure}
		\begin{subfigure}{0.2\textwidth}
			\includegraphics[width=\textwidth]{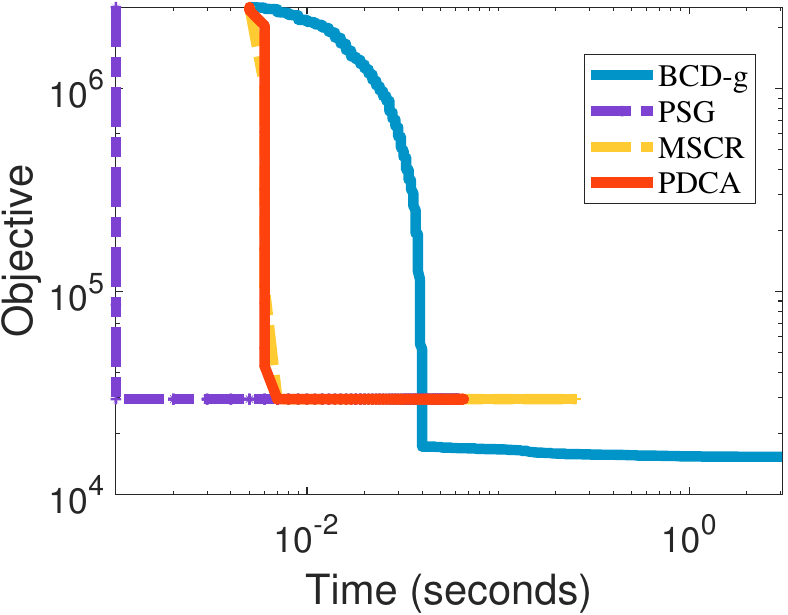}
			\caption{DCPB1 in Cifar-a}
		\end{subfigure}
		\begin{subfigure}{0.2\textwidth}
			\includegraphics[width=\textwidth]{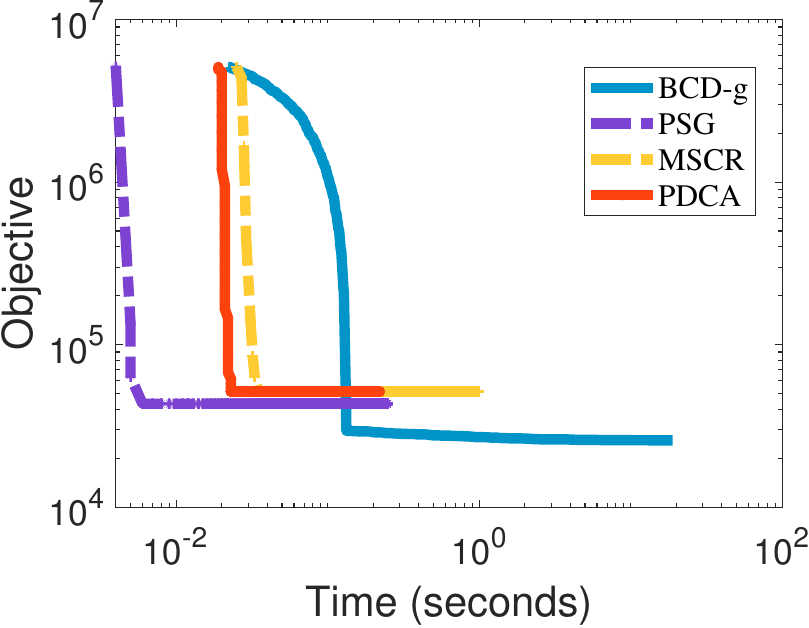}
			\caption{DCPB1 in Cifar-b}
		\end{subfigure}
		\begin{subfigure}{0.2\textwidth}
			\includegraphics[width=\textwidth]{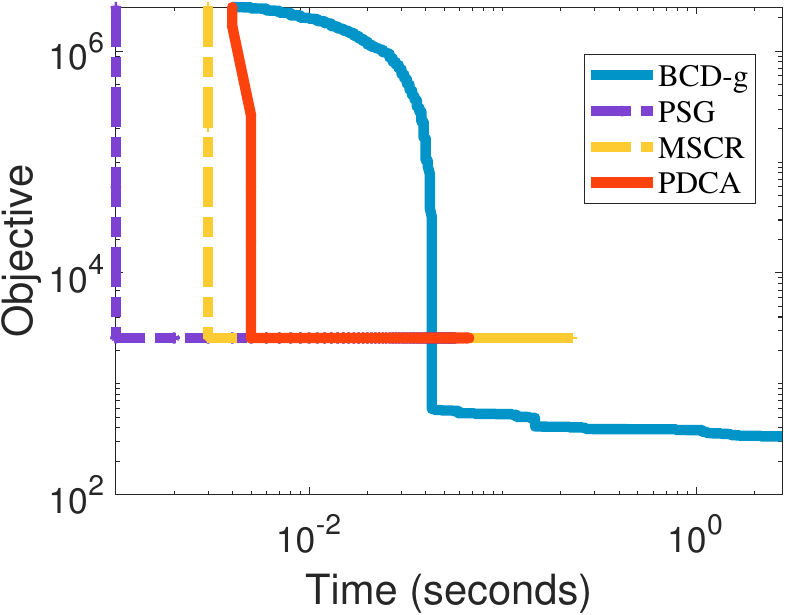}
			\caption{DCPB1 in MNIST-a}
		\end{subfigure}
		\begin{subfigure}{0.2\textwidth}
			\includegraphics[width=\textwidth]{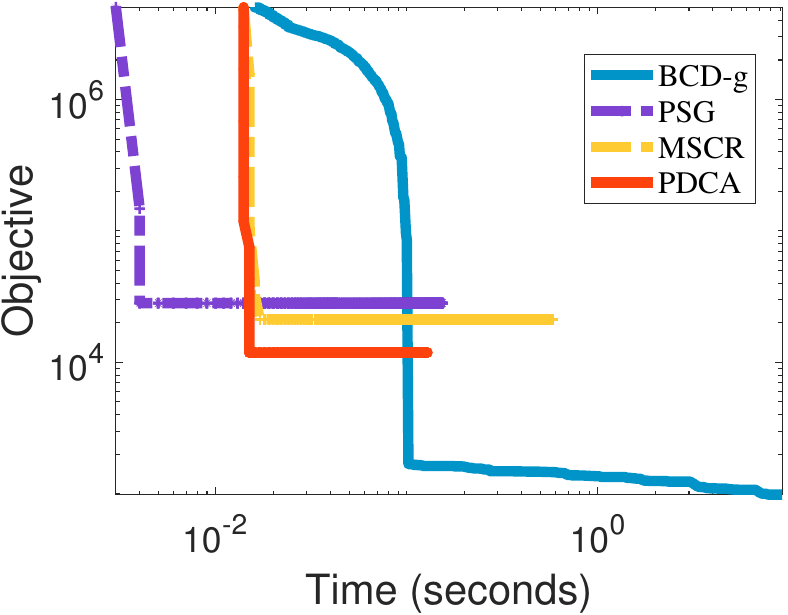}
			\caption{DCPB1 in MNIST-b}
		\end{subfigure}
		\begin{subfigure}{0.2\textwidth}
			\includegraphics[width=\textwidth]{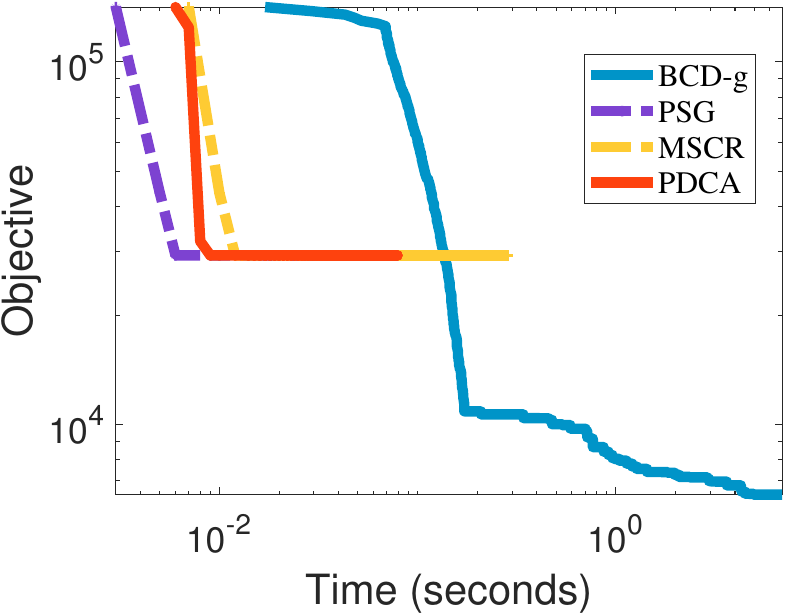}
			\caption{DCPB2 in randn-a}
		\end{subfigure}
		\begin{subfigure}{0.2\textwidth}
			\includegraphics[width=\textwidth]{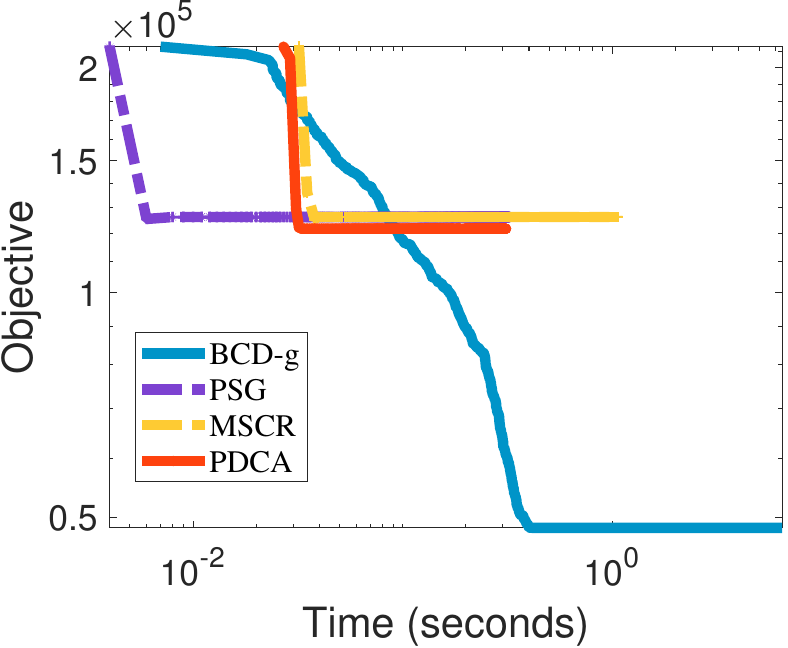}
			\caption{DCPB2 in randn-b}
		\end{subfigure}
		\begin{subfigure}{0.2\textwidth}
			\includegraphics[width=\textwidth]{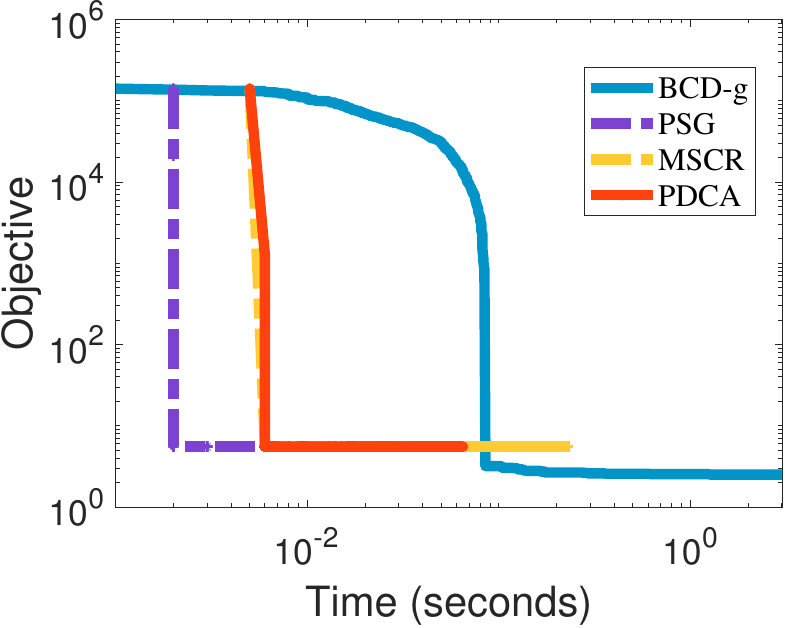}
			\caption{DCPB2 in TDT2-a}
		\end{subfigure}
		\begin{subfigure}{0.2\textwidth}
			\includegraphics[width=\textwidth]{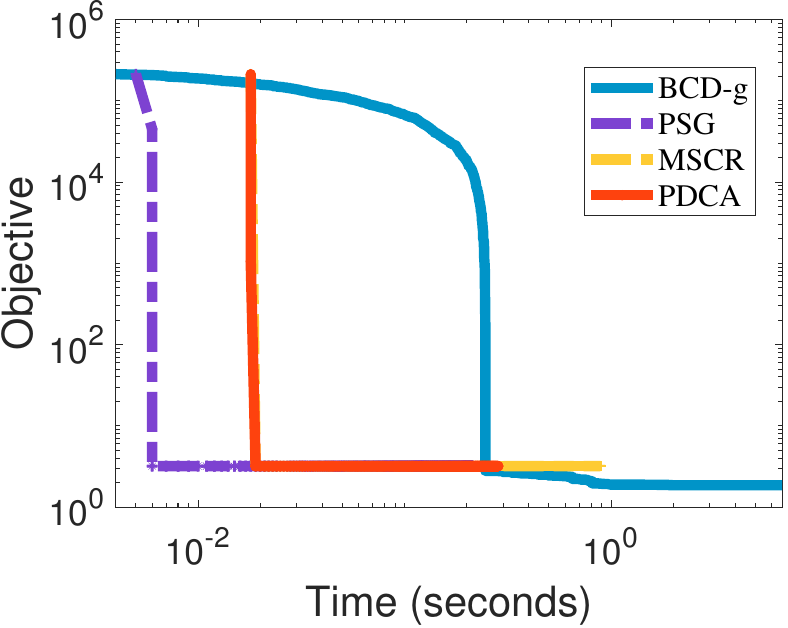}
			\caption{DCPB2 in TDT2-b}
		\end{subfigure}
		\begin{subfigure}{0.2\textwidth}
			\includegraphics[width=\textwidth]{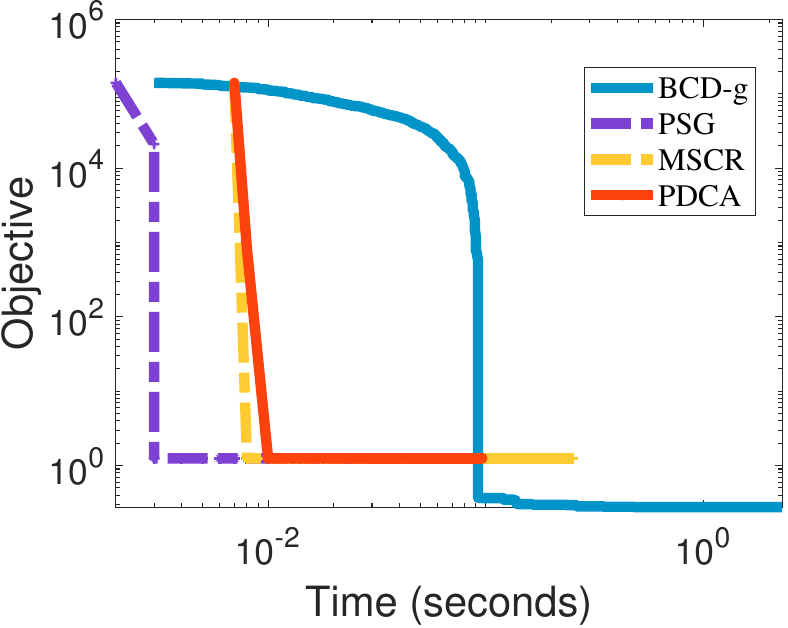}
			\caption{DCPB2 in 20News-a}
		\end{subfigure}
		\begin{subfigure}{0.2\textwidth}
			\includegraphics[width=\textwidth]{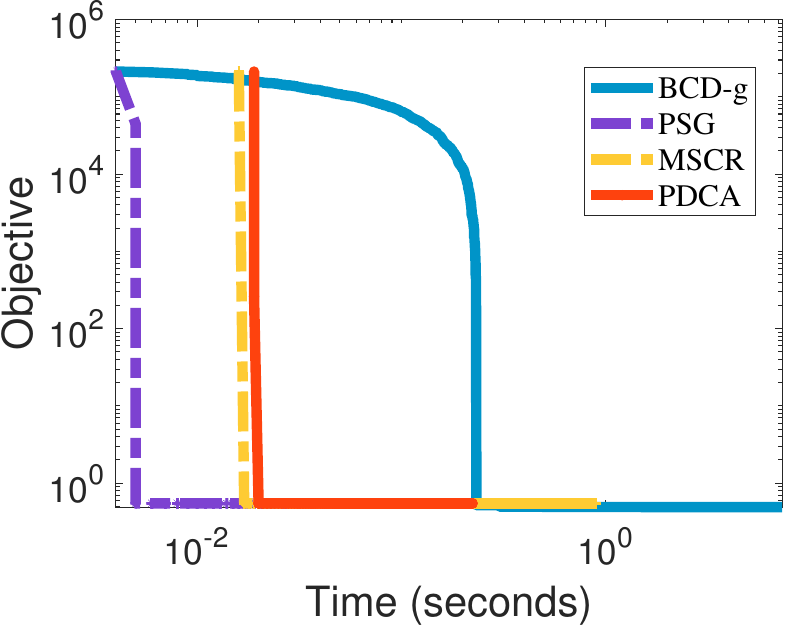}
			\caption{DCPB2 in 20News-b}
		\end{subfigure}
		\begin{subfigure}{0.2\textwidth}
			\includegraphics[width=\textwidth]{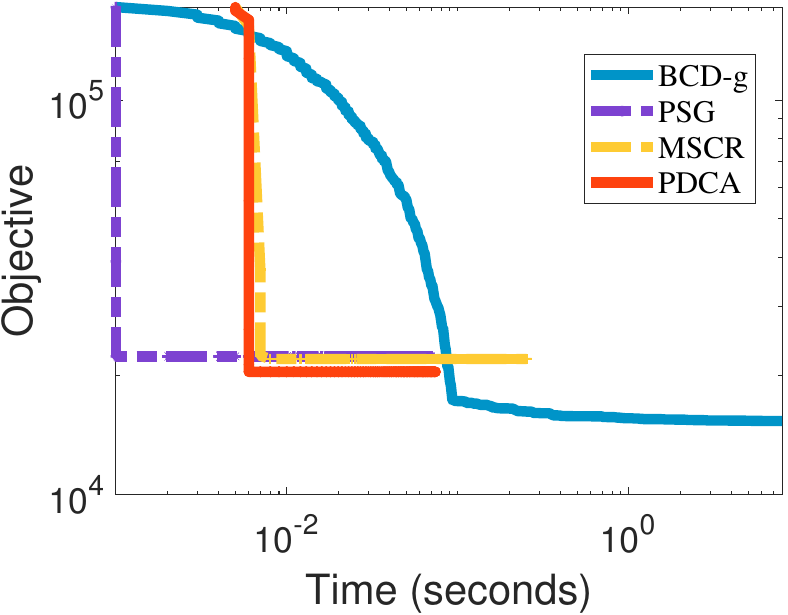}
			\caption{DCPB2 in Cifar-a}
		\end{subfigure}
		\begin{subfigure}{0.2\textwidth}
			\includegraphics[width=\textwidth]{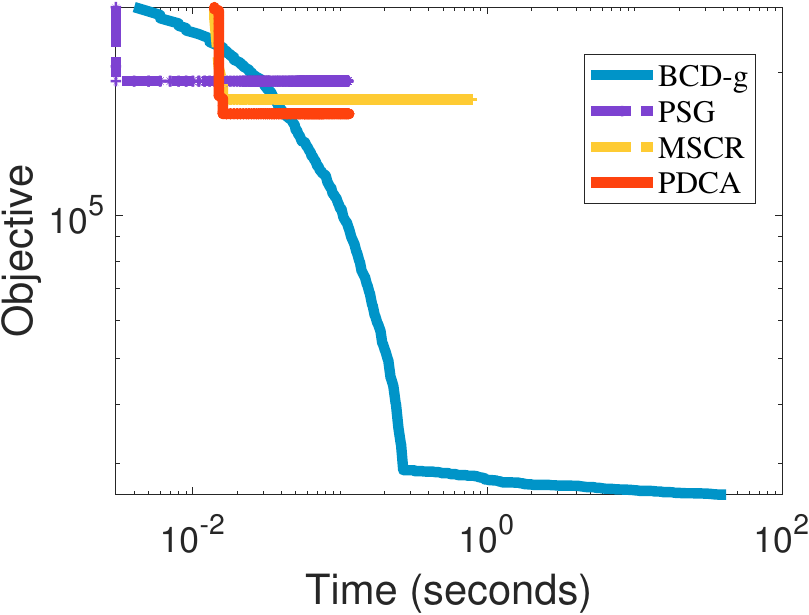}
			\caption{DCPB2 in Cifar-b}
		\end{subfigure}
		\begin{subfigure}{0.2\textwidth}
			\includegraphics[width=\textwidth]{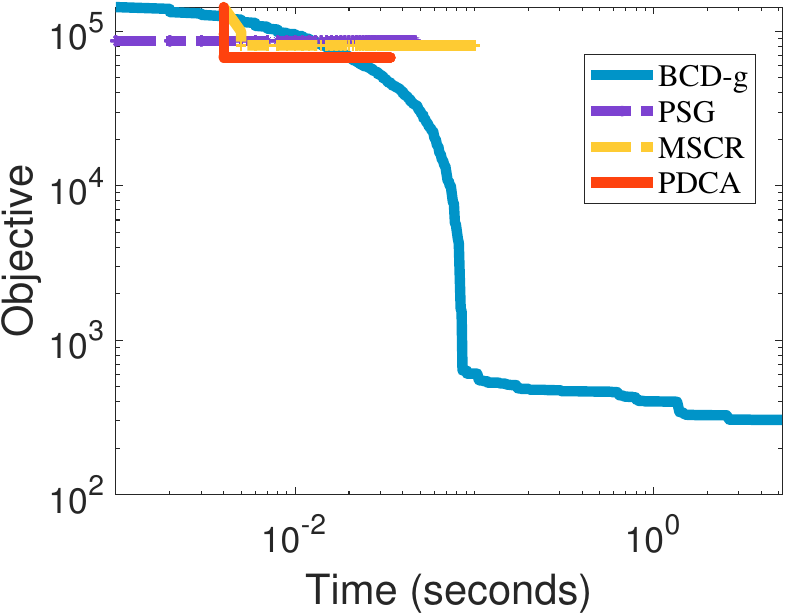}
			\caption{DCPB2 in MNIST-a}
		\end{subfigure}
		\begin{subfigure}{0.2\textwidth}
			\includegraphics[width=\textwidth]{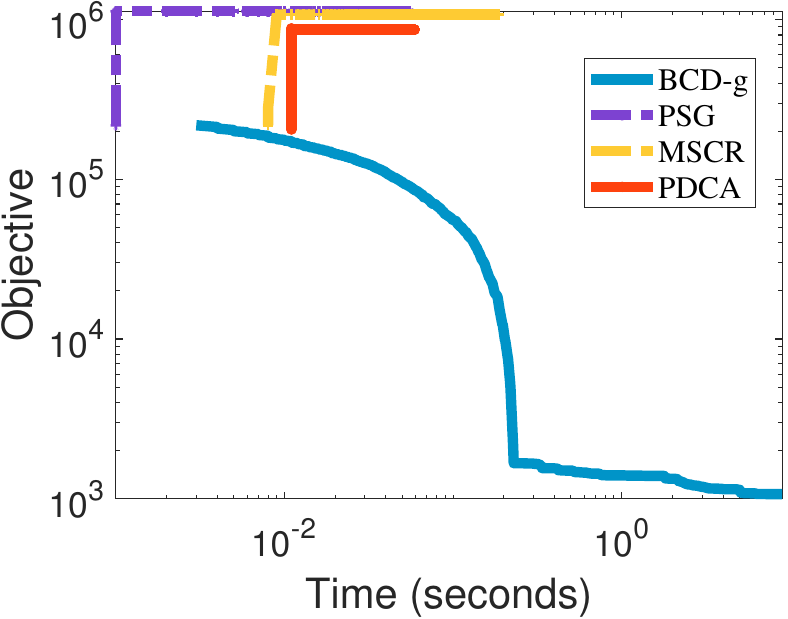}
			\caption{DCPB2 in MNIST-b}
		\end{subfigure}
		\caption{The convergence curve of the compared methods for solving two DC penalized binary optimization problems on 20 different datasets.}
		\label{fig:cputimebin}
	\end{figure}

	\begin{figure}[htbp]
		\centering
		\begin{subfigure}{0.2\textwidth}
			\includegraphics[width=\textwidth]{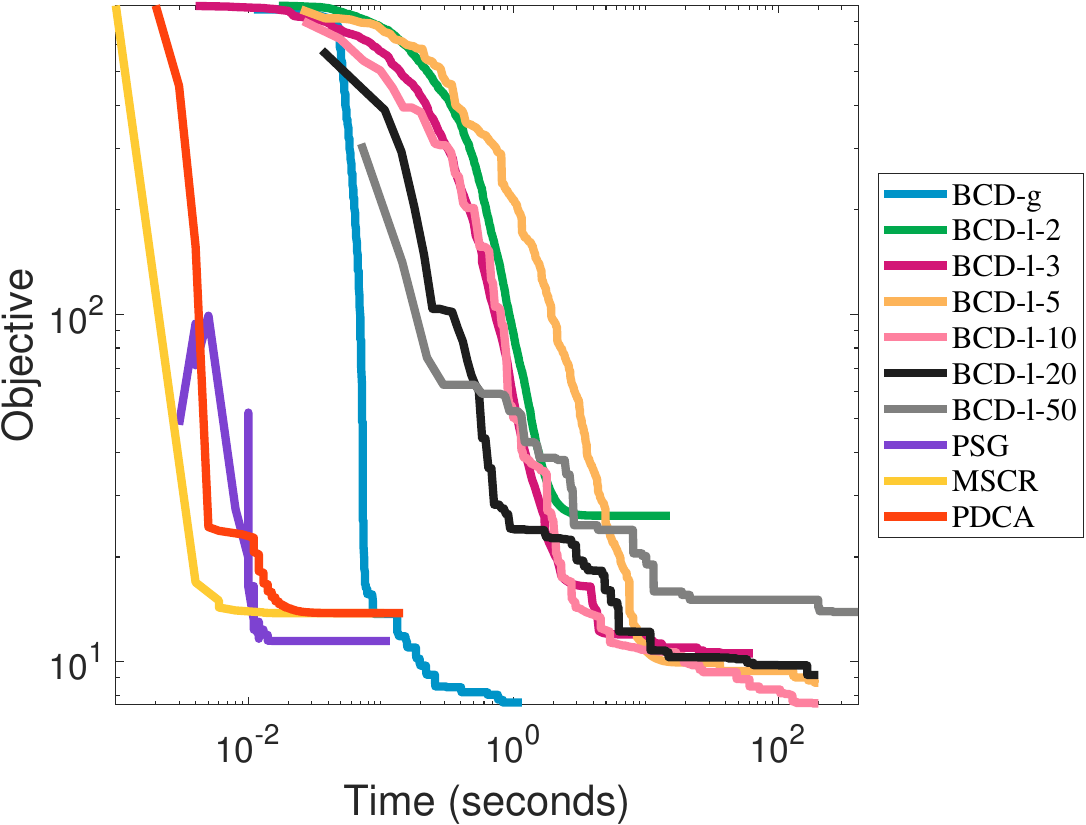}
			\caption{SIT in NAS-16}
		\end{subfigure}
		\begin{subfigure}{0.2\textwidth}
			\includegraphics[width=\textwidth]{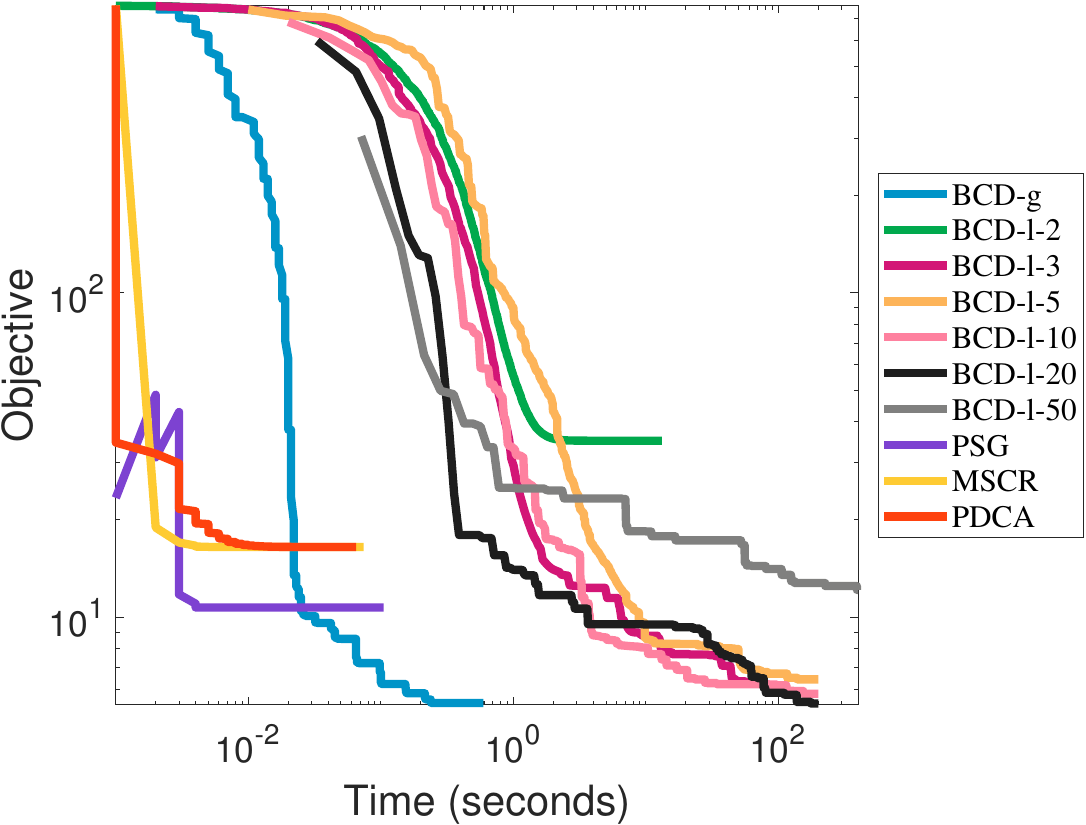}
			\caption{SIT in NAS-17}
		\end{subfigure}
		\begin{subfigure}{0.2\textwidth}
			\includegraphics[width=\textwidth]{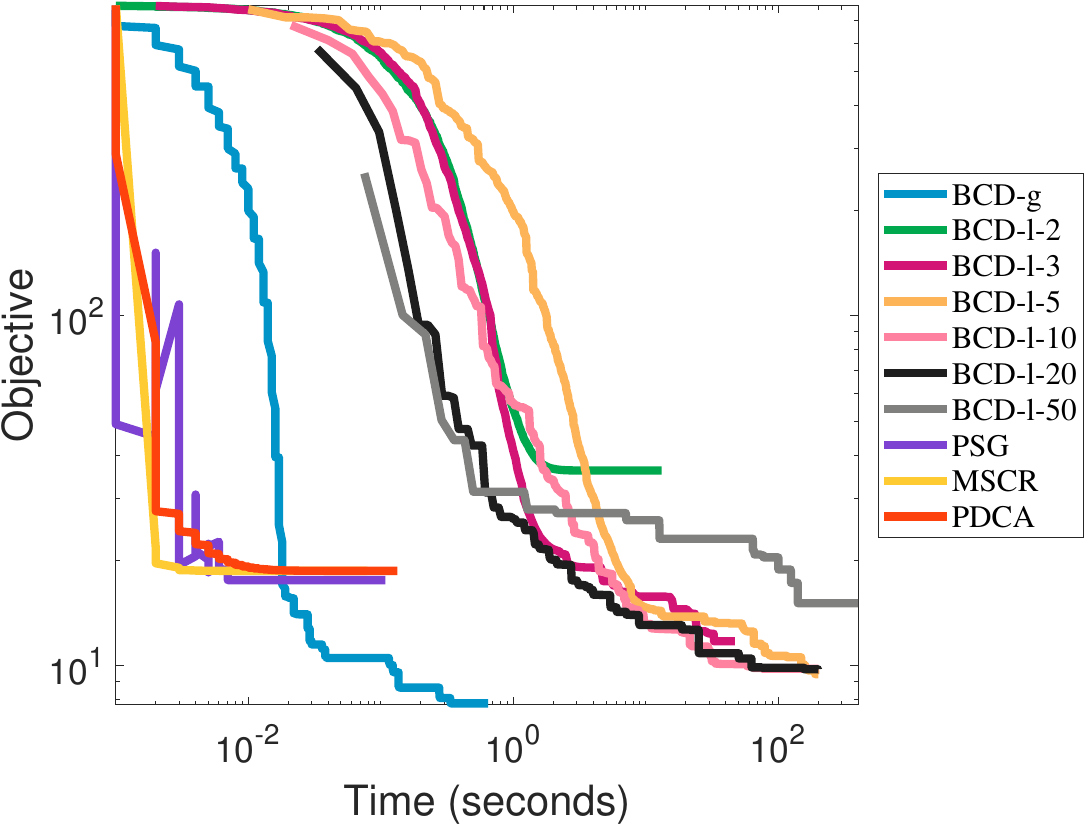}
			\caption{SIT in NAS-18}
		\end{subfigure}
		\begin{subfigure}{0.2\textwidth}
			\includegraphics[width=\textwidth]{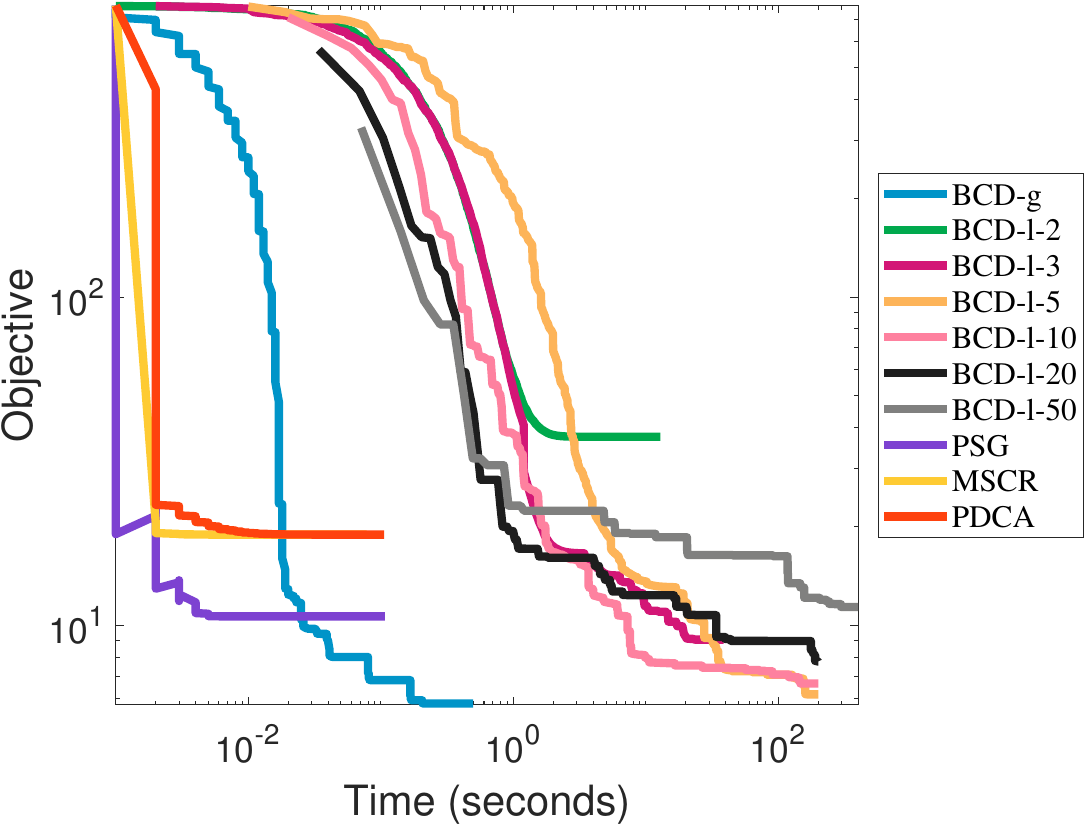}
			\caption{SIT in NAS-19}
		\end{subfigure}
		\begin{subfigure}{0.2\textwidth}
			\includegraphics[width=\textwidth]{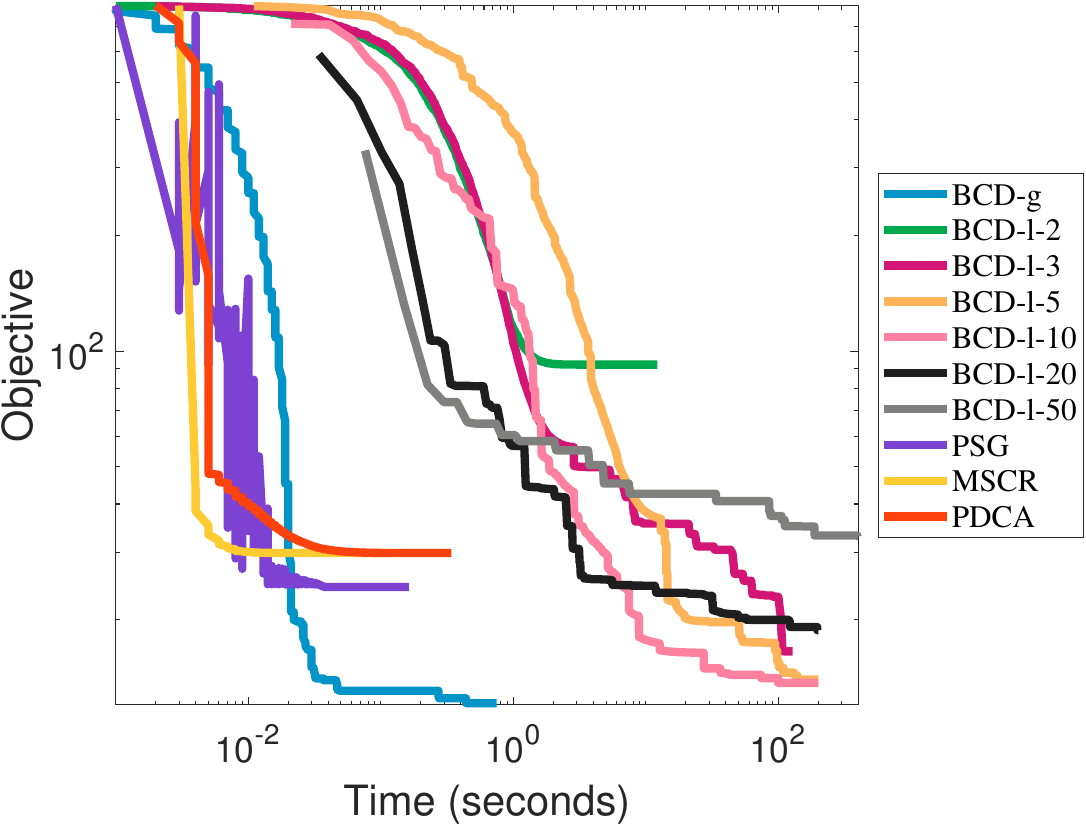}
			\caption{SIT in NAS-20}
		\end{subfigure}
		\begin{subfigure}{0.2\textwidth}
			\includegraphics[width=\textwidth]{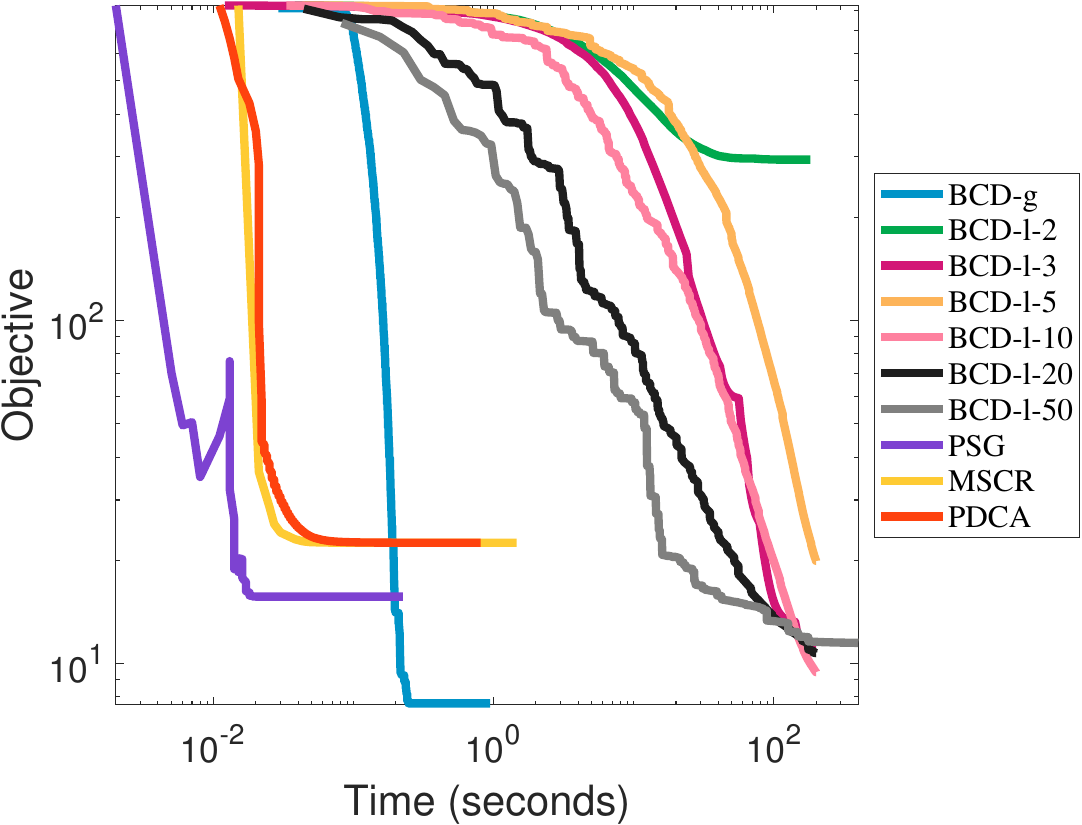}
			\caption{SIT in SP-16}
		\end{subfigure}
		\begin{subfigure}{0.2\textwidth}
			\includegraphics[width=\textwidth]{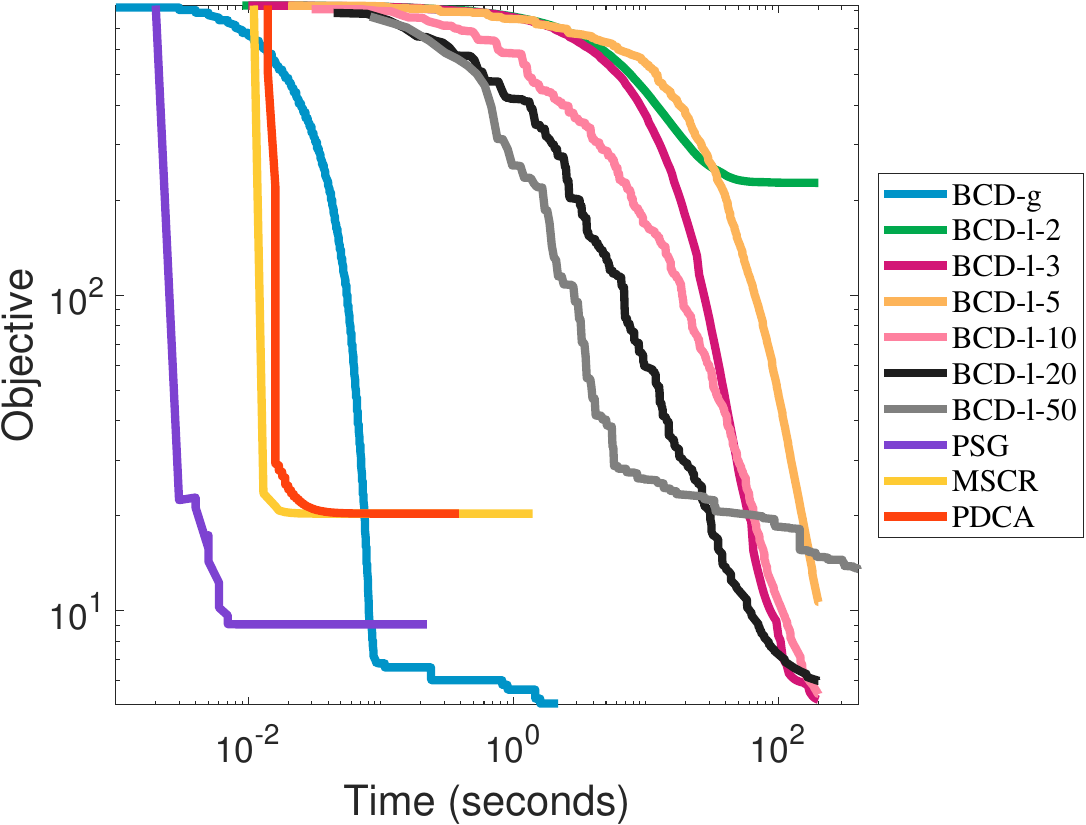}
			\caption{SIT in SP-17}
		\end{subfigure}
		\begin{subfigure}{0.2\textwidth}
			\includegraphics[width=\textwidth]{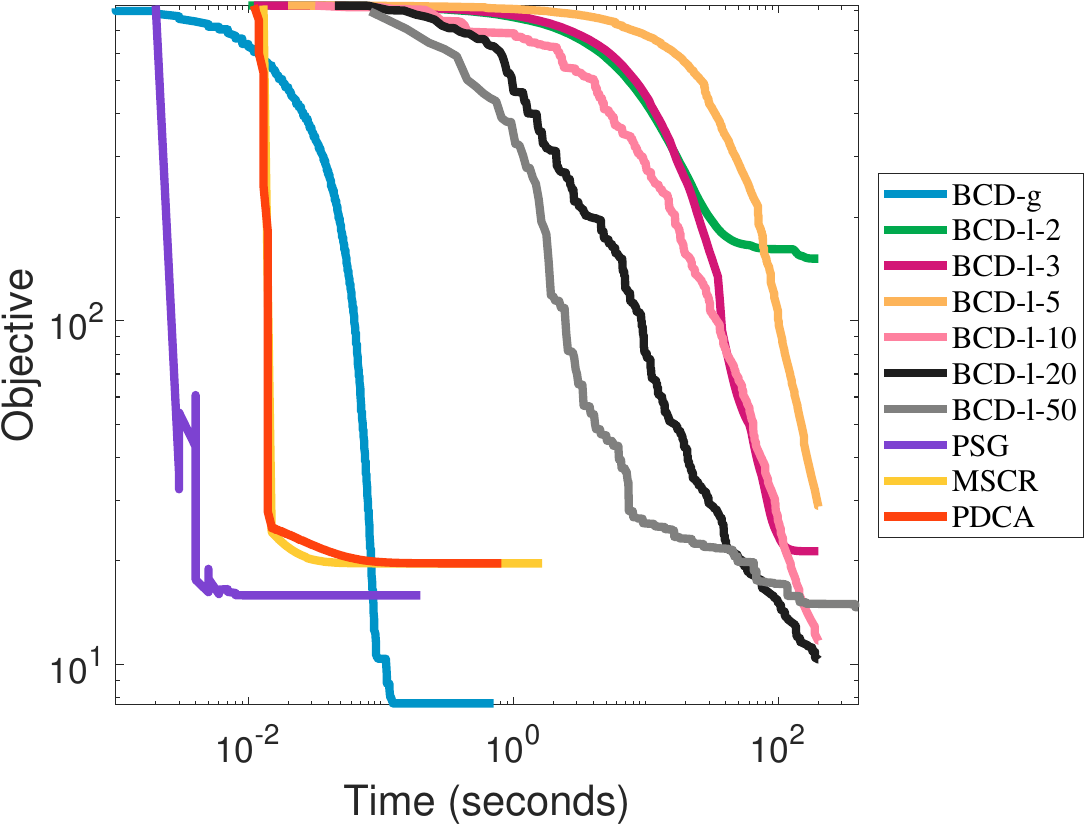}
			\caption{SIT in SP-18}
		\end{subfigure}
		\begin{subfigure}{0.2\textwidth}
			\includegraphics[width=\textwidth]{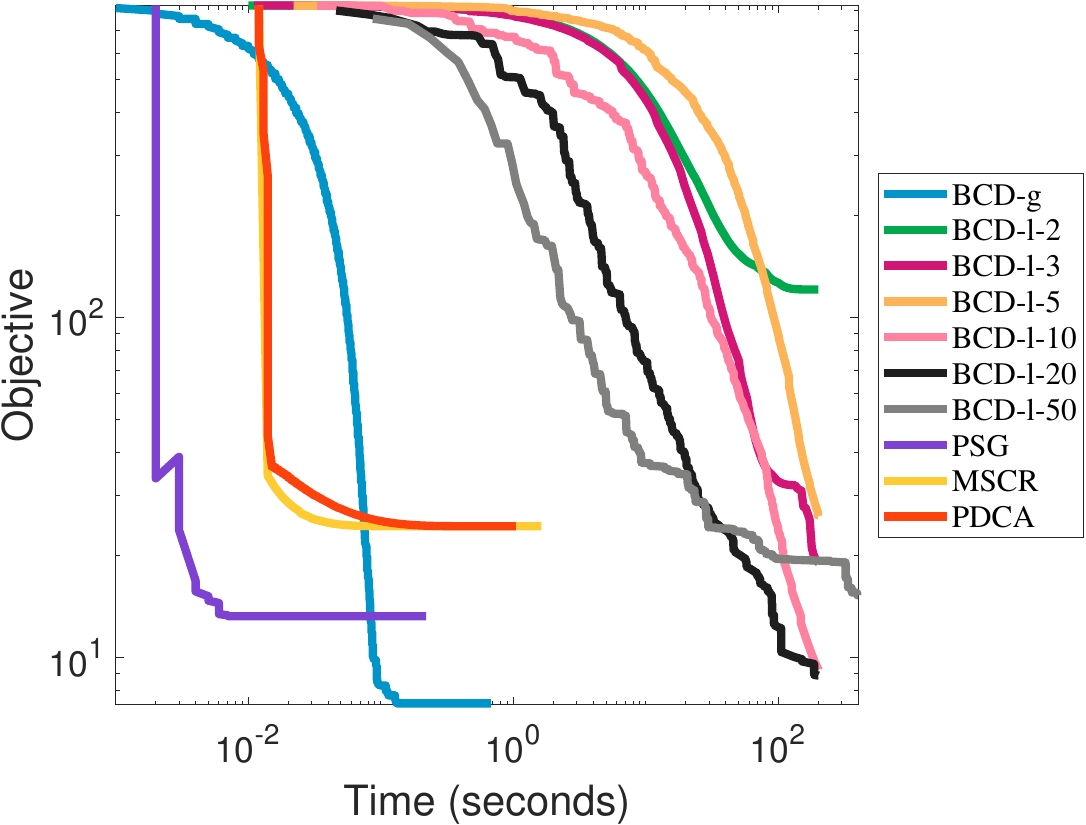}
			\caption{SIT in SP-19}
		\end{subfigure}
		\begin{subfigure}{0.2\textwidth}
			\includegraphics[width=\textwidth]{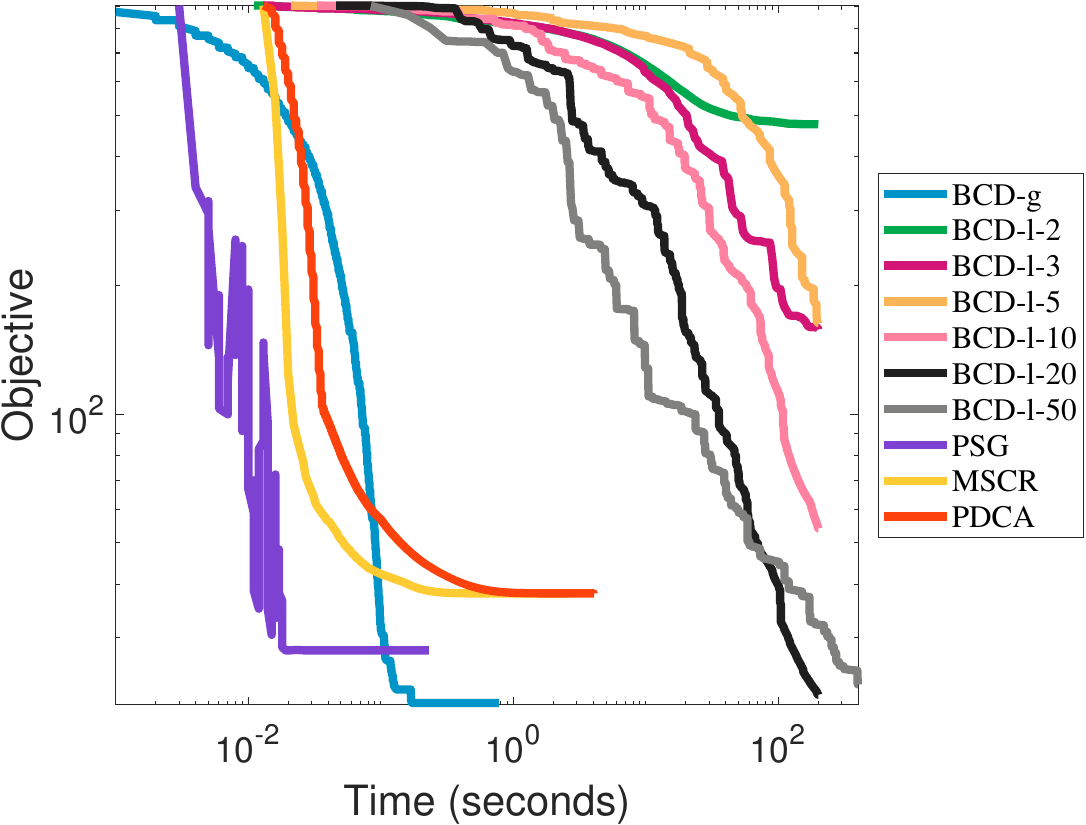}
			\caption{SIT in SP-20}
		\end{subfigure}
		\caption{The convergence curve of the compared methods for solving index tracking problem with varying $k$.}
		\label{fig:cputimeblock}
	\end{figure}

	\begin{figure}[htbp]
		\centering
		\begin{subfigure}{0.2\textwidth}
			\includegraphics[width=\textwidth]{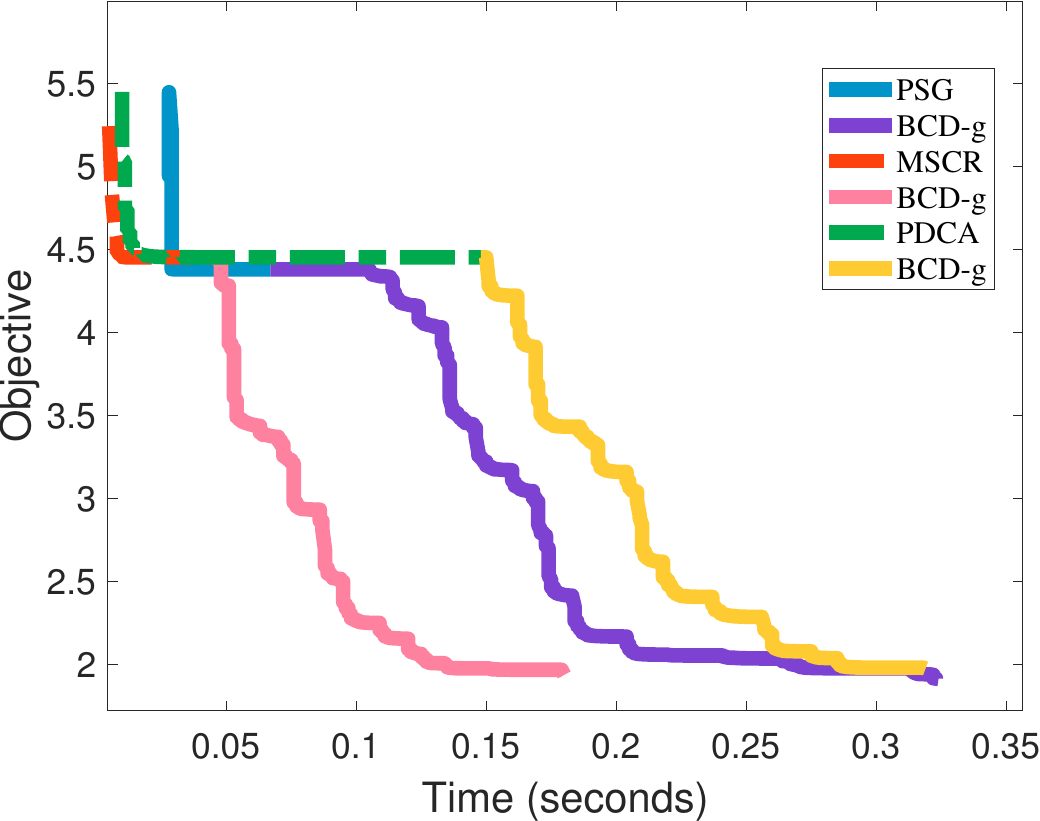}
			\caption{SIT in NAS-16}
		\end{subfigure}
		\begin{subfigure}{0.2\textwidth}
			\includegraphics[width=\textwidth]{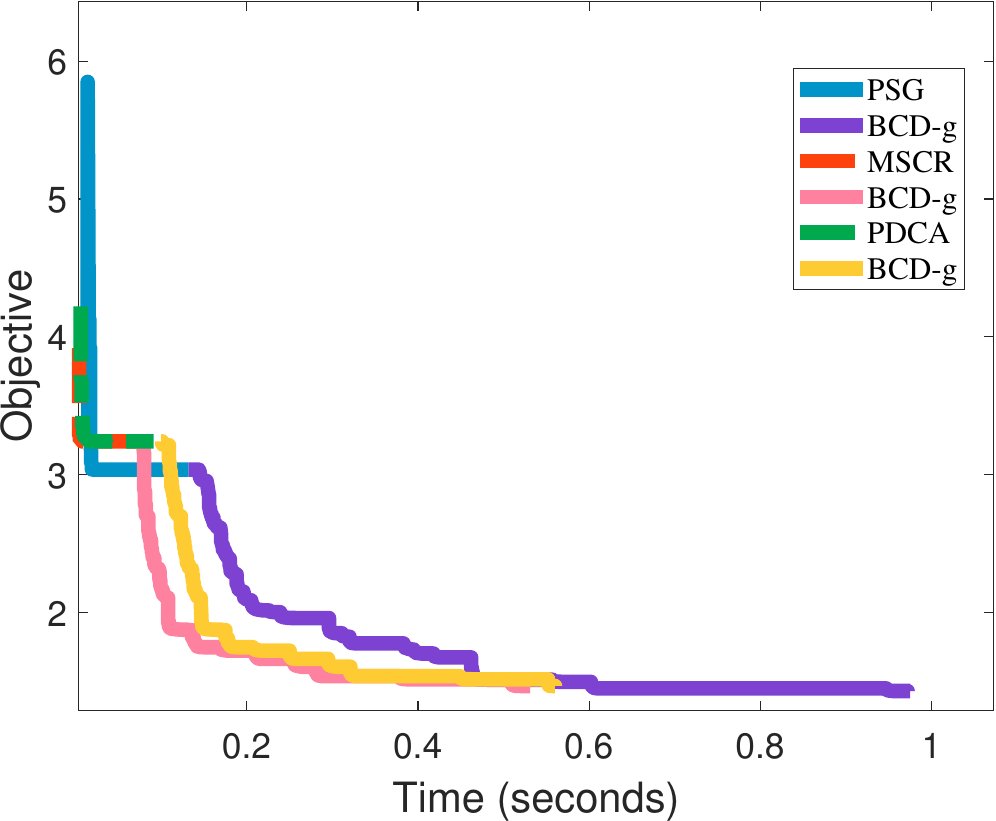}
			\caption{SIT in NAS-17}
		\end{subfigure}
		\begin{subfigure}{0.2\textwidth}
			\includegraphics[width=\textwidth]{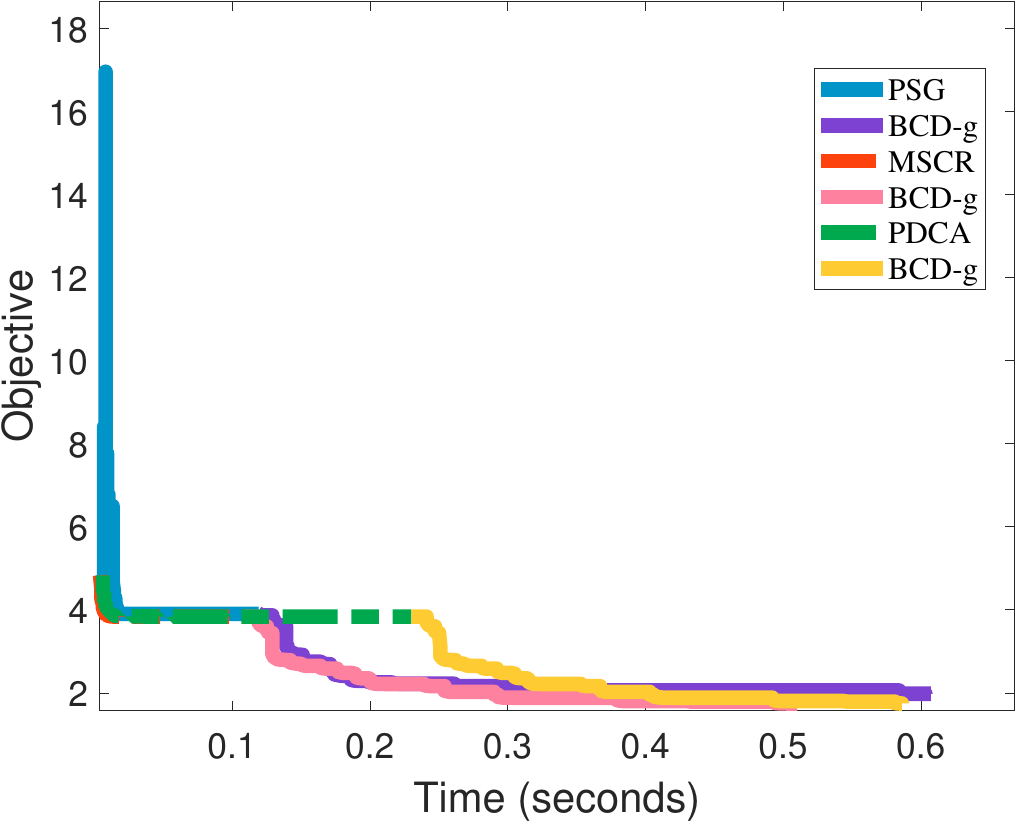}
			\caption{SIT in NAS-18}
		\end{subfigure}
		\begin{subfigure}{0.2\textwidth}
			\includegraphics[width=\textwidth]{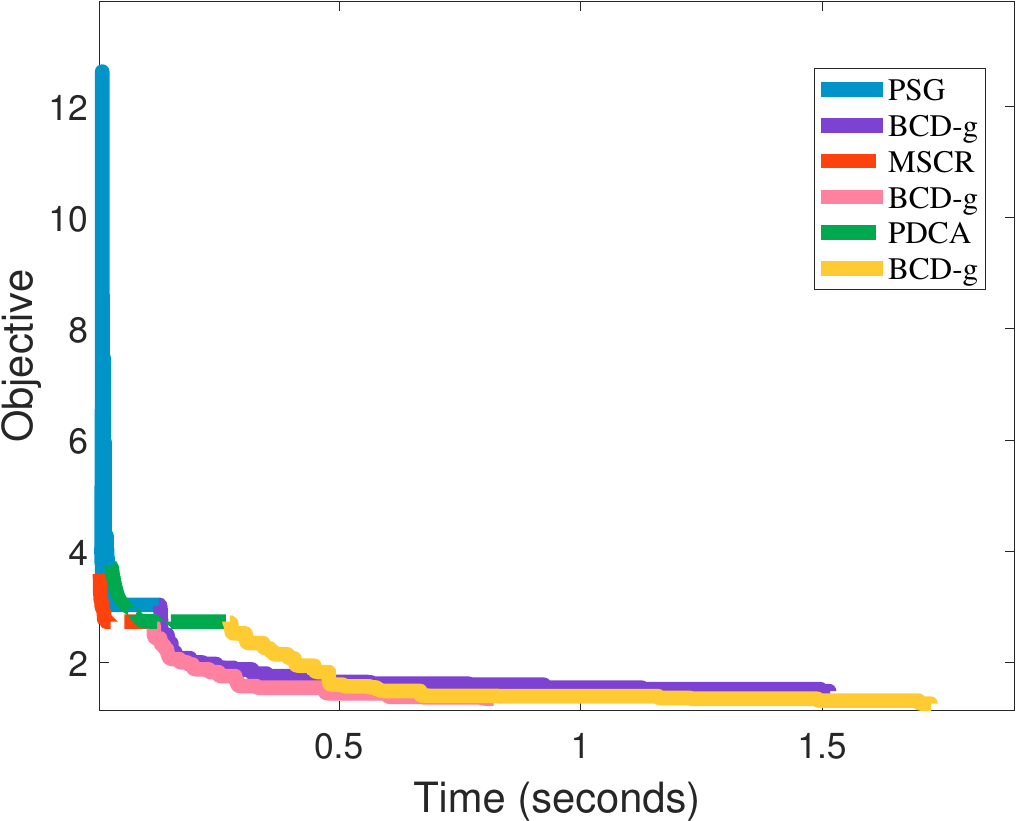}
			\caption{SIT in NAS-19}
		\end{subfigure}
		\begin{subfigure}{0.2\textwidth}
			\includegraphics[width=\textwidth]{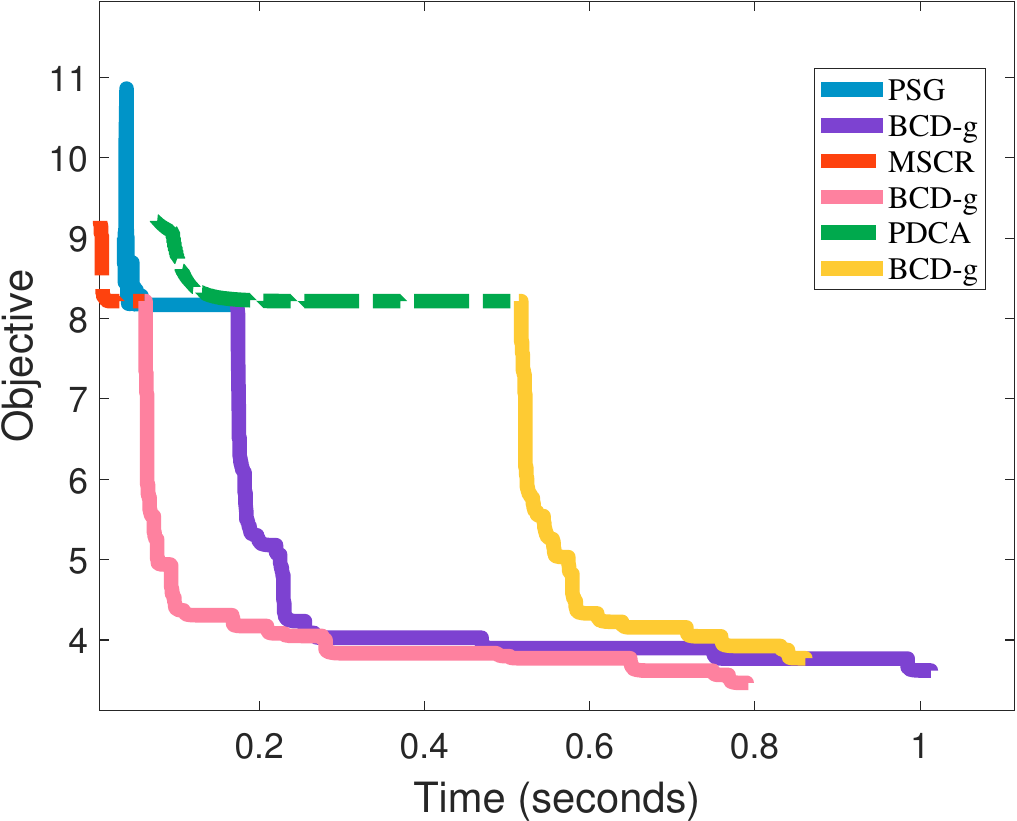}
			\caption{SIT in NAS-20}
		\end{subfigure}
		\begin{subfigure}{0.2\textwidth}
			\includegraphics[width=\textwidth]{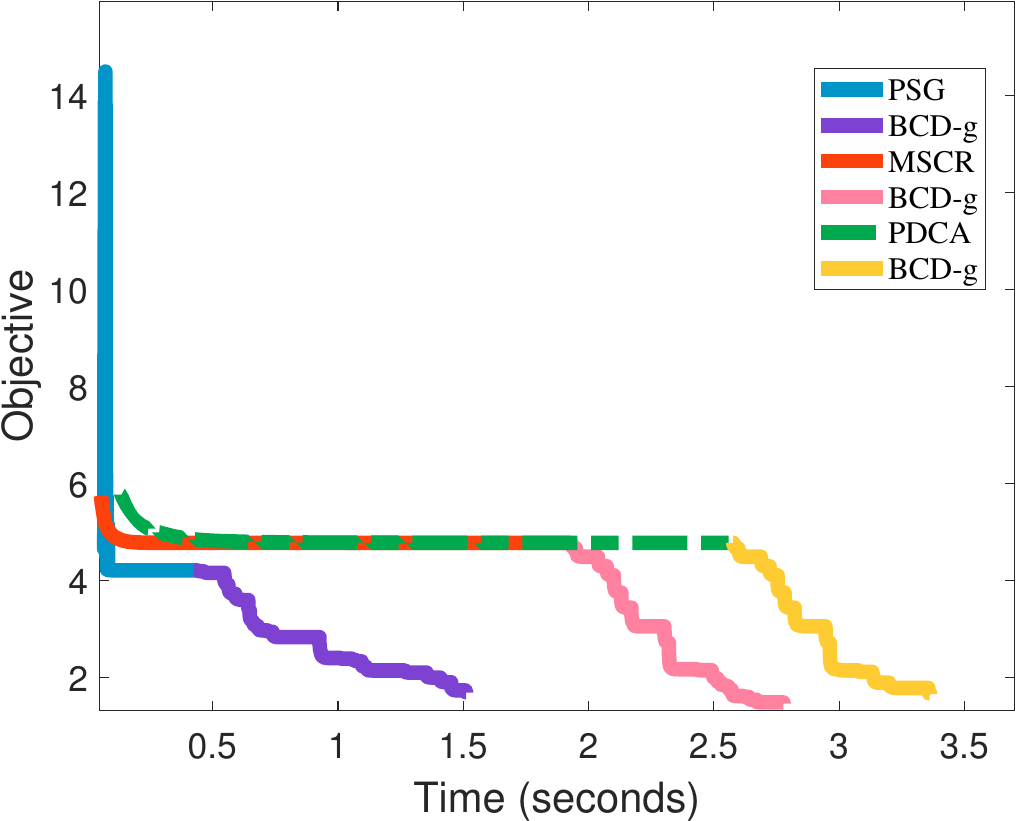}
			\caption{SIT in SP-16}
		\end{subfigure}
		\begin{subfigure}{0.2\textwidth}
			\includegraphics[width=\textwidth]{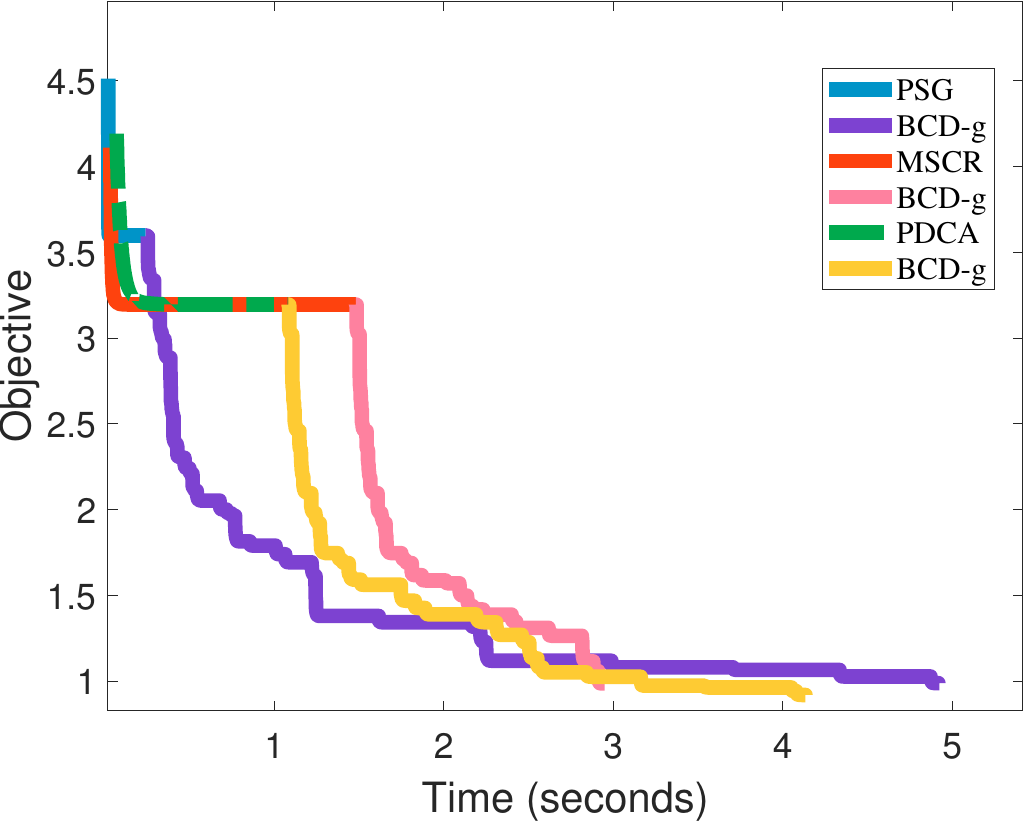}
			\caption{SIT in SP-17}
		\end{subfigure}
		\begin{subfigure}{0.2\textwidth}
			\includegraphics[width=\textwidth]{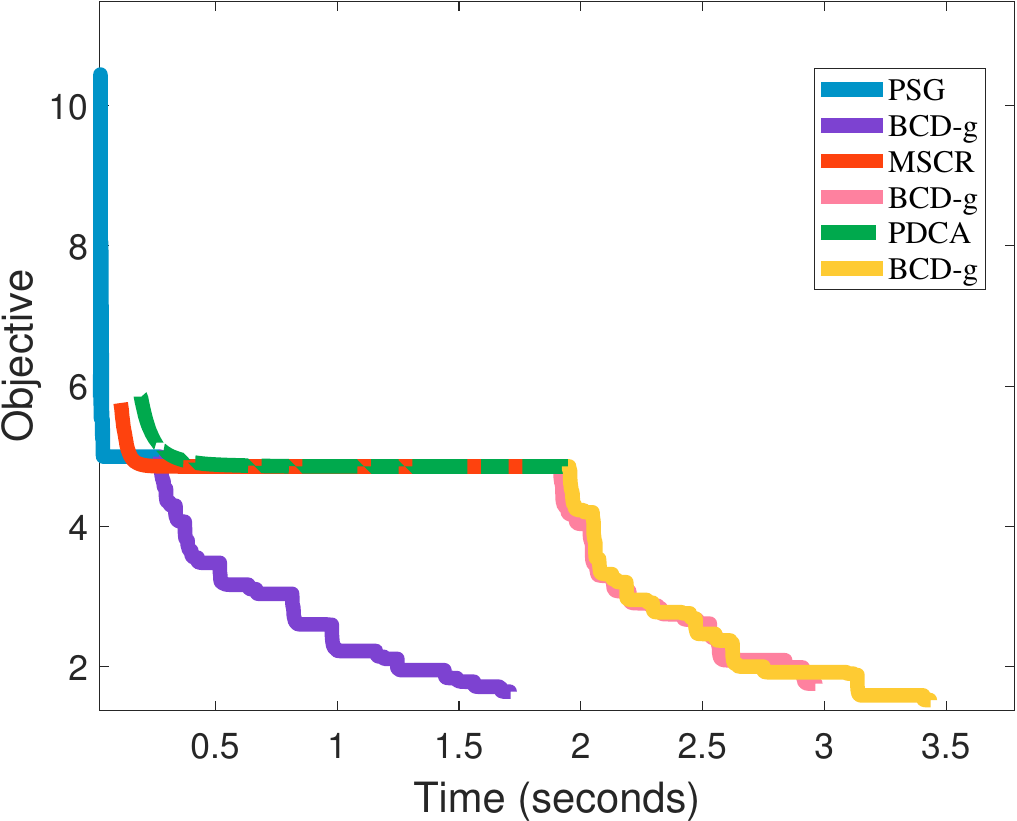}
			\caption{SIT in SP-18}
		\end{subfigure}
		\begin{subfigure}{0.2\textwidth}
			\includegraphics[width=\textwidth]{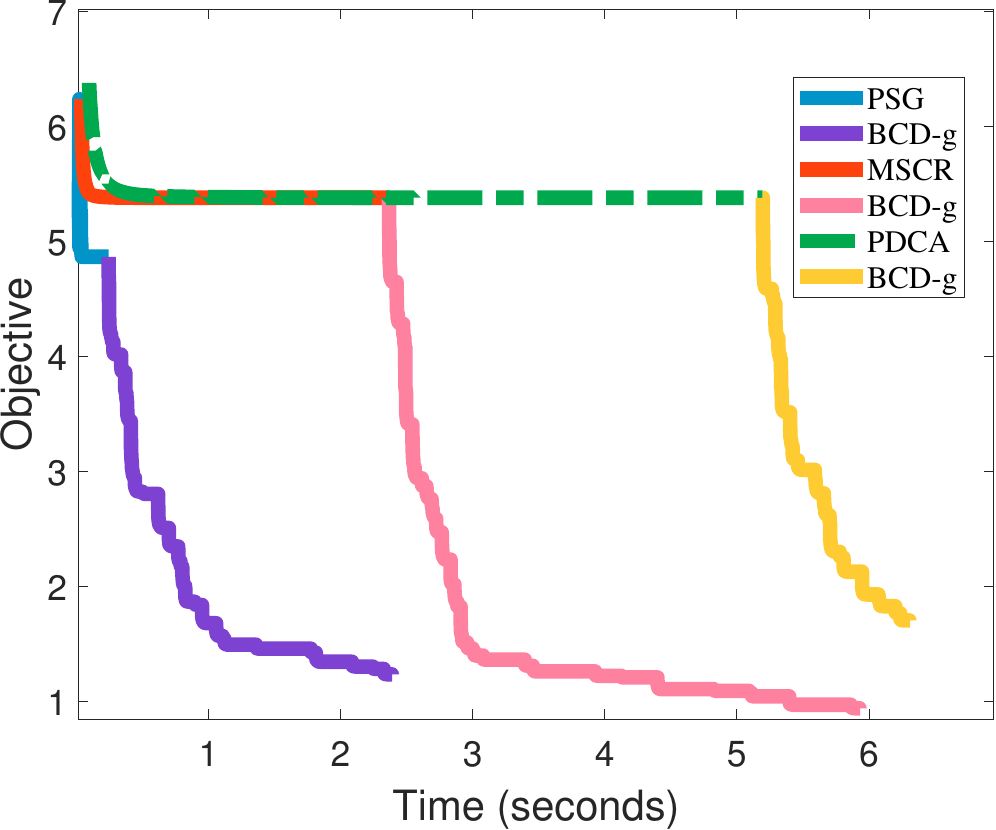}
			\caption{SIT in SP-19}
		\end{subfigure}
		\begin{subfigure}{0.2\textwidth}
			\includegraphics[width=\textwidth]{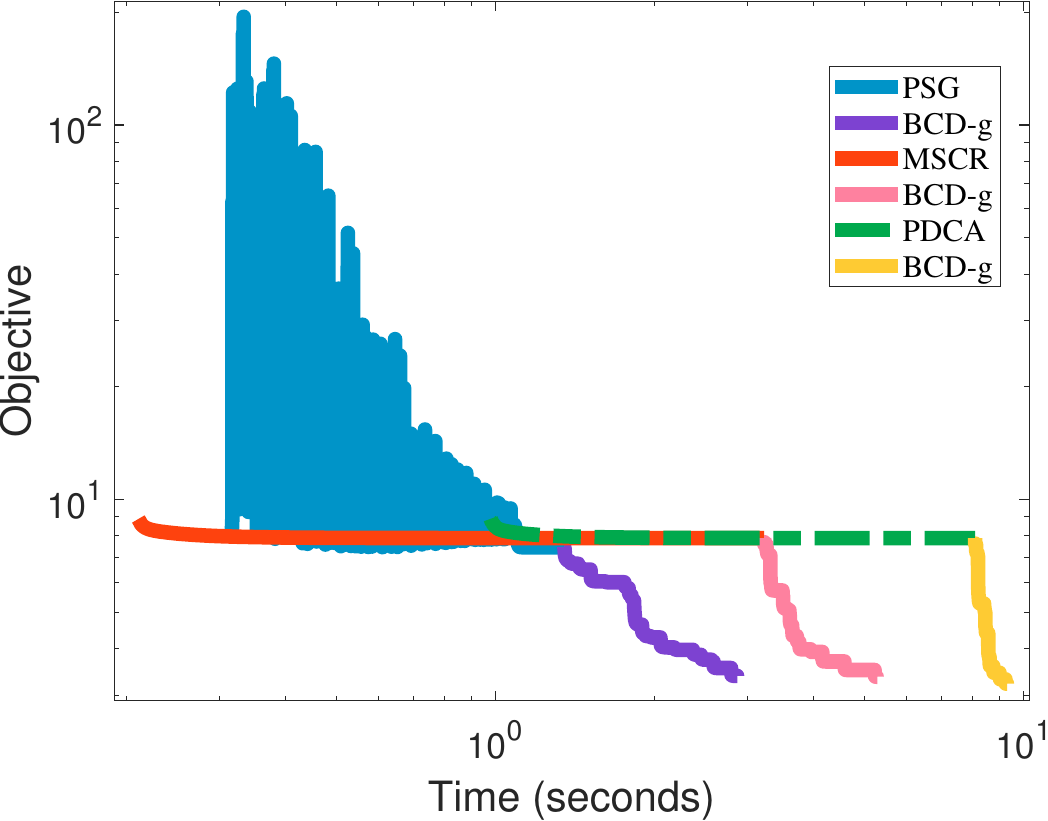}
			\caption{SIT in SP-20}
		\end{subfigure}
		\begin{subfigure}{0.2\textwidth}
			\includegraphics[width=\textwidth]{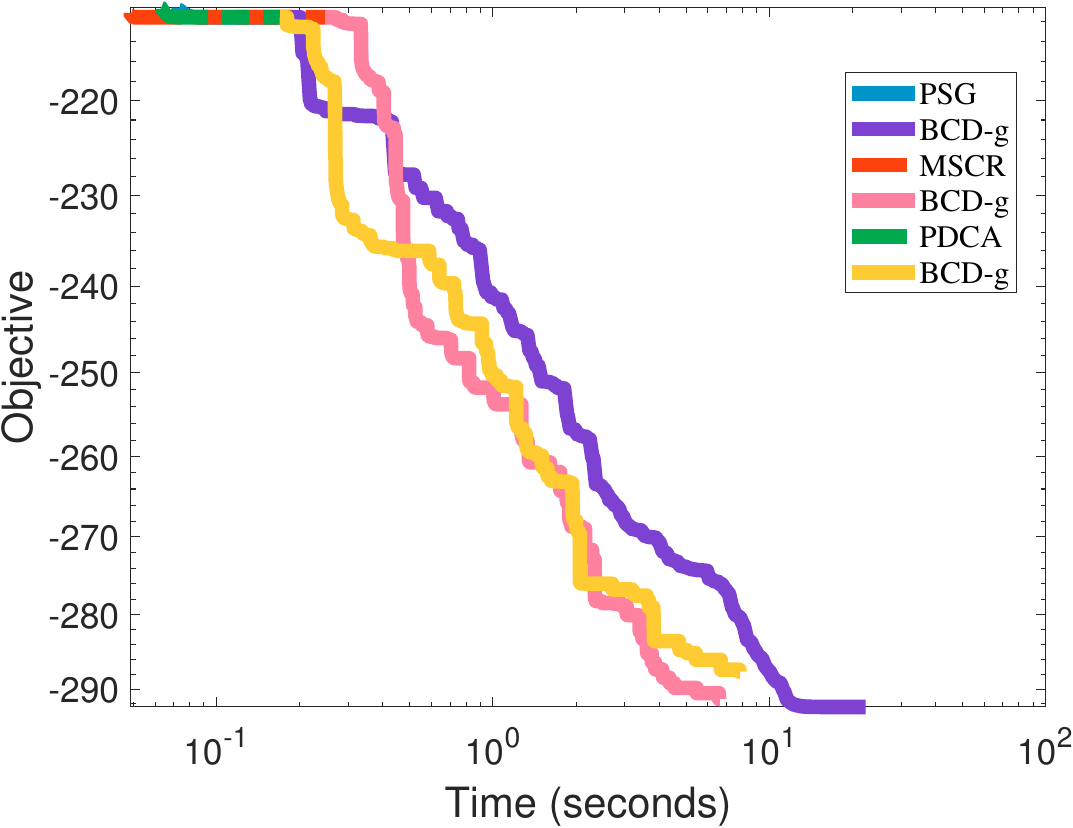}
			\caption{NNSPCA in randn-a}
		\end{subfigure}
		\begin{subfigure}{0.2\textwidth}
			\includegraphics[width=\textwidth]{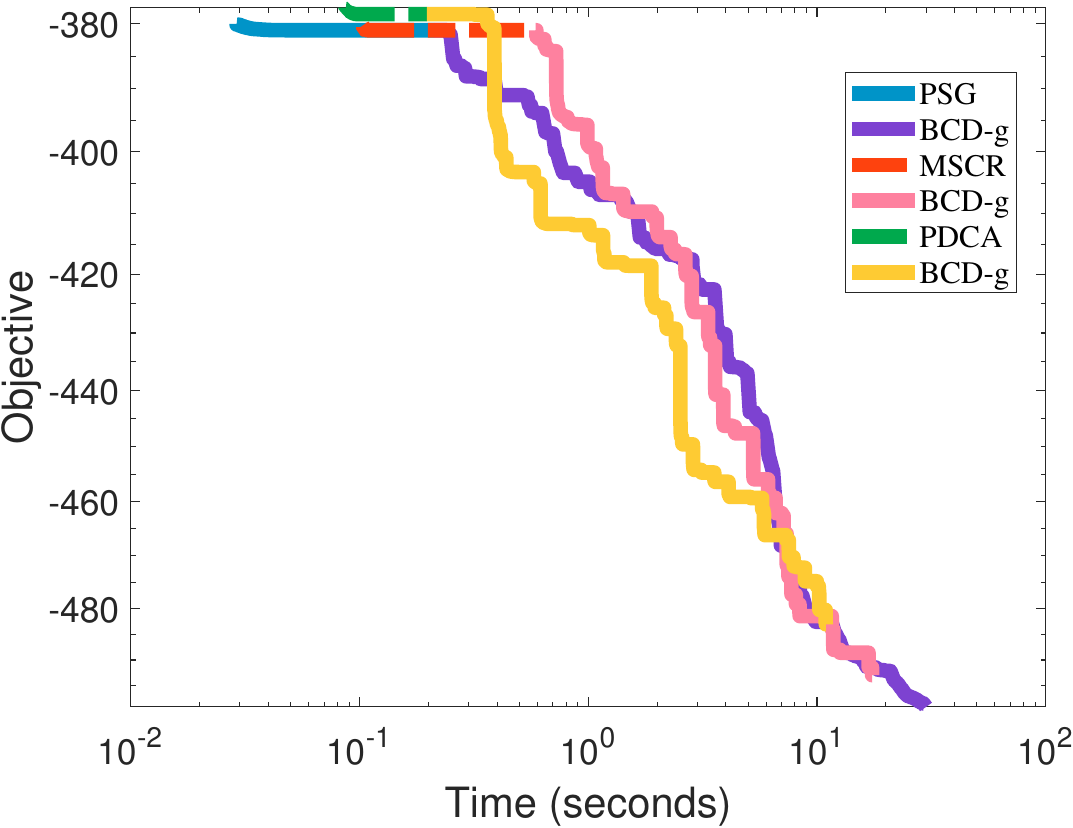}
			\caption{NNSPCA in randn-b}
		\end{subfigure}
		\begin{subfigure}{0.2\textwidth}
			\includegraphics[width=\textwidth]{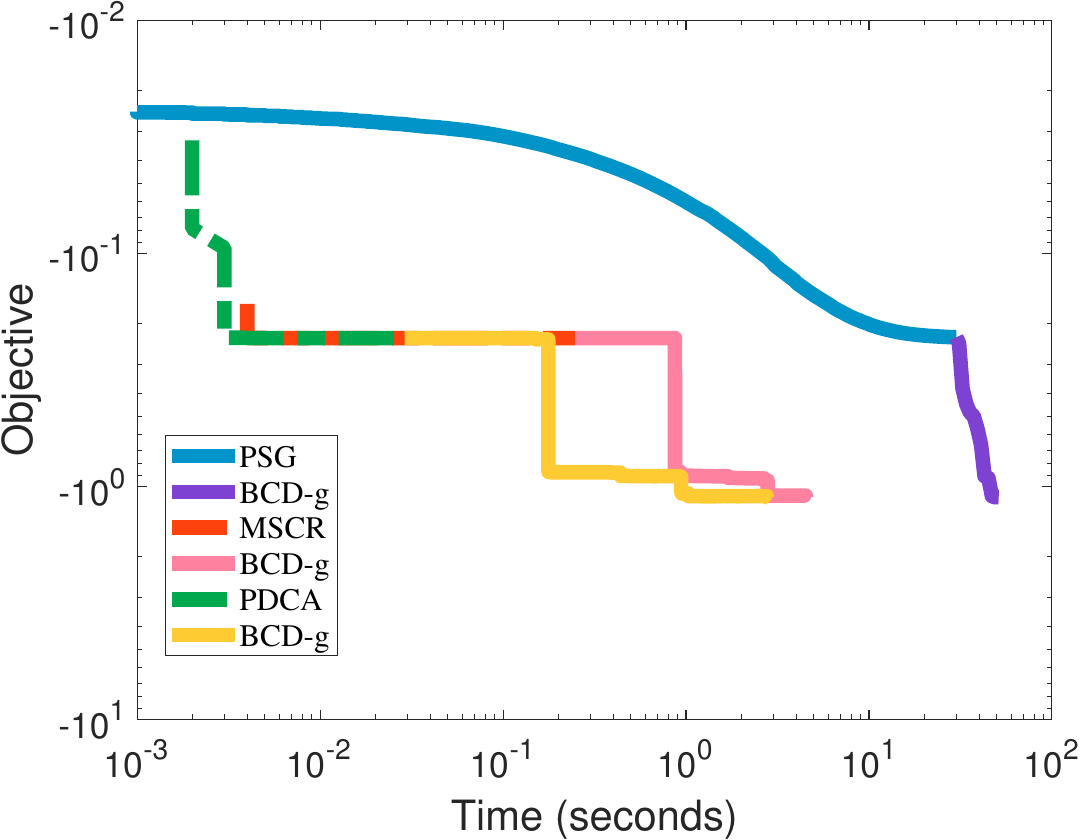}
			\caption{\tiny NNSPCA in TDT2-a}
		\end{subfigure}
		\begin{subfigure}{0.2\textwidth}
			\includegraphics[width=\textwidth]{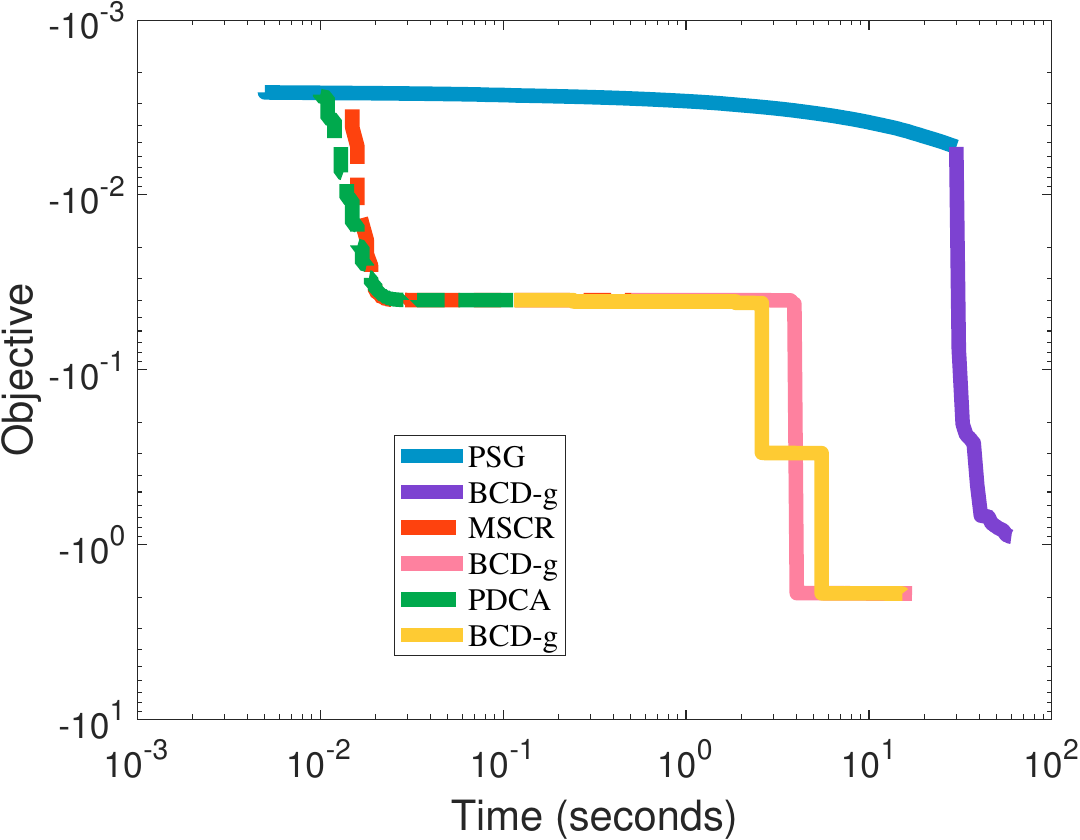}
			\caption{NNSPCA in TDT2-b}
		\end{subfigure}
		\begin{subfigure}{0.2\textwidth}
			\includegraphics[width=\textwidth]{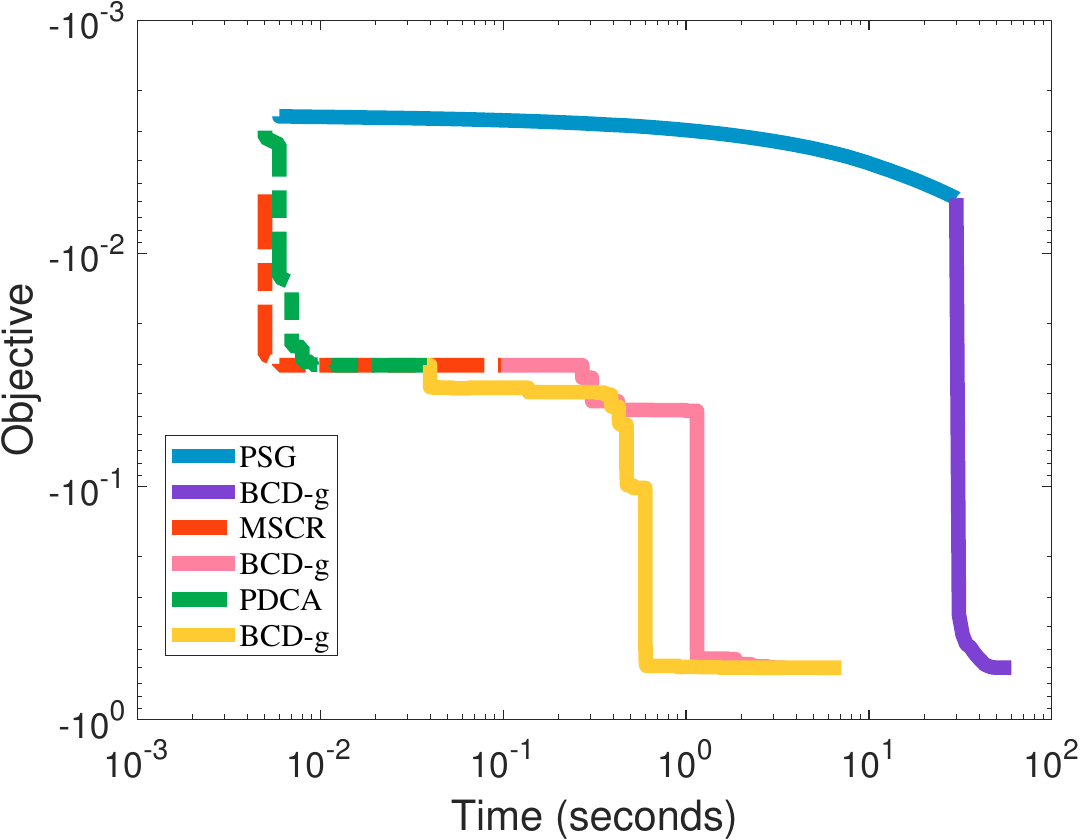}
			\caption{NNSPCA in 20News-a}
		\end{subfigure}
		\begin{subfigure}{0.2\textwidth}
			\includegraphics[width=\textwidth]{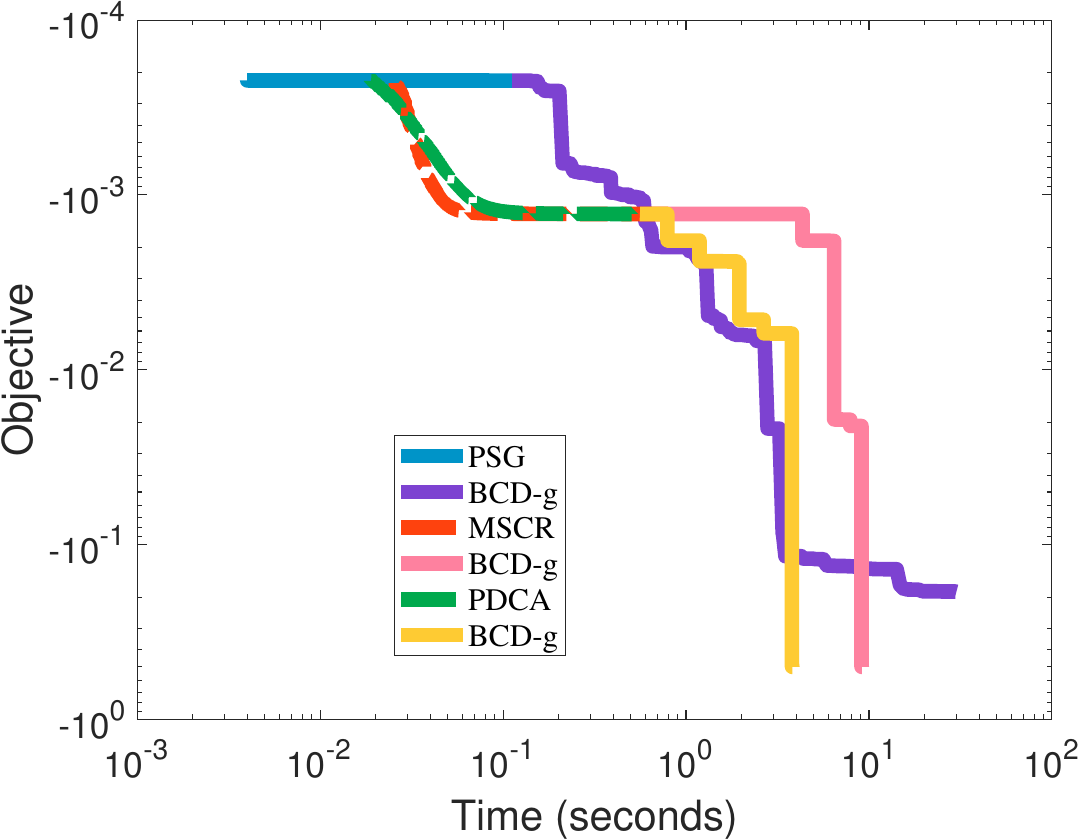}
			\caption{NNSPCA in 20News-b}
		\end{subfigure}
		\begin{subfigure}{0.2\textwidth}
			\includegraphics[width=\textwidth]{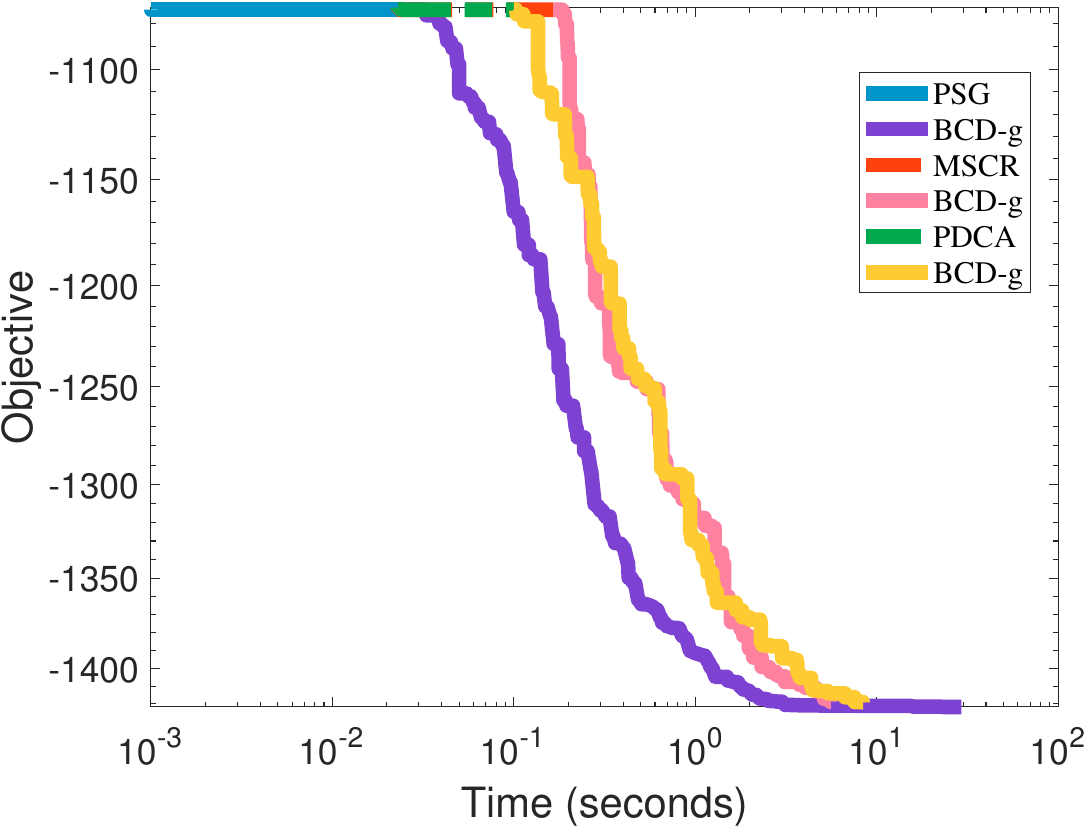}
			\caption{NNSPCA in Cifar-a}
		\end{subfigure}
		\begin{subfigure}{0.2\textwidth}
			\includegraphics[width=\textwidth]{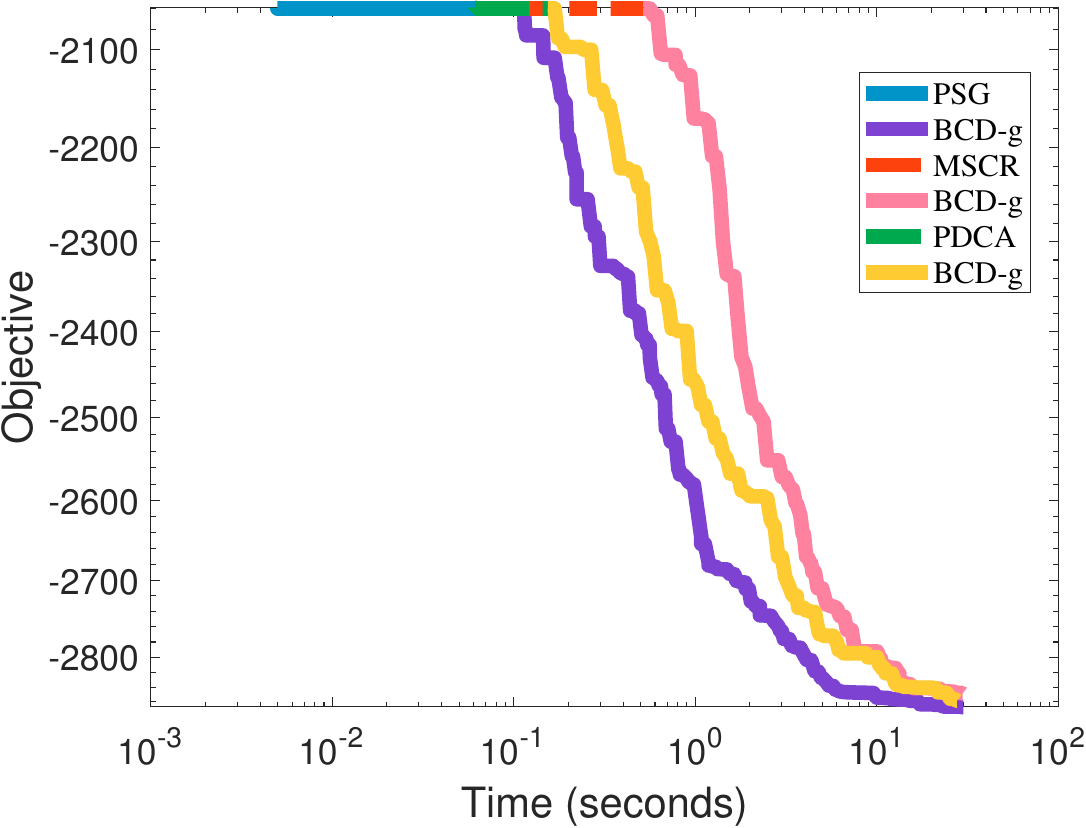}
			\caption{NNSPCA in Cifar-b}
		\end{subfigure}
		\begin{subfigure}{0.2\textwidth}
			\includegraphics[width=\textwidth]{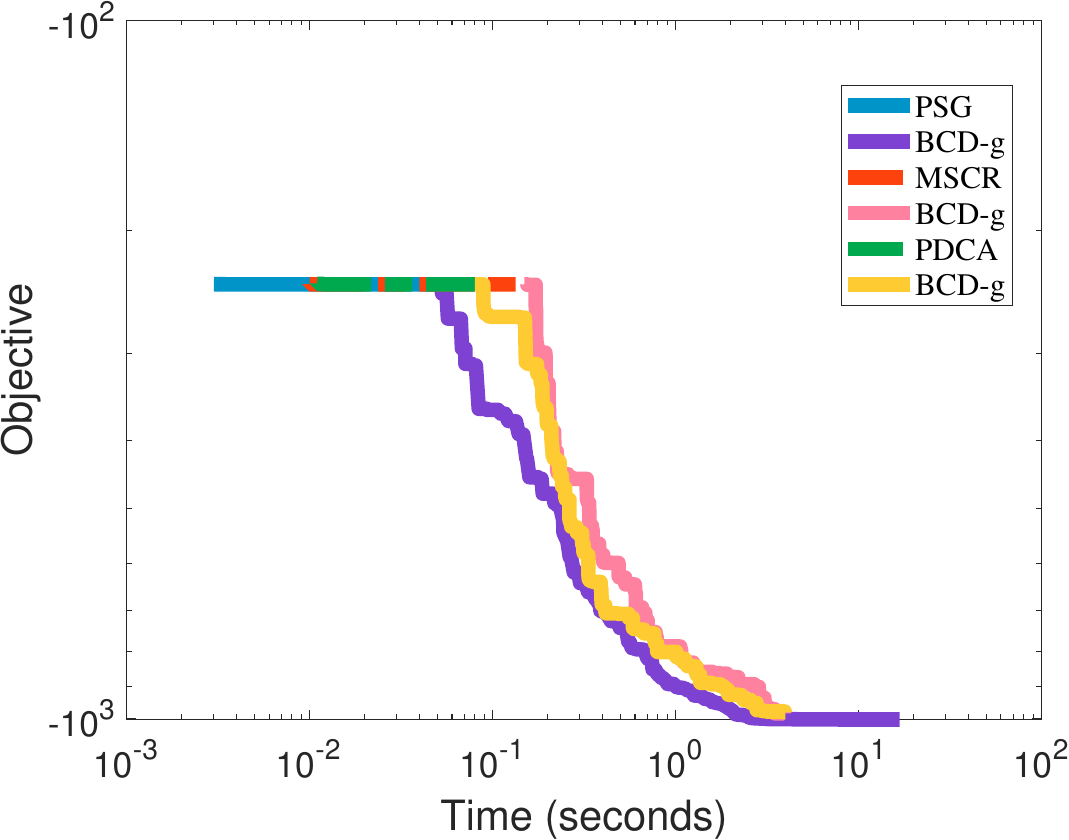}
			\caption{NNSPCA in MNIST-a}
		\end{subfigure}
		\begin{subfigure}{0.2\textwidth}
			\includegraphics[width=\textwidth]{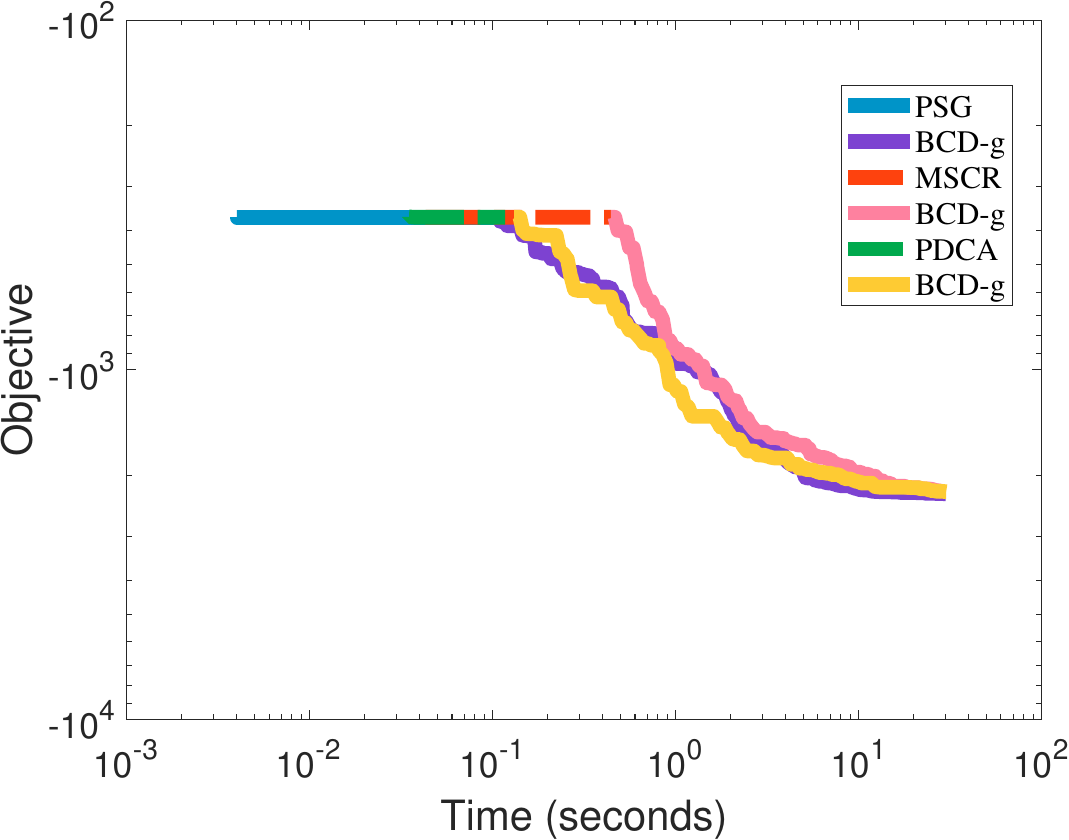}
			\caption{NNSPCA in MNIST-b}
		\end{subfigure}
		\caption{The convergence curve of the hybrid methods for sparse index tracking problem and non-negative sparse PCA problem on 20 different datasets.}
		\label{fig:compareindex}
	\end{figure}

	\captionsetup[subfigure]{font=tiny}
	\begin{figure}[htbp]
		\centering
		\begin{subfigure}{0.2\textwidth}
			\includegraphics[width=\textwidth]{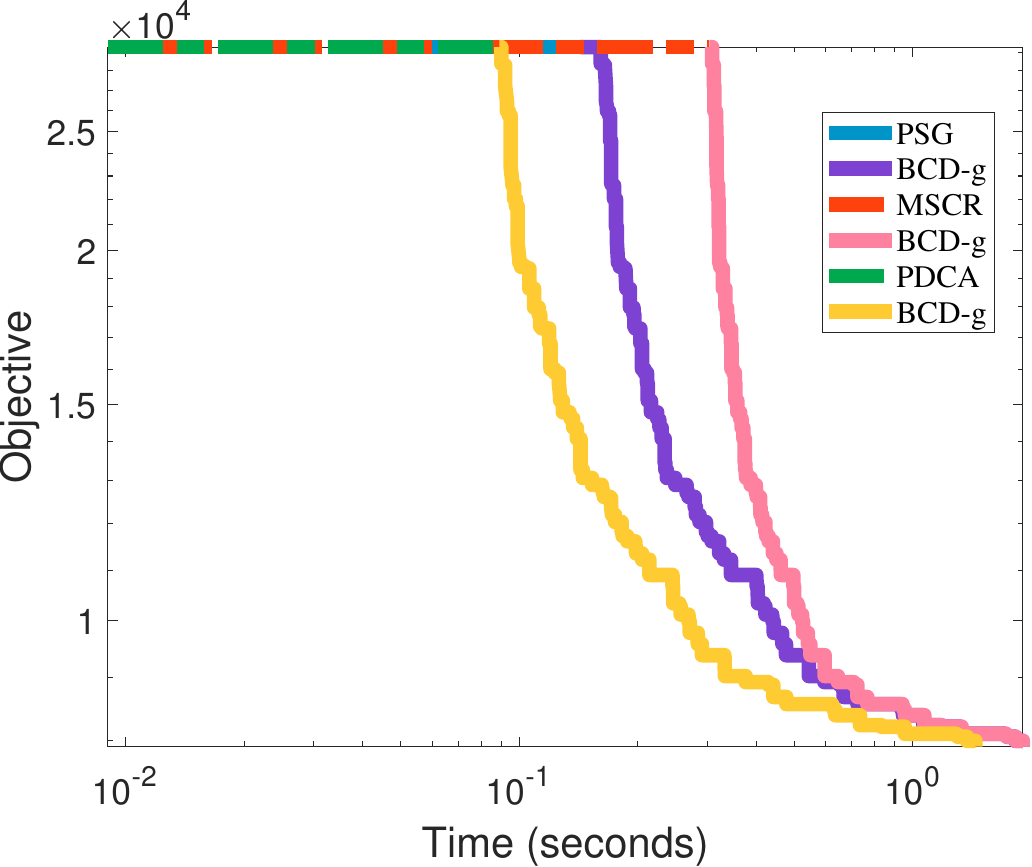}
			\caption{DCPB1 in randn-a}
		\end{subfigure}
		\begin{subfigure}{0.2\textwidth}
			\includegraphics[width=\textwidth]{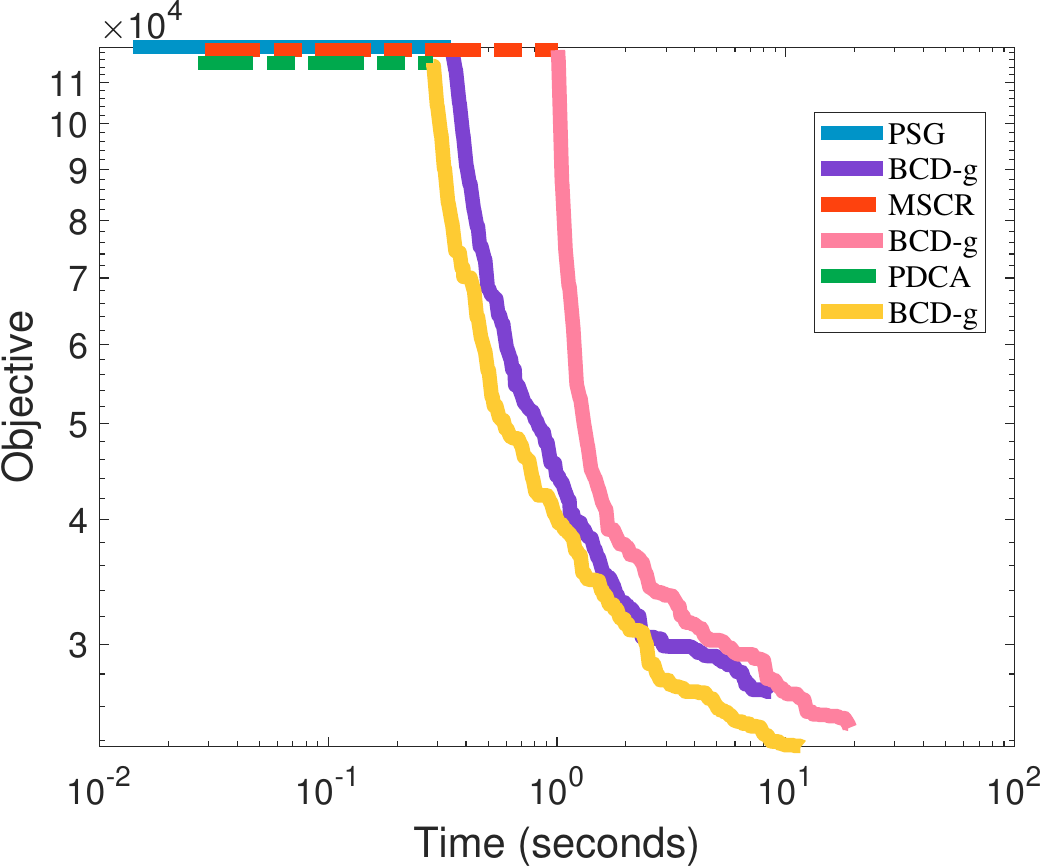}
			\caption{DCPB1 in randn-b}
		\end{subfigure}
		\begin{subfigure}{0.2\textwidth}
			\includegraphics[width=\textwidth]{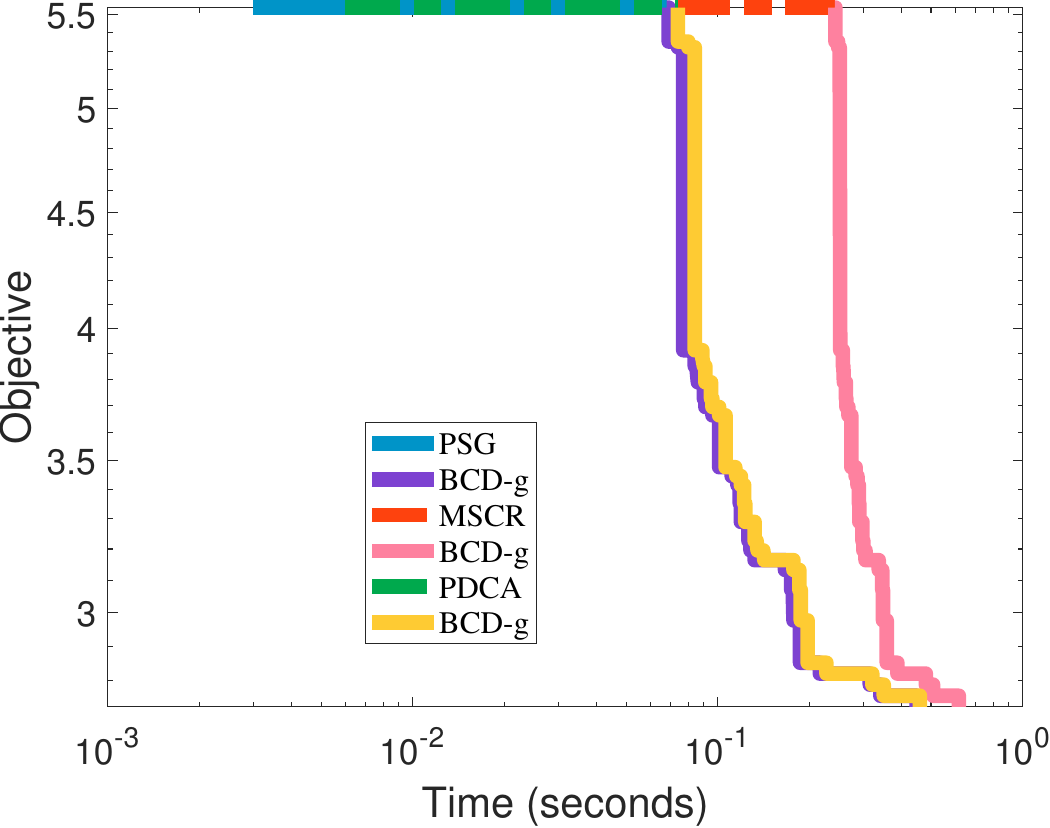}
			\caption{DCPB1 in TDT2-a}
		\end{subfigure}
		\begin{subfigure}{0.2\textwidth}
			\includegraphics[width=\textwidth]{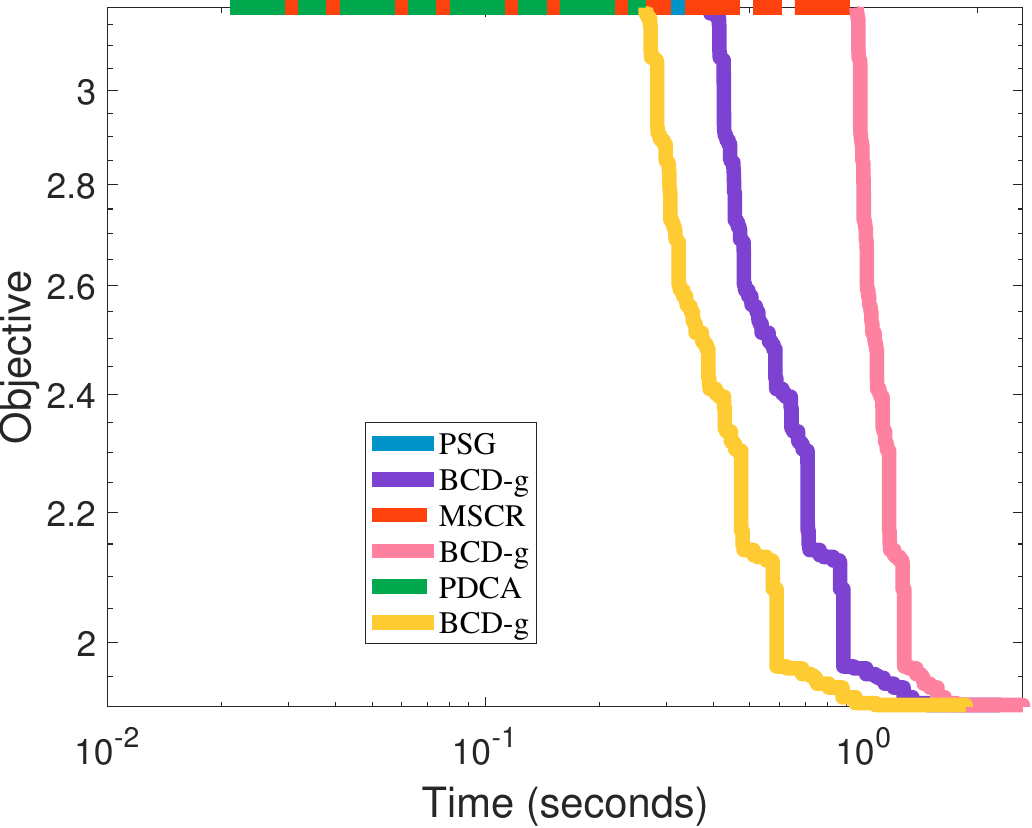}
			\caption{DCPB1 in TDT2-b}
		\end{subfigure}
		\begin{subfigure}{0.2\textwidth}
			\includegraphics[width=\textwidth]{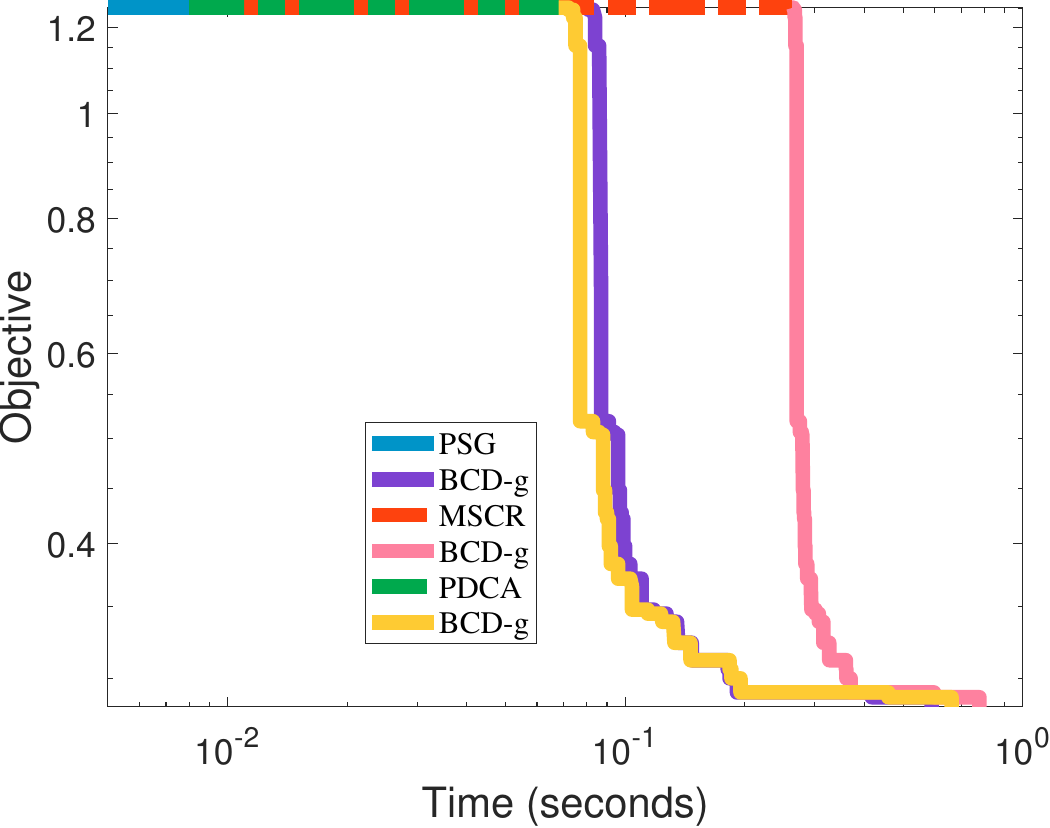}
			\caption{DCPB1 in 20News-a}
		\end{subfigure}
		\begin{subfigure}{0.2\textwidth}
			\includegraphics[width=\textwidth]{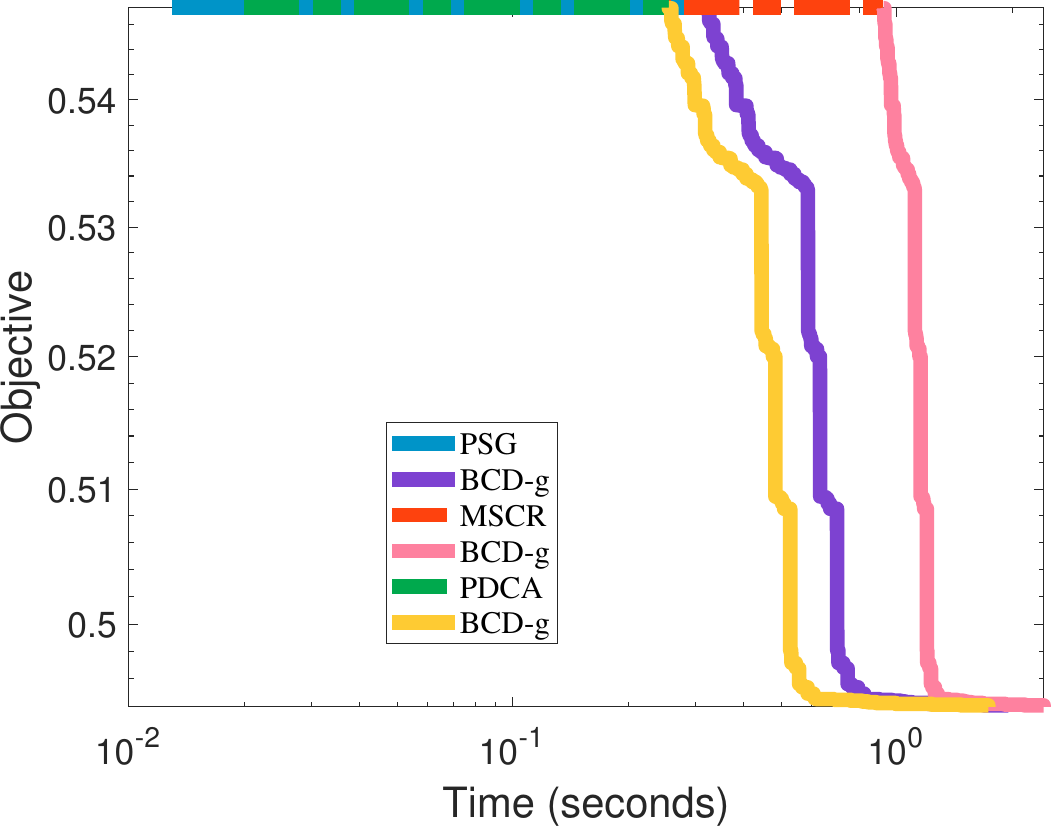}
			\caption{DCPB1 in 20News-b}
		\end{subfigure}
		\begin{subfigure}{0.2\textwidth}
			\includegraphics[width=\textwidth]{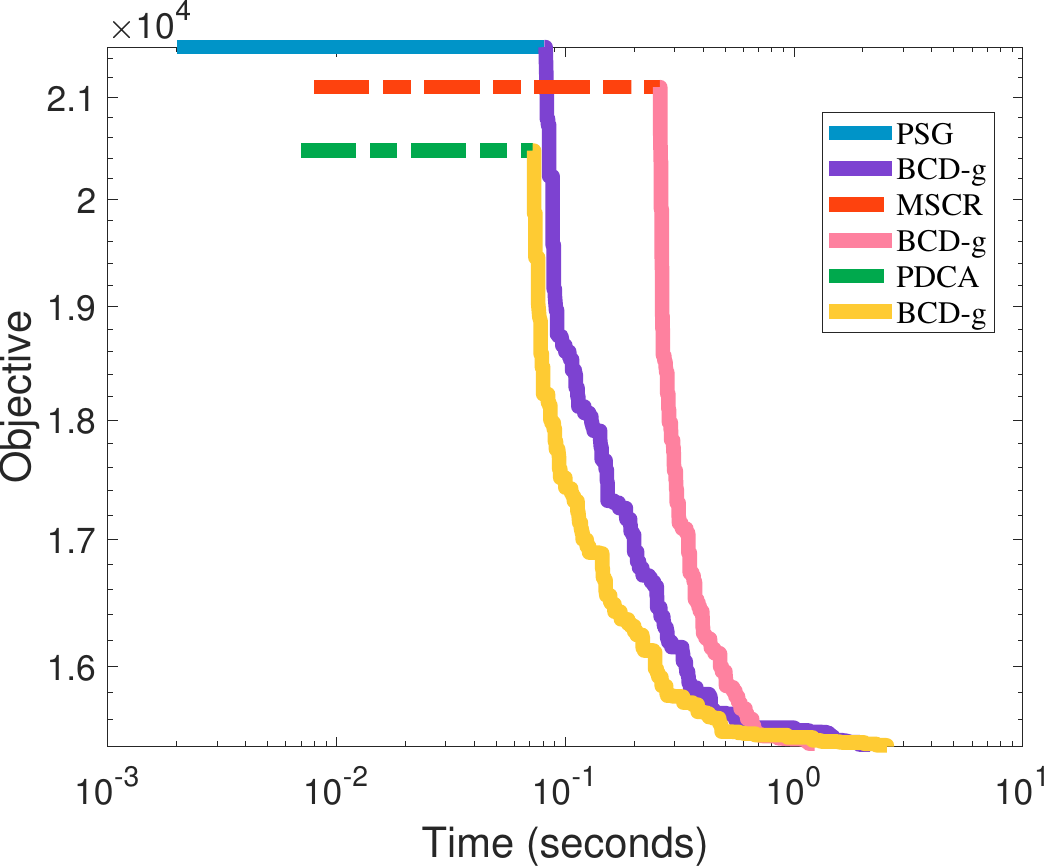}
			\caption{DCPB1 in Cifar-a}
		\end{subfigure}
		\begin{subfigure}{0.2\textwidth}
			\includegraphics[width=\textwidth]{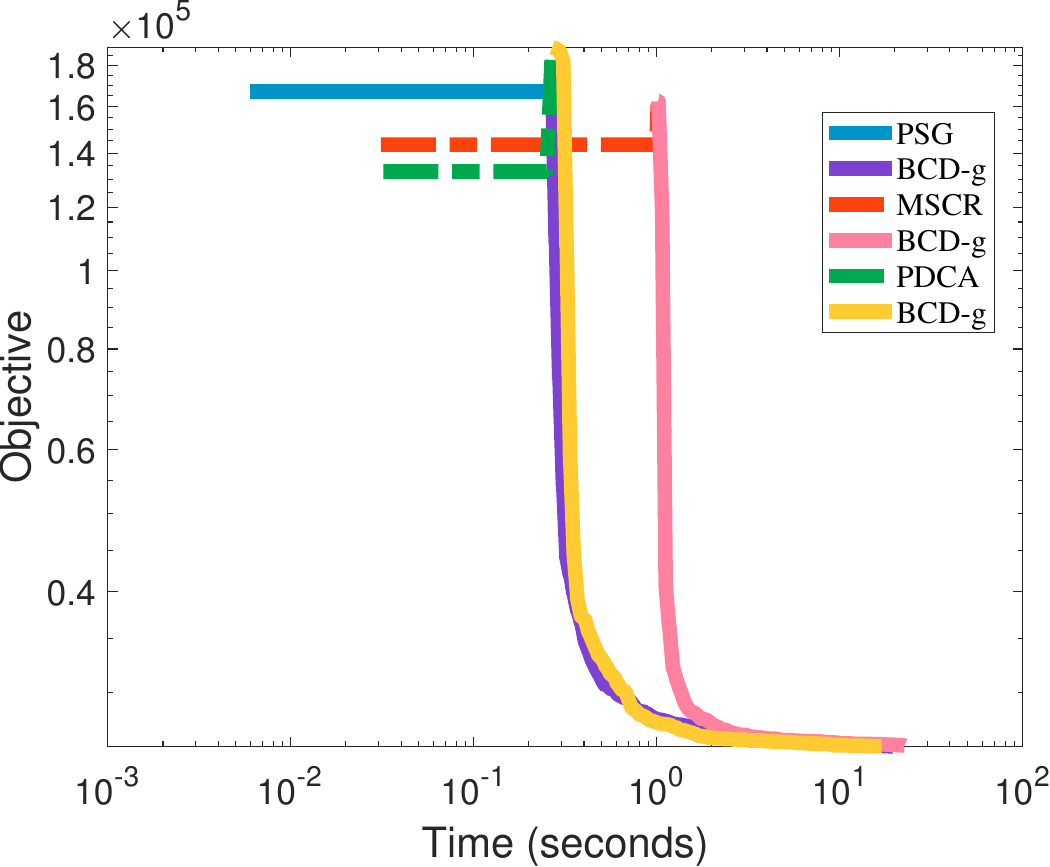}
			\caption{DCPB1 in Cifar-b}
		\end{subfigure}
		\begin{subfigure}{0.2\textwidth}
			\includegraphics[width=\textwidth]{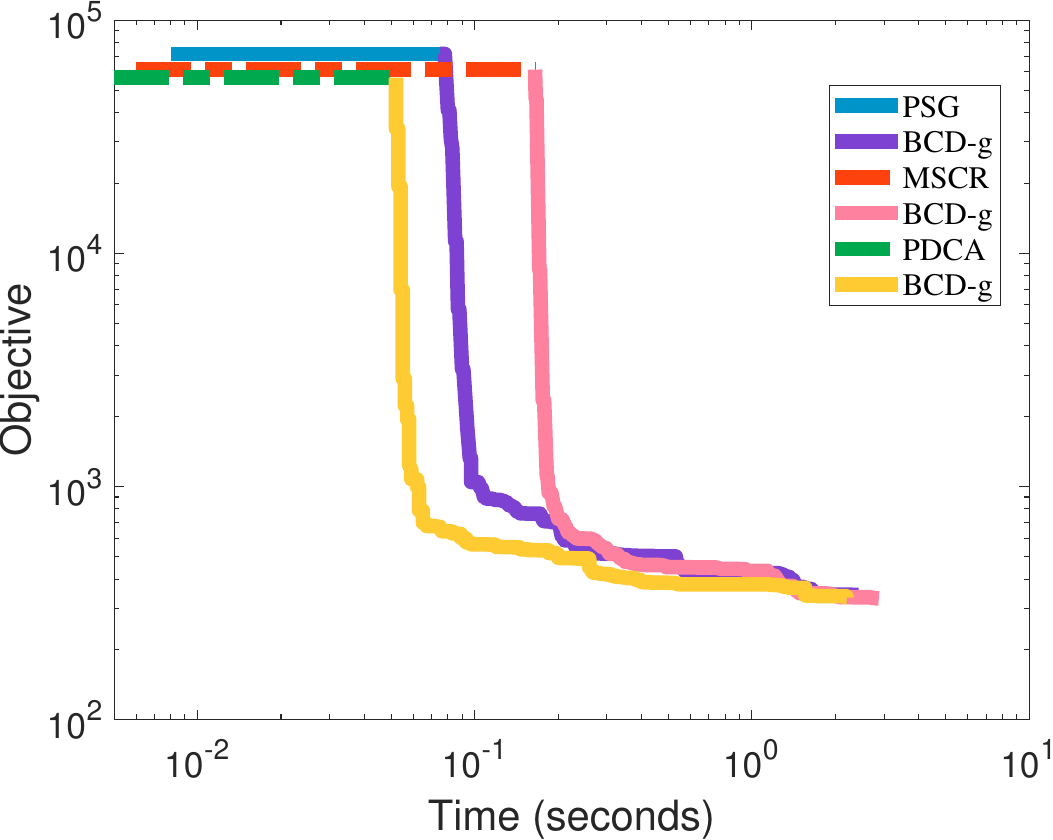}
			\caption{DCPB1 in MNIST-a}
		\end{subfigure}
		\begin{subfigure}{0.2\textwidth}
			\includegraphics[width=\textwidth]{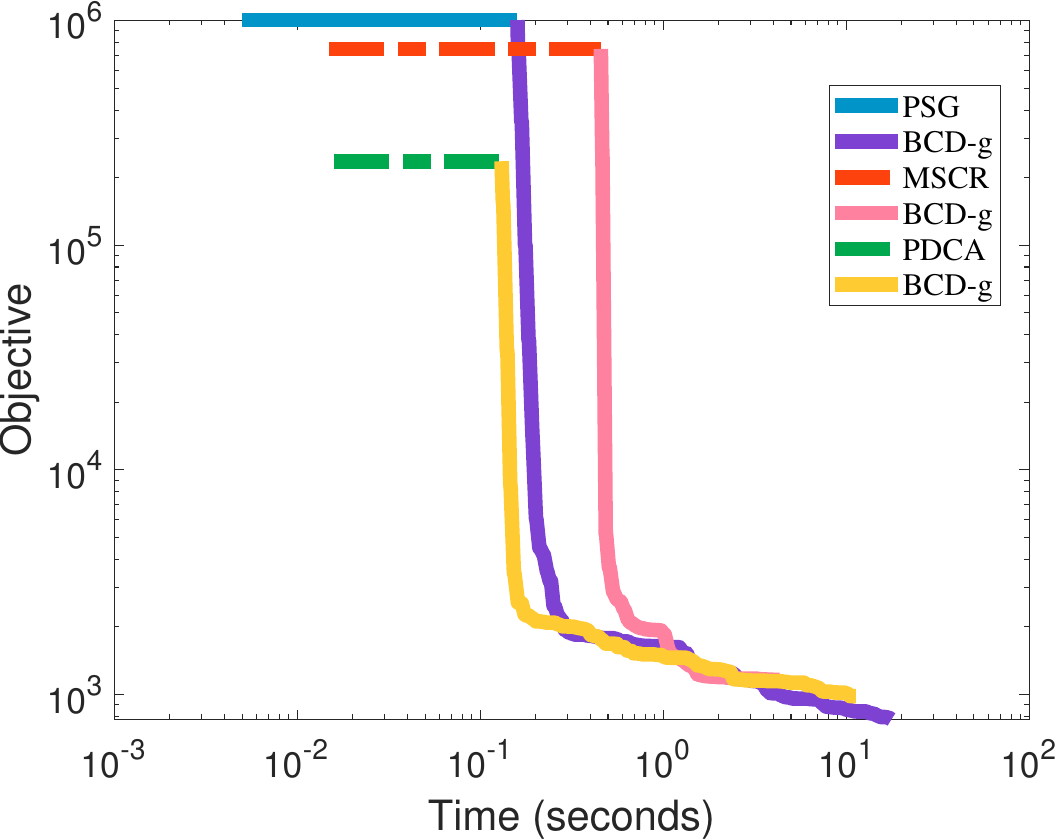}
			\caption{DCPB1 in MNIST-b}
		\end{subfigure}
		\begin{subfigure}{0.2\textwidth}
			\includegraphics[width=\textwidth]{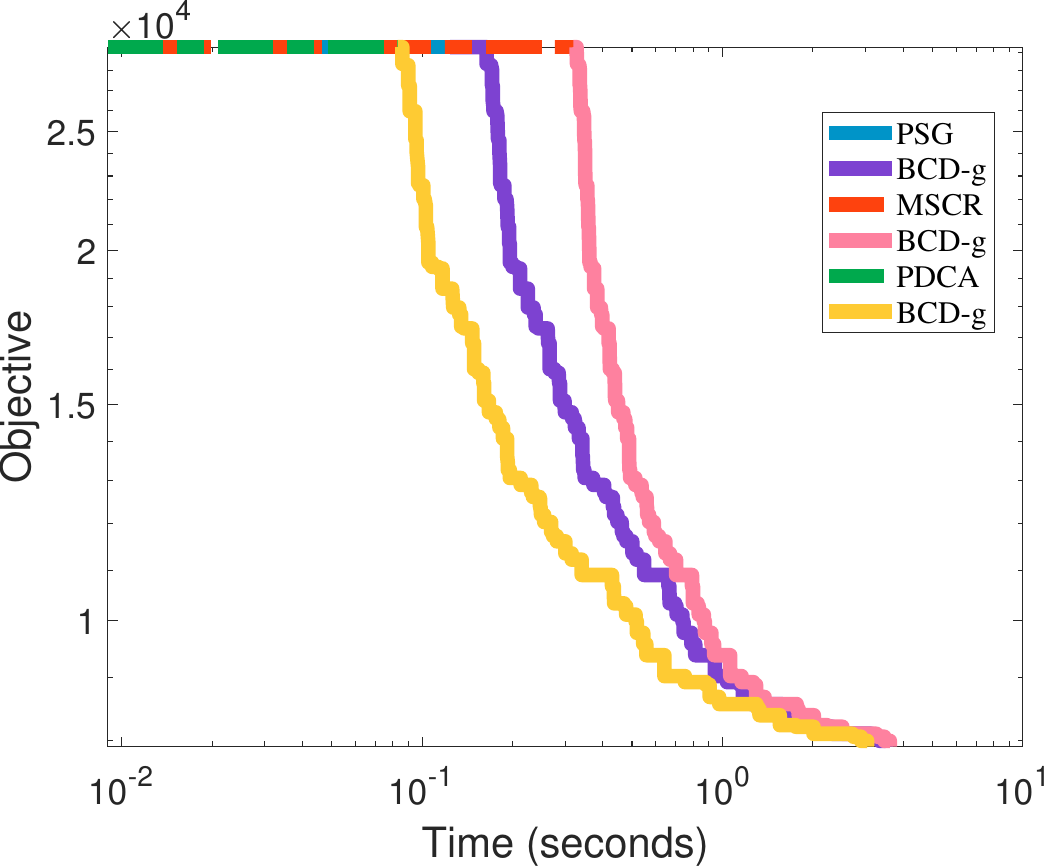}
			\caption{DCPB2 in randn-a}
		\end{subfigure}
		\begin{subfigure}{0.2\textwidth}
			\includegraphics[width=\textwidth]{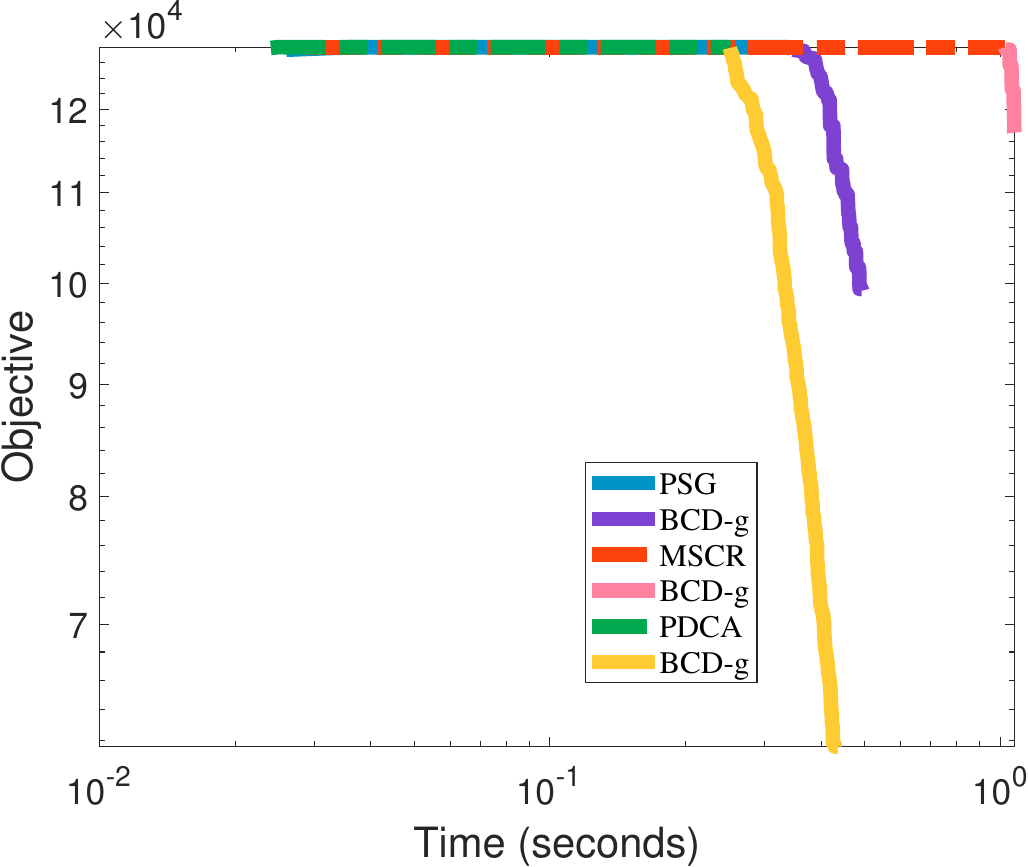}
			\caption{DCPB2 in randn-b}
		\end{subfigure}
		\begin{subfigure}{0.2\textwidth}
			\includegraphics[width=\textwidth]{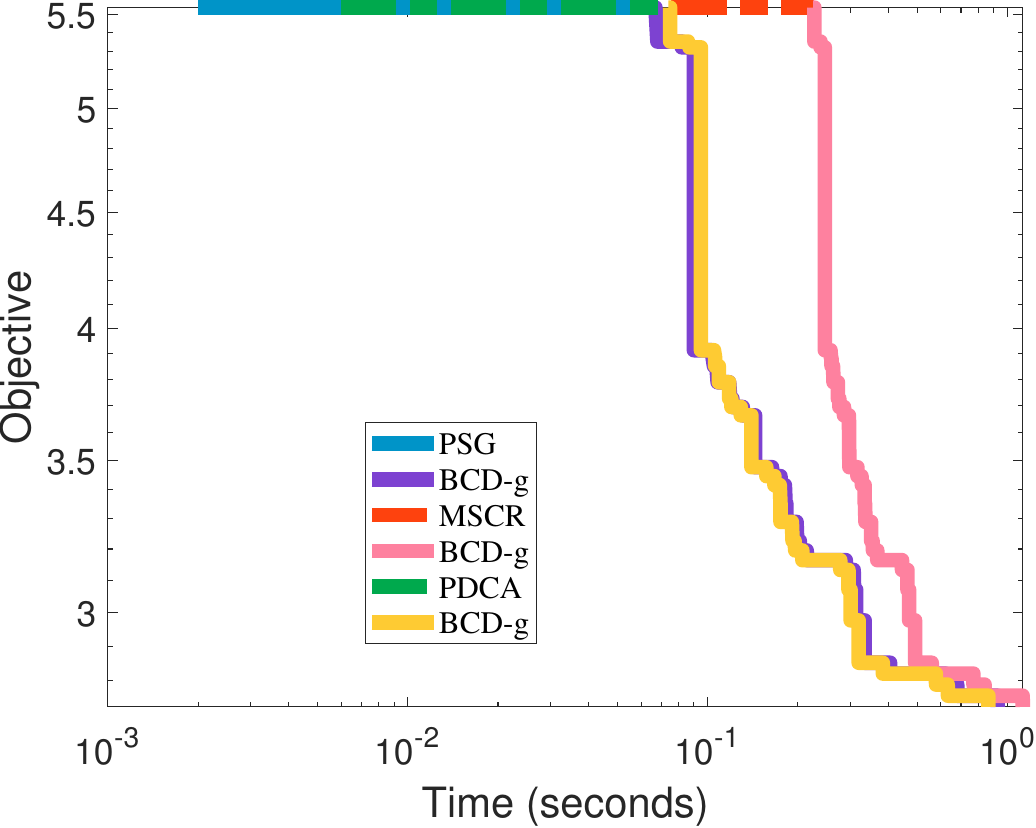}
			\caption{DCPB2 in TDT2-a}
		\end{subfigure}
		\begin{subfigure}{0.2\textwidth}
			\includegraphics[width=\textwidth]{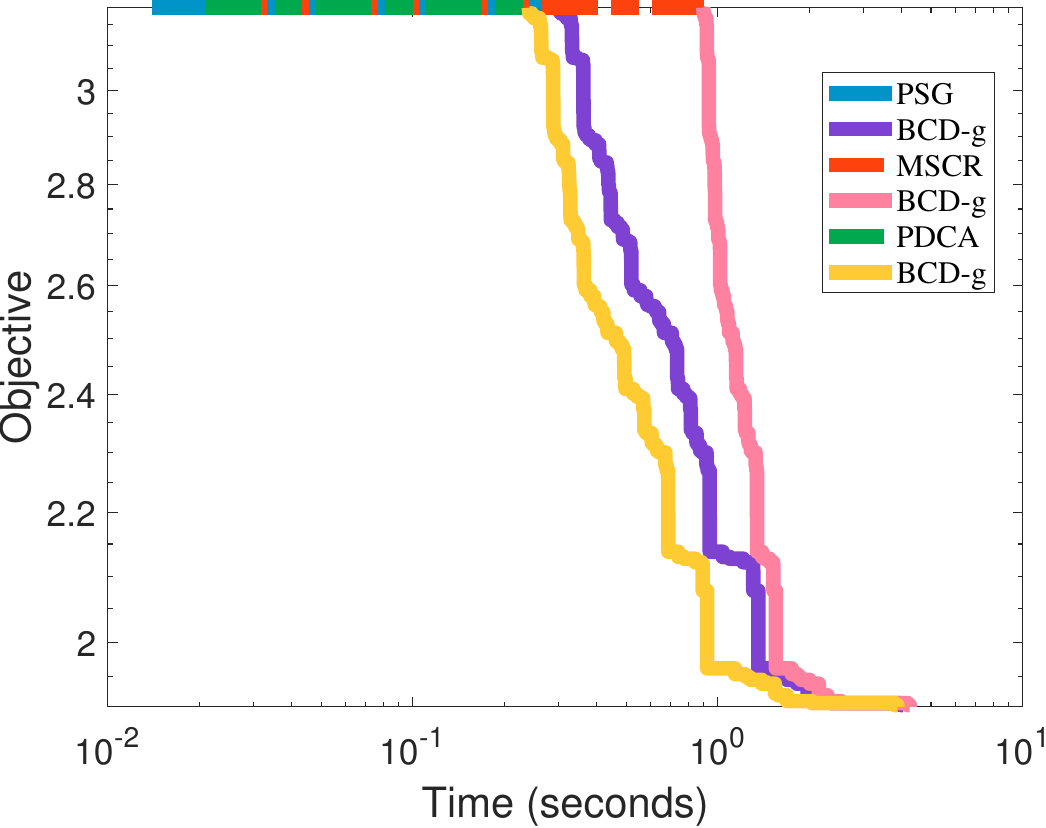}
			\caption{DCPB2 in TDT2-b}
		\end{subfigure}
		\begin{subfigure}{0.2\textwidth}
			\includegraphics[width=\textwidth]{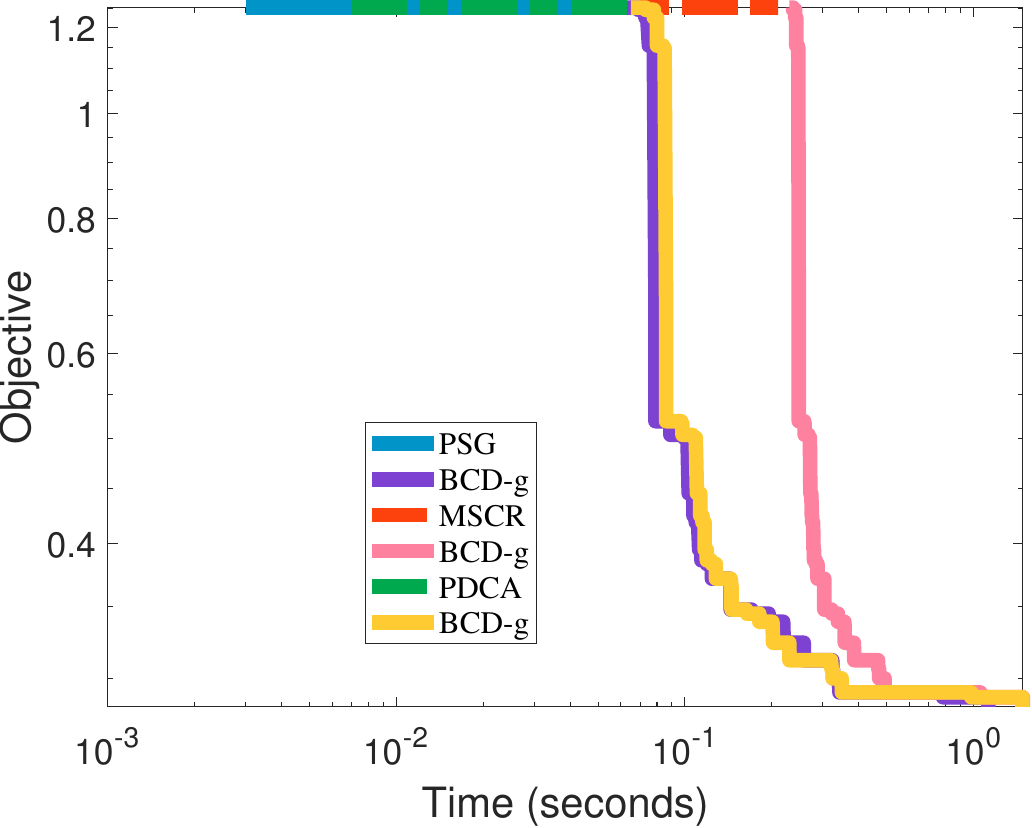}
			\caption{DCPB2 in 20News-a}
		\end{subfigure}
		\begin{subfigure}{0.2\textwidth}
			\includegraphics[width=\textwidth]{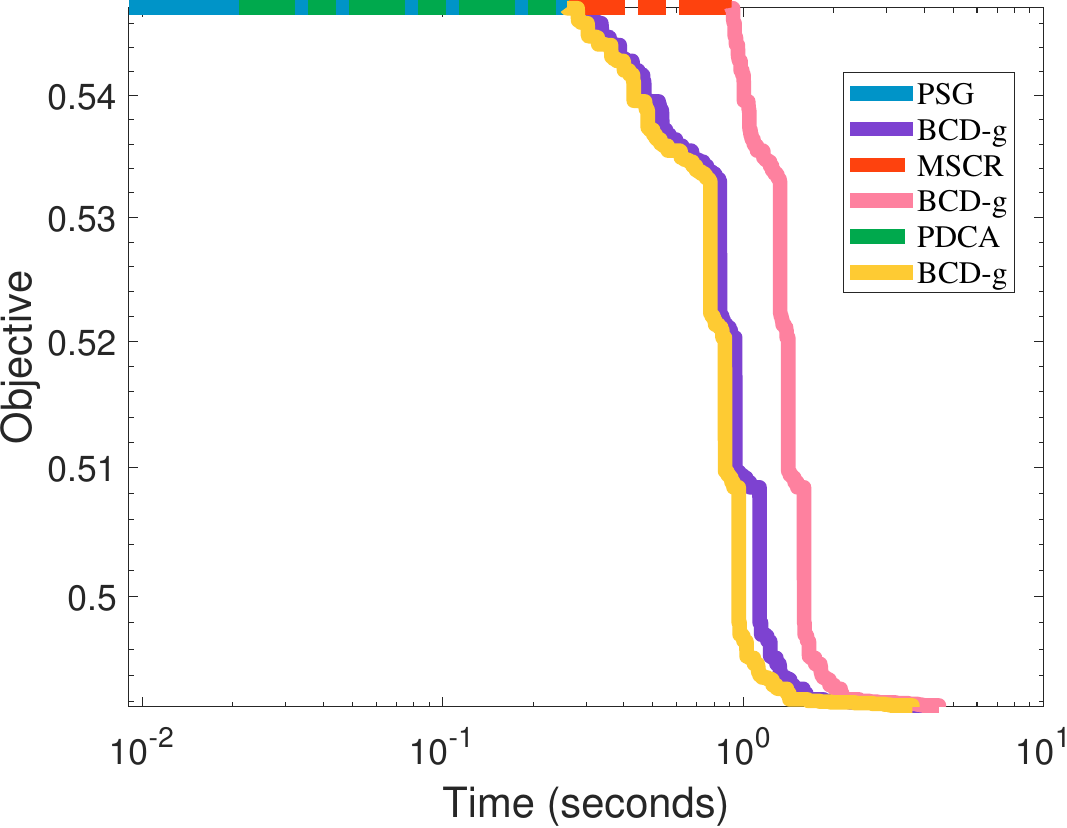}
			\caption{DCPB2 in 20News-b}
		\end{subfigure}
		\begin{subfigure}{0.2\textwidth}
			\includegraphics[width=\textwidth]{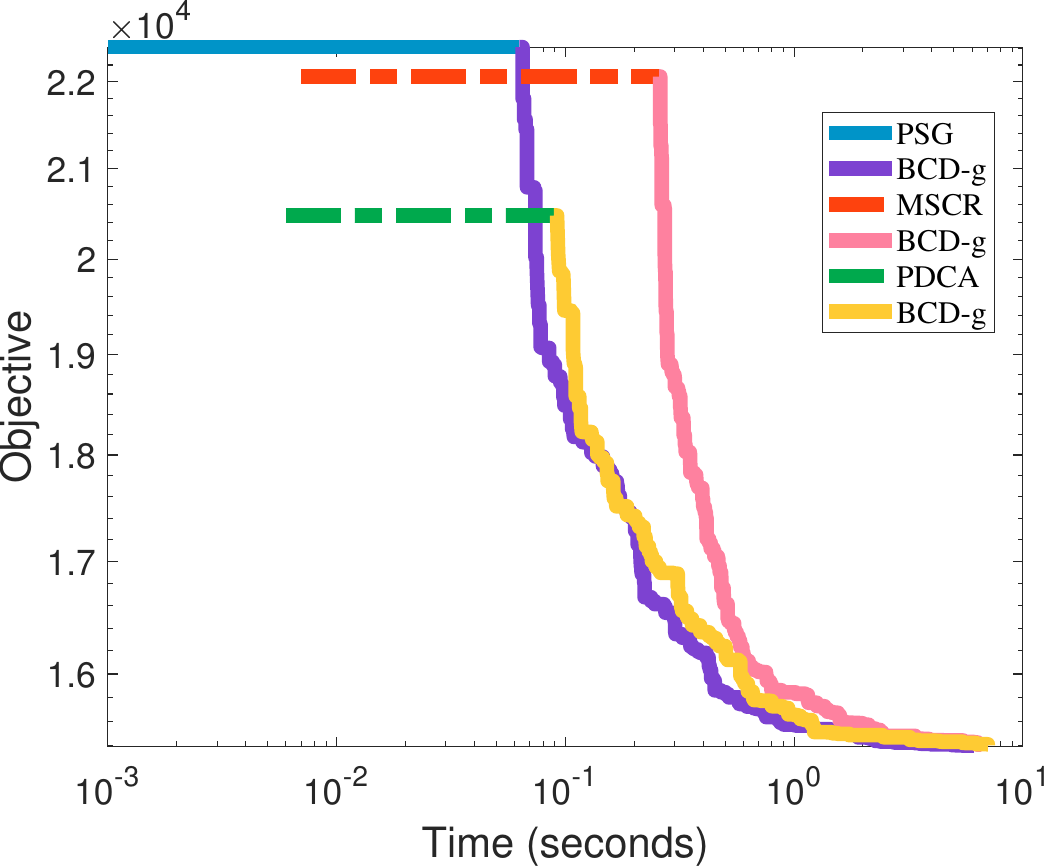}
			\caption{DCPB2 in Cifar-a}
		\end{subfigure}
		\begin{subfigure}{0.2\textwidth}
			\includegraphics[width=\textwidth]{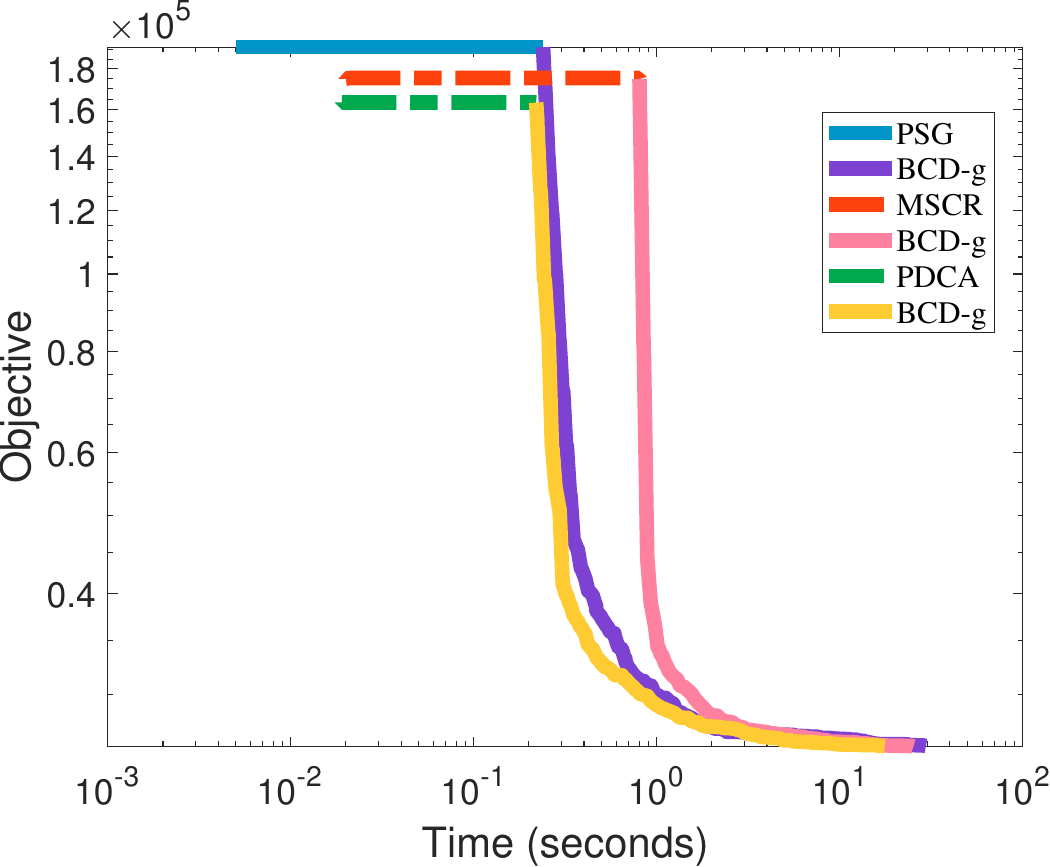}
			\caption{DCPB2 in Cifar-b}
		\end{subfigure}
		\begin{subfigure}{0.2\textwidth}
			\includegraphics[width=\textwidth]{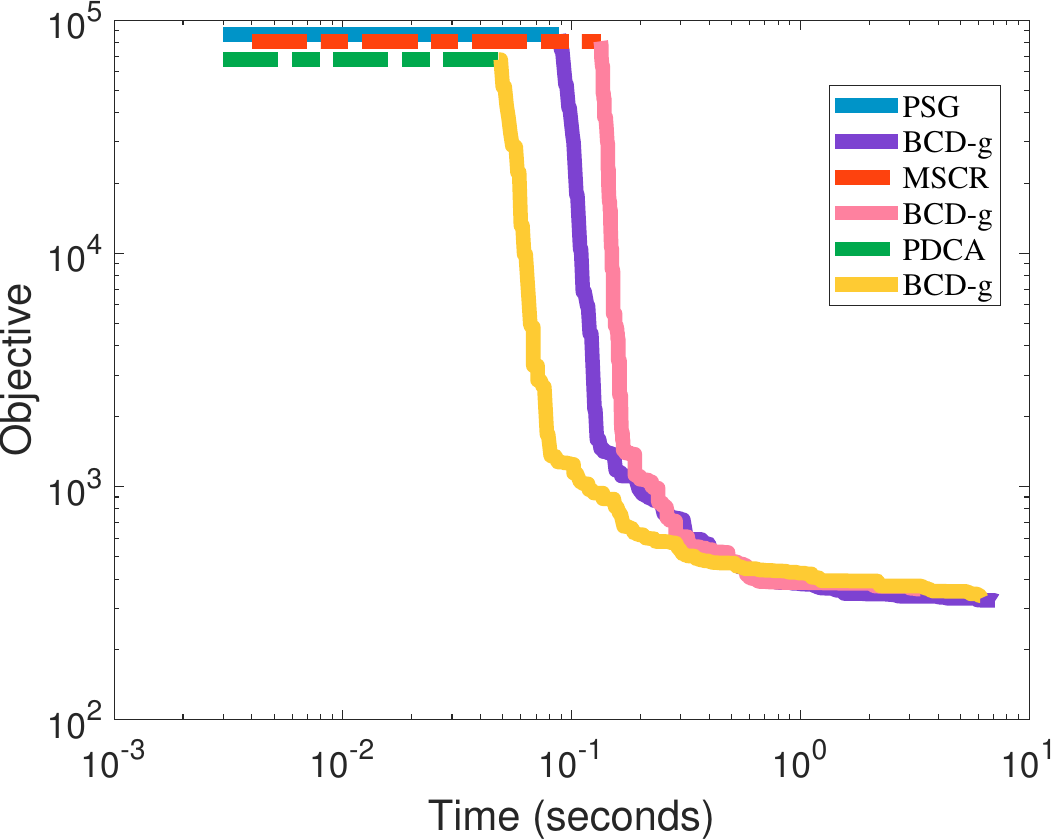}
			\caption{DCPB2 in MNIST-a}
		\end{subfigure}
		\begin{subfigure}{0.2\textwidth}
			\includegraphics[width=\textwidth]{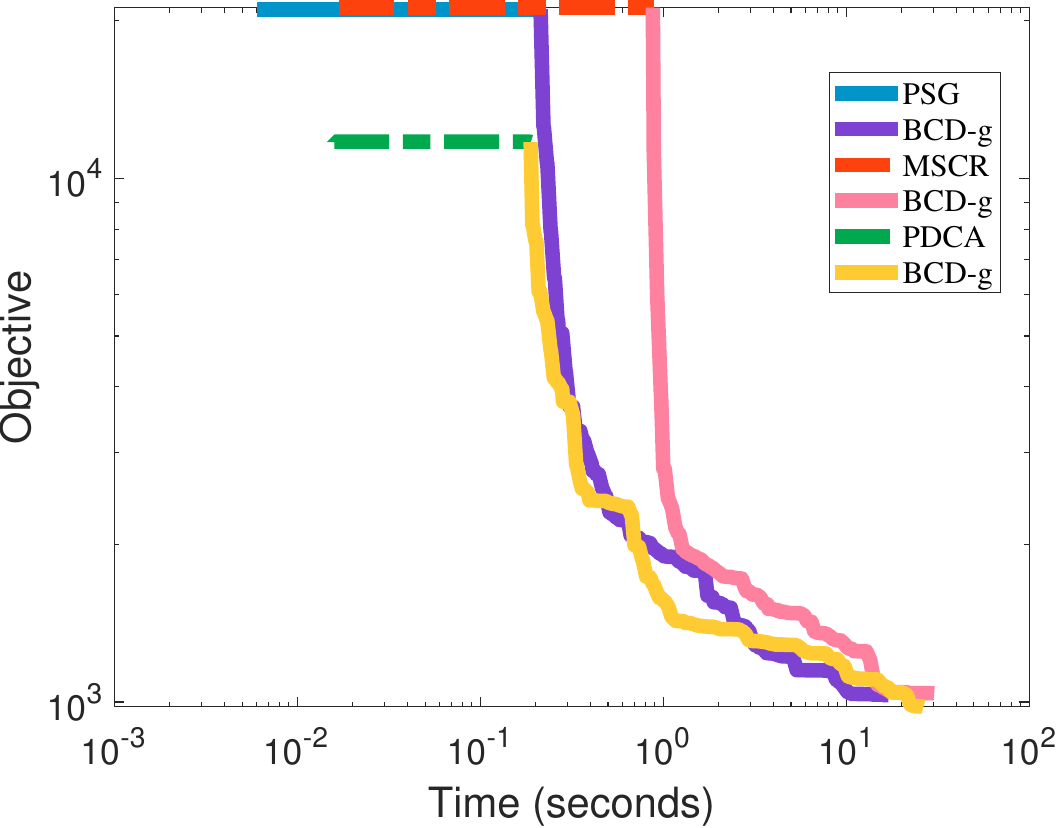}
			\caption{DCPB2 in MNIST-b}
		\end{subfigure}
		\caption{The convergence curves of the hybrid methods for two DC penalized binary optimization problems across 20 different datasets.}
		\label{fig:comparebin}
	\end{figure}

\end{document}